\documentclass{amsart}
\usepackage[utf8]{inputenc}
\usepackage{amsmath}
\usepackage{amsfonts}
\usepackage{amssymb}
\usepackage{amsthm}
\usepackage{amscd}
\usepackage{float}
\usepackage{mathabx}
\usepackage{tikz}
\usepackage{pgfplots}
\usetikzlibrary{shapes}
\usetikzlibrary{calc,shadings,fadings,backgrounds}
\usepackage{graphicx,caption}
\usepackage[colorlinks=true]{hyperref}
\hypersetup{urlcolor=blue, citecolor=red}
\usepackage{mathtools}

\mathtoolsset{showonlyrefs=true}

\newtheorem{theorem}{Theorem}[section]
\newtheorem{corollary}[theorem]{Corollary}

\newtheorem{lemma}[theorem]{Lemma}
\newtheorem{proposition}[theorem]{Proposition}

\newtheorem{assumption}[theorem]{Assumption}

\theoremstyle{definition}
\newtheorem{definition}[theorem]{Definition}
\newtheorem{remark}[theorem]{Remark}

\newcommand{\R}{\mathbb{R}}
\newcommand{\N}{\mathbb{N}}
\newcommand{\C}{\mathbb{C}}
\newcommand{\T}{\mathbb{T}}
\newcommand{\Z}{\mathbb{Z}}

\begin{document}

\title[Wave packet decompositions and bilinear estimates]{Wave packet decompositions and sharp bilinear estimates for rough Hamiltonian flows}

\author[Robert Schippa]{Robert Schippa*}
\address{University of Bonn, Mathematical Institute, Endenicher Allee 60, D-53115 Bonn, Germany}
\email{rschippa@uni-bonn.de}

\author{Daniel Tataru}
\address{UC Berkeley, Department of Mathematics, Evans Hall
Berkeley, CA 94720-3840, United States}
\email{tataru@berkeley.edu}


\begin{abstract}
The goal of this paper is to prove bilinear $L^p$ estimates for rough dispersive evolutions satisfying non-degeneracy and transversality assumptions. The estimates generalize the sharp Fourier extension estimates for the cone and the paraboloid. We obtain the sharp dependence on transversality and sharp up to endpoints $L^p$ range. To this end, we combine an axiomatic induction-on-scales with the construction of a wave packet parametrix for dispersive equations with $C^{1,1}$-coefficients by using the FBI transform.
\end{abstract}
\thanks{*Corresponding author}

\maketitle

\setcounter{tocdepth}{1}
\tableofcontents

\section{Introduction}

\subsection{Bilinear Fourier extension estimates}
The purpose of this article is to explore bilinear estimates for solutions to dispersive PDE's with rough coefficients under transversality assumptions. These generalize the classical bilinear adjoint Fourier restriction estimates, which correspond to the constant coefficient case, and which we briefly describe now.

The first (sharp up to endpoints) results are due to Wolff \cite{Wolff01} (see \cite{Bourgain1995,TaoVargasVega1998} for earlier results). Formulated for the cone extension operator,
\begin{equation*}
\mathcal{E}_c f(x,t) = \int_{\xi \in \R^d: 1/2 \leq |\xi| \leq 1} e^{i(x \cdot \xi + t |\xi|)} f(\xi) d\xi, \quad (x,t) \in \R^d \times \R,
\end{equation*}
Wolff's bilinear estimate reads
\begin{equation}
\label{eq:BilinearEstimateWave}
\| \mathcal{E}_c f_1 \mathcal{E}_c f_2 \|_{L^p_{t,x}(B_{d+1}(0,R))} \lesssim_\varepsilon R^\varepsilon \| f_1 \|_2 \| f_2 \|_2 
\end{equation}
for $p = \frac{d+3}{d+1}$ with the angular separation condition
\begin{equation}
\label{eq:AngularSeparationIntro}
\angle ( \text{supp}(f_1), \text{supp}(f_2)) \geq c > 0.
\end{equation}
The advent of sharp bilinear estimates had a transformative effect on Fourier restriction theory, whose study was initiated by Stein in the 1960s.  In particular, the bilinear estimates solve the linear cone restriction problem for $d=3$ (\cite{Wolff01}). For $d=4$ Ou--Wang \cite{OuWang2022} proved the sharp estimates via polynomial partitioning.

\smallskip

 The next results were due to Tao \cite{Tao03}, who  showed the corresponding sharp estimates for the paraboloid, for which the extension operator is given by
\begin{equation*}
\mathcal{E}_{\Delta} f(x,t) = \int_{\xi \in \R^d: |\xi| \leq 1} e^{i(x \cdot \xi + t |\xi|^2)} f(\xi) d\xi.
\end{equation*}
The corresponding bilinear estimate reads
\begin{equation*}
\| \mathcal{E}_{\Delta} f_1 \mathcal{E}_{\Delta} f_2 \|_{L^p_{t,x}(B_{d+1}(0,R))} \lesssim_\varepsilon R^{\varepsilon} \| f_1 \|_2 \| f_2 \|_2
\end{equation*}
for $p=\frac{d+3}{d+1}$. Here, due to non-vanishing curvature of the paraboloid, the transversality condition reads
\begin{equation*}
\text{dist}( \text{supp}(f_1), \text{supp}(f_2)) \geq c > 0.
\end{equation*}
Henceforth, the bilinear estimates have been widely generalized. We cite works due to Tao \cite{Tao2001} who proved the estimate \eqref{eq:BilinearEstimateWave} at the endpoint $p=\frac{d+3}{d+1}$ with $\varepsilon = 0$, Vargas \cite{Vargas2005} who proved a bilinear restriction estimate for the hyperboloid, and S. Lee \cite{Lee2006DifferentSignsCurvature,Lee2006}, who generalized the bilinear estimate to possibly non-elliptic surfaces and oscillatory integral operators. Possibly degenerate surfaces were considered by Lee--Vargas in \cite{LeeVargas2010}; Buschenhenke--M\"uller--Vargas considered surfaces of finite type \cite{BuschenhenkeMuellerVargas2017}. Bejenaru \cite{Bejenaru2017} provided a unified description in terms of the shape operator.

\smallskip

Since they are $L^2$-based, the Fourier extension estimates  can be interpreted as bilinear estimates for solutions to PDE's. Variants and extensions with precise dependence on the transversality and regularity of the phase have been worked out more recently by Candy \cite{Candy2019};
see \cite{CandyHerr2018,CandyHerr2018II,CandyHerr2018III} for more applications.

\smallskip

Regarding the cone extension operator, the bilinear Fourier extension estimates can be formulated for
solutions to the half-wave equation:
\begin{equation}
\label{eq:BilinearEstimateEuclideanHalfwaves}
\| e^{it \sqrt{-\Delta}} f_1 e^{it \sqrt{-\Delta}} f_2 \|_{L^p_{t,x}(B_{d+1}(0,R))} \lesssim_\varepsilon R^\varepsilon \| f_1 \|_2 \| f_2 \|_2
\end{equation}
for $p=\frac{d+3}{d+1}$ and $\text{supp}(\hat{f}_i)$ contained in the unit annulus, satisfying the angular separation condition for the Fourier transforms.

In view of applications, we can also consider solutions to distinct wave equations with quantitatively separated group velocity $|\nu_1 - \nu_2 | \geq c > 0$, $\nu_i \in (1,2)$; see below. This can replace the assumption on the angular support separation. 

\subsection{Wave packet decompositions}

A fundamental tool in all modern developments of Fourier restriction theory is the \emph{wave packet decomposition}. This provides a decomposition of the Fourier extension $\mathcal{E}_{\Delta} f$ into non-oscillating components, the \emph{wave packets}. This goes back to pioneering work of Fefferman \cite{Fefferman1973} and C\'ordoba \cite{Cordoba1979} with subsequent important contributions by Bourgain \cite{Bourgain1991I,Bourgain1991II,Bourgain1995}. Regarding applications to PDE's we want to refer to another pioneering work of Fefferman \cite{Fefferman1983}.

\smallskip

The construction becomes significantly more involved for variable rough coefficients. H. Smith \cite{Smith1998} constructed a wave packet decomposition (or \emph{frame}) for wave equations with $C^{1,1}$-coefficients, using a paradifferential decomposition and propagating the wave packets along the Hamiltonian flow. Smith's wave packets are anisotropic and of maximal size in frequency, taking advantage of the flatness in the radial direction of the characteristic surface of the wave equation. Similar anisotropic wave packets were used by Smith--Tataru \cite{SmithTataru2005} to prove sharp local well-posedness of quasilinear wave equations. Geba--Tataru \cite{GebaTataru2007} refined Smith's parametrix construction to metrics with minimal averaged space-time regularity $\| \partial^2 g \|_{L_t^1 L_{x}^{\infty}} < \infty$ without paradifferential decomposition. 

\smallskip

For non-degenerate characteristic surfaces, which is the case e.g. for Schr\"odinger equations, isotropic wave packets are distinguished. 
Ralston \cite{Ralston1982} constructed isotro\-pic wave packets referred to as \emph{Gaussian beams}. The beams are obtained from propagating Gaussians along the Hamiltonian flow. Waters \cite{Waters2011} constructed a parametrix for wave equations with $C^{1,1}$-coefficients without paradifferential decomposition and obtained the corresponding space and frequency localization of the wave packets; see also \cite{Waters2019}.
Here we carry out a unified approach and construct wave packet parametrices for symbols with size and regularity assumptions, which naturally arise from paradifferential decompositions for low regularity coefficients. In the case of wave equations the regularity corresponds to $C^{1,1}$-coefficients. The symbol regularity additionally allows us to prove time-frequency localization on short times, which is crucial for the proof of the bilinear estimates.

\smallskip

The wave packets are obtained from a maximal simultaneous decomposition in position and frequency space. On a cube $|(x,t)| \lesssim R$, this amounts to an initial spatial decomposition on the scale $\Delta x \sim R^{\frac{1}{2}}$ and the frequency scale $\Delta \xi \sim R^{-\frac{1}{2}}$. This decomposition works just as well for $\mathcal{E}_{c} f$, for which we note that due to the flatness of the cone in the radial direction, there are many wave packets with frequency support in different small balls pointing in the same direction. As noted above, this overlap can be avoided by 
unbalancing the localization scales \cite{Smith1998,Wolff01,SmithTataru2005,GebaTataru2007}. We do not pursue this idea here as the ultimate aim is to show $L^2$-based estimates for which we can sum the wave packets traveling in the same direction by almost orthogonality.

For the Fourier extension operator of a graph surface $S = \{(\xi,\phi(\xi)) : \xi \in A \subseteq B_d(0,2) \}$ the frequencies remain fixed under time evolution; the position of the wave packet, which is initially centered at $x_0$, changes according to the group velocity: $x(t) = x_0 - t \nabla \phi(\xi_0)$. Finally, the third and crucial ingredient for the combinatorial argument to prove the bilinear estimates is to take into account the time-frequency localization $\tau = \phi(\xi_0)$.

It is the transverse intersection, which significantly reduces the overlap of wave packets. This allows one to go beyond the range of the sharp linear Tomas--Stein estimate, which reads $\| \mathcal{E}_{\Delta} f \|_{L_{t,x}^\frac{2(d+2)}{d}(\R^{d+1})} \lesssim \| f \|_2$.

\subsection{Linear and bilinear estimates for solutions to PDEs with rough coefficients}

In this paper we look into bilinear estimates for solutions to wave and Schr\"odinger equations governed by rough coefficients. The question of bilinear estimates for solutions to wave equations with rough coefficients generalizing Wolff's estimates was raised by Klainerman--Rodnianski--Tao \cite[Section~8]{KlainermanRodnianskiTao2002}. Null form estimates for rough wave equations were previously explored in~\cite{Tataru03} (see also \cite{AiIfrimTataru2024} for recent developments concerning null forms for quasilinear wave equations). The question of generalizations of Wolff's and Tao's estimates is presently answered via wave packet decompositions and bilinear estimates for rough Hamiltonian flows. We begin by formulating an example of our bilinear estimates. 

\medskip 

Consider two frequency-localized solutions to the system of wave equations:
\begin{equation}
\label{eq:TransverseSystem}
\left\{ \begin{array}{cl}
\partial_t^2 u_{1} &= \partial_i g^{ij}_{1, \leq cN} \partial_j u_{1} + f_1, \\
\partial_t^2 u_{2} &= \partial_i g^{ij}_{2, \leq cN} \partial_j u_{2} + f_2, \quad 
\end{array} \right.
\end{equation}
for $(t,x) \in [0,1] \times \R^d$,
with $f_i \in L_t^1([0,1];L_x^2)$. We assume that $g^{ij}_k \in C^{1,s}(\R^{d+1})$, $k=1,2$, $s \in [0,1]$ and that the metrics satisfy the uniform ellipticity condition
\begin{equation}
\label{eq:UniformEllipticity}
\exists \lambda, \Lambda > 0: \forall (x,t) \in \R^{d+1}: \forall \xi \in \R^d \backslash 0: \, \lambda |\xi|^2 \leq g_k^{ij}(x,t) \xi_i \xi_j \leq \Lambda |\xi|^2, \quad k \in \{1,2\}.
\end{equation}
Further, in order to guarantee the transverse propagation of the two waves we require the speed separation condition
\begin{equation*}
| g^{ij}_k - \nu_k \delta_{ij}| \leq \epsilon^*,
\end{equation*}
with $\nu_i \in [1/2,2]$, and $|\nu_1 - \nu_2| \geq 1/4$ and a sufficiently small dimensional constant $\epsilon^* \ll 1$.  Above $N \geq N_0(\lambda,\Lambda)$ denotes the dyadic frequency localization scale and $c = c(\lambda,\Lambda)$. Precisely, we assume that
\begin{equation}
\label{eq:FrequencyDirection}
\text{supp}(\hat{u}_i) \subseteq \{ \xi \in \R^d : |\xi| \sim N, \big| \xi / |\xi| - \xi_{*} \big| \leq \epsilon^* \},
\end{equation}
for some $\xi_{*} \in \mathbb{S}^{d-1}$.
\eqref{eq:TransverseSystem} is typically derived from an untruncated system by commutator arguments, and we refer to the proof of Theorem \ref{thm:EndpointStrichartzWave} for details.
In summary, here we consider two approximate solutions to wave equations with quantitatively transverse propagation and show the following:

\begin{theorem}[Bilinear~estimates~for~rough~wave~equations]
\label{thm:RoughWaveEquationsIntro}
Let $u_i$ be solutions to \eqref{eq:TransverseSystem} as above. Then the following bilinear estimate holds with $p = \frac{d+3}{d+1}$
\begin{equation*}
N^2 \| u_{1} u_{2} \|_{L^p_{t,x}([0,1] \times \R^d)} \lesssim_\varepsilon N^{\frac{4}{3+s} \frac{d-1}{d+3} + \varepsilon} \prod_{i=1}^2 ( \| \nabla_{t,x} u_{i}(0) \|_{L^2} + N \| f_i \|_{L^1 L^2}).
\end{equation*}
\end{theorem}

For comparison we record the following estimate for Euclidean wave equations with transversely intersecting cones: Let $u_1, u_2 : \R^{1+d} \to \C$ denote solutions to constant-coefficient wave equations
\begin{equation*}
\left\{ \begin{array}{cl}
\partial_t^2 u_1 & = \Delta_{g_1} u_1, \\
\partial_t^2 u_2 &= \Delta_{g_2} u_2, 
\end{array} \right.
\end{equation*}
with $g^{ij}_k(x,t) \equiv g^{ij}_k$ and under the same assumptions as above. Here $\Delta_{g}: =  \partial_i g^{ij} \partial_j$.
Then for functions $u_i$  frequency-localized as in \eqref{eq:FrequencyDirection} 
we have the bilinear estimate 
\begin{equation*}
\| \nabla_x u_1 \cdot \nabla_x u_2 \|_{L^{\frac{d+3}{d+1}}_{t,x}([0,1] \times \R^d)} \lesssim_{\varepsilon} N^{\frac{d-1}{d+3}+\varepsilon} \| \nabla_{t,x} u_{1}(0) \|_{L^2_x} \| \nabla_{t,x} u_{2}(0) \|_{L^2_x}.
\end{equation*}
This follows in the Euclidean case from writing
\begin{equation*}
u_i(t,x) = \cos(t \sqrt{- \Delta_{g_i}}) u_{i}(0) + \frac{\sin(t \sqrt{- \Delta_{g_i}})}{\sqrt{- \Delta_{g_i}}} \dot{u}_{i}(0),
\end{equation*}
multiplying out the solutions to the half-wave equation, rescaling and applying Wolff's result \cite{Wolff01}, which is known to be sharp up to endpoints. 

\smallskip

The estimate in the above theorem goes beyond the range of linear Strichartz estimates (after applying H\"older's inequality), which can be applied without transversality assumptions. We remark that $L^2$-bilinear estimates for variable-coefficient wave equations with rough metrics were previously considered by Jao--Gavrus--Tataru \cite{GavrusJaoTataru2021}. Linear Strichartz estimates for wave equations with rough coefficients were proved by the second author \cite{Tataru2000,Tataru2002,Tataru03} using the FBI transform. The FBI transform decomposes a function into \emph{coherent states}, i.e., Gaussians, which likewise exhaust the uncertainty relation. It goes back to this series of papers that the coherent states propagate along the Hamiltonian flow in phase space:
\begin{equation*}
\left\{ \begin{array}{cl}
\dot{x}_t &= \frac{\partial p_N}{\partial \xi}, \\
\dot{\xi}_t &= - \frac{\partial p_N}{\partial x},
\end{array} \right.
\end{equation*}
where $p_N$ denotes the Hamiltonian after normalization of frequency to unit scale
(which changes the space-time scale to $R = N$).
But for $C^{1,1}$-coefficients, after a paradifferential decomposition, $C^k$-bounds ($k \geq 3$) on the Hamiltonian flow can only be obtained for times $|t| \leq C R^{\frac{1}{2}}$ with $C$ depending on the $C^{1,1}$-norm of the coefficients. For larger times, we can use the localization properties of the Gaussians.

\smallskip

We now provide a brief overview of our results and of our strategy. Our general bilinear estimates are formulated in Theorems \ref{thm:GeneralizationBilinearParaboloid} and \ref{thm:GeneralizationBilinearCone}, for Hamiltonian flows under low regularity assumptions and conditional upon the existence of a wave packet decomposition. We require the wave packet decomposition to satisfy localization in space, frequency localization, and time-frequency localization,
as described in Assumption \ref{ass:WP}.
This turns out to be more robust than simply imposing specific regularity assumptions on the coefficients.

Rather than simply considering two flows
which are uniformly transverse, we introduce a transversality parameter $\nu$, which should be thought of as the angle between the two propagation directions. Then we study the dependence 
of the implicit constant in the bilinear
estimates on $\nu$.
In the general case we are able to  obtain the sharp dependence on the transversality parameter, but there 
is a price to pay, namely that we  require a flatness assumption on the coefficients matching the transversality, in order to control interactions with very small transversality. Notably, in the case of self-interactions we can drop this additional assumption; see Section \ref{subsection:SelfInteractions}.

\smallskip

In the second part of the paper we use a modified FBI transform in order to establish a wave packet decomposition  for solutions to dispersive PDE's in the space-time region $|(x,t)| \leq R$ governed by symbols $p \in C^{\infty}(\R^{d+1} \times \R^d)$, which satisfy the following regularity estimates (see Assumption \ref{ass:HamiltonianFlowRegularity} for details):
\begin{equation*}
        |\partial_z^{\alpha} \partial_{\xi}^{\beta} p(x,t;\xi)| \lesssim_{\alpha,\beta} R^{-|\alpha| + \frac{(|\alpha|-2)_+}{2}}, \quad |\alpha| \geq 0.
\end{equation*}
 Our wave packet decomposition will be adapted to a lattice $\Lambda_R$ in phase space with the position and frequency scales $\Delta x \sim R^{\frac{1}{2}}$ and $\Delta \xi \sim R^{-\frac{1}{2}}$. Such a decomposition can be constructed globally, but for our purposes we need to 
 consider data which is phase space localized to a set $\mathcal{Y}$, which has frequencies localized to small balls of size $\nu$, and for an intermediate scale $r$ use wave packets associated to an enlargement $\overline{\mathcal{Y}}_r$ as in Assumption \ref{ass:RegularitySet} and Definition \ref{def:Thickening}.
We remark that we impose a smallness condition on the spatial regularity of the Hamiltonian to achieve the absence of caustics on the interval $[0,R]$. Alternatively, we could analyze the Hamiltonian for large spatial regularity on intervals $[0,\epsilon_0 R]$ for some small $\epsilon_0$ depending on the regularity. The following provides a wave packet decomposition for phase space localized functions:
\begin{theorem}
\label{thm:WavePacketDecompositions}
Let $R \geq 1$, $0< \delta_0 < \frac12$, and $\nu \in [R^{-\frac{1}{2}+\delta_0},1]$. Assume that $a$ is a Hamiltonian, which satisfies Assumption \ref{ass:HamiltonianFlowRegularity}, with the set $\mathcal{Y}$ satisfying Assumption \ref{ass:RegularitySet}. Let $r \in [\nu^{-2-\delta_0},R]$ and $\mathcal{Y}_r, \, \overline{\mathcal{Y}}_r$ be defined as in Definition \ref{def:Thickening}. There is a pseudo-differential operator $\chi_{\overline{\mathcal{Y}}_r}(x,D): L^2(\R^d) \to L^2(\R^d)$ whose symbol satisfies $\chi_{\overline{\mathcal{Y}}_r}(x,\xi) = 1$ for $(x,\xi) \in \mathcal{Y}_r$ such that for solutions to
\begin{equation*}
\left\{ \begin{array}{cl}
(D_t + p^w) u &= 0, \quad (t,x) \in \R^{d+1}, \\
u(0) &= \chi_{\overline{\mathcal{Y}}_r}(x,D) u_0 \in L^2(\R^d)
\end{array} \right.
\end{equation*}
we have the decomposition in $L^2$ for $t \in [-r,r]$
\begin{equation*}
u(t) = \sum_{(x_0,\xi_0) \in \Lambda_r \cap \overline{\mathcal{Y}}_r} \alpha_{(x_0,\xi_0)} \phi_{(x_0,\xi_0)}(t) + g(t)
\end{equation*}
with
\begin{equation*}
    \| g(t) \|_{L^p} \leq C_m \nu^{d \big( \frac12 - \frac1p \big)} (1+r \nu^2)^{-m} \| u_0 \|_{L^2} \text{ for } 2 \leq p \leq \infty
\end{equation*} 
and $(\phi_{(x_0,\xi_0)})_{(x_0,\xi_0) \in \Lambda_r \cap \overline{\mathcal{Y}}_r}$ are normalized wave packets satisfying Assumptions \ref{ass:WP} and \ref{ass:Frames}.
\end{theorem}
For the reader's convenience, we give an extended version of Theorem \ref{thm:WavePacketDecompositions} in Section \ref{subsection:WavePacketDecompositions} with a detailed overview of the assumptions.

In the bilinear estimates $\nu$ will play the role of the transversality parameter. The lower bound $\nu^{-2}$ for the time scale is the minimal one for which the wave packet scales are compatible with the $\nu$ scale for the frequency localization. 

The bounds on $p$ imposed above arise naturally as follows. Consider for definiteness a $1$-homogeneous symbol $p(x,t,\xi) = \sqrt{g^{ij}_{\leq N^{\frac{1}{2}}}(x,t) \xi_i \xi_j}$ with coefficients in $C^{1,1}$, and a solution $u$ to $(D_t + p^w) P_N u = f$. At frequencies $N$ the paradifferential decomposition suggests to truncate the coefficients of the metric at frequencies $\leq N^{\frac{1}{2}}$ because the incurred error can be estimated perturbatively. After rescaling to unit frequencies we find
\begin{equation*}
p_N(x,t;\xi) = \sqrt{g^{ij}_{\leq N^{\frac{1}{2}}}(N^{-1} x,N^{-1} t) \xi_i \xi_j} \chi_1(\xi),
\end{equation*}
with $\chi_1$ denoting a smooth cutoff to the unit annulus, and we see that the regularity estimates from Theorem \ref{thm:WavePacketDecompositions} are satisfied with $R= N$.

\smallskip

$C^{1,1}$-regularity of the symbols is natural to assume as the flow ceases to be well-defined in general for rougher regularities. We show that basic wave packet properties (see Section \ref{section:WavePacketDecomposition}) are satisfied when defining packets through conjugation to phase space via the FBI transform and phase space localization.
In a similar spirit, we point out how for wave or Schr\"odinger equations with $C^{1,s}$-coefficients the Euclidean wave packet decompositions can be recovered on a frequency-dependent time scale. It is instructive to observe that wave packet decompositions can be regarded as more fundamental than dispersive decay properties. With a wave packet decomposition  at hand, which is governed by a $C^{1,1}$-Hamiltonian flow, it is essentially well-known how to infer the same dispersive properties as in Euclidean space provided that the flow satisfies non-degeneracy assumptions as is the case for the Schr\"odinger or wave evolution. And with the dispersive estimate at hand, Keel--Tao \cite{KeelTao1998} pointed out how to obtain the endpoint Strichartz estimates via bilinear interpolation; see \cite[Theorem~4]{Tataru2000} for an extension to suitable parametrices. 

\smallskip

\subsection{Linear Fourier extension and oscillatory integral estimates}

The following remarks on linear (adjoint) restriction estimates are in order. 
For the Fourier extension operator these read
\begin{equation}
\label{eq:LinearFourierRestrictionIntro}
\| \mathcal{E}_{x} f \|_{L^p_{t,x}(B_{d+1}(0,R))} \lesssim_\varepsilon R^{\varepsilon} \| f \|_{L^\infty}
\end{equation}
and are conjectured to hold for 
\begin{equation*}
p >
\begin{cases}
\frac{2(d+1)}{d}, \quad &x = \Delta, \\
\frac{2d}{d-1}, \quad &x = c.
\end{cases}
\end{equation*}
The conjectural range of $p$ hinges on the curvature properties of the surface.

There is a large body of literature  \cite{Bourgain1991II,Wisewell2005,Lee2006,BennettCarberyTao2006,BourgainGuth2011, GuthHickmanIliopoulou2019,Schippa2024Oscillatory,GuoWangZhang2024} on $L^p$-boundedness properties for oscillatory integral operators with smooth phase functions and amplitude functions
\begin{equation*}
T^\lambda f(z) = \int e^{i \phi^\lambda(z,\xi)} a^\lambda(z,\xi) f(\xi) d\xi.
\end{equation*}
The list of references is by no means exhaustive, and we refer to the references within for further reading.
The oscillatory integral operators clearly generalize \eqref{eq:LinearFourierRestrictionIntro}. Above $\phi^\lambda(z,\xi) = \lambda \phi(z/\lambda,\xi)$ and $a^\lambda(z,\xi) = a(z/\lambda,\xi)$; $a, \phi \in C^\infty$, and we suppose that the phase function can be regarded as smooth perturbation of
the cone or paraboloid. Naturally, these phases arise from parametrices of solutions to equations governed by $p \in S^k_{1,0}$:
\begin{equation*}
\left\{ \begin{array}{cl}
D_t u + p(x,t,D) u &= 0, \quad (x,t) \in \R^d \times \R, \\
u(0) &= u_0.
\end{array} \right.
\end{equation*}
The solution can be represented by a Fourier Integral Operator
\begin{equation*}
u(x,t) = \int e^{i \phi(x,t,\xi)} a(x,t;\xi) \hat{u}(\xi) d\xi + R_t u
\end{equation*}
with an $L^p$-bounded error term $R_t(x,D)$. The phase function $\phi$ solves the eikonal equation
\begin{equation*}
\left\{ \begin{array}{cl}
\partial_t \phi(x,t;\xi) + p(x,t,\nabla_x \phi) &= 0, \quad (x,t,\xi) \in \R^d \times \R \times \R^d, \\
\phi(x,0;\xi) &= \langle x, \xi \rangle.
\end{array} \right.
\end{equation*}
For instance when $p(x,\xi) = \sqrt{g^{ij}(x) \xi_i \xi_j}$ as is the case for solutions to (half-)wave equations on compact Riemannian manifolds, we find that $\phi(x,t;\xi) = \langle x, \xi \rangle - t \sqrt{g^{ij}(x) \xi_i \xi_j} + \ldots$ can be regarded as a smooth perturbation of the phase function for a cone extension operator. Moreover, in the translation-invariant case $p(x,t,\xi) = p_0(\xi)$ the eikonal equation is solved by $\phi(x,t,\xi) = \langle x, \xi \rangle - t p_0(\xi)$, which recovers the Fourier extension case.

\medskip

One may compare the $L^p$-estimates for oscillatory integral operators with the similar bounds for the Fourier extension. Oversimplifying matters a bit, it was found out that the multilinear estimates hinging on transversality remained valid, as pointed out for bilinear estimates by S. Lee \cite{Lee2006} and for the $d+1$-multilinear estimate by Bennett--Carbery--Tao \cite{BennettCarberyTao2006}. 

Some of the techniques to recover linear from multilinear estimates also carry over. In the bilinear case this is accomplished by Whitney decomposition and rescaling; see \cite{TaoVargasVega1998}. The Bourgain--Guth \cite{BourgainGuth2011} argument to infer linear estimates from higher multilinear estimates transpires likewise. But due to the phenomenon of \emph{Kakeya compression} (see \cite{Wisewell2005,GuthHickmanIliopoulou2019}), that is wave packets clustering in lower-dimensional manifolds, the range of linear estimates for oscillatory integral operators is strictly smaller than for the linear Fourier extension operators.

Indeed, following the multilinear approach to linear estimates, sharp linear estimates for oscillatory integral operators could be proved although the restriction conjecture remains open (see \cite{Lee2006,BourgainGuth2011,GuthHickmanIliopoulou2019,Schippa2024Oscillatory}).

\medskip

In summary, our contribution can be regarded as a first step to examine the ramifications of the more recent developments in multilinear Fourier restriction estimates for dispersive PDEs with rough coefficients. Given that the linear estimates for smooth phase functions are strictly worse than in the constant-coefficient case, it is remarkable that the bilinear estimates are retained for $C^{1,1}$-coefficients. In particular, the sharp bilinear estimates in the case of self-interaction with small transversality, cf. Theorems \ref{thm:BilinearSelfinteractionsSEQ}, \ref{thm:BilinearSelfInteractionsWave}, go far beyond the oscillatory integral estimates for smooth Hamiltonian flows. Our analysis here involves rescaled flows, which cease to be $C^{1,1}$. The estimates for smooth flows rather compare to our estimates on the restricted scale $|(x,t)| \leq R^{\frac{1}{2}}$, on which scale the Hamiltonian flow satisfies the same regularity estimates as a smooth flow (cf. Section \ref{section:RegularityFlow}).

\bigskip

\textbf{Basic notations:} 
\begin{itemize}
\item Spatial variables are denoted by $x \in \R^d$. Time is denoted by $t \in \R$. We let $z=(x,t) \in \R^{d+1}$ denote space-time. The dual frequency variables are denoted by $(\xi,\tau) \in \R^d \times \R$. We let $D_t = -i \partial_t$ and $D_x = - i \nabla_x$.  
\item For $x \in \R^d$ we denote the Euclidean norm by $|x|^2 = \sum_{i=1}^d x_i^2$ and 
\begin{equation*}
    B_d(y,R)= \{ x \in \R^d : |x-y| < R \}.
\end{equation*}
\item For $A, B > 0$ we write $A \lesssim B$ provided that there is a harmless constant (depending on fixed parameters like dimension, etc) $C \geq 1$ such that $A \leq C B$. Similarly, we let $A \sim B$ provided that there is $C \geq 1$ as above such that $A \leq C B$ and $B \leq C A$. More involved dependencies of the constant are indicated by subscripts, e.g., $A \lesssim_{\alpha} B$ means that $A \leq C(\alpha) B$ for a constant depending on $\alpha$.
\item For $R \geq 1$ and $F: \R \to \R_{>0}$ we write $F(R) = \text{RapDec}(R)$
provided that for any $N \in \N$ there is $C_N \geq 1$ such that for $R \geq 1$ we have
\begin{equation*}
F(R) \leq C_N (1 + R)^{-N}.
\end{equation*}
\item Let $A \subseteq \R^d$. By $\chi_A$ we denote the indicator function
\begin{equation*}
   \chi_A(x) = \begin{cases}
        1, &\quad x \in A, \\
        0, &\quad x \notin A,
   \end{cases}
\end{equation*}
and for $r > 0$ by $\mathcal{N}_r(A) = \{ x \in \R^d : \text{dist}(x,A) < r \}.$
\item For $x \in \R$ we denote by $x_+ = \max(x,0)$.
\end{itemize}

\medskip

\emph{Outline of the paper.} In Section \ref{section:BilinearRoughEstimates} we formulate the bilinear estimates for transversely intersecting curved tubes governed by rough Hamiltonian flows. The bilinear estimates are established in Section \ref{section:ProofBilinear} through induction-on-scales and combinatorial arguments. In Section \ref{section:WavePacketDecomposition} we show wave packet decompositions for symbols satisfying decay and regularity estimates relying on the FBI transform. Finally, in Section \ref{section:Applications} we apply the wave packet decomposition to recover sharp Strichartz estimates and bilinear estimates for PDE's with rough coefficients. Moreover, we show sharp bilinear estimates for self-interactions of Schr\"odinger and wave equations with small transversality without additional flatness assumptions on the metrics in Section \ref{subsection:SelfInteractions}.

\section{Bilinear estimates for rough Hamiltonian flows}
\label{section:BilinearRoughEstimates}

We formulate the validity of the bilinear estimate in terms of the properties of the Hamiltonian flow. Examples of symbols we consider are
\begin{equation*}
p_1(x,t,\xi) = |\xi|_{g} = (g^{ij}(x,t) \xi_i \xi_j)^{\frac{1}{2}}, \quad p_2(x,t,\xi) = g^{ij}(x,t) \xi_i \xi_j.
\end{equation*}
We always require non-degeneracy of the metric;  in the first case, where the PDE describes a half-wave evolution,  ellipticity of the metric. In the second case the evolution describes a (possibly non-elliptic) Schr\"odinger equation. We provide an overview of the involved constants and their dependencies at the end of this section.

\subsection{Basic regularity assumptions}
With applications to quasilinear PDE's in mind, we want to consider rough coefficients. After a standard paradifferential decomposition, we may assume the coefficients are smooth, but our bounds shall depend only on as few $x$-derivatives as possible. Below $R \geq 1$ denotes a spatial parameter and $\nu \in [R^{-\frac{1}{2}+\delta_0},1]$ a transversality parameter.

\begin{assumption}[Regularity~assumptions~on~the~Hamiltonian~at~scale~$R$~and~trans\-versality~$\nu$]
\label{ass:RegularityFlow}
Let $R \geq 1$, $ 0 < \delta_0 < 1/4$, and $\nu \in [R^{-\frac{1}{2}+\delta_0},1]$. We let $\epsilon_0 \leq c \nu$ for a small universal constant $0<c<1$.
Let $p \in C_{z}^{1,1} C_{\xi}^3(\R^d \times \R \times \R^d)$. The regularity in phase space is quantified for $1 \leq |\alpha| \leq 2$, $0 \leq |\beta| \leq 3$ by
\begin{equation}
\label{eq:MixedRegularity}
|\partial_z^\alpha \partial_\xi^{\beta} p(z,\xi)| \leq \epsilon_0 R^{-|\alpha|}, \quad (z,\xi) \in \R^d \times \R \times \R^d.
\end{equation}
\end{assumption}

The regularity in frequency is quantified by $C_2 \geq 1$ such that for $0 \leq |\beta| \leq 3$:
\begin{equation}
\label{eq:XiRegularity}
|\partial_{\xi}^{\beta} p(z,\xi)| \leq C_2, \quad (z,\xi) \in \R^d \times \R \times \R^d.
\end{equation}

\begin{remark}[Absence~of~caustics]
\label{rem:Caustics}
    For wave and Schr\"odinger Hamiltonians the flow does not develop caustics up to time $t=R$, provided that $c < 1$ is chosen small enough. 
\end{remark}

We remark that the flatness assumption \eqref{eq:MixedRegularity} seems fairly strong for $\nu \ll 1$. For self-interactions we show in Section \ref{subsection:SelfInteractions} how rescaling allows for the less stringent assumption $\epsilon_0 \leq c$ in \eqref{eq:MixedRegularity}, even when the transversality parameter satisfies $\nu \ll c$.

We let $\overline{C}_2 = e^{4 C_2}$ and record the following bi-Lipschitz consequence for the bicharacteristics. Introduce the following metric on phase space:
\begin{equation*}
    d_R((x,\xi),(y,\eta)) = R^{-\frac{1}{2}} |x-y| + R^{\frac{1}{2}} |\xi - \eta|.
\end{equation*}

\begin{lemma}[Bi-Lipschitz~property~of~bicharacteristics]
\label{lem:GeometryWavepackets}
Let $p \in  C_{t}^0 C_{x,\xi}^{1,1}$. Let $(x_i^t,\xi_i^t)$, $i=1,2$ be two bicharacteristics with $(x_i^0,\xi_i^0) \in \R^{2d}$. 
For $0 \leq |\alpha|, |\beta| \leq 2$ with $|\alpha|+|\beta| = 2$ let
\begin{equation*}
    |\partial_x^{\alpha} \partial_{\xi}^{\beta} p(x,t,\xi)| \leq C R^{-|\alpha|}.
\end{equation*}
Then it holds for $|t| \leq R$:
\begin{equation*}
d_R((x_1^t,\xi_1^t),(x_2^t,\xi_2^t)) \leq d_R((x_1^0,\xi_1^0),(x_2^0,\xi_2^0)) e^{2 C R^{-1} |t|}.
\end{equation*}
\end{lemma}

\begin{proof}
This will be a consequence of Gr\o nwall's lemma. In the following we suppress a possible explicit time-dependence of the Hamiltonian for brevity. First, we write by the fundamental theorem:
\begin{equation*}
    x_1^t-x_2^t = x_1^0 - x_2^0 + \int_0^t \big[ \frac{\partial p}{\partial \xi}(x_1^s,\xi_1^s) - \frac{\partial p}{\partial \xi}(x_2^s,\xi_2^s) \big] ds .
\end{equation*}
Next, we use the mean value theorem:
\begin{equation*}
\begin{split}
    \frac{\partial p}{\partial \xi}(x_1^s,\xi_1^s) - \frac{\partial p}{\partial \xi}(x_2^s,\xi_2^s) &= \int_0^1 \frac{\partial^2 p}{\partial x \partial \xi}(x^s_1 + \theta(x_2^s - x_1^s), \xi_1^s + \theta (\xi_2^s - \xi_1^s)) (x_2^s - x_1^s) \\
    &\quad \quad  + \frac{\partial^2 p}{(\partial \xi)^2}(x_1^s+\theta(x_2^s-x_1^s),\xi_1^s + \theta(\xi_2^s - \xi_1^s)) (\xi_2^s - \xi_1^s) \, d\theta.
\end{split}
\end{equation*}
We obtain by Minkowski's inequality:
\begin{equation*}
    \begin{split}
    R^{-\frac{1}{2}}|x_1^t-x_2^t| &\leq R^{-\frac{1}{2}} |x_1^0 - x_2^0| + C R^{-1} \int_0^t R^{-\frac{1}{2}} |x_2^s - x_1^s| ds + C R^{-1} \int_0^t R^{\frac{1}{2}} |\xi_1^s - \xi_2^s| ds \\
    &\leq R^{-\frac{1}{2}} |x_1^0 - x_2^0| + C R^{-1} \int_0^t d_R((x_1^s,\xi_1^s),(x_2^s,\xi_2^s)) ds.
    \end{split}
\end{equation*}

For the frequency component we find by the fundamental theorem:
\begin{equation*}
\xi_1^t - \xi_2^t = \xi_1^0 - \xi_2^0 - \int_0^t \big( \frac{\partial p}{\partial x}(x_1^s,\xi_1^s) - \frac{\partial p}{\partial x}(x_2^s,\xi_2^s) \big) ds.
\end{equation*}
A further application of the mean-value theorem yields
\begin{equation*}
\frac{\partial p}{\partial x}(x_1^s,\xi_1^s) - \frac{\partial p}{\partial x}(x_2^s,\xi_2^s) = \frac{\partial^2 p}{(\partial x)^2}(x^*,\xi^*) (x_1^s - x_2^s) + \frac{\partial^2 p}{(\partial x \partial \xi)}(x^*,\xi^*) (\xi_1^s - \xi_2^s).
\end{equation*}
By the regularity estimates for $p$ and an application of Minkowski's inequality we find
\begin{equation*}
\begin{split}
R^{\frac{1}{2}} |\xi_1^t - \xi_2^t| &\leq R^{\frac{1}{2}} |\xi_1^0 - \xi_2^0|+ C R^{-1} \big( \int_0^t R^{-\frac{1}{2}} |x_1^s - x_2^s| + R^{\frac{1}{2}} |\xi_1^s - \xi_2^s| ds \big) \\
&\leq R^{\frac{1}{2}} |\xi_1^0 - \xi_2^0|+ C R^{-1} \int_0^t d_R((x_1^s,\xi_1^s),(x_2^s,\xi_2^s)) ds.
\end{split}
\end{equation*}

Taking the estimates together we find
\begin{equation*}
    d_R((x_1^t,\xi_1^t),(x_2^t,\xi_2^t)) \leq d_R((x_1^0,\xi_1^0),(x_2^0,\xi_2^0)) +2  C R^{-1} \int_0^t d_R((x_1^s,\xi_1^s),(x_2^s,\xi_2^s)) ds,
\end{equation*}
and the conclusion follows from applying Gr\o nwall's inequality in integral form.
\end{proof}

\smallskip

We shall consider initial data essentially localized to subsets $\mathcal{Y} \subseteq T^* \R^d$ in the phase space, which will be realized via application of a pseudo-differential operator (PDO). $\mathcal{Y}$ is supported on a cube of size $R$ and in frequencies on a scale $\nu \in [R^{-\frac{1}{2}+\delta_0},1]$. Later, $\nu$ will quantify the transversality of intersecting bicharacteristics. We thicken $\mathcal{Y}$ with a margin of size $R \nu$ in space and with a margin of $\nu$ in frequency; on the scale $r = \nu^{-2}$ the symbol of the PDO will be supported on the $(2 R \nu,2 \nu)$-thickening. At larger scales we can allow for smaller margin, which leads us to sets $\overline{\mathcal{Y}}_r$. Then, when applying the phase space transform on a scale $r \in [\nu^{-2},R]$, we obtain for the phase space portion on $\overline{\mathcal{Y}}_r^c$ an error of $\text{RapDec}(r \nu^2) \| f \|_2$. In the following we denote by $\pi_a: T^* \R^d \to \R^d$, $(x,\xi) \mapsto a$ for $a \in \{x,\xi\}$ the coordinate projections from phase space. 

The thickening is motivated by constructing wave packets on the smallest spatial scale $\nu^{-2}$, which corresponds to the largest frequency scale $\nu$. The resolution allows us to observe single wave packets. On the other hand, the spatial thickening $R \nu$ is motivated by the metric $d_R$ on phase space, which balances the spatial and frequency margin; the distance between $\mathcal{Y}$ and $\overline{\mathcal{Y}}^c$ of $R^{\frac{1}{2}} \nu$ is propagated up to times $|t| \leq R$.
\begin{definition}[Thickenings~of~sets~in~phase~space]
\label{def:Thickening}
Let $0<\nu<1$, $0<\delta_0 < 1/4$, $R \geq 1$, and $\mathcal{Y} \subseteq T^* \R^d$. $\mathcal{Y}_{\nu^{-2}}$ denotes the $(R \nu,\nu)$-thickening, that is a thickening of $R \nu$ spatially and $\nu$ in frequency, of $\mathcal{Y}$:
\begin{equation*}
    \mathcal{Y}_{\nu^{-2}} = \{ (y,\eta) \in T^* \R^d: \exists (x,\xi) \in \mathcal{Y} : |x-y| < R \nu \text{ and } |\eta - \xi| < \nu \}.
\end{equation*}
For a scale $r \in [\nu^{-2-\delta_0},R]$ we consider the $(R / r^{\frac{1}{2}-\delta_0}, r^{-\frac{1}{2}+\delta_0})$ thickening of $\mathcal{Y}_{\nu^{-2}}$ denoted by $\mathcal{Y}_r$ and its thickening $(C R / r^{\frac{1}{2}-\delta_0},C r^{-\frac{1}{2}+\delta_0})$ by $\overline{\mathcal{Y}}_r$.
\end{definition}

\begin{assumption}[Size~assumptions~on~$\mathcal{Y}$]
\label{ass:RegularitySet}
Let $R \geq 1$, $\emptyset \neq \mathcal{Y} \subseteq B_d(0,10 R) \times \R^d$ be an open set, and $C = c^{-1/2}$ with $c$ the small constant from Assumption \ref{ass:RegularityFlow}. We require that $\mathcal{Y}$ satisfies the following size 
conditions:
\begin{equation*}
    \pi_{\xi}(\mathcal{Y}) \subseteq B_d(\xi^*, C \nu).
\end{equation*}
We consider a thickening $\overline{\mathcal{Y}}$ which is defined as follows:
\begin{equation*}
    \overline{\mathcal{Y}} = \{ (y,\eta) \in T^* \R^d : \; \exists (x,\xi) \in \mathcal{Y} : |y-x| < C R \nu \text{ and } |\eta - \xi | < C \nu \}.
\end{equation*}

For $t \in [-10R,10R]$ let $\chi_t : T^* \R^d \to T^* \R^d$ denote the symplectic flow generated by $p$, i.e., $(x,\xi) \mapsto (x^t,\xi^t)$. We assume that the flow-out of the second thickening $\overline{\overline{\mathcal{Y}}}$ under the Hamiltonian flow governed by $p$ for $t \in [-10R,10R]$
\begin{equation}
\left\{ \begin{array}{cl}
\label{eq:RescaledHamiltonianFlowAssumptions}
\dot{x}_t &= \frac{\partial p}{\partial \xi}, \\
\dot{\xi}_t &= - \frac{\partial p}{\partial x}
\end{array} \right.
\end{equation}
is contained in the open set
\begin{equation}
\label{eq:FlowOutCondition}
\overline{\mathcal{Y}}_R \supseteq \bigcup_{|t| \leq 10R} \{(x_t,t,\xi_t) \in \R^{d} \times \R \times \R^d : (x_0,\xi_0) \in \overline{\overline{\mathcal{Y}}}  \}.
\end{equation}
In case $p$ is $1$-homogeneous in $\xi$, we replace the frequency balls with sectors. For the first condition the modified assumption reads for some $\xi^* \in \mathbb{S}^{d-1}$
\begin{equation*}
    \{ \xi / |\xi| : \xi \in \pi_{\xi}(\mathcal{Y}) \} \subseteq B_d(\xi^*, C \nu).
\end{equation*}
\end{assumption}

That $\mathcal{Y}$ is spatially contained in an $R$-ball centered at the origin is an assumption to ease notation. In practice, by finite speed of propagation, we can suppose that the solution is localized to balls of size $R$. Since the pointwise regularity assumptions are translation-invariant, the centering at the origin causes no loss of generality.

\smallskip

Note that for $\overline{\mathcal{Y}}_r$ to be contained in $\overline{\mathcal{Y}}$ we require  $\nu^{-2-\delta_0} \leq r \leq R$. Note again that the phase space distance from $\mathcal{Y}_{\nu^{-2}}$ to $\overline{\mathcal{Y}}_r^c$ is likewise propagated. The sharp localization corresponds to $\delta_0 = 0$. Allowing for $\delta_0 > 0$ gives us rapid decay off the localization. Furthermore, when carrying out the induction on scales, we shall also carry out the induction on cubes $Q_r$ not centered at the origin. This leads us to sets $\mathcal{Y}_{r,t}$, which are obtained from evolving $\mathcal{Y}_r$ under the Hamiltonian flow, the set $\overline{\mathcal{Y}}_{r,t}$ is obtained from taking its $C (R r^{-\frac{1}{2}+\delta_0}, r^{-\frac{1}{2}+\delta_0})$ thickening. Note that the margin conditions are preserved.

The thickenings come into the picture as follows: we consider a pseudo-differential operator,
whose symbol is identically one on $\mathcal{Y}_{\nu^{-2}}$ and supported on $\mathcal{Y}_{C \nu^{-2}}$ which is contained in $\overline{\mathcal{Y}}_r$ for $r \geq \nu^{-2-\delta_0}$.

Subtracting the wave packets in the whole of $\overline{\mathcal{Y}}_r$ allows us to prove the rapid decay $\text{RapDec}(r \nu^2)$ for the remainder. 
When constructing the wave packets in $\overline{\mathcal{Y}}_r \subseteq \overline{\mathcal{Y}}$, we consider bicharacteristics emanating from a lattice set. However, for technical reasons we want to control bicharacteristics in an even larger set $\overline{\overline{\mathcal{Y}}}$, which is obtained from thickening $\overline{\mathcal{Y}}$ once more.

\smallskip

As it turns out, once a flow is defined, most important for the bilinear estimates
will be geometric transversality assumptions for which it will suffice to consider regularity bounds obtained from integrating along the bicharacteristics.

\subsection{Wave packet decompositions}

To match the literature on bilinear extension estimates\footnote{Another rescaling would lead us to consider the flow on the unit time scale and in the unit cube. Then the wave packets would have a different shape compared to the literature on multilinear estimates, matching instead the body of literature on nonlinear dispersive equations.}, we consider the flow governed by $p(x,t,\xi)$ for $|(x,t)| \leq R$. We consider wave packets $\phi_{T_i}$ on scale $\Delta x \sim R^{\frac{1}{2}}$, $\Delta \xi \sim R^{-\frac{1}{2}}$. More precisely, we introduce a lattice in phase space $\Lambda_R =  R^{\frac{1}{2}} \Z^d \times R^{-\frac{1}{2}} \Z^d $ - the initial data will be decomposed according to this lattice in phase space with $\mathcal{Y}$ indicating its \emph{essential} support.

Lattice points, and by extension along the flow, \emph{wave packets},  are labeled by $T \in \Lambda_R$. We abuse notation and when the context permits, we denote the curved cylinder obtained from enlarging the projection of the bicharacteristic also by $T_R =
\{ R^{\frac{1}{2}} \pi_x(x_T(t),\xi_T(t)) : t \in [-R,R] \} \subseteq \R^{d+1} $. We formulate the bilinear estimates in terms of wave packets, which satisfy the axioms below:
\begin{assumption}[Wave~packet~axioms]
\label{ass:WP}
Let $1 \leq r \leq R$, $\delta > 0$, and $(x^T_t,\xi^T_t)$ denote a bicharacteristic of the Hamiltonian flow \eqref{eq:RescaledHamiltonianFlowAssumptions} with $(x_0^T,\xi_0^T) = T$. Let $T \in \Lambda_r \cap \overline{\mathcal{Y}}$. A Schwartz function $\phi_T \in \mathcal{S}(\R^{1+d})$ satisfies the \emph{wave packet axioms} if for any $t \in [-r,r]$ it holds $|\phi_T(x,t)| \leq C_{\delta} r^{-\frac{d}{4}+C\delta}$ and $\| \phi_T \|_2 \leq C$ for a universal constant $C \geq 1$.

And $\phi_T$ is required to satisfy the following for any $\delta > 0$ uniformly in $t \in [-r,r]$:
\begin{itemize}
\item[1.] \emph{Localization in position:} 
\begin{equation}
\label{eq:LocPos}
\| \chi_{\mathcal{N}^c_{r^{\frac{1}{2}+\delta}}(x^t)} \phi_T(t) \|_{L^2_x} \leq C_{m,\delta} (1+ r)^{-m},
\end{equation}
\item[2.] \emph{Localization in frequency:}
\begin{equation}
\label{eq:LocFrequ}
 \| \chi_{\mathcal{N}^c_{r^{-\frac{1}{2}+\delta}}(\xi^t)} \widehat{\phi_T}(t) \|_{L^2_{\xi}} \leq C_{m,\delta} (1+ r)^{-m},
\end{equation}
\item[3.] \emph{Localization in time-frequency:} For a non-negative function $\chi \in C^\infty_c(\R)$ with $\chi(t) \equiv 1$ for $t \in [-1,1]$ and $\text{supp}(\chi) \subseteq (-2,2)$ we have the following estimate uniform in $x \in \R^d$ and $t_0 \in [-r,r]$:
\begin{equation}
\label{eq:LocTimeFrequ}
\| \chi_{\mathcal{N}^c_{r^{\frac{1}{2}+\delta}}(-p(x^{t_0},t_0,\xi^{t_0}))}(\tau) \mathcal{F}_t [ \chi(r^{-\frac{1}{2}} (t-t_0)) \phi_T ](x,\tau) \|_{L^2_{\tau}} \leq C_{m,\delta} (1+ r)^{-m}.
\end{equation}
\end{itemize}
\end{assumption}

\begin{remark}
In our application, the bounds in the localization properties of the wave packets will depend on the seminorm-like quantities $C_{\alpha,\beta}$ of the symbol
\begin{equation*}
    |\partial_{z}^{\alpha} \partial_{\xi}^{\beta} p(x,t,\xi)| \leq C_{\alpha,\beta} R^{-|\alpha|+\frac{(|\alpha|-2)_+}{2}}, \quad |\alpha| \geq 0.
\end{equation*}
For a symbol, which is obtained from paradifferential decomposition, for $|\alpha| > 2$ in turn, these quantities only depend on the $C^{1,1}$-norm of the coefficients.
\end{remark}

The following almost orthogonality property is immediate:
\begin{lemma}
\label{lem:AlmostOrthogonalityWavePackets}
Let $(\phi_T)_{T \in \Lambda_r \cap \overline{\mathcal{Y}}}$ be a family of wave packets, which satisfy Assumption \ref{ass:WP}. Then for any subcollection $\T' \subseteq \Lambda_r \cap \overline{\mathcal{Y}}$ and $t \in [-r,r]$ we have:
\begin{equation*}
\big\| \sum_{T' \in \T'} \phi_{T'}(t) \big\|_{L_x^2} \leq C_{\delta} r^{\delta} (\# \T')^{\frac{1}{2}}
\end{equation*}
for any $\delta > 0$.
\end{lemma}

In practice, the wave packets will constitute solutions to a PDE evolution for initial data localized in phase space. 
We make the following assumption, which relates the two-parameter family $(S(t,s))_{t,s}$ governing the evolution with the wave packet decomposition. In the following let $Q_R$ denote a space-time cube of size $R$ centered at the origin.

\begin{assumption}[Frame~property~for~approximate~solutions]
\label{ass:Frames}
    Let $p$ be a Hamiltonian flow satisfying Assumptions \ref{ass:RegularityFlow} and \ref{ass:RegularitySet}. Let $(S(t,s))_{t,s \in [-R,R]}$ denote the associated two-parameter family of flow maps $S(t,s):L^2(\R^d) \to L^2(\R^d)$ which satisfies uniform bounds $\| S(t,s) \|_{L^2} \lesssim 1$ and
\begin{itemize}
    \item $S(t,t) = {Id}_{L^2}$, 
    \item $S(t_3,t_2) S(t_2,t_1) =S(t_3,t_1)$.
\end{itemize}
We assume that for any $t \in [-R,R]$ there is a family of sets $\big( \mathcal{Y}_{r,t} \big)_{r \in [\nu^{-2-\delta_0},R]}$, such that the thickenings satisfy $\overline{\mathcal{Y}}_{R,t} \subseteq \overline{\mathcal{Y}}_{r_2,t} \subseteq \overline{\mathcal{Y}}_{r_1,t} \subseteq \overline{\mathcal{Y}} \subseteq T^* \R^d$ with $R \geq r_2 \geq r_1 \geq \nu^{-2-\delta_0}$, $r \in [\nu^{-2-\delta_0},R]$ with $\mathcal{Y}_{r,t} \supseteq \chi_t(\mathcal{Y}_{r})$ and a family of $L^2$-bounded pseudo-differential operators $\chi_{\overline{\mathcal{Y}}_{r,t}}(x,D)$ with $\chi_{\overline{\mathcal{Y}}_{r,t}}(x,\xi) = 1$ for $(x,\xi) \in \mathcal{Y}_{r,t}$.

We require that, for $t_i \in [-R,R]$ and $t \in (t_i + [-r,r]) \cap [-R,R]$, with $Q_r \subseteq Q_R$ denoting a space-time cube of size $r$ with time-center $t_i$, we have the expansion in $L^2$
\begin{equation}
\label{eq:ExpansionInitialData}
\begin{split}
    S(t,t_i) (\chi_{\overline{\mathcal{Y}}_{r,t_i}} f) &= S(t,t_i) f_0 + S(t,t_i) f_1 \\
    &= \sum_{T \in \mathbb{T} \subseteq \Lambda_r \cap \overline{\mathcal{Y}}_{r,t_i}} \alpha_T \phi_T(t) + g(t) \text{ with } \sum_{T \in \mathbb{T}} |\alpha_T|^2 \lesssim \| f \|_{L^2}^2
\end{split}
\end{equation}
with $\# \T \leq r^{10d}$ and 
\begin{equation}
 \| g \|_{L_{x,t}^p(Q_r)} \lesssim \nu^{d \big( \frac{1}{2} - \frac{1}{p} \big)} \text{RapDec}(r \nu^{2}) \| f \|_{L^2}, \qquad 2 \leq p \leq \infty.
\end{equation}
Let $\chi_{\mathcal{Y}} = \chi_{\overline{\mathcal{Y}}_{R,0}}$. We require for a normalized wave packet $\phi_T$ at scale $r_2$ with $r_1 \leq r_2^{1-\delta}$
\begin{equation}
\label{eq:WavePacketLocalizationFinal}
    \| (1-\chi_{\overline{\mathcal{Y}}_{r_1,t_f}} (x,D)) \phi_T(t_f) \|_{L^2} \lesssim_{\delta} \text{RapDec}(r_1)
\end{equation}
provided that $(x_T(t_f),\xi_T(t_f))$ is in the $C(R r_2^{-\frac{1}{2}+\delta_0},r_2^{-\frac{1}{2}+\delta_0})$-thickening of $\overline{\mathcal{Y}}_{r_2,t_f}$ and $\delta \geq \delta_0$.
\end{assumption}
\begin{remark}
The $L^2$-estimate for $g$ corresponds to $\chi_{\mathcal{Y}} f$ in phase space being rapidly decaying off $\overline{\mathcal{Y}}_r^c$. For $r \to \nu^{-2}$ the rapid decrease in $r$ becomes rapid decay in $r \nu^2$.
 For the $L^p$-estimates we use that the frequencies are essentially localized to a ball of radius $\nu$, which makes Bernstein's inequality applicable. We comment on the realization of Assumption \ref{ass:Frames} in Subsection \ref{subsection:WavePacketDecompositions}.

The above assumption describes our requirements on wave packet parametrices, which are needed to carry out the proof of the bilinear estimates. The microlocalization described in \eqref{eq:ExpansionInitialData} and \eqref{eq:WavePacketLocalizationFinal} will be used to carry out induction-on-scales and more precisely, to pass from wave packets at large spatial scales to wave packets at smaller scales.
\end{remark}

\subsection{Transversality assumptions and bilinear estimates}

We seek to obtain the following bilinear estimates for Hamiltonian flows $p_i$, $i=1,2$, which are compatible with evolutions $(S_i(t,s))$, $i=1,2$;
\begin{equation}
\label{eq:RoughBilinearEstimate}
\big\| S_1(t,0) (\chi_{\mathcal{Y}_1} f_1) S_2(t,0) (\chi_{\mathcal{Y}_2} f_2) \big\|_{L^{\frac{d+3}{d+1}}(Q_R)} \lesssim_\varepsilon R^\varepsilon \nu^{-\frac{2}{d+3}} \| f_1 \|_{L^2} \| f_2 \|_{L^2}.
\end{equation}
To this end, we require transversality assumptions on the propagation along the Hamiltonian flows $p_i$, $i=1,2$. 

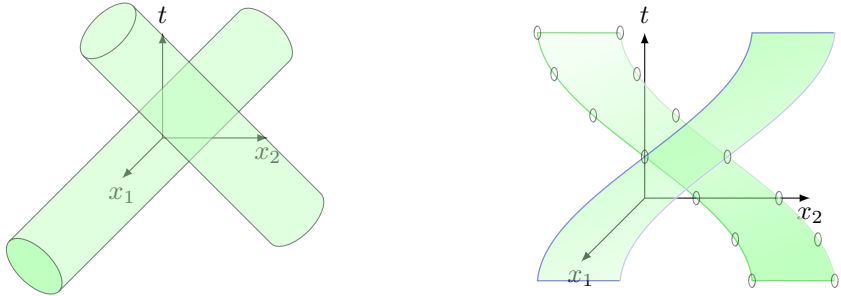
\begin{figure}[ht!]
    \centering
    \begin{minipage}[t]{0.45\textwidth} 
    	\vspace{0pt}
        \centering
        \begin{tikzpicture}[scale=0.7]

  \begin{scope}[shift={(6,0)}]

  \coordinate (O) at (0,0,0);
  \coordinate (A) at (2,0,0);
  \coordinate (B) at (0,2,0);
  \coordinate (C) at (0,0,2);

  \draw[-latex] (O) -- (A) node[below] {$x_2$};
  \draw[-latex] (O) -- (B) node[above] {$t$};
  \draw[-latex] (O) -- (C) node[below] {$x_1$};

    \node[cylinder, draw, shape aspect=.5, 
      cylinder uses custom fill, cylinder end fill=green!50, 
      minimum height=1.2cm,
      cylinder body fill=green!25, opacity=0.5, 
    scale=4, rotate=-135]{};
    302
  \end{scope}

  \begin{scope}[shift={(7,0)}]
  	\node[cylinder, draw, shape aspect=.5, 
      cylinder uses custom fill, cylinder end fill=green!25,
      minimum height=1cm,
      cylinder body fill=green!25, opacity=0.5, 
    scale=4, rotate=135]{};
   \end{scope}

\end{tikzpicture}

    \end{minipage}%
    \hfill 
    \begin{minipage}[t]{0.45\textwidth} 
    \vspace{0pt}
        \centering
\label{fig:Scales}
\centering
\begin{tikzpicture}[scale=1.1]

  \begin{scope}[shift={(6,0)}]
    \coordinate (O) at (0,0,0);
    \coordinate (A) at (2,0,0);
    \coordinate (B) at (0,2,0);
    \coordinate (C) at (0,0,2);

    \draw[-latex] (O) -- (A) node[below] {$x_2$};
    \draw[-latex] (O) -- (B) node[above] {$t$};
    \draw[-latex] (O) -- (C) node[below] {$x_1$};
  \end{scope}

  \begin{scope}[shift={(9,0.5)},rotate=90,scale=1]

    \tikzset{
      tubeA/.style={
        shading=axis, left color=green!10, right color=green!55,
        shading angle=20, draw=green!70!black, line width=0.5pt, opacity=0.45
      },
      tubeB/.style={
        shading=axis, left color=green!10, right color=green!55,
        shading angle=200, draw=blue!70!black, line width=0.5pt, opacity=0.45
      },
      rimBack/.style={line width=0.3pt, draw=black, dashed, opacity=0.18},
      rimFront/.style={line width=0.4pt, draw=black, opacity=0.45},
      highlight/.style={white, line cap=round, line width=0.9pt, opacity=0.55}
    }

    \begin{scope}[on background layer]
      \foreach \t in {-1.5,-1,...,1.5} {
        \draw[rimBack] ({\t},{1.3*sin(\t r)+2}) ellipse (0.08 and 0.04);
        \draw[rimBack] ({-\t},{-1.3*sin(\t r)+3}) ellipse (0.08 and 0.04);
      }
    \end{scope}

    \shade[tubeA]
      plot[domain=-1.5:1.5, variable=\t, samples=100] ({\t},{1.3*sin(\t r)+2})
      -- plot[domain=-1.5:1.5, variable=\t, samples=100] ({-\t},{-1.3*sin(\t r)+3})
      -- cycle;

    \draw[highlight]
      plot[domain=-1.5:1.5, variable=\t, samples=120] ({\t},{1.3*sin(\t r)+2});

    \shade[tubeB]
      plot[domain=-1.5:1.5, variable=\t, samples=100] ({-\t},{1.3*sin(\t r)+2})
      -- plot[domain=-1.5:1.5, variable=\t, samples=100] ({\t},{-1.3*sin(\t r)+3})
      -- cycle;

    \draw[highlight]
      plot[domain=-1.5:1.5, variable=\t, samples=120] ({-\t},{1.3*sin(\t r)+2});

    \foreach \t in {-1.5,-1,...,1.5} {
       \draw[rimFront] ({\t},{1.3*sin(\t r)+2}) ellipse (0.08 and 0.04);
       \draw[rimFront] ({-\t},{-1.3*sin(\t r)+3}) ellipse (0.08 and 0.04);
    }

  \end{scope}

\end{tikzpicture}
        \label{fig:TransverseCurves}
    \end{minipage}
    \captionsetup{width=.9\linewidth}
    \caption{On the left: two transversely interacting wave packets for Fourier extension operators. The bottoms of the cylinders intersect the $x$-plane in a circle of radius $R^{\frac{1}{2}}$. The direction of the long-side with length $R$ is determined by the support $\theta$ in Fourier space. On the right: the transverse interaction of curved wave packets. The wave packets retain their size, but the cores become curved.}
     \label{figure:TransverseWavePackets}
\end{figure}

For translation-invariant flows $p(z,\xi) = p_0(\xi)$ and $S(t,s) = e^{-i(t-s) p_0(\nabla/i)}$ the estimates of the form \eqref{eq:RoughBilinearEstimate} become the bilinear Fourier extension estimates due to Wolff \cite{Wolff01} and Tao \cite{Tao03}. Sharp bilinear estimates for oscillatory integral operators with smooth phase functions
\begin{equation*}
T^\lambda f(x,t) = \int e^{i \phi^\lambda(x,t;\xi)} a^\lambda(x,t;\xi) f(\xi) d\xi
\end{equation*}
and transversality $\nu \sim 1$ were proved by Lee \cite{Lee2006} via linearization and following Tao's argument \cite{Tao03} (see also \cite{Bourgain1991I} for an earlier contribution). These correspond to smooth flows.

\bigskip
We are now ready to state our results, which 
we split into two scenarios depending on the 
class of symbols $p$.
We begin with non-degenerate symbols $p_1,p_2$, i.e. when the Hessian $\partial^2_\xi p_i$ is nondegenerate, and associated localization regions $\mathcal{Y}_{i} \subset \overline{\mathcal{Y}}_{i}$ as in Assumption \ref{ass:RegularitySet}.  For sufficient control over the bicharacteristics relative to the transversality $\nu$, we confine the frequency support $\pi_{\xi}(\mathcal{Y}_i)$ essentially to a $\nu$-ball. In this case we formulate the following assumptions.

\begin{assumption}[Transversality~assumptions~for~non-degenerate~Hamiltonians]
\label{AssumptionNondegenerateHamiltonian}
Assume that the symbols $p_i \in C_{z}^{1,1} C_{\xi}^3(\R^d \times \R \times \R^d)$ satisfy Assumption \ref{ass:RegularityFlow} for $i=1,2$ on the spatial scale $R \geq 1$ with transversality parameter $R^{-\frac{1}{2}+\delta_0} \leq \nu \leq 1$, and that there are sets $\mathcal{Y}_i \subseteq \overline{\mathcal{Y}}_i \subseteq T^* \R^d$ which satisfy Assumption \ref{ass:RegularitySet}. Assume that the non-degeneracy assumption is satisfied:
\begin{equation}
\label{eq:NonDegeneracyAssumption}
d_2 \geq |\det(\partial_{\xi}^2 p_i(z,\xi))| \geq d_1 > 0
\end{equation}
with $d_1,d_2$ uniform in $(z,\xi) \in \overline{\mathcal{Y}}_{i,R}$.

Let $(x_i(t),\xi_i(t))$ be two bicharacteristics governed by $p_i$ with $(x_i(0),\xi_i(0)) \in \overline{\mathcal{Y}}_i$ for $i=1,2$. Suppose that the bicharacteristics intersect at $t_I \in [-R,R]$: $x_1(t_I) = x_2(t_I)$.
We assume the following for $\Delta_v = \nabla_{\xi} p_1(x_1(t),t,\xi_1(t)) - \nabla_{\xi} p_2(x_2(t),t,\xi_2(t))$ for any $t \in [-R,R]$:
\begin{align}
\label{eq:Transversality}
| \Delta_v | &\in [\nu/2,2\nu], \\
\label{eq:NonDegeneracyTransversality}
| \langle \Delta_v, (\partial^2_{\xi} p_j(x_j(t),t,\xi_j(t))^{-1}  \Delta_v \rangle | &\in [\nu^2/2,2\nu^2].
\end{align}

\end{assumption}

\begin{remark}[Basic~transversality]
We remark that in case $\nu \sim 1$ \eqref{eq:NonDegeneracyTransversality} implies \eqref{eq:Transversality}. Secondly, if $\nu \sim 1$ holds, then it suffices to require \eqref{eq:NonDegeneracyTransversality} at $t=t_I$ as for $t \in [-R,R]$ \eqref{eq:NonDegeneracyTransversality} persists by the fundamental theorem and regularity of the flow.
\end{remark}

Moreover, by the margin condition and the bi-Lipschitz property of the flow, \eqref{eq:Transversality} and \eqref{eq:NonDegeneracyTransversality} hold also for almost intersecting bicharacteristics: 
\begin{lemma}
\label{lem:RegularityAlmostIntersectingI}
Let $(x_i(t),\xi_i(t))$ be two bicharacteristics governed by $p_i$ with \\ $(x_i(0),\xi_i(0)) \in \overline{\mathcal{Y}}_i$ for $i=1,2$ and the flows satisfy Assumption \ref{AssumptionNondegenerateHamiltonian}.
Suppose that there is $t_I \in [-R,R]$ with
    \begin{equation*}
        |x_1(t_I) - x_2(t_I)| \leq R^{\frac{1}{2}+\delta_0}.
    \end{equation*}
    Then 
\begin{align*}
| \Delta_v | &\in [\nu/16,16\nu], \\
| \langle \Delta_v, (\partial^2_{\xi} p_j(x_j(t),t,\xi_j(t))^{-1}  \Delta_v \rangle | &\in [\nu^2/16,16 \nu^2].
\end{align*}    
     hold also for the bicharacteristics $(x_i(t),\xi_i(t))$ on $t \in [-R,R]$. 
\end{lemma}
\begin{proof}
    In the first step we show that $(x_1(t_I),t_I,\xi_2(t_I))$ is in the flow-out of $\overline{\mathcal{Y}}_2$ under $p_2$.
    To this end, we argue by contradiction. We link $(x_2(t_I),t_I,\xi_2(t_I))$ and $(x_1(t_I),t_I,\xi_2(t_I))$ with a straight line, parametrized by $\gamma$. Since $\overline{\mathcal{Y}}_{2,R}$ is open, we have
    $0<t_* = \inf ( \{ t \in [0,1] : \gamma(t) \not\in \overline{\mathcal{Y}}_{2,R} \})$. However, for $0<t<t_*$ we can apply Lemma \ref{lem:GeometryWavepackets} to argue that the pullback $(x_2'(0),\xi_2'(0))$ along $p_2$ to $t=0$ satisfies $|x_2'(0) - x_2(0)| \leq \overline{C}_2 R^{\frac{1}{2}+\delta_0}$ and $|\xi_2'(0)-\xi_2(0)| \leq \overline{C}_2 R^{-\frac{1}{2}+\delta_0}$. The estimate is uniform up to $t_*$, which contradicts the margin condition.

    Let the intersecting bicharacteristic be denoted by $T_2'$. Another application of Lemma \ref{lem:GeometryWavepackets} shows that $T_2$ and $T_2'$ are at any time $t \in [-R,R]$ at most separated by $\overline{C}_2 R^{\frac{1}{2}+\delta_0}$ in space. The claim then follows from perturbing \eqref{eq:Transversality} and \eqref{eq:NonDegeneracyTransversality} by a Taylor expansion and regularity of the flow.
\end{proof}

Our main bilinear estimate in the scenario of non-degenerate Hamiltonian flows
is as follows:

\begin{theorem}[Sharp~bilinear~estimates~for~non-degenerate~flows]
\label{thm:GeneralizationBilinearParaboloid}
Let $R \geq 1$ and $\epsilon > 0$. Assume that $p_i \in C_{z}^{1,1} C^3_{\xi}(\R^d \times \R \times \R^d)$ satisfies Assumption \ref{AssumptionNondegenerateHamiltonian} for sets $\mathcal{Y}_i$ for $i=1,2$. We assume that the flows $p_i$ admit wave packet decompositions, which satisfy Assumption \ref{ass:WP}, and there are compatible two-parameter groups $(S_i(t,s))_{t,s \in [-R,R]}$ such that Assumption \ref{ass:Frames} is satisfied.

Then the bilinear estimate 
\begin{equation}\label{eq:bi-nondeg}
\big\| S_1(t,0) (\chi_{\mathcal{Y}_1} f_1) \, S_2(t,0) (\chi_{\mathcal{Y}_2} f_2) \big\|_{L^{\frac{d+3}{d+1}}(Q_R)} \leq C_\varepsilon R^{\varepsilon} \nu^{-\frac{2}{d+3}} \| f_1 \|_{L^2} \| f_2 \|_{L^2}
\end{equation}
holds true with implicit constants only depending on $C_2$, $\epsilon_0$, and $d_i$ defined in \eqref{eq:MixedRegularity}, \eqref{eq:XiRegularity}, and \eqref{eq:NonDegeneracyAssumption}, $\delta_0=\delta_0(\varepsilon)$ and the constants in the wave packet axioms Assumption \ref{ass:WP} for $\delta = \delta_0(\varepsilon)$ and $m=m(\varepsilon)$.
\end{theorem}
We note that in the constant coefficient case, e.g. for the self-interactions of a single Schr\"odinger flow, the transversality parameter $\nu$ plays the role of a scaling parameter, and the $\nu$ dependence of the implicit constant comes from scaling. This indicates that 
the power of $\nu$ above is sharp in general.

\smallskip

We now turn our attention to the second scenario, and formulate a corresponding theorem for $1$-homogeneous symbols $p_1,p_2$. In this case we require the frequencies to be separated from the origin:
\begin{assumption}[Transversality~assumptions~for~$1$-homogeneous~Hamiltonians]
\label{AssumptionNondegenerateHamiltonianHomogeneous}
Let $R \geq 1$ and $p_i \in C_{z}^{1,1} C^3_{\xi}(\R^d \times \R \times \R^d)$  for $i=1,2$ satisfy Assumption \ref{ass:RegularityFlow} in $\overline{\mathcal{Y}}_{i,R}$. Assume that $p_i$ are $1$-homogeneous in $\xi$, i.e., for $\lambda > 0$ it holds $p_i(x,t,\lambda \xi) = \lambda p_i(x,t,\xi)$.

Assume that for $(z,\xi) \in \overline{\mathcal{Y}}_{i,R}$  there are $d-1$ eigenvectors $w_1,\ldots,w_{d-1}$ for $A = \partial^2_{\xi} p_j(z,\xi)$ with eigenvalues $\lambda_1,\ldots,\lambda_{d-1}$ which satisfy $0< d_1 \leq |\lambda_i| \leq d_2$ \emph{(non-degeneracy assumption)}.

\smallskip

Let $(x_i(t),\xi_i(t))$ be two bicharacteristics governed by $p_i$ with $(x_i(0),\xi_i(0)) \in \overline{\mathcal{Y}}_i$, $i=1,2$. Suppose that the bicharacteristics  intersect for some $t_I \in [-R,R]$: $x_1(t_I) = x_2(t_I)$. 
We assume the following for $\Delta_v = \nabla_{\xi} p_1(x_1(t),t,\xi_1(t)) - \nabla_{\xi} p_2(x_2(t),t,\xi_2(t))$ for $j \in \{1,2\}$:
\begin{align}
\label{eq:BasicTransversalityHomogeneous}
|\Delta_v| \in [\nu/2,2\nu], \\
\label{eq:NonDegeneracyTransversalityHomogeneous}
| \langle \xi_j^*, \Delta_v \rangle | \geq \nu \geq R^{-\frac{1}{2}+\delta_0} >0.
\end{align}
\end{assumption}

We have the following analog of Lemma \ref{lem:RegularityAlmostIntersectingI} about the regularity of almost intersecting bicharacteristics whose proof is omitted to avoid repetition.

\begin{lemma}
Let $(x_i(t),\xi_i(t))$ be two bicharacteristics governed by $p_i$ with \\ $(x_i(0),\xi_i(0)) \in \mathcal{Y}_i$ for $i=1,2$, and the flows satisfy Assumption \ref{AssumptionNondegenerateHamiltonianHomogeneous}. Suppose that there is $t_I \in [-R,R]$ with 
\begin{equation*}
|x_1(t_I) - x_2(t_I)| \leq R^{\frac{1}{2}+\delta_0}.
\end{equation*}
Then \eqref{eq:BasicTransversalityHomogeneous} and \eqref{eq:NonDegeneracyTransversalityHomogeneous} hold also for the bicharacteristics $(x_i(t),\xi_i(t))$ on $t \in [-R,R]$.

\end{lemma}

Now we can state our main result in this case:

\begin{theorem}[Sharp~bilinear~estimates~for~$1$-homogeneous~flows]
\label{thm:GeneralizationBilinearCone}
Let $R \geq 1$ be a spatial scale, $\epsilon > 0$ and suppose that $p_i \in C_{z}^{1,1} C^3_{\xi}(\R^d \times \R \times \R^d)$ satisfies Assumption \ref{AssumptionNondegenerateHamiltonianHomogeneous} for $i=1,2$.
Assume that the flows $p_i$ admit wave packet decompositions as described in Assumption \ref{ass:WP}, and there are compatible two-parameter group $(S_i(t,s))_{t,s \in [-R,R]}$ such that Assumption \ref{ass:Frames} is satisfied.

Then the bilinear estimate \eqref{eq:bi-nondeg} 
holds true with implicit constants only depending on $C_2$, $\epsilon_0$, and $d_i$ defined in \eqref{eq:MixedRegularity}, \eqref{eq:XiRegularity}, and \eqref{eq:NonDegeneracyTransversalityHomogeneous}, $\delta_0=\delta_0(\varepsilon)$ and the constants in the wave packet axioms Assumption \ref{ass:WP} for $\delta \geq \delta_0(\varepsilon)$ and $m=m(\varepsilon)$.
\end{theorem}

We end the section with remarks on the geometric interpretation of the transversality assumptions and the regularity assumption for small transversality.

\begin{remark}[Geometric~interpretation~of~transversality~assumptions]
We can as well consider the mixed case of $1$-homogeneous flows and non-degenerate flows: Suppose that $p_1$ is non-degenerate and $p_2$ is $1$-homo\-geneous. Then, in addition to the quantitative transversality assumption \eqref{eq:Transversality} we have to require that \eqref{eq:NonDegeneracyTransversality} holds for $j=1$ and \eqref{eq:NonDegeneracyTransversalityHomogeneous} holds for $j=2$. The conditions reflect the following geometric condition on the energy shell
given by
\begin{equation*}
\mathcal{E}^z_{\xi_1,\xi_2'} = \{ \eta \in \R^d : p_1(z,\xi_1) + p_2(z,\eta+\xi_2'-\xi_1) - p_1(z,\eta) - p_2(z,\xi_2') = 0 \}.
\end{equation*}
Let $A = \partial^2_{\xi} p_1(z,\eta)$ describe the leading order variation of the Hamiltonian flow for a fixed base point along frequencies. 
The transversality conditions \eqref{eq:NonDegeneracyTransversality} and \eqref{eq:NonDegeneracyTransversalityHomogeneous} become instances of
\begin{equation*}
\text{dist}(A \, T_{\eta} \mathcal{E}^z_{\xi_1,\xi_2'}, \nabla_{\xi} p_1(z,\xi_1) - \nabla_{\xi} p_2(z,\xi_2) ) \gtrsim \nu > 0.
\end{equation*}
The regularity assumptions with $\epsilon_0 \leq c \nu \ll \nu$ allow us to vary the parameters $z$ and $\xi_i$ without changing the condition. We refer to works by Lee--Vargas \cite{LeeVargas2010}, Bejenaru \cite{Bejenaru2017}, and Candy \cite{Candy2019} in the case of constant coefficients.
\end{remark}

\begin{remark}[Transversality~vs.~regularity]
The regularity assumption \eqref{eq:MixedRegularity} with $\epsilon_0 \leq c \nu$ for small transversality parameter $\nu \ll 1$ seems to be a restrictive assumption. We note that in case of self-interactions $p_1 = p_2=p$ for wave $p= \sqrt{g^{ij} \xi_i \xi_j}$ or Schr\"odinger flows $p = g^{ij} \xi_i \xi_j$ with uniformly elliptic $(g^{ij})$ the transversality condition on the initial frequencies corresponds to $\nu$-separation for Schr\"odinger flows and to $\nu$-angle separation for wave flows. In the Schr\"odinger case, after a Galilean transform and  parabolic rescaling we recover unit separation of the frequencies. Correspondingly, in the wave case, after a Lorentz transformation and rescaling, we recover unit angular separation. The regularity assumptions for the rescaled flows require careful inspection, and the full discussion is postponed to Section \ref{subsection:SelfInteractions}.
\end{remark}

\subsection{Overview of parameters and constants}

\begin{itemize}
    \item $R \geq 1$ denotes a (typically large) spatial parameter. We have explained above how the spatial parameter enters the picture in our applications when we consider frequency localization $|\xi| \sim N$ at the outset. Below $1 \leq r \leq R$ describes an intermediate scale to carry out an induction on scales.
    \item 
    $\nu \in [R^{-\frac{1}{2}+\delta_0},1]$, $0 < \delta_0 < 1/4$ to be chosen later, denotes a transversality parameter, which describes the angle of two intersecting bicharacteristics from the Hamiltonian flows $p_i$, $i=1,2$, for which we aim to prove the bilinear estimate. $c < 1$ denotes a universal small constant in Assumption \ref{ass:RegularityFlow}, which gives the required control over the Hamiltonian flow up to times $|t| \leq R$ (cf. Lemma \ref{lem:GeometryWavepackets}). $C \geq 1$ denotes a universal large constant in Assumptions \ref{ass:RegularityFlow}, \ref{ass:RegularitySet}. We choose $C=c^{-\frac{1}{2}}$ to simplify matters. The regularity parameter $\epsilon_0$ in Assumption \ref{ass:RegularityFlow} is given by $\epsilon_0 \leq c \nu$. We can suppose that this is the same universal constant as in Assumption \ref{ass:RegularityFlow}.
    \item We will choose the ``leeway" parameter $\delta_0 >0$ depending on the parameter $\varepsilon > 0$ in Theorems \ref{thm:GeneralizationBilinearParaboloid} and \ref{thm:GeneralizationBilinearCone}. For $0<\varepsilon<1/4$, $\delta_0 = \varepsilon^2/(10Cd)$ suffices for the purpose of the argument.
    \item In Theorems \ref{thm:GeneralizationBilinearParaboloid}, \ref{thm:GeneralizationBilinearCone} the choice of $\varepsilon > 0$ at the outset of the argument allows us to fix $\delta_0 > 0$. With $\delta_0 > 0$ fixed, we will use the implicit constants appearing in the assumptions on the wave packets \ref{ass:WP} for $\delta = \delta_0$, $m = m(\delta_0)$, for which reason the constants in the wave packet axioms ultimately depend on $\varepsilon$. This is permissible.
\end{itemize}

\section{Proof of the bilinear estimates for rough Hamiltonian flows}
\label{section:ProofBilinear}

We rely on Tao's argument \cite{Tao03} (see also \cite{Lee2006}) to show Theorems \ref{thm:GeneralizationBilinearParaboloid} and \ref{thm:GeneralizationBilinearCone}. The proofs follow the same general strategy:
\begin{enumerate}
\item We induct on the space-time cubes of size $C \nu^{-2-\delta_0} \leq r \leq R$ and divide the tubes into focusing and non-focusing part on the coarse scale $r^{1-\delta}$. 
Roughly speaking, the highly concentrated coarse scales, where the tubes hit the most $r^{\frac{1}{2}}$-cubes after some pigeonholing arguments, can be estimated directly via induction on scales. This takes up Subsections \ref{subsection:Induction}-\ref{subsection:FineScaleDecomposition}.
\item It remains to estimate the non-focusing contribution. Here we rely on a biorthogonality estimate. It is obtained from a double counting argument, which differs in the non-degenerate and $1$-homogeneous case. The combinatorial argument is established in Subsection \ref{subsection:CombinatorialEstimate}.
\end{enumerate}

\subsection{Induction on scales and coarse-scale decomposition}
\label{subsection:Induction}
To show the theorems, we induct on $D_{\varepsilon} \nu^{-2-\delta_0} \leq r \leq R$. $D_{\varepsilon} \geq C$ will be a (large) constant depending on $\epsilon > 0$. We carry out the induction on the uniform bilinear estimates for Hamiltonians satisfying the regularity assumption \ref{ass:RegularityFlow} with a parameter $R$, which are associated with the two-parameter groups $(S_i(t,s))_{t,s \in [-r,r]}$ and associated wave packet decompositions. 

\emph{Induction Assumption}: For $C \nu^{-2-\delta_0} \leq r \leq R/2$ 
we have for a cube $Q_r$ of size $r$, with time center $t_i$, and $Q_r \subseteq Q_R$ and free solutions 
\begin{equation}
\label{eq:InductionHypothesis}
\| \prod_{j=1}^2 S_j(t,t_i) (\chi_{\overline{\mathcal{Y}}_{r,t_i,j}}(x,D) f_j) \|_{L^{\frac{d+3}{d+1}}_{t,x}(Q_r)} \leq C_\varepsilon r^{\varepsilon} \nu^{-\frac{2}{d+3}} \| f_1 \|_{L^2} \| f_2 \|_{L^2}.
\end{equation}
We shall show in the \emph{induction step} that this implies the above for cubes up to size $2r$. Before we carry out the induction by establishing the \emph{base case} at $r = C \nu^{-2-\delta_0}$ and the \emph{induction step} to propagate the estimate up to $r=R$, we note how the estimate at $r=R$ yields the original claim of Theorems \ref{thm:GeneralizationBilinearParaboloid} and \ref{thm:GeneralizationBilinearCone} by the definition of $\chi_{\bar{\mathcal{Y}}}(x,D)$. 

\smallskip

We turn to the \emph{base case}, which is dictated by the minimum scale for the validity of wave packet decompositions $r = C \nu^{-2-\delta_0}$ and still obtaining an admissible error for the wave packet decomposition.
Indeed, this matches that the frequencies are roughly supported on a ball of size $\nu$ (or in the homogeneous case on a slab of size $\nu$). In the non-degenerate case we have $\# (\Lambda_r \cap \overline{\mathcal{Y}}_{i,r}) = \mathcal{O}(\nu^{-d \delta_0})$. By our choice of $\delta_0 = c \varepsilon$ we can suppose that the initial frequency remains unchanged, as this incurs only an error of $r^{\varepsilon}$. In the homogeneous case there are possibly more lattice points, what matters is that the direction of the frequency does not change essentially. 

By Assumption \ref{ass:Frames} we have the wave packet expansion
\begin{equation*}
S_j(t,t_i) \big( \chi_{\overline{\mathcal{Y}}_{r,t_i,j}}(x,D) f_j \big) = \sum_{T_j \in \T_j \subseteq \Lambda_r \cap \overline{\mathcal{Y}}_{r,t_i,j}} a_{T_j} f_j + S_j(t,t_i) g_j
\end{equation*}
with an error that is negligible on $Q_r$.

We pigeonhole the amplitude of the wave packets, which incurs a logarithmic loss in $r$. This follows from assuming $\|f \|_2 = 1$ by linearity and recalling that $\# \T_j \leq r^{10d}$ by Assumption \ref{ass:Frames}. Next, we recall that the $\ell^2$-norm of the amplitude coefficients is bounded by the $L^2$-norm of the initial data, and it is enough to show the following estimate for wave packets on the $r$-scale:
\begin{equation*}
   \big\| \sum_{T_1 \in \T_1} \phi_{T_1} \sum_{T_2 \in \T_2} \phi_{T_2} \big\|_{L^{\frac{d+3}{d+1}}_{t,x}(Q_r)} \lesssim_{\varepsilon} r^{\varepsilon} \nu^{-\frac{2}{d+3}} \big( \# \T_1)^{\frac{1}{2}} \big( \# \T_2 \big)^{\frac{1}{2}}.
\end{equation*}
This will be proved via interpolating between the energy estimate:
\begin{equation*}
    \big\| \sum_{T_1 \in \T_1} \phi_{T_1} \sum_{T_2 \in \T_2} \phi_{T_2} \big\|_{L^1_{t,x}(Q_r)} \lesssim_{\varepsilon} r^{1+\varepsilon} \big( \# \T_1)^{\frac{1}{2}} \big( \# \T_2 \big)^{\frac{1}{2}},
\end{equation*}
which follows from Lemma \ref{lem:AlmostOrthogonalityWavePackets}, after applying H\"older's inequality in time.
Secondly, we show the following based on biorthogonality:
\begin{equation*}
    \big\| \sum_{T_1 \in \T_1} \phi_{T_1} \sum_{T_2 \in \T_2} \phi_{T_2} \big\|_{L^2_{t,x}(Q_r)}^2 \lesssim \nu^{-1} r^{C \delta} r^{- \frac{d-1}{2}} \big( \# \T_1 \big) \big( \# \T_2 \big).
\end{equation*}
We write the left-hand side as
\begin{equation*}
    \big\| \sum_{T_1 \in \T_1} \phi_{T_1} \sum_{T_2 \in \T_2} \phi_{T_2} \big\|_{L^2_{t,x}(Q_r)}^2 = \sum_{\substack{T_1,T_1', \\ T_2,T_2'}} \iint \phi_{T_1} \phi_{T_2} \overline{\phi_{T_1'}} \overline{\phi_{T_2'}} dx dt.
\end{equation*}
Since the initial frequency (or in the homogeneous case: the initial frequency direction) is always the same, for tubes $T_1$, $T_2$ there are only $\lesssim r^{C \delta}$ tubes $T_1'$ and $\lesssim r^{C \delta}$ tubes $T_2'$ such that 
\begin{equation*}
    \iint \phi_{T_1} \phi_{T_2} \overline{\phi_{T_1'}} \overline{\phi_{T_2'}} dx dt
\end{equation*}
is not rapidly decreasing in $r$. Indeed, as a consequence of Lemma \ref{lem:GeometryWavepackets}, the above expression can only be essentially contributing provided that the initial positions of $T_1$, $T_1'$ and $T_2$, $T_2'$ are not apart by $\gtrsim r^{\frac{1}{2}+\delta}$. Here we use that by the Hamiltonian flow equations the frequency does not change essentially and the bi-Lipschitz property of the flow.

By an application of Cauchy-Schwarz, we find
\begin{equation*}
\begin{split}
   \sum_{\substack{T_1,T_1', \\ T_2,T_2'}} \iint \phi_{T_1} \phi_{T_2} \overline{\phi_{T_1'}} \overline{\phi_{T_2'}} dx dt &\leq r^{C \delta} \sum_{T_1,T_2} \iint |\phi_{T_1}|^2 |\phi_{T_2}|^2 dx dt \\ 
   &\quad + \text{RapDec}(r) \big( \# \T_1 \big) \big( \# \T_2 \big).
   \end{split}
\end{equation*}
We can now carry out the estimate by computing the maximum essential overlap of $\phi_{T_1}$, $\phi_{T_2}$ as $r^{\delta} \big( (r^{\frac{1}{2}})^d \times (\nu^{-1} r^{\frac{1}{2}}) \big)$, using the transversality, and taking into account the amplitude size bounded by $r^{-\frac{d}{4}+C \delta}$:
\begin{equation*}
    \iint |\phi_{T_1}|^2 |\phi_{T_2}|^2 dx dt \lesssim r^{-d} \nu^{-1} r^{\frac{1}{2}} r^{\frac{d}{2}} r^{C \delta}.
\end{equation*}
This finishes the proof of the base case.

\medskip

We turn to the induction step $r /2 \to r$, establishing \eqref{eq:InductionHypothesis} on a cube of size $r$, taking the estimate \eqref{eq:InductionHypothesis} at $r/2$ for granted.  
We rely again on Assumption \ref{ass:Frames} yielding the wave packet decomposition
\begin{equation*}
S_j(t,t_i) \big( \chi_{\overline{\mathcal{Y}}_{r,t_i,j}}(x,D) f_j \big) = \sum_{T_j \in \T_j \subseteq \Lambda_r \cap \overline{\mathcal{Y}}_{r,t_i,j}} a_{T_j} f_j + S_j(t,t_i) g_j
\end{equation*}
with an error that is negligible on $Q_r$ and pigeonhole the amplitudes $|a_{T_j}| \sim \alpha$, which incurs only a logarithmic loss.

\smallskip

Decompose $Q_r = \bigcup Q_{i,r^{1-\delta}}$ modulo zero measure sets into a disjoint family of cubes with sidelength $r^{1-\delta}$. The collection is denoted by $\mathcal{S}_{1-\delta}$. 

\smallskip

We have by Minkowski's inequality:
\begin{equation*}
\big\| \sum_{T_1} \phi_{T_1} \sum_{T_2} \phi_{T_2} \big\|_{L^{\frac{d+3}{d+1}}(Q_r)} \leq \sum_{S \in \mathcal{S}_{1-\delta}} \big\| \sum_{T_1} \phi_{T_1} \sum_{T_2} \phi_{T_2} \big\|_{L^{\frac{d+3}{d+1}}(S)}. 
\end{equation*}
For tubes $T_i \in \T_i$ and cubes $S \in \mathcal{S}_{1-\delta}$, through a series of pigeonholing arguments, we will below define a relation $T_i \sim S$ that broadly says that $T_i$ is concentrated on $S$. We split the highly concentrated contribution of tubes $T_i \sim S$ (see Definition~\ref{def:FocusingRelation}) and the contribution of tubes, which are not concentrated on $S$:
\begin{equation}
\label{eq:SplittingConcentration}
\begin{split}
&\quad \big\| \sum_{T_1 \in \T_1} \phi_{T_1} \sum_{T_2 \in \T_2} \phi_{T_2}  \big\|_{L^{\frac{d+3}{d+1}}(S)} \leq \underbrace{\big\| \sum_{T_1 \sim S} \phi_{T_1} \sum_{T_2 \sim S} \phi_{T_2} \big\|_{L^{\frac{d+3}{d+1}}(S)}}_{(I)} \\
&\quad + \underbrace{\big\| \sum_{T_1 \not\sim S} \phi_{T_1} \sum_{T_2 \sim S} \phi_{T_2} \big\|_{L^{\frac{d+3}{d+1}}(S)}}_{(II)} + \underbrace{\big\| \sum_{T_1 \sim S} \phi_{T_1} \sum_{T_2 \not\sim S} \phi_{T_2} \big\|_{L^{\frac{d+3}{d+1}}(S)}}_{(III)} \\
&\quad + \underbrace{\big\| \sum_{T_1 \not\sim S} \phi_{T_1} \sum_{T_2 \not\sim S} \phi_{T_2} \big\|_{L^{\frac{d+3}{d+1}}(S)}}_{(IV)}.
\end{split}
\end{equation}
The contribution of $(I)$ is estimated by induction on scales and the key property of the focusing relation:
\begin{equation}
\label{eq:EssentialDisjointConcentration}
\forall T_i \in \T_i: \# \{ S \in \mathcal{S}_{1-\delta} : T_i \sim S \} \lesssim \log(r)^3.
\end{equation}
Let $t_S$ be the time-center of $S$. Applying the induction assumption to estimate $(I)$ at scale $r^{1-\delta}$ deserves clarification as we only have it at scale $r^{1-\delta}$ for products of functions $S_j(t,t_S) \chi_{\overline{\mathcal{Y}}_{r^{1-\delta},t_S,j}}(x,D) f_j$ localized with the PDO $\chi_{\overline{\mathcal{Y}}_{r^{1-\delta},t_S,j}}(x,D)$. By the assumption on the phase space localization of the wave packets we have
\begin{equation}
\label{eq:WavePacketMicroLocalization}
\| (1- \chi_{\overline{\mathcal{Y}}_{r^{1-\delta},t_S,j}}(x,D)) \phi_T(t_S) \|_{L^2} \lesssim_{\delta} \text{RapDec}(r)
\end{equation}
for $(x_{T}(t_S),\xi_{T}(t_S))$ contained in the $C (R r^{-\frac12+\delta_0},r^{-\frac{1}{2}+\delta_0})$ thickening of \\ ${\mathcal{Y}}_{r,t_S,j}$.
This condition is satisfied for the wave packets initiated at $t_i$ as above, by the bi-Lipschitz property of the flow; see also Subsection \ref{subsection:WavePacketDecompositions}.

By the polynomial bound on $\# \T_j$ we have in $L^2$ for $t \in t_S + [-r^{1-\delta},r^{1-\delta}]$ and the bound \eqref{eq:WavePacketMicroLocalization}:
\begin{equation*}
\sum_{T_j \in \T_j} \phi_{T_j}(t) = S_j(t,t_S) \chi_{\overline{\mathcal{Y}}_{r,t_S,j}}(x,D) \big( \sum_{T_j \in \T_j} \phi_{T_j}(t_S) \big) + \text{RapDec}(r).
\end{equation*}
In summary, we have
\begin{equation*}
\begin{split}
\big\| \sum_{T_1 \sim S} \phi_{T_1} \sum_{T_2 \sim S} \phi_{T_2} \big\|_{L^{\frac{d+3}{d+1}}(S)} &\leq \big\| \prod_{j=1}^2 S_j(t,t_S) \chi_{\overline{\mathcal{Y}}_{r^{1-\delta},t_S,j}}(x,D) \big( \sum_{T_j \sim S} \phi_{T_j}(t_S) \big) \big\|_{L^{\frac{d+3}{d+1}}(S)}  \\ &\quad + \text{RapDec}(r).
\end{split}
\end{equation*}
This can now be estimated by the induction hypothesis and Lemma \ref{lem:AlmostOrthogonalityWavePackets}:
\begin{equation*}
\begin{split}
&\quad \big\| \prod_{j=1}^2 S_j(t,t_S) \chi_{\overline{\mathcal{Y}}_{r^{1-\delta},t_S,j}}(x,D) \big( \sum_{T_i \sim S} \phi_{T_j}(t_S) \big) \big\|_{L^{\frac{d+3}{d+1}}(S)} \\
 &\leq C_\varepsilon r^{\varepsilon(1-\delta)} \nu^{-\frac{2}{d+3}} C_{\delta'} r^{\delta'} \prod_{j=1}^2 \# \{ \T_j \sim S \}^{\frac{1}{2}}.
\end{split}
\end{equation*}

\medskip

Coming back to \eqref{eq:SplittingConcentration}: Using the induction hypothesis, Cauchy-Schwarz, changing the order of summation and \eqref{eq:EssentialDisjointConcentration}, we find up to rapidly decaying errors:
\begin{equation*}
\begin{split}
(I) &\lesssim_{\varepsilon,\delta'} r^{(1-\delta)\varepsilon+\delta'}  \nu^{-\frac{2}{d+3}} \sum_{S \in \mathcal{S}_{1-\delta}} \# \{ T_1 \in \T_1 : T_1 \sim S \}^{\frac{1}{2}} \# \{ T_2 \in \T_2 : T_2 \sim S \}^{\frac{1}{2}} \\
&\lesssim_{\varepsilon,\delta'} r^{(1-\delta)\varepsilon+\delta'} \nu^{-\frac{2}{d+3}} \prod_{i=1}^2 \big( \sum_{S \in \mathcal{S}_{1-\delta}} \# \{ T_i \in \T_i : T_i \sim S \} \big)^{\frac{1}{2}} \\
&\lesssim_{\varepsilon,\delta'} r^{(1-\delta)\varepsilon+\delta'} \nu^{-\frac{2}{d+3}} \prod_{i=1}^2 \big( \sum_{T_i \in \T_i} \# \{ S \in \mathcal{S}_{1-\delta} : T_i \sim S \} \big)^{\frac{1}{2}} \\
&\lesssim_{\varepsilon,\delta'} r^{(1-\delta) \varepsilon+\delta'} \nu^{-\frac{2}{d+3}} \log(r)^3 (\# \T_1)^{\frac{1}{2}} (\# \T_2)^{\frac{1}{2}}.
\end{split}
\end{equation*}
So, induction closes for the highly concentrated part as a consequence of the property \eqref{eq:EssentialDisjointConcentration} of the focusing relation, choosing $\delta'$ small and $r$ large enough.
Consequently, choosing $\delta'=\varepsilon \delta/2$ and by $r \geq D_{\varepsilon}$ such that $C_{\delta'} r^{- \delta \varepsilon + \delta'} C \log(r)^3 \leq 1/2$ induction closes. Recall that $\delta = \delta(\varepsilon)$, for which reason indeed $D=D_\varepsilon$.

In the following we again denote $r$ by $R$ to ease notation. It will suffice to prove an estimate for the non-concentration part $(II)$:
\begin{equation*}
\big\| \sum_{\substack{T_1 \in \T_1, \\ T_1 \not\sim S}} \phi_{T_1} \sum_{T_2 \in \T_2} \phi_{T_2} \big\|_{L^{\frac{d+3}{d+1}}(S)} \lesssim_{\delta} R^{C \delta} \nu^{-\frac{2}{d+3}} (\# \T_1 )^{\frac{1}{2}} (\#  \T_2 )^{\frac{1}{2}}. 
\end{equation*}
By symmetry it suffices to consider $(II)$; the estimate of $(III)$ is carried out analogously. $(IV)$ can be estimated either like $(II)$ or $(III)$. The dependence of the implicit constant on $\delta$ (i.e. $\varepsilon$ and hence an acceptable dependence) will be omitted in the following to lighten the notation.

\subsection{Pigeonholing and the focusing relation}
\label{subsection:Pigeonholing}

We turn to the estimate of the non-concentrated parts $(II)$ and $(III)$. By symmetry it suffices to estimate $(II)$. 
To make sense of these terms, we need  
to define the focusing relation, which will be achieved via several pigeonholings.

\medskip

Decompose $Q_R$ into essentially disjoint cubes $q$ of side length $R^{\frac{1}{2}}$, the collection being denoted by $\mathfrak{q}$. Now let for $1 \leq \mu_1,\mu_2 \leq R^{100d}$, $\mu_i \in 2^{\N_0}$:
\begin{equation*}
\mathfrak{q}[\mu_1,\mu_2] = \{ q \in \mathfrak{q} : \# \{ T_i \in \T_i : q \cap R^\delta T_i \neq \emptyset \} \sim \mu_i \}.
\end{equation*}
Next we pigeonhole tubes from $\T_1$. For $\lambda_1 \in 2^{\N_0}$, $\lambda_1 \leq R^{100d}$ let
\begin{equation*}
\T_1[\lambda_1,\mu_1,\mu_2] = \{ T_1 \in \T_1 : \# \{ q \in \mathfrak{q}[\mu_1,\mu_2] : R^\delta T_1 \cap q \neq \emptyset \} \sim \lambda_1 \}.
\end{equation*}

\begin{definition}[The relation $\sim$ between tubes and cubes]
\label{def:FocusingRelation}
Let $S \in \mathcal{S}_{1-\delta}$ and for $T_1 \in \T_1[\lambda_1,\mu_1,\mu_2]$ we let $T_1 \sim_{\lambda_1,\mu_1,\mu_2} S$ if the number of cubes $q \in \mathfrak{q}[\mu_1,\mu_2]$, $q \subseteq S$ with $R^{\delta} T_1 \cap  q \neq \emptyset$ is maximal. If there are several $S \in \mathcal{S}_{1-\delta}$, choose any one of them. With $S$ fixed, we let moreover $T_1 \sim_{\lambda_1,\mu_1,\mu_2} S'$ for $S$ and $S'$ neighbouring cubes.
 We let $T_1 \sim S$ if $T_1 \sim_{\lambda_1,\mu_1,\mu_2} S$ for some $\lambda_1,\mu_1,\mu_2$. 
\end{definition}
 
\begin{remark} 
The relation $\sim$ satisfies the counting relation \eqref{eq:EssentialDisjointConcentration} because there are $\lesssim \log(R)^3$ choices for dyadic numbers $1 \leq \lambda_1,\mu_1,\mu_2 \leq R^{100d}$.
\end{remark}

\smallskip

It suffices to estimate on $Q_R$ 
\begin{equation}
\label{eq:pruned}
\begin{split}
   &\quad \big\| \sum_{\substack{ T_1 \in \T_1, \\ T_1 \not\sim \T_1} } \phi_{T_1} \sum_{T_2 \in \T_2} \phi_{T_2} \big\|_{L^{\frac{d+3}{d+1}}(S)} \\
   &\leq \sum_{1 \leq \mu_1, \mu_2, \lambda_1 \leq R^{100d}} \big\| \sum_{ \substack{T_1 \in \T_1[\lambda_1,\mu_1,\mu_2], \\ T_1 \not\sim T_1}} \phi_{T_1} \sum_{T_2} \phi_{T_2} \big\|_{L^{\frac{d+3}{d+1}}(\bigcup \mathfrak{q}[\mu_1,\mu_2] \cap S )}.
\end{split}
\end{equation}

For $T_1 \in \T_1[\lambda_1,\mu_1,\mu_2]$ with $T_1 \sim_{\lambda_1,\mu_1,\mu_2} S$
we have by averaging
\begin{equation}
\label{eq:AveragingEstimate}
\# \{ q \in \mathfrak{q}[\mu_1,\mu_2] : q \subseteq S, \quad q \cap R^\delta T_1 \neq \emptyset \} \gtrsim \lambda_1 R^{-C \delta}.
\end{equation}

We turn to the estimate of the non-focusing part after pigeonholing, i.e., fixing the maximizing dyadic values $\lambda_1, \mu_1, \mu_2$ in \eqref{eq:pruned}, which incurs a permissible logarithmic loss:
\begin{equation}
\label{eq:NonfocusingEstimate}
\big\| \sum_{\substack{T_1 \not\sim S, \\ T_1 \in \T_1[\lambda_1,\mu_1,\mu_2]}} \phi_{T_1} \sum_{T_2} \phi_{T_2} \big\|_{L^{\frac{d+3}{d+1}}(\mathcal{Q} = \bigcup \mathfrak{q}[\mu_1,\mu_2] \cap S)} \lesssim R^{C \delta} (\# \T_1)^{\frac{1}{2}} (\# \T_2)^{\frac{1}{2}}.
\end{equation}

The estimate \eqref{eq:NonfocusingEstimate} will follow from interpolating the energy estimate and an estimate obtained from biorthogonality. The energy estimate does not depend on the pigeonholing and is given by:
\begin{equation}
\label{eq:EnergyEstimate}
\big\| \sum_{T_1 \in \T_1} \phi_{T_1} \sum_{T_2} \phi_{T_2} \big\|_{L^1(S)} \lesssim R^{1-\delta} R^{\delta} (\# \T_1)^{\frac{1}{2}} (\# \T_2)^{\frac{1}{2}}.
\end{equation}
This is based on the almost orthogonality of the wave packets at $t > 0$ by Lemma \ref{lem:AlmostOrthogonalityWavePackets}, after applying H\"older's inequality.

\smallskip

The second estimate is a consequence of biorthogonality, whose proof relies on the pigeonholing and focusing relation: 
\begin{equation}
\label{eq:Biorthogonality}
\begin{split}
\big\| \sum_{\substack{T_1 \not\sim S, \\ T_1 \in \T_1[\lambda_1,\mu_1,\mu_2]}} \phi_{T_1} \sum_{T_2} \phi_{T_2} \big\|^2_{L^2(\mathcal{Q})} &\leq \sum_{q \in \mathfrak{q}[\mu_1,\mu_2]\cap S} \sum_{T_1,T_2,T_1',T_2'} \iint_q \chi_q^4 \phi_{T_1} \phi_{T_2} \overline{\phi_{T_1'}} \overline{\phi_{T_2'}} dx dt \\
&\lesssim \nu^{-1} R^{10 C \delta} R^{-\frac{d-1}{2}} (\# \T_1) (\# \T_2).
\end{split}
\end{equation}
where $\chi_q$ denotes a smooth bump function adapted to $q \in \mathfrak{q}$, which we assume for simplicity to be of product type, such that
\begin{equation*}
\chi_q \geq c > 0 \text{ on } q, \quad \text{supp}(\chi_q) \subseteq 2 q, \text{ and } \sum_{q \in \mathfrak{q}} \chi_q^4 \equiv 1.
\end{equation*}

Interpolating \eqref{eq:EnergyEstimate} with \eqref{eq:Biorthogonality} at $p=\frac{d+3}{d+1}$ yields \eqref{eq:NonfocusingEstimate}. The remainder of the argument consists of establishing \eqref{eq:Biorthogonality}.

\subsection{Fine-scale decomposition and conservation laws}
\label{subsection:FineScaleDecomposition}

We start by investigating when  $T_1, T_2, T_1', T_2'$  make a significant contribution on $q \in \mathfrak{q}[\mu_1,\mu_2]$ to the integral in \eqref{eq:Biorthogonality}. Recall that $T_1,T_1' \in \T_1[\lambda_1,\mu_1,\mu_2]$ and $T_1, T_1' \notin S$. To derive the conservation laws, we for the moment omit this information from notation, nor is it relevant to this point.

Let $I_{t,q} = \pi_t(q)$ denote the $R^{\frac{1}{2}}$-interval of times covered in $q$ and $t_q = c(I_{t,q})$ denote the center in time. Let $x_q = c_x(q)$ denote the center in space coordinates and $z_q = (x_q,t_q)$ the space-time center.

The cubes $q \in \mathfrak{q}$ for which
\begin{equation*}
\iint \chi_q^4 \phi_{T_1}(z) \overline{\phi_{T_1'}(z)} \phi_{T_2}(z) \overline{\phi_{T_2'}(z)} dz = \text{RapDec}(R)
\end{equation*}
are referred to as \emph{inessential} as these can be neglected in the proof of the estimate, recalling that we always have a bound in $\# T_j$ that is polynomial in $R$.
In the following we collect identities, which must hold true, else the cube will be inessential.
We obtain
\begin{equation}
\label{eq:EssentialDecomposition}
\begin{split}
    &\; \iint_{\bigcup \mathfrak{q}[\mu_1,\mu_2] \cap S} \big| \sum_{T_1 \not\sim S} \phi_{T_1} \big|^2 \big| \sum_{T_2 } \phi_{T_2} \big|^2 \\
    &\lesssim \sum_{\substack{q \in \mathfrak{q}[\mu_1,\mu_2], \\ q \subseteq S}} \iint \chi_q^4 \big| \sum_{T_1 \not \sim S} \phi_{T_1} \big|^2 \big| \sum_{T_2} \phi_{T_2} \big|^2 \\
    &= \sum_{\substack{T_1, T_1'. \\ T_2, T_2'}} \sum_{\substack{q \in \mathfrak{q}[\mu_1,\mu_2], \\ q \subseteq S}} \iint \chi_q^4 \phi_{T_1} \overline{\phi_{T_1'}} \phi_{T_2} \overline{\phi_{T_2'}} \\
    &= \sum_{T_i, T_j'} \sum_{\substack{q \in \mathfrak{q}[\mu_1,\mu_2], \\ \text{essential } q \subseteq S}} \iint \chi_q^4 \phi_{T_1} \overline{\phi_{T_1'}} \phi_{T_2} \overline{\phi_{T_2'}} + \text{RapDec}(R) \big( \# \T_1 \big) \big( \# \T_2 \big).
\end{split}
\end{equation}

For $t \in I_{t,q}$ we can suppose the following identities:

\smallskip

\emph{Congruence of position:}
\begin{equation}
\label{eq:CongruencePosition}
\begin{split}
x_{T_1}(t) + \mathcal{O}(R^{\frac{1}{2}+\delta}) &= x_{T_1'}(t) + \mathcal{O}(R^{\frac{1}{2}+\delta}) = x_{T_2}(t) + \mathcal{O}(R^{\frac{1}{2}+\delta}) \\
&\quad = x_{T_2'}(t) + \mathcal{O}(R^{\frac{1}{2}+\delta}).
\end{split}
\end{equation}

\smallskip

\emph{Conservation of momentum:}
\begin{equation}
\label{eq:MomentumConservation}
\xi_{T_1}(t) + \xi_{T_2}(t) = \xi_{T_1'}(t) + \xi_{T_2'}(t) + \mathcal{O}(R^{-\frac{1}{2}+\delta}).
\end{equation}

To prove the above two statements, we argue by contradiction and suppose that the congruence of position or conservation of momentum fails at one time $t \in I_{t,q}$. For this time-slice, we have
\begin{equation}
\label{eq:InessentialCubeSlice}
\int \phi_{T_1}(t,x) \overline{\phi_{T_1'}(t,x)} \phi_{T_2}(t,x) \overline{\phi_{T_2'}(t,x)} dx = \text{RapDec}(R).
\end{equation}

 We shall see that in this case we find \eqref{eq:InessentialCubeSlice} for all times $t \in I_{t,q}$. But then the contribution over the whole cube becomes negligible. Here we use that position and momentum are not significantly changing over $q$. This can be justified by linearizing the Hamiltonian flow on $q$.

\begin{lemma}
For $t \in I_{t,q}$ we have
\begin{equation*}
|x_t - ( x_{t_q} + (t-t_q) \frac{\partial p}{\partial \xi}(x_{t_q},t_q,\xi_{t_q})) | \lesssim 1 \text{ and } |\xi_t - \xi_{t_q} | \lesssim R^{-\frac{1}{2}}.
\end{equation*}
In particular, if \eqref{eq:CongruencePosition} or \eqref{eq:MomentumConservation} are violated, then it holds
\begin{equation*}
\iint \chi_q^4 \phi_{T_1}(t,x) \overline{\phi_{T_1'}(t,x)} \phi_{T_2}(t,x) \overline{\phi_{T_2'}(t,x)} dx dt = \text{RapDec}(R).
\end{equation*}
\end{lemma}
\begin{proof}
Suppose $t_q = 0$ to ease notation. We have by Taylor expansion:
\begin{equation*}
\begin{split}
x_t &= x_0 + t \frac{\partial p}{\partial \xi}(x_0,0,\xi_0) + \mathcal{O}( t^2 \big( \frac{\partial^2 p}{\partial \xi \partial x} \frac{\partial p}{\partial \xi} - \frac{\partial^2 p}{\partial \xi^2} \frac{\partial p}{\partial x} \big) ), \\
\xi_t &= \xi_0 + \mathcal{O}(t \frac{\partial p}{\partial x}).
\end{split}
\end{equation*}
We estimate both terms by \eqref{eq:MixedRegularity} and \eqref{eq:XiRegularity}:
\begin{equation*}
\big| t^2 \big( \frac{\partial^2 p}{\partial \xi \partial x} \frac{\partial p}{\partial \xi} - \frac{\partial^2 p}{\partial \xi^2} \frac{\partial p}{\partial x} \big) \big| \lesssim_{C_2} \epsilon_0.
\end{equation*}

For the second term we have
\begin{equation*}
\big| t \frac{\partial p}{\partial x} \big| \lesssim_{C_2} \epsilon_0 R^{-\frac{1}{2}}.
\end{equation*}
This shows that the Hamiltonian flow is not essentially changing over $q$. Then, it follows from the wave packet axioms concerning the localization in position and frequency that violation of \eqref{eq:CongruencePosition} leads to rapid decay. If  \eqref{eq:MomentumConservation} is violated, the claim follows from Plancherel's theorem. 
\end{proof}

The third conservation law is the \emph{conservation of energy}: 
\begin{equation}
\label{eq:EnergyConservation}
p_1(z_q,\xi_{T_1}(t_q)) + p_2(z_q,\xi_{T_2}(t_q)) = p_1(z_q,\xi_{T_1'}(t_q)) + p_2(z_q,\xi_{T_2'}(t_q)) + \mathcal{O}(R^{-\frac{1}{2}+\delta}),
\end{equation}
which is a consequence of time orthogonality.

\begin{lemma}
Suppose that \eqref{eq:EnergyConservation} fails. Then
\begin{equation*}
\big| \iint_{\R^{d+1}} \chi_q^4 \phi_{T_1} \overline{\phi_{T_1'}} \phi_{T_2} \overline{\phi_{T_2'}} dz \big| \lesssim \text{RapDec(R)}.
\end{equation*}
\end{lemma}
\begin{proof}
We apply Plancherel's theorem in time
\begin{equation*}
\iint \chi_q \phi_{T_1} \overline{\chi_q \phi_{T_1'}} \chi_q \phi_{T_2} \overline{\chi_q \phi_{T_2'}} dz = \iint \langle \mathcal{F}_t[ \chi_q \phi_{T_1} ] * \mathcal{F}_t [ \chi_{q} \phi_{T_2} ], \mathcal{F}_t [ \chi_q \phi_{T_1'}] * \mathcal{F}_t [ \chi_q \phi_{T_2'}] \rangle dx.
\end{equation*}

By using the essential support of the time frequencies according to Assumption \ref{ass:WP} (3) we can suppose that
\begin{equation*}
\begin{split}
&\big| p_1(x_{T_1}(t_q),t_q, \xi_{T_1}(t_q)) + p_2(x_{T_2}(t_q),t_q,\xi_{T_2}(t_q)) \\
&\quad - p_1(x_{T'_1}(t_q),t_q, \xi_{T'_1}(t_q)) - p_2(x_{T'_2}(t_q),t_q,\xi_{T'_2}(t_q) \big| \lesssim R^{-\frac{1}{2}+\delta}.
\end{split}
\end{equation*}
Now by the congruence of position and Taylor expansion we conclude
\begin{equation*}
|p_1(z_q,\xi_{T_1}(t_1)) + p_2(z_q,\xi_{T_2}(t_q)) - p_1(z_q,\xi_{T_1'}(t_q)) - p_2(z_q,\xi_{T_2'}(t_q)) | \lesssim R^{-\frac{1}{2}+\delta}.
\end{equation*}

\end{proof}

 To summarize, at this point we have determined that all essential contributions to 
the quadrilinear expression in \eqref{eq:EssentialDecomposition} must satisfy the conservation of 
position \eqref{eq:CongruencePosition}, momentum \eqref{eq:MomentumConservation} and energy \eqref{eq:EnergyConservation}. This allows us to restrict our attention only to such quadruples of tubes, which we now seek to describe.
By momentum conservation we have for essential cubes 
\begin{equation*}
\xi_{T_2}(t_q) = \xi_{T_1'}(t_q)+\xi_{T_2'}(t_q) -\xi_{T_1}(t_q) + \mathcal{O}(R^{-\frac{1}{2}+\delta})
\end{equation*}
 and by Taylor expansion, we can plug this identity into \eqref{eq:EnergyConservation}.

\smallskip

Fixing two of the tube frequencies, this leads us to define the energy difference function, as a function of a third frequency $\eta$,
\begin{equation*}
F^z_{\xi_1,\xi_2'}(\eta) =p_1(z,\xi_1)+p_2(z,\eta+\xi_{2}'-\xi_{1})-p_1(z,\eta) - p_2(z,\xi_{2}'),
\end{equation*}
where $z=(x,t)$, $\xi_1 \in \Xi$, $\xi_2' \in \Xi'$ and $T_1 \in \T_1$ such that $|x_{T_1}(t) -x| \lesssim R^{\frac{1}{2}+\delta}$.

\smallskip

With this notation, the set of frequencies $\eta$ which yields essential contributions 
is confined to an \emph{energy shell}, defined as the collection of frequencies of $T_1$-tubes at time $t_q$, which essentially annihilate the energy difference function:
\begin{equation*}
\mathcal{E}^{q}_{\xi_1,\xi_2'} = \{ \eta_{T_1}(t_q) : T_1 \in \T_1, \; |F^{z_q}_{\xi_1,\xi_2'}(\eta_{T_1}(t_q))| \lesssim R^{-\frac{1}{2} +\delta } \}.
\end{equation*}

\smallskip
For $q \in \mathfrak{q}[\mu_1,\mu_2]$ let 
 $\T_1^{\not\sim S}(q,\lambda_1,\mu_1,\mu_2)$ denote the tubes $T_1 \in \T_1[\lambda_1,\mu_1,\mu_2]$ associated to $q$, i.e. with $T_1 \not\sim S$, $R^\delta T_1 \cap q \neq \emptyset$.
Given frequencies $(\xi_1,\xi'_2)$, let $\T_1^{\not\sim S}(\mathcal{E}^{z_q}_{\xi_1,\xi_2'},\lambda_1,\mu_1,\mu_2)$ be the subset $\T_1^{\not\sim S}(q,\lambda_1,\mu_1,\mu_2)$ 
whose spatial frequency $\eta_{T_1}(t_q)$ at $t_q=c_t(q)$ belongs to the associated energy shell
\begin{equation*}
|F^{z_q}_{\xi_1,\xi_2'}(\eta_{T_1}(t_q))| \lesssim R^{-\frac{1}{2}+\delta}.
\end{equation*}
We find from counting and integration, taking into account the wave packet amplitude size as in Assumption \ref{ass:WP}:
\begin{equation}
\label{eq:AuxEstimateEnergyI}
\begin{split}
\sum_{\substack{q \in \mathfrak{q}[\mu_1,\mu_2] \\  q \subseteq S}} \sum_{\substack{T_i,T_j',\\ T_1 \not\sim S, T_1' \not\sim S}} \iint \chi_q^4 \phi_{T_1} \overline{\phi_{T_1'}} \phi_{T_2} \overline{\phi_{T_2'}} dz &\lesssim R^{-\frac{d-1}{2}} R^{4C \delta} \sum_q | \T_1^{\not\sim S}(q,\lambda_1,\mu_1,\mu_2)|   \\
&\, \times |\T_2(q)| \, \big( \sup_{\xi_1,\xi_2'} \big| \T_1^{\not\sim S}(\mathcal{E}^{q}_{\xi_1,\xi_2'},\lambda_1,\mu_1,\mu_2) \big| \big).
\end{split}
\end{equation}

\medskip

The following combinatorial estimate for $q \in \mathfrak{q}[\mu_1,\mu_2]$, $q \subseteq S$ is decisive:
\begin{equation}
\label{eq:CombEstIntro}
|\T_1^{\not\sim S}(\mathcal{E}^{q}_{\xi_1,\xi_2'},\lambda_1,\mu_1,\mu_2)| \lesssim \nu^{-1} R^{C \delta} \frac{\# \T_2}{\lambda_1 \mu_2}.
\end{equation}
Here $q, \xi_1, \xi'_2$ are fixed, and the fact that we are counting tubes not concentrated on $S$ is essential.

If this is the case we can continue \eqref{eq:AuxEstimateEnergyI} as follows:
\begin{equation}
\label{eq:AuxEstimateEnergyII}
\begin{split}
\eqref{eq:AuxEstimateEnergyI} &\lesssim R^{-\frac{d-1}{2}} \sum_{q \in \mathfrak{q}[\mu_1,\mu_2]}|\T^{\not\sim S}_1(q,\lambda_1,\mu_1,\mu_2)| R^{5C \delta} \nu^{-1} \frac{\# \T_2}{\lambda_1 \mu_2} |\T_2(q)| \\
&\lesssim R^{5C \delta} \nu^{-1} \# \T_2 \sum_{q\in \mathfrak{q}[\mu_1,\mu_2]} \frac{|\T^{\not\sim S}_1(q,\lambda_1,\mu_1,\mu_2)|}{\lambda_1}.
\end{split}
\end{equation}
Finally, we find from changing the summation:
\begin{equation}
\label{eq:ChangingSummation}
\begin{split}
\sum_{q \in \mathfrak{q}[\mu_1,\mu_2]} |\T_1^{\not\sim S}(q,\lambda_1,\mu_1,\mu_2)| &\leq \sum_{T_1 \in \T_1[\lambda_1,\mu_1,\mu_2]} \# \{ q \in \mathfrak{q}[\mu_1,\mu_2] : R^\delta T_1 \cap q \neq \emptyset \} \\ &\lesssim \lambda_1 |\T_1|,
\end{split}
\end{equation}
where the requirement that we are only counting tubes not concentrated on $S$ is nonessential and is dropped.
Plugging \eqref{eq:ChangingSummation} into \eqref{eq:AuxEstimateEnergyII} we conclude the proof of the estimate \eqref{eq:Biorthogonality}.

\subsection{The combinatorial estimate}
\label{subsection:CombinatorialEstimate}

Here we lay the last brick in the proofs of Theorems \ref{thm:GeneralizationBilinearParaboloid} and \ref{thm:GeneralizationBilinearCone} by establishing \eqref{eq:CombEstIntro}, which we recall here:
\begin{equation}
\label{eq:CombEst}
|\T_1^{\not\sim S}(\mathcal{E}^{q}_{\xi_1,\xi_2'},\lambda_1,\mu_1,\mu_2)| \lesssim R^{C \delta} \nu^{-1} \frac{\# \T_2}{\lambda_1 \mu_2}.
\end{equation}

For this we critically use the non-focusing property of $T_1$ from the above set. Firstly, observe that for any such $T_1$  there are $\gtrsim R^{-c \delta} \lambda_1$ cubes $q' \in \mathfrak{q}[\mu_1,\mu_2]$ with $R^\delta T_1 \cap q' \neq \emptyset$ which are at distance $\gtrsim R^{1-\delta}$ from $q \subseteq S$ because $T_1$ is not concentrating on $S$. By pigeonholing there are roughly $\mu_2$ tubes $T_{2} \in \T_2$ passing through $q' \in \mathfrak{q}[\mu_1,\mu_2]$ with $\text{dist}(q',q) \geq R^{1-\delta}$. 
Recall the estimate \eqref{eq:AveragingEstimate} from averaging, which gives 
\begin{equation*}
\begin{split}
&| \{ (q',T_1) \in \mathfrak{q}[\mu_1,\mu_2] \times \T_1^{\not \sim S}(\mathcal{E}^{q}_{\xi_1,\xi_2'},\lambda_1,\mu_1,\mu_2) : \; \text{dist}(q',q) \gtrsim R^{1-\delta}, 
 R^{\delta} T_1 \cap q' \neq \emptyset \} | \\ &\qquad \qquad\gtrsim R^{- C \delta} \lambda_1 | \T_1^{\not \sim}(\mathcal{E}^{q}_{\xi_1,\xi_2'},\lambda_1,\mu_1,\mu_2) |.
\end{split}
\end{equation*}

Consequently, adding tubes $T_2$ through $q'$ to the counting, we have
\begin{equation}
\label{eq:ProductSetCounting}
\begin{split}
&\quad | \{ (q',T_1,T_2) \in \mathfrak{q}[\mu_1,\mu_2] \times \T_1^{\not\sim S}(\mathcal{E}^{q}_{\xi_1,\xi_2'},\lambda,\mu_1,\mu_2) \times \T_2 :  q' \cap R^\delta T_{1} \neq \emptyset, \\
&\quad \text{dist}(q',q) \geq R^{1-\delta}, \;  q' \cap R^\delta T_{2} \neq \emptyset \}| \gtrsim R^{-C \delta} \lambda_1 \mu_2 |\T_1^{\not\sim S}(\mathcal{E}^{q}_{\xi_1,\xi_2'},\lambda,\mu_1,\mu_2)|.
\end{split}
\end{equation}

In order to change the counting order on the left, we now fix $T_{2} \in \T_2$, and we estimate the size of the set 
\begin{equation*}
\begin{split}
\mathcal{Q}_{q}[\lambda,\mu_1,\mu_2] &= \{ (q',T_1) \in \mathfrak{q}[\mu_1,\mu_2] \times \T_1^{\not\sim S}(\mathcal{E}^{q}_{\xi_1,\xi_2'},\lambda,\mu_1,\mu_2) : \text{dist}(q,q') \geq R^{1-\delta},   \\
&\quad R^{\delta} T_{2} \cap  q' \neq \emptyset, \; R^{\delta} T_1 \cap q \neq \emptyset, \; R^{\delta} T_1 \cap q' \neq \emptyset \}.
\end{split}
\end{equation*}
We claim that
\begin{equation}
\label{eq:CubeDoubleEndCounting0}
|\mathcal{Q}_{q}[\lambda,\mu_1,\mu_2]| \lesssim \nu^{-1} R^{C \delta}.
\end{equation}
Suppose that \eqref{eq:CubeDoubleEndCounting0} holds. Then we can estimate
\begin{equation*}
lhs \eqref{eq:ProductSetCounting} \lesssim \nu^{-1} R^{C \delta} |\T_2|,
\end{equation*}
which yields the desired bound \eqref{eq:CombEst}.

\medskip

At this point the simplifying observation is that 
the parameters $\lambda,\mu_1,\mu_2$ as well as the nonconcentration on $S$ no longer play a role. Dropping 
these constraints we define the larger and simpler set 
\begin{equation*}
\begin{split}
\mathcal{Q}_{q} &= \{ (q',T_1) \in \mathfrak{q} \times \T_1(\mathcal{E}^{q}_{\xi_1,\xi_2'}) : \text{dist}(q,q') \geq R^{1-\delta},   \\
&\quad R^{\delta} T_{2} \cap  q' \neq \emptyset, \; R^{\delta} T_1 \cap q \neq \emptyset, \; R^{\delta}  T_1 \cap q' \neq \emptyset \}.
\end{split}
\end{equation*}
and replace \eqref{eq:CubeDoubleEndCounting0} with the stronger bound
\begin{equation}
\label{eq:CubeDoubleEndCounting}
|\mathcal{Q}_{q}| \lesssim \nu^{-1} R^{C \delta}.
\end{equation}
It remains to prove \eqref{eq:CubeDoubleEndCounting}. This is where we use the transversality condition, and therefore the proofs of Theorems \ref{thm:GeneralizationBilinearParaboloid} and \ref{thm:GeneralizationBilinearCone} bifurcate.

\subsubsection{Proof of Theorem \ref{thm:GeneralizationBilinearParaboloid}}
In this section we deal with the non-degenerate case, i.e., the Hessians of $p_j$ are assumed to be non-degenerate.

We have the following lemma as a consequence of the non-degeneracy of the Hamiltonian flow, which reduces us to count the intersections $q'$ with $T_2$ and $T_1$. The lemma states that only $R^{C\delta}$ tubes from an energy shell can pass through two $R^{1-\delta}$ separated points.
\begin{lemma}
\label{lem:TubeCountingAuxNondegenerate}
Let $q' \in \mathfrak{q}$ such that there is $T_1 \in \T_1$ so that $ (q',T_1) \in \mathcal{Q}_{q}$. Let
\begin{equation*}
\mathcal{Q}_{q}(q') = \{ T_1 \in \T_1(\mathcal{E}^{q}_{\xi_1,\xi_2'}) : (q',T_1) \in \mathcal{Q}_{q} \}.
\end{equation*}
Then it holds
\begin{equation*}
|\mathcal{Q}_{q}(q')| \lesssim R^{c \delta}.
\end{equation*}
\end{lemma}
\begin{proof}

We shall see that $T_1$ is essentially determined by $q$ and $q'$ as property of bi-Lipschitz flow. 
Note that the cubes $q$ and $q'$ are separated in time at least by $R^{1-\delta}$ (because we have the bound $|\partial_{\xi} p_i| \lesssim 1$ and the tubes could not hit both cubes otherwise). We decompose the frequency domain into small finitely overlapping balls $(B(\xi^*,c_0))_{\xi^* \in \Xi}$ such that for any $x_0$ it holds for any $\xi_1$, $\xi_2 \in B(\xi^*,c)$:
\begin{equation*}
\big| \frac{\partial^2 p_1}{(\partial \xi)^2}(x_0,\xi_1) - \frac{\partial^2 p_1}{(\partial \xi)^2}(x_0,\xi_2) \big| \leq C_2 |\xi_1 - \xi_2| \leq 2 C_2 c_0.
\end{equation*}
Let $A_{x_0} = \partial^2_{\xi} p_1(x_0,\xi^*)$. $c_0$ will be chosen later, depending on the regularity and non-degeneracy properties of $p$ to ensure that with
\begin{equation*}
\mathcal{Q}_q(q',\xi^*) = \{ T_1 \in \mathbb{T}_1(\mathcal{E}^q_{\xi_1,\xi_2'}) : \xi_{T_1} \in B(\xi^*,c) \text{ and } (q',T_1) \in \mathcal{Q}_q \},
\end{equation*}
it holds
\begin{equation*}
|\mathcal{Q}_q(q',\xi^*)| \lesssim R^{c \delta}.
\end{equation*}

\smallskip

For the $p_1$ bicharacteristic starting at $x_0$ with frequency $\xi_0$ we have the Taylor expansion
\begin{equation*}
x(t) = x_0 + t \frac{\partial p_1}{\partial \xi}(x_0,\xi_0) + t^2 \zeta(t,x_0,\xi_0).
\end{equation*}
For the remainder of the Taylor expansion we have the following Lagrange representation:
\begin{equation*}
\int_0^t (t-s) \big[ \frac{\partial^2 p_1}{\partial x \partial \xi}(x_s,\xi_s) \frac{\partial p_1}{\partial \xi}(x_s,\xi_s) - \frac{\partial^2 p_1}{\partial \xi^2}(x_s,\xi_s) \frac{\partial p_1}{\partial x}(x_s,\xi_s) \big] ds = t^2 \zeta(t,x_0,\xi_0)
\end{equation*}
with $\zeta$ a continuous function and bounds only depending on $ \epsilon_0$ and $C_2$.

We compare two bicharacteristics starting at the same point but with different frequencies $\xi_i \in B(\xi^*,c)$, at times $R^{1-\delta} \leq |t| \leq R$:
\begin{equation*}
\begin{split}
x_1(t) &= x_0 + t \frac{\partial p_1 }{\partial \xi}(x_0,\xi_1) + t^2 \zeta(t,x_0,\xi_1), \\
x_2(t) &= x_0 + t \frac{\partial p_1 }{\partial \xi}(x_0,\xi_2) + t^2 \zeta(t,x_0,\xi_2).
\end{split}
\end{equation*}
By the fundamental theorem we have
\begin{equation*}
\begin{split}
|t| \big| \frac{\partial p_1}{\partial \xi}(x_0,\xi_1) - \frac{\partial p_1}{\partial \xi}(x_0,\xi_2) \big| &= |t| \big| \int_0^1 \frac{d}{ds} \frac{\partial p_1}{\partial \xi}(x_0,\xi_2+s \Delta \xi) ds \big| \\
 &= |t| \big| \int_0^1 \frac{\partial^2 p_1}{(\partial \xi)^2} (x_0, \xi_2 + s \Delta \xi) ds \Delta \xi \big| \\
 &\geq |t| | | A_{x_0} \Delta \xi| - 2C_2 c_0 |\Delta \xi| | \geq |t| \bar{C} d_1 |\Delta \xi|.
 \end{split}
\end{equation*}
Indeed, we can estimate $|A_{x_0} \Delta \xi| \gtrsim |\Delta \xi|$ from below in terms of the smallest eigenvalue, that is in terms of $d_1$ and the upper Hessian bound $C_2$, which is subsumed in $\bar{C}$. And the ultimate estimate follows from choosing $c_0$ small enough.

Since $p \in C^{1,1}_{z,\xi} \cap C^{1}_{z} C^3_{\xi}$ we have
\begin{equation*}
|t|^2 |\zeta(t,x_0,\xi_1) - \zeta(t,x_0,\xi_2) | \lesssim \frac{ \epsilon_0 |t|^2}{R} |\xi_1-\xi_2|.
\end{equation*}
Consequently, for $\epsilon_0$ small enough, relative to $\bar{C} d_1$, we find
\begin{equation*}
|x_1(t) - x_2(t)| \geq \frac{\bar{C} d_1 |t|}{2} |\xi_1 - \xi_2|.
\end{equation*}
With $\xi_i \in R^{-\frac{1}{2}} \Z^d$ separated by $R^{-\frac{1}{2}}$ this bounds the numbers of frequencies resulting in tubes which can intersect $q$ after a time $|t| \gtrsim R^{1-\delta}$, by $R^{C \delta}$.
\end{proof}

\smallskip

With the previous lemma at hand, it suffices to estimate the intersections of $q'$ for which $\mathcal{Q}_q(q') \neq \emptyset$ with $T_2$. With the next lemma, we can bound the cubes $q'$ as follows.

\begin{lemma}\label{l:transverse}
It holds
\begin{equation*}
\big( \bigcup_{q': \mathcal{Q}_{q} (q') \neq \emptyset } q' \big) \cap T_2 \subseteq  T_2 \vert_{I_{\tau}}
\end{equation*}
with $I_{\tau}$ being a time interval of size $\nu^{-1} R^{\frac{1}{2}+\delta}$.
\end{lemma}
With the last two lemmas at hand, the desired bound \eqref{eq:CubeDoubleEndCounting} immediately follows. It remains to prove this last lemma.
\begin{proof}

Recall that the energy shell is given by
\begin{equation*}
\mathcal{E}^{q}_{\xi_1,\xi_2'} = \{ \eta \in B_d(0,2) : \, |p_1(z_q,\xi_1) + p_2(z_q,\eta +\xi_2'-\xi_1) - p_1(z_q,\eta) - p_2(z_q,\xi_2') | \lesssim R^{-\frac{1}{2}+\delta} \}.
\end{equation*}
We define the zero energy shell  $\mathcal{E}^{q}_{0,\xi_1,\xi_2'} = \{ \eta \in B_d(0,2) : \, F^{z_q}_{\xi_1,\xi_2'}(\eta) = 0 \}$ where it suffices to consider bicharacteristics $(x(t),\xi(t))$ with $x(t_q) = x_q$ and $F^{z_q}_{\xi_1,\xi_2'}(\xi(t_q)) = 0$. Note that for the energy difference functional, 
by the transversality condition we know that its gradient $\nabla_{\xi} F = \nabla_{\xi} ( p_2(z_q,\eta + \xi_2'-\xi_1) - p_1(z_q,\eta) )$ is of size $\gtrsim \nu$ and it is normal to the energy shell. This allows us to estimate the thickness of the energy shell in normal direction by $\nu^{-1} R^{-\frac{1}{2}+\delta}$. For this reason the variation of a bicharacteristic taking into account the thickness of the shell can be estimated by
\begin{equation}
\label{eq:VariationLevelSets}
R^{-\frac{1}{2}+\delta} \nu^{-1} | \frac{\partial x}{\partial \eta} | \lesssim R^{\frac{1}{2} + \delta } \nu^{-1},
\end{equation}
which makes it acceptable to restrict to the zero enery shell $F^{z_q}_{\xi_1,\xi_2'}(\eta) = 0$ as this is precisely the scale of $T_2 \big\vert_{I_{\tau}}$.

Without loss of generality
we will assume  that $0= t_q < t_{q'} \in \mathcal{I}_T = [\frac{R^{1-\delta}}{10},R]$. The restriction on the time interval stems from the timewise difference between $q$ and $q'$. 
 For $t \in \mathcal{I}_T$ we consider the $d-1$-dimensional manifold
\begin{equation*}
\Gamma_{q}(t) = \{ (x(t),t) : (x(t),\xi(t)) \text{ is a bicharacteristic of } p_1, \; x(t_q) = x_q, \; F(\xi(t_q)) = 0 \}
\end{equation*}
describing the time slice of the spatial projection of the tubes passing through $q$ and with frequency in the energy shell. We let
\begin{equation*}
\Gamma_{q} = \bigcup_{t \in \mathcal{I}_T} \Gamma_{q}(t),
\end{equation*}
which is a $d$-dimensional manifold.

It will suffice to show that any intersection of $\Gamma_q$ with $T_2$ occurs with a transversality of at least $\nu$. The following facts come together:
\begin{itemize}
\item We have restricted the frequencies to $F^{z_q}_{\xi_1,\xi_2'}(\xi(t_q)) = 0$. 
Below are further remarks on varying $\Gamma_q$ along the frequencies.
\item We can estimate the number of total intersections of $\Gamma_q$ with $T_2$ by $\mathcal{O}(1)$. This will be verified at the end of the proof.
\end{itemize}

We take an intersection point $(x_{t_2},t_2)$ of $T_{2}$ with $\Gamma_{q}$. Let $\eta_1$ denote the ``intersection frequency", which satisfies $\eta_1 \in \mathcal{E}^{z_q}_{0,\xi_1,\xi_2'}$ and $(x_{t_2},t_2) \in T_{1}(z_q,\eta_1).$ We examine the tangent space of $\Gamma_q$ at this point. 

We obtain $d-1$ spatial tangent vectors $ \mathfrak{t}_1^{\Gamma_q}, \ldots, \mathfrak{t}_{d-1}^{\Gamma_q} $ of $\Gamma_q$ from varying the frequencies in $T_{\eta_1} \mathcal{E}^{z_q}_{0,\xi_1,\xi_2'} = \langle \mathfrak{t}_1, \ldots, \mathfrak{t}_{d-1} \rangle$, and observing the variation of the Hamiltonian flow $\frac{\partial x_{t_q}}{\partial \eta}$ at $q$. For $i=1,\ldots,d-1$ we obtain
\begin{equation*}
\mathfrak{t}_i^{\Gamma_q} = \frac{\partial x_{t_q}}{\partial \xi} \mathfrak{t}_i.
\end{equation*}

This is complemented with the tangential direction of the intersecting tube $T_2$. Note that here we use the fact that we are working on a scale of $R^{\frac{1}{2}}$ (up to admissible losses of $R^{c \delta}$).

\medskip

Let $\mathfrak{t}_1,\ldots,\mathfrak{t}_{d-1}$ denote an orthonormal basis of the tangent space $T_{\eta_1} \mathcal{E}^{z_q}_{0,\xi_1,\xi_2'}$ at $\eta_1$ (corresponding to an intersection frequency of $\Gamma_{q}(t)$ with $T_{2}$). Note that an (unnormalized) normal of $T_{\eta_1} \mathcal{E}^{z_q}_{0,\xi_1,\xi_2'}$ is given by
\begin{equation*}
\mathfrak{n}_{\mathcal{E}^{z_q}_{0}} = \nabla_{\xi} [ p_2(z_q,\eta_1 + \xi_2' - \xi_1) - p_1(z_q,\eta_1)] = \nabla_{\xi} p_2(z_q,\xi_2(t_q)) - \nabla_{\xi} p_1(z_q,\xi_1(t_q)).
\end{equation*}
As a consequence of the transversality assumption we have $|\mathfrak{n}_{\mathcal{E}^{z_q}_{0}}| \gtrsim \nu$.

\smallskip

We compute the variation of $(x_{\eta_1 + \varepsilon \mathfrak{t}_i}(t),t)$ at $\varepsilon = 0$ with $\eta_1 \in \mathcal{E}^{z_q}_{0,\xi_1,\xi_2'}$ and $x_{\eta}(t)$ denoting the spatial projection of the bicharacteristic governed by $p_1$ passing through $(x(t_q)=x_q,\eta)$ at $t=t_q$. To this end, we compute the variation of the Hamiltonian flow with respect to frequencies:
\begin{equation*}
\left\{ \begin{array}{cl}
\dot{x}_t &= \frac{\partial p_1}{\partial \xi}(x_t,t,\xi_t), \\
\dot{\xi}_t &= - \frac{\partial p_1}{\partial x}(x_t,t,\xi_t).
\end{array} \right.
\end{equation*}
Taking the derivative with respect to $\xi$ we find
\begin{equation*}
\frac{d}{dt} \frac{\partial x_t}{\partial \xi} = \frac{\partial^2 p_1}{\partial \xi \partial x}(x_t,\xi_t) \frac{\partial x_t}{\partial \xi} + \frac{\partial^2 p_1}{(\partial \xi)^2} \frac{\partial \xi_t}{\partial \xi}.
\end{equation*}
Note that $\partial_{\xi} x_{t = t_q} = 0$ because the bicharacteristics have the common starting point $x_q$. We obtain by integration and projection to $\mathfrak{t}_i$:
\begin{equation*}
\frac{\partial x_t}{\partial \xi} \mathfrak{t}_i = \int_{t_q}^{t} \frac{\partial^2 p_1}{\partial \xi \partial x}(x_s,s,\xi_s) \frac{\partial x_s}{\partial \xi} \mathfrak{t}_i ds + \int_{t_q}^{t} \frac{\partial^2 p_1}{(\partial \xi)^2}(x_s,s,\xi_s) \frac{\partial \xi_s}{\partial \xi} \mathfrak{t}_i ds.
\end{equation*}
The first integral is estimated as $\mathcal{O}(\epsilon_0 R^{-1} (t-t_q)^2)$, the second integral is $\mathcal{O}(t-t_q)$. Here we use
\begin{equation*}
\big| \frac{\partial x_t}{\partial \xi} \big| \leq  |t-t_q| \text{ and } \big| \frac{\partial \xi_t}{\partial \xi} - I_d \big| \leq \frac{\epsilon_0}{R} |t-t_q|
\end{equation*}
as an immediate consequence of the regularity of the Hamiltonian $p_1$ and integration
in the linearized Hamilton flow.

 Consequently, upon normalization denoted by $(\cdot)_{\mathbf{n}}$, we find
\begin{equation*}
\big( \frac{\partial x_t}{\partial \xi} \mathfrak{t}_i \big)_{\mathbf{n}} = \frac{1}{t-t_q} \int_{t_q}^{t} \frac{\partial^2 p_1}{\partial \xi^2}  ds \, \mathfrak{t}_i + \mathcal{O}(\epsilon_0 R^{-1} (t-t_q)).
\end{equation*}
Recall that $\epsilon_0 \ll \nu$ and we aim to show a transversality of order $\nu$. This means an error of order $\epsilon_0$ is negligible. We let
\begin{equation*}
A = \frac{1}{t_{q'}-t_q} \int_{t_q}^{t_{q'}} \frac{\partial^2 p_1}{\partial \xi^2}(x_s,s,\xi_s) ds,
\end{equation*}
which by non-degeneracy assumption on $p$ can be regarded as normalized.
Then it suffices to check $\nu$-transversality using the $\epsilon_0$-approximate spatial tangent vectors of $\Gamma_{q}(t)$:
\begin{equation*}
\big( A \mathfrak{t}_i, 0), \quad i =1,\ldots,d-1.
\end{equation*}
To complete the tangent space to $\Gamma_q$ we add one last  tangent vector which also points in the time direction. By $\eta_1(t)$ we denote the frequency evolution of the bicharacteristic $(x(t),\xi(t))$ with $(x(t_q),\xi(t_q)) = (x_q,\eta_1)$.
The derivative in $t$ yields a tangent vector of $\Gamma_{q}$ with non-trivial time component:
\begin{equation*}
V_1 = (\frac{\partial p_1}{\partial \xi}(z_{q'},\eta_1(t_{q'})),1).
\end{equation*}
On the other hand the tangent vector of $T_2$ is given by
\begin{equation*}
V_2 = (\frac{\partial p_2}{\partial \xi} (z_{q'},\eta_2(t_{q'})), 1).
\end{equation*}

Hence for the desired transversality condition we need to prove the space-time relation 
\begin{equation}
\label{eq:nuTransversality}
|\mathfrak{t}_1^{\Gamma_q} \wedge \ldots \wedge \mathfrak{t}_{d-1}^{\Gamma_q} \wedge V_1 \wedge V_2 | \gtrsim \nu,
\end{equation}
which is in turn  equivalent to the purely spatial relation
\[
|\mathfrak{t}_1^{\Gamma_q} \wedge \ldots \wedge \mathfrak{t}_{d-1}^{\Gamma_q} \wedge (V_1-V_2) | \gtrsim \nu.
\]
It suffices to check that the difference of $V_1$ and $V_2$ is quantitatively transverse to the spatial tangential vectors. We substitute $\mathfrak{t}_j^{\Gamma_q}$ 
by $ A \mathfrak{t}_j$ as discussed above. Inverting $A$ (which is allowed due to the non-degeneracy of $\partial^2_{\xi} p_1$ \eqref{eq:NonDegeneracyAssumption}), this is equivalent to
\begin{equation}
\label{eq:TransversalityAuxI}
| \mathfrak{t}_1 \wedge \ldots \wedge \mathfrak{t}_{d-1} \wedge A^{-1} \big( \nabla_{\xi} p(z_{q'},\eta_1(t_{q'})) - \nabla_{\xi} p(z_{q'},\eta_2(t_{q'})) | \gtrsim \nu.
\end{equation}

To this end, we use the normalized normal vector of $\mathcal{E}^{z_q}_{0,\xi_1,\xi_2'}$ given by $\mathfrak{n}_{\mathcal{E}_0^{z_q}} / \nu$ and write
\begin{equation*}
A^{-1} \big( \nabla_{\xi} p_1(z_{q'},\eta_1(t_{q'})) - \nabla_{\xi} p_2(z_{q'},\eta_2(t_{q'}) \big) = \sum_{i=1}^{d-1} \alpha_i \mathfrak{t}_i + \alpha \mathfrak{n}_{\mathcal{E}^{z_0}_{0}}
\end{equation*}
such that
\begin{equation*}
| \mathfrak{t}_1 \wedge \ldots \wedge \mathfrak{t}_{d-1} \wedge A^{-1} \big( \nabla_{\xi} p_1(z_{q'},\eta_1(t_{q'})) - \nabla_{\xi} p_2(z_{q'},\eta_2(t_{q'})) | = |\alpha|.
\end{equation*}

Consequently, \eqref{eq:TransversalityAuxI} amounts to
\begin{equation*}
\begin{split}
 | \langle \nabla_{\xi} p_1(z_q,\xi_1(t_q)) - \nabla_{\xi} p_2(z_q,\xi_2(t_q)), &A^{-1} (\nabla_{\xi} p_1(z_{q'},\eta_1(t_{q'})) - \nabla_{\xi} p_2(z_{q'},\eta_2(t_{q'}))) \rangle | \\&\gtrsim \nu^2.
 \end{split}
\end{equation*}
Let $\Delta_v = \nabla_{\xi} p_1(z_{q'},\eta_1(t_{q'})) - \nabla_{\xi} p_2(z_{q'},\eta_2(t_{q'}))$.
Finally, by Taylor expansion, the regularity, $\epsilon_0$-smallness of the frequency support and $\epsilon_0 \ll \nu$, we find that the leading order term with an error of size $\mathcal{O}(\epsilon_0 \nu)$ is given by
\begin{equation}
\label{eq:LeadingOrderTermTransversality}
|\langle \Delta_v , (\partial_{\xi}^2 p_1(z_{q'},\xi_1(t_{q'})))^{-1} \Delta_v \rangle | \gtrsim \nu^2
\end{equation}
where the final estimate is a consequence of our  hypothesis \eqref{eq:NonDegeneracyTransversality}.

\medskip

It remains to check for multiple intersections: First, suppose that the same bicharacteristic $T_1$ intersects $T_2$ at times $t_1$ and $t_2 \geq t_1 + C R^{\frac{1}{2}+\delta} \nu^{-1}$. We compute the difference between the two as
\begin{equation*}
\begin{split}
(x_{T_1} - x_{T_2})(t) &= \int_{t_1}^t \partial_{\xi} p_1(x_{T_1}(s),s,\xi_{T_1}(s)) - \partial_{\xi} p_2(x_{T_2}(s),s,\xi_{T_2}(s)) ds \\
&= \int_{t_1}^t \partial_{\xi} p_1(x_{T_1}(t_1),t_1,\xi_{T_1}(t_1)) - \partial_{\xi} p_2(x_{T_2}(t_1),t_1,\xi_{T_2}(t_1)) ds \\
&\quad + \mathcal{O}(\epsilon_0 C_2).
\end{split}
\end{equation*}
Consequently, we have $|x_{T_1}(t) - x_{T_2}(t)| \gtrsim |t-t_1| \nu \gtrsim R^{\frac{1}{2}+\delta}$, which proves the separation of the bicharacteristics after one intersection and contradicts the hypothesis that the same bicharacteristic intersects $T_2$ at separated times.

\smallskip

Next, suppose we have two bicharacteristics $T_1$, $T_1'$ satisfying $x_{T_1}(t_q) = x_{T'_1}(t_q) = x_q$ and with frequencies $\xi_{T_1}, \xi_{T_1'} \in \mathcal{E}^{z_q}_{0,\xi_1,\xi_2'}$. Suppose that the first intersects $T_2$ at time $t_1$, the other at time $t_2 \geq t_1$ with $|t_2 - t_1| \geq R^{\frac{1}{2}+\delta} \nu^{-1}$. We have
\begin{equation*}
x_{T_1}(t_1) - x_{T_1'}(t_1) + x_{T_2}(t_2) - x_{T_2}(t_1) = x_{T_1'}(t_2) - x_{T_1'}(t_1).
\end{equation*}
We can write equivalently
\begin{equation*}
\begin{split}
x_{T_1}(t_1) - x_{T_1'}(t_1) &= (x_{T_1'}(t_2) - x_{T_1'}(t_1)) - (x_{T_2}(t_2) - x_{T_2}(t_1)) \\
&= (t_2 - t_1) (\nabla_{\xi} p_1(x_{T_1}(t_*),t_*, \xi_{T_1}(t_*)) - \nabla_{\xi} p_2(x_{T_2}(t_*),t_*,\xi_{T_2}(t_*))).
\end{split}
\end{equation*}

Let $\gamma:[0,1] \to \mathcal{E}^{z_q}_{0,\xi_1,\xi_2'}$ which connects the two frequencies $\xi_{T_1}, \xi_{T_1'}$. We write the lefthand side of the preceding display as
\begin{equation*}
\frac{\partial x_t}{\partial \eta} \frac{d \eta(\gamma(s^*))}{ds} = (t_2 - t_1) (\nabla_{\xi} p_1(x_{T_1}(t_*),t_*, \xi_{T_1}(t_*)) - \nabla_{\xi} p_2(x_{T_2}(t_*),t_*,\xi_{T_2}(t_*))).
\end{equation*}
Clearly, $\mathfrak{t} = \frac{d \eta(\gamma(s^*))}{ds}$ is an element of the tangent space $T_{\eta^*} \mathcal{E}^{z_q}_{0,\xi_1,\xi_2'}$ at some frequency $\eta^*$ in $\mathcal{E}^{z_q}_{0,\xi_1,\xi_2'}$. This frequency yields a point $x^* \in \Gamma_q(t)$. It follows a reprise of the above argument to rule out the multiple intersections.

We write again by the Hamiltonian flow
\begin{equation*}
\begin{split}
\frac{\partial x_t}{\partial \xi} &= \int_{t_q}^{t_1} \frac{\partial^2 p}{(\partial \xi)^2} \frac{\partial \xi_s}{\partial \xi} + \frac{\partial^2 p}{\partial x \partial \xi} \frac{\partial x_s}{\partial \xi} ds \\
&= (t_1-t_q) (\partial^2_{\xi} p(x',\xi') (1 + \mathcal{O}(\varepsilon)) + \mathcal{O}(\varepsilon )) = (t_1-t_q) A^*.
\end{split}
\end{equation*}
In particular, by our non-degeneracy assumption $\partial_{\xi} x_t$ is invertible.

We can summarize the above as
\begin{equation*}
(t_1 - t_q) / (t_2 - t_1) \mathfrak{t} = (A^*)^{-1} (\nabla_{\xi} p_1(x_{T_1}(t_*),t_*,\xi_{T_1}(t_*)) - \nabla_{\xi} p_2(x_{T_2}(t_*),t_*,\xi_{T_2}(t_*))).
\end{equation*}

Now we project to an unnormalized normal vector $\mathfrak{n}$ at $\mathfrak{t}$ which is given by 
\begin{equation*}
\mathfrak{n} = \nabla_{\xi} p_1(x^*,t_1,\eta_1^*) - \nabla_{\xi} p_2(x^*,t_1,\eta_2^*),
\end{equation*}
and we find
\begin{equation*}
0 = \langle \mathfrak{n}, (A^*)^{-1} (\nabla_{\xi} p_1(x_{T_1}(t_*),t_*,\xi_{T_1}(t_*)) - \nabla_{\xi} p_2(x_{T_2}(t_*),t_*,\xi_{T_2}(t_*))) \rangle.
\end{equation*}
Let $\Delta_v = \nabla_{\xi} p_1(x_{T_1}(t_*),t_*,\xi_{T_1}(t_*)) - \nabla_{\xi} p_2(x_{T_2}(t_*),t_*,\xi_{T_2}(t_*)))$.
To conclude a contradiction, we carry out a Taylor expansion, which yields by size and regularity estimates on $p_i$:
\begin{equation*}
0 = \langle \Delta_v, (\partial^2_{\xi} p_1(x_{T_1}(t_*),t_*,\xi_{T_1}(t_*)))^{-1} \Delta_v \rangle + \mathcal{O}(\nu \varepsilon) \gtrsim \nu^2.
\end{equation*}
This contradiction shows that there can be no multiple intersections of $\Gamma_q$ with $T_2$.
\end{proof}

\subsection{Proof of Theorem \ref{thm:GeneralizationBilinearCone}}

To conclude the proof of the theorem, it remains to show the combinatorial estimate \eqref{eq:CubeDoubleEndCounting} for $1$-homogeneous phase functions. As in the proof of Theorem~\ref{thm:GeneralizationBilinearParaboloid}, we proceed in two steps. In the first step, we prove the counterpart of Lemma \ref{lem:TubeCountingAuxNondegenerate}.
\begin{lemma}
\label{lem:TubeCountingAuxNondegenerateCone}
Let $q' \in \mathfrak{q}$ such that there is $T_1 \in \T_1 : (q',T_1) \in \mathcal{Q}_{q}$. Let
\begin{equation*}
\mathcal{Q}_{q}(q') = \{ T_1 \in \T_1(\mathcal{E}^{q}_{\xi_1,\xi_2'}) : (q',T_1) \in \mathcal{Q}_{q} \}.
\end{equation*}
Then it holds
\begin{equation*}
|\mathcal{Q}_{q}(q')| \lesssim R^{C \delta}.
\end{equation*}
\end{lemma}
\begin{proof}
Again we shall see that $T_1$ is essentially determined by $q$ and $q'$ due to the bi-Lipschitz property of the flow. For $1$-homogeneous $p_1$ the following is immediate: 

\smallskip

\emph{Observation $(\star)$: Let $(x_t,\xi^\lambda_t)$ be the bicharacteristic with initial data $(x_0,\lambda \xi_0)$. Then it holds $(x_t,\xi^\lambda_t) = (x_0, \lambda \xi_t)$.}

\smallskip

In the argument we can only use the non-degeneracy in the angular direction of $p_i$. Regarding the degeneracy in the radial direction, we observe that for a frequency $\eta = \eta_{T_1}$ we have by the transversality assumptions
\begin{equation*}
\frac{d}{d\lambda} F^{z_q}_{\xi_1,\xi_2'}(\lambda \eta) = (\nabla_{\eta} p_2(z_q,\eta + \xi_2'-\xi_1) - \nabla_{\eta} p_1(z_q,\eta)) \cdot \eta, \quad |\frac{d}{d\lambda} F^{z_q}_{\xi_1,\xi_2'}(\lambda \eta)| \geq R^{-\frac{1}{2}+\delta_0}.
\end{equation*}
For this reason it is enough to consider angular variations.

 Indeed, for two frequencies $\xi_1$, $\xi_2$ with very small angular separation:
\begin{equation*}
\big| \frac{\xi_1}{|\xi_1|} - \frac{\xi_2}{|\xi_2|} \big| \lesssim R^{-\frac{1}{2}+\delta}
\end{equation*}
it is easy to see that the tubes with the same starting point are comparable on a scale of $R^{\frac{1}{2}+\delta}$ by the below.

We obtain for the bicharacteristic emanating from $x_0$ with frequency $\xi_0$, $| \xi_0 | = 1$:
\begin{equation*}
x(t) = x_0 + t \frac{\partial p_1}{\partial \xi}(x_0,0,\xi_0) + t^2 \zeta(t,x_0,\xi_0)
\end{equation*}

For the remainder of the Taylor expansion we have the following Lagrange representation:
\begin{equation*}
\int_0^t (t-s) \big[ \frac{\partial^2 p_1}{\partial x \partial \xi}(x_s,\xi_s) \frac{\partial p_1}{\partial \xi}(x_s,\xi_s) - \frac{\partial^2 p_1}{\partial \xi^2}(x_s,\xi_s) \frac{\partial p_1}{\partial x}(x_s,\xi_s) \big] ds = t^2 \zeta(t,x_0,\xi_0)
\end{equation*}
with $\zeta$ a continuous function and bounds only depending on $\varepsilon$ and $C_2$.

We compare two tubes with $R^{1-\delta} \leq |t| \leq R$ and frequencies $\xi_1$, $\xi_2$:
\begin{equation*}
\begin{split}
x_1(t) &= x_0 + t \frac{\partial p_1 }{\partial \xi}(x_0,\xi_1) + t^2 \zeta(t,x_0,\xi_1), \\
x_2(t) &= x_0 + t \frac{\partial p_1 }{\partial \xi}(x_0,\xi_2) + t^2 \zeta(t,x_0,\xi_2).
\end{split}
\end{equation*}
By $(\star)$ we can normalize $| \xi_i | = 1$ and retain the same curves $x_i(t)$.
It follows from the non-degeneracy in the angular directions
\begin{equation*}
\big| \frac{\partial p_1}{\partial \xi}(x_0,\xi_1) - \frac{\partial p_1}{\partial \xi}(x_0,\xi_2) \big| = \big| \frac{\partial^2 p_1}{(\partial \xi)^2} (x_0,\xi_*) \big( \frac{\xi_1}{|\xi_1|} - \frac{\xi_2}{|\xi|} \big) \big| \sim_{d_1,d_2} \big| \frac{\xi_1}{|\xi_1|} - \frac{\xi_2}{|\xi|} \big|.
\end{equation*}
Secondly,
\begin{equation*}
\big| \zeta(t,x_0,\xi_1) - \zeta(t,x_0,\xi_2) \big| \lesssim \frac{\epsilon_0 C t^2 }{R} \big| \frac{\xi_1}{|\xi_1|} - \frac{\xi_2}{|\xi|} \big|. 
\end{equation*}
And it follows that the angular frequency is essentially uniquely determined by cubes $q$, $q'$.
\end{proof}

We continue with the counterpart of Lemma~\ref{l:transverse}. Together, the next lemma and the previous one complete the proof of the desired bound \eqref{eq:CubeDoubleEndCounting}.
\begin{lemma}
It holds
\begin{equation*}
\big( \bigcup_{q': \mathcal{Q}_{q} (q') \neq \emptyset } q' \big) \cap T_2 \subseteq T_2 \big\vert_{I_T}
\end{equation*}
with $I_T$ an interval of length $\nu^{-1} R^{\frac{1}{2}+C\delta}$.
\end{lemma}
\begin{proof}
The argument follows along the above lines, but since $p$ is homogeneous, we cannot use the non-degeneracy. So, as above, let
\begin{equation*}
\Gamma_{q}(t) = \{(x(t),t) : (x(t),\xi(t)) \text{ solves } \eqref{eq:RescaledHamiltonianFlowAssumptions}, \; x(t_q) = x_q, \; \xi(t_q) \in \mathcal{E}^{z_q}_{0,\xi_1,\xi_2'} \}
\end{equation*}
describe the time-slice of the spatial projection of the bicharacteristic curves \\ $(x(t),\xi(t))$ with $x(t_q) = x_q$ and $\xi(t_q) = \eta \in \mathcal{E}^{z_q}_{0,\xi_1,\xi_2'}$ and let $\Gamma_{q} = \bigcup_{t \in \mathcal{I}_T} \Gamma_{q}(t)$.

Again suppose that $t_q = 0$ and $0=t_q < t_{q'} \in \mathcal{I}_T = [\frac{R^{1-\delta}}{20},R]$.

\smallskip

Consider a tube $T_2 \in \T_2$ which intersects $\Gamma_{q}$ at $(x_{q'},t')$. We show again the $\nu$-transversality of $T_2$ with the tangent space of $\Gamma_{q}$. The perturbative arguments, i.e. checking the $R^{-\frac{1}{2}+\delta}$-neighborhood of the energy shell and for multiple intersections can be done essentially as above, so this is omitted here.

Presently, we have $d-1$ spatial directions obtained from varying $x_{\eta_1}$ with variational frequencies in $T_{\eta_1} \mathcal{E}^{z_q}_{0,\xi_1,\xi_2'}$, the intersection frequency, and \\
 $V_1 = (\partial_{\xi} p_1(x_{q'},t',\eta_1(t')),1)$. We suppose by $(\star)$ that $| \eta_1 | = 1$. Also, let $V_2 = (\partial_{\xi} p_2(x_{q'},t',\eta_2(t')),1)$ denote the tangent direction of $T_2$ at the intersection point. Recall the variation formula from above:
\begin{equation*}
\frac{\partial x_t}{\partial \xi} = \int_{0}^{t'} \frac{\partial^2 p}{\partial \xi \partial x} \frac{\partial x_t}{\partial \xi}  ds + \int_0^{t'} \frac{\partial^2 p}{\partial \xi^2} \frac{\partial \xi_t}{\partial \xi} ds.
\end{equation*}

By the mean-value theorem we let
\begin{equation*}
\frac{\partial x_t}{\partial \xi} = t' \big( \frac{\partial^2 p_1}{\partial x \partial \xi}(x_*,\xi_*) \frac{\partial x_t}{\partial \xi} + \frac{\partial^2 p_1}{(\partial \xi)^2}(x_*,\xi_*) \frac{\partial \xi_{t_*}}{\partial \xi} \big) = t' A^*.
\end{equation*}
It will suffice to check the transversality
\begin{equation}
\label{eq:TransveralityHomAux}
\big| \frac{\partial^2 p_1}{(\partial \xi)^2}(x_*,\xi_*) \mathfrak{t}_1 \wedge \ldots \wedge \frac{\partial^2 p_1}{(\partial \xi)^2}(x_*,\xi_*) \mathfrak{t}_{d-1} \wedge V_1 \wedge V_2 \big| \gtrsim \nu
\end{equation}
where $\mathfrak{t}_i$ denote the normalized tangent vectors spanning $T_{\eta_1} \mathcal{E}_{0,\xi_1,\xi_2'}^{z_q}$.

This reduction follows from estimating
\begin{equation*}
\big| \frac{\partial^2 p_1}{\partial x \partial \xi}(x_*,\xi_*) \frac{\partial x_t}{\partial \xi} \big| \lesssim \epsilon_0 \text{ and } \big| \frac{\partial \xi_{t_*}}{\partial \xi} - I_d \big| \lesssim \epsilon_0.
\end{equation*}
By elimination, another Taylor expansion and using the confinement of the frequencies, it suffices to check that
\begin{equation*}
\big| \frac{\partial^2 p_1}{(\partial \xi)^2}(x_q',\xi_1^*) \mathfrak{t}_1 \wedge \ldots \wedge \frac{\partial^2 p_1}{(\partial \xi)^2}(x_q',\xi_1^*) \mathfrak{t}_{d-1} \wedge V_1 - V_2 \big| \gtrsim \nu.
\end{equation*}
To prove the above estimate, we can argue equivalently for $A_{\xi_1^*} = \frac{\partial^2 p_1}{(\partial \xi)^2}(x_q',\xi_1^*)$:
\begin{equation*}
\text{dist}(V_1-V_2,\text{R}(A_{\xi_1^*})) \gtrsim \nu.
\end{equation*}
Now we use self-adjointness and $\text{R}(A_{\xi_1^*}) = \ker(A_{\xi_1^*})^{\perp} = \langle \xi_1^* \rangle^{\perp}$ by $1$-homogeneity.

Consequently, this is equivalent to
\begin{equation*}
|\langle \xi_1^*, \nabla_{\xi} p_1(x_{q'},t',\eta_1(t')) - \nabla_{\xi} p_2(x_{q'},t',\eta_2(t')) \rangle| \gtrsim \nu,
\end{equation*}
which is true by our hypothesis \eqref{eq:NonDegeneracyTransversalityHomogeneous}.
\end{proof}

\section{Wave packet decompositions for rough flows}
\label{section:WavePacketDecomposition}

In this section we construct wave packet decompositions to 
\begin{equation}
\label{eq:FirstOrderRoughCoefficients}
(D_t + p^w(x,t,D))u = f, \quad u(0) = u_0
\end{equation}
for $p \in C^{\infty}_{t,x} C_{\xi}^\infty$ and phase-space localized $u_0$. To comprehend the regularity assumptions on $p$, recall that in the applications we consider a dyadic frequency localization $N \gg 1$ with a corresponding paradifferential decomposition. For real $p$, we find that $p^w$ is self-adjoint, and the above PDE generates an isometric evolution operator $S(t,s)$ in $L^2(\R^d)$. 
The key examples are $1$-homogeneous symbols, corresponding to variable-coefficient wave equations, and symbols quadratic in $\xi$ ($2$-homogeneous), corresponding to variable-coefficient Schr\"odinger equations. We denote the homogeneity by $\mathfrak{h}$. For $\mathfrak{h} = 1$, we have the important example
\begin{equation}
\label{eq:1HomogeneousA}
p(x,t,\xi) = (g^{ij}_{\leq N^{\frac{1}{2}}}(x,t) \xi_i \xi_j)^{\frac{1}{2}} \chi_N(\xi), \quad |\xi| \sim N.
\end{equation}
In this case we consider $|t| \leq 1$ and we can localize to $|x| \lesssim 1$.

For $\mathfrak{h} = 2$, we have
\begin{equation}
\label{eq:SchroedingerSymbol}
p(x,t,\xi) = g^{ij}_{\leq N, \leq N^{\frac{1}{2}}}(x,t) \xi_i \xi_j \chi_N(\xi).
\end{equation} 
Here we consider times $|t| \lesssim N^{-1}$ such that we can still localize spatially $|x| \lesssim 1$ and due to parabolic scaling, we can consider a more relaxed paradifferential decomposition with respect to time. We shall see that it suffices to consider $C^1$-regularity in time.

\smallskip

Our phase space transform will be formulated for the unit frequency evolution on a large time scale $R$. The transition between the examples above and this normalized setting is carried out by scaling. Precisely, with $R=N$, in the wave case $\mathfrak{h} = 1$ we rescale both space and time by $R$, while in the Schr\"odinger case we use the anisotropic rescaling by $(R^2,R)$ for time respectively space. 
The regularity assumptions on the Hamiltonian flow are formulated as follows:
\begin{assumption}
    \label{ass:HamiltonianFlowRegularity}
    Let $R \geq 1$. Let $\mathcal{Y} \subseteq T^* \R^d$ which satisfies Assumption \ref{ass:RegularitySet} with $\overline{\mathcal{Y}}_R \subseteq \R^d \times \R \times \R^d$, and $p \in C^\infty_{z,\xi}(\R^d \times \R \times \R^d)$. We require that the flow-out of $\mathcal{Y}$ under the Hamiltonian flow $p$ for $t \in [-10R,10R]$:
    \begin{equation*}
        \left\{ \begin{array}{cl}
             \dot{x}^t &= \frac{\partial p}{\partial \xi}, \\
             \dot{\xi}^t &= - \frac{\partial p}{\partial x}
        \end{array} \right.
    \end{equation*}
    is contained in $\overline{\mathcal{Y}}_R$. We require the following regularity estimates:
    \begin{equation}
        \label{eq:RegularityEstimatesHamiltonianFlowAssumption}
        |\partial_z^{\alpha} \partial_{\xi}^{\beta} p(x,t;\xi)| \leq C_{\alpha,\beta} R^{-|\alpha| + \frac{(|\alpha|-2)_+}{2}}.
    \end{equation}
\end{assumption}

We obtain wave packets from applying the modified FBI transform with parameter $R$ on the time-scale $|t| \leq R$. The wave packets are of size $\Delta x \sim R^{\frac{1}{2}}$, $\Delta \xi \sim R^{-\frac{1}{2}}$. 
We shall see that wave packets constructed this way satisfy the axioms from Assumption \ref{ass:WP}. For fixed $\delta$ and $m$ in the regularity hypotheses the implicit constants will depend on finitely many seminorm quantities $C_{\alpha,\beta}$ in \eqref{eq:RegularityEstimatesHamiltonianFlowAssumption}.
Moreover, the wave packets will comply with the evolution generated by the PDE as in Assumption \ref{ass:Frames}.

Our objective in this section will be to prove 
Theorem~\ref{thm:WavePacketDecompositions}, which is  stated in the introduction but we recall it here for convenience:
\begin{theorem}
Let $R \geq 1$, $0<\delta_0<1/2$, and $\nu \in [R^{-\frac{1}{2}+\delta_0},1]$. Assume that $a$ is a Hamiltonian, which satisfies Assumption \ref{ass:HamiltonianFlowRegularity}, with the set $\mathcal{Y}$ satisfying Assumption \ref{ass:RegularitySet}. Let $r \in [\nu^{-2-\delta_0},R]$ and $\mathcal{Y}_r, \, \overline{\mathcal{Y}}_r$ be defined like in Definition \ref{def:Thickening}. 
There is a pseudo-differential operator $\chi_{\overline{\mathcal{Y}}_r}(x,D): L^2(\R^d) \to L^2(\R^d)$ whose symbol satisfies $\chi_{\overline{\mathcal{Y}}_r}(x,\xi) = 1$ for $(x,\xi) \in \mathcal{Y}_r$ such that for solutions to
\begin{equation*}
\left\{ \begin{array}{cl}
(D_t + p^w) u &= 0, \quad (x,t) \in \R^{d+1}, \\
u(0) &= \chi_{\overline{\mathcal{Y}}_r}(x,D) u_0 \in L^2(\R^d)
\end{array} \right.
\end{equation*}
we have the decomposition in $L^2$ for $t \in [-r,r]$ with a rapidly decaying error term
\begin{equation*}
u = \sum_{(x_0,\xi_0) \in \Lambda_r \cap \overline{\mathcal{Y}}_r} \alpha_{(x_0,\xi_0)} \phi_{(x_0,\xi_0)}(t) + g(t)
\end{equation*}
with 
\begin{equation*}
\| g(t) \|_{L^p} \lesssim \nu^{d \big( \frac12 - \frac1p \big)} \text{RapDec}(r \nu^2) \| u_0 \|_{L^2}
\end{equation*}
for $2 \leq p \leq \infty$ and $(\phi_{(x_0,\xi_0)})_{(x_0,\xi_0) \in \Lambda_r \cap \overline{\mathcal{Y}}_r}$ are normalized wave packets satisfying Assumption \ref{ass:WP}. The implicit constants $C_{m,\delta}$ in \eqref{eq:LocPos}, \eqref{eq:LocFrequ}, and \eqref{eq:LocTimeFrequ} depend for fixed $m$ and $\delta$ only on the seminorm quantities $C_{\alpha,\beta}$ \eqref{eq:RegularityEstimatesHamiltonianFlowAssumption} for $|\alpha| + |\beta| \leq M(m,\delta)$.
\end{theorem}

\subsection{Decomposition into coherent states}

The parametrix for rough coefficients will be based on the modified FBI transform:
\begin{equation}
\label{eq:RescaledPhaseSpaceTransform}
\begin{split}
T_R : L^2(\R^d) &\to L^2(\R^{2d}), \\ T_R f(x,\xi) &= C_R \int_{\R^d} e^{i \xi (x-y)} e^{-\frac{\rho}{2} (x-y)^2} f(y) dy, \quad (x,\xi) \in \R^d \times \R^d
\end{split}
\end{equation}
for functions $f \in L^2(\R^d)$, and $\rho = R^{-1}$. The normalizing constant is given by $C_R = R^{-\frac{d}{4}} 2^{-\frac{d}{2}} \pi^{-\frac{3d}{4}}$. $T_R: L^2(\R^d) \to L^2(\R^{2d})$ is an isometry and maps onto the closed subspace of analytic functions under the identification $\R^{2d} \tilde{=} \C^d$, letting $z = x - i \xi \in \C^d$. 
One inversion formula is provided by the adjoint:
\begin{equation*}
f(y) = C_R \int_{\R^{2d}} e^{-\frac{\rho}{2}(x-y)^2} e^{-i \xi (x-y)} F(x,\xi) dx d\xi.
\end{equation*}

We remark that after applying $T_R$ the frequencies are essentially normalized. 
The FBI transform decomposes a function into coherent states,
which exhaust the uncertainty relation, and are given as follows
\begin{equation}
\label{eq:CoherentStates}
\phi_{x,\xi}(y) = e^{- i \xi(x-y)} e^{-\frac{\rho}{2}(x-y)^2}.
\end{equation}
These are localized in space on a scale $R^{\frac{1}{2}}$ around $x$ and on the frequency side on a scale of $R^{-\frac{1}{2}}$ around $\xi$. We can write the modified FBI transform as
\begin{equation*}
T_R f(x,\xi) = C_R \int_{\R^{d}} \bar{\phi}_{x,\xi}(y) f(y) dy.
\end{equation*}

The FBI transform conjugates the evolution of \eqref{eq:FirstOrderRoughCoefficients} with controlled error terms to the Hamiltonian flow governed by $a$ in phase space.
By these means, in a series of papers, the second author \cite{Tataru2000, Tataru2002,Tataru03} proved Strichartz estimates for wave equations with $C^{1,1}$-coefficients with the same derivative loss as in Euclidean space (see also Tataru~\cite{T:phase-space} and Koch--Tataru \cite{KochTataru05} for parametrices for symbols in a restricted $S^0_{0,0}$-class).
We remark that for $C^{1,1}$-coefficients, which represents the natural threshold for classical well-posedness of the Hamiltonian flow, after the paradifferential decomposition explained in the beginning of the section, Assumption \ref{ass:HamiltonianFlowRegularity} is satisfied.

We denote the solutions to the Hamiltonian flow with initial data $(x^0,\xi^0)=(x,\xi)$ by $x^t(x,\xi)$, $\xi^t(x,\xi)$ and define the phase shift $\psi$ as solution to the equation
\begin{equation*}
\left\{ \begin{array}{cl}
\frac{d}{dt} \psi (x,t,\xi) &= - p(x^t,t,\xi^t)+ \xi^t \frac{\partial p}{\partial \xi}(x^t,t,\xi^t), \quad (x,t,\xi) \in \overline{\mathcal{Y}} \\
\psi(x,0,\xi) &= 0.
\end{array} \right.
\end{equation*}
We note that for $\mathfrak{h} \in \{1,2\}$ 
\begin{equation*}
\psi(x,t,\xi) = (\mathfrak{h}-1) \int_0^t p(x^s,s,\xi^s) ds.
\end{equation*}

\medskip

An explicit formula for the solution to the homogeneous equation \eqref{eq:FirstOrderRoughCoefficients} with $f =0$ was provided by the second author in \cite{T:phase-space}, see also the later paper \cite{KochTataru05}. Here we obtain the following quantification:
\begin{proposition}
\label{prop:ParametrixConstruction}
Let $R \geq 1$. The kernel $K$ of the fundamental solution to $D_t + p^w$  can be represented in the form
\begin{equation}
\label{eq:KernelFundamentalSolution}
\begin{split}
K(y,t,\tilde{y},s) &= \int_{\R^{2d}} e^{-\frac{\rho}{2}(\tilde{y}-x^s)^2} e^{-i \xi^s (\tilde{y}-x^s)} e^{i ( \psi(x,t,\xi) - \psi(x,s,\xi))} e^{i  \xi^t \cdot (y-x^t)} \\
&\quad \quad \times G(t,s,x,\xi,y) dx d\xi,
\end{split}
\end{equation}
where the function $G$ satisfies
\begin{equation}
\label{eq:RegularityGreen}
\| (R^{-\frac{1}{2}}(x^t - y))^{\gamma} \partial_t^{\sigma} \partial_x^{\alpha} \partial_{\xi}^{\beta} \partial_y^{\nu} G(t,s,x,\xi,y) \|_{L_y^2} \lesssim c_{\nu,\alpha,\beta,\gamma} R^{-\frac{|\alpha| - |\beta| + |\nu|+|\sigma|}{2}} C_R^2 \| \phi_{x,\xi} \|_2
\end{equation}
for $|t-s| \lesssim R$. For $|\alpha| + |\beta| + |\gamma| + \sigma + |\nu| \leq m$ the implicit constants depend only on the quantities $C_{\gamma,\delta}$ for $|\gamma| + |\delta| \leq M(m)$ in \eqref{eq:RegularityEstimatesHamiltonianFlowAssumption}.
\end{proposition}
In the case of normalized frequencies on the unit time interval for symbols in a restricted $S^{0}$-class, this is \cite[Theorem 5]{T:phase-space}. We carry out a detailed analysis to obtain the dependence of the constants on frequency and time.

\smallskip

\emph{Beginning~of~the~Proof~of~Proposition~\ref{prop:ParametrixConstruction}.} The strategy follows \cite[Proposition~4.3]{KochTataru05}. But here we use the rescaled phase space transform.

We let $s=0$ without loss of generality. For $u_0 \in L^2$ we use the modified FBI transform to write $u_0$ as superposition into coherent states:
\begin{equation}
\label{eq:SuperpositionCoherent}
u(y,t) =  S(t,0) T_R^* T_R u_0 = C_{R} \int S(t,0) \phi_{x,\xi}(y) (T_R u_0)(x,\xi) dx d\xi
\end{equation}
with coherent states given by \eqref{eq:CoherentStates}. The function $G$ is correspondingly defined by
\begin{equation*}
G(t,x,\xi,y) = C^2_{R} e^{-i \xi^t (y-x^t)} e^{-i \psi(t,x,\xi)} (S(t,0) \phi_{x,\xi})(y).
\end{equation*}
It remains to verify the bounds on $G$, which rely on bounds for the regularity of the Hamiltonian flow. This will be carried out in Section \ref{section:RegularityFlow}. 

\smallskip

With \eqref{eq:SuperpositionCoherent} at hand, we proceed to the next step:

\medskip

\emph{Construction of wave packets:} We use a smooth partition in phase space. Let $\psi_{x_0,\xi_0}(x,\xi)$ denote a smooth bump function with $\text{supp}(\psi_{x_0,\xi_0}) \subseteq B(x_0,3 R^{\frac{1}{2}}) \times B(\xi_0,3 R^{-\frac{1}{2}})$.
Let $\Lambda_{R} = R^{\frac{1}{2}} \Z^d \times R^{-\frac{1}{2}} \Z^d$. We require that for any $(x,\xi) \in \R^{2n}$
\begin{equation*}
\sum_{(x_0,\xi_0) \in \Lambda_R} \psi_{x_0,\xi_0}(x,\xi) \equiv 1.
\end{equation*}

We obtain a decomposition of the solution
\begin{equation*}
\begin{split}
u &= S(t,0) T_R^* T_R u_0 = C_R S(t,0) T_R^* \sum_{(x_0,\xi_0) \in \Lambda_R} \psi_{x_0,\xi_0}(x,\xi) ( T_R u_0 ) \\
&= C_R \sum_{(x_0,\xi_0) \in \Lambda_R} \int (S(t,0) \phi_{x,\xi}) \psi_{x_0,\xi_0}(x,\xi) (T_R u_0)(x,\xi),
\end{split}
\end{equation*}
and we obtain a decomposition of the kernel
\begin{equation*}
\begin{split}
&\quad K(y,t,\tilde{y},s) \\
&= \sum_{(x_0,\xi_0) \in \Lambda_R} \int_{\R^{2d}} e^{- \frac{\rho}{2}(\tilde{y}-x^s)^2} e^{-i \xi^s (\tilde{y}-x^s)} e^{i \xi^t (y-x^t)} e^{i (\psi(t,x,\xi) - \psi(s,x,\xi))} \\
&\quad \quad \times \psi_{x_0,\xi_0}(x,\xi) G(t,s,x,\xi,y) dx d\xi = \sum_{(x_0,\xi_0) \in \Lambda_R} K_{x_0,\xi_0}(t,y,s,\tilde{y}).
\end{split}
\end{equation*}
What we expect is that the solution obtained from the kernel localized to $(x_0,\xi_0)$ acting on the initial data,
\begin{equation*}
u_{x_0,\xi_0}(y,t) = \int K_{x_0,\xi_0}(y,t,\tilde{y},0) u_0(\tilde{y}) d\tilde{y},
\end{equation*}
has frequency support essentially in an $R^{-\frac{1}{2}}$-ball around $\xi_0^t$ and spatial support in an $R^{\frac{1}{2}}$-ball around $x_0^t$. Indeed, the above formula, together with regularity of the Hamiltonian flow and $G$, already indicates that the wave packets have the desired shape: they can be regarded as curved tubes centered at bicharacteristics with radius of size $R^{\frac{1}{2}}$; see Lemma \ref{lem:GeometryWavepackets}.
\begin{remark}
We note that the phase space localization by $\psi_{x_0,\xi_0}$ combines well with the bi-Lipschitz property proved in Lemma \ref{lem:GeometryWavepackets} in the sense that with respect to the phase space metric the unit distance localization is propagated up to times $|t| \lesssim R$. 
\end{remark}

The bounds for $G$ are proved in the next section.  Once we have established the regularity of $G$ claimed in Proposition \ref{prop:ParametrixConstruction}, we conclude the localization properties of the wave packets in Section \ref{section:Localization}. At the end of Section \ref{section:Localization} we introduce the weight in phase space and finish the proof of Theorem~\ref{thm:WavePacketDecompositions}.

\subsection{The regularity of the Hamiltonian flow}
\label{section:RegularityFlow}
This section is devoted to the analysis of the Hamiltonian flow. We shall see that on the frequency-dependent time $|t| \leq R^{\frac{1}{2}}$ the Hamiltonian flow is smooth - more precisely, the variation with respect to $(x,t,\xi)$ is bounded in terms of the seminorm quantities $C_{\alpha,\beta}$. Naturally, the coherent states $\phi_{x,\xi}$ retain their shape and can be regarded as $\phi_{x_t,\xi_t}$. For $|t| \gg R^{\frac{1}{2}}$, the Hamiltonian flow ceases to be smooth and blows up exponentially - for the precise statement see the lemmas below on the variation with respect to $x,t,\xi$. We can take advantage of the localization of the Gaussians and argue that the coherence persists up to times $|t| \leq R$. 

The main tool to close the bootstrap assumption on the localization of $S(t,0) \phi_{x,\xi}$ are energy estimates. With the above in mind, we would like to regard $S(t,0) \phi_{x,\xi} \approx e^{i \xi^t \cdot (x^t - y)} e^{- \frac{\rho}{2}(x^t - y)^2}$ such that
\begin{equation*}
G(t,x,\xi,y) \approx C_R^2 e^{- \frac{\rho}{2}(x^t - y)^2}.
\end{equation*}
This explains the estimates in Proposition \ref{prop:ParametrixConstruction}.

For ease of presentation, we state the regularity of the Hamiltonian flow in a series of lemmas. We begin with the regularity in frequency variables, which up to a rescaling is \cite[Lemma~9]{Tataru2002}. We revisit the proof because the arguments will be important to obtain the joint regularity in position and frequency variables.

To be precise, we need to prove bounds on the Hamiltonian flow associated with $a$ satisfying Assumption \ref{ass:HamiltonianFlowRegularity}.

\smallskip

The regularity of the Hamiltonian flow is easier to analyze and state for normalized time as various Gr\o nwall arguments are involved. This leads us to analyze Hamiltonians on the unit time interval which satisfy the regularity assumption
\begin{equation}
\label{eq:BoundHamiltonSymbol}
|\partial_{z}^{\alpha} \partial_{\xi}^{\beta} p(x,t,\xi) | \leq C_{\alpha,\beta} R^{\frac{(|\alpha|-2)_+}{2}}.
\end{equation}

\medskip

In the following we analyze in detail the Hamiltonian flow associated with $p$ satisfying \eqref{eq:BoundHamiltonSymbol} on the unit time interval $|t| \lesssim 1$:
\begin{equation}
\label{eq:RescaledHamiltonFlow}
\left\{ \begin{array}{cl}
\dot{x}^t &= \partial_{\xi} p(x^t,t,\xi^t), \\
\dot{\xi}^t &= - \partial_x p(x^t,t,\xi^t).
\end{array} \right.
\end{equation}
When the explicit time-dependence plays no role in the following, it will be suppressed to lighten the notation. Moreover, the variation bounds obtained in the following clearly depend on the quantities $C_{\alpha,\beta}$ for $|\alpha|+|\beta| \leq M$ from \eqref{eq:BoundHamiltonSymbol} - with $M$ determined by the order of the variation. This is likewise suppressed to lighten the notation.

\begin{lemma}[Regularity~in~frequency~variables]
\label{lem:HamiltonFlowFrequencyRegularity}
For $|\beta| \geq 1$, the Hamiltonian flow \eqref{eq:RescaledHamiltonFlow} satisfies for $|t| \leq 2$ the following regularity estimates:
\begin{equation}
\label{eq:RegularityHamiltonFlowFrequencies}
\left\{ \begin{array}{cl}
|\partial_{\xi}^{\beta} x^t| &\lesssim_{\beta} |t| (1+ R^{\frac{1}{2}} |t|)^{|\beta|-1}, \\
|\partial_{\xi}^{\beta} \xi^t| &\lesssim_{\beta} (1+ R^{\frac{1}{2}} |t|)^{|\beta|-1}.
\end{array} \right.
\end{equation}
\end{lemma}
\begin{proof}
For $|\beta| = 1$, using the differentiability of the flow, we find
\begin{equation}
\label{eq:VariationalEquationI}
\frac{d}{dt} \begin{pmatrix}
\partial x^t / \partial \xi \\ \partial \xi^t / \partial \xi
\end{pmatrix}
=
\underbrace{\begin{pmatrix}
\partial_{x \xi}^2 p  & \partial^2_{\xi \xi} p \\
- \partial_{xx}^2 p & - \partial_{x \xi}^2 p 
\end{pmatrix}}_{A(t)}
\begin{pmatrix}
\partial x^t / \partial \xi \\
\partial \xi^t / \partial \xi
\end{pmatrix}
.
\end{equation}
$A(t)$ is clearly bounded for $p \in C^{1,1}$ uniformly. Moreover, $\partial x^0 / \partial \xi = 0$, $\partial \xi^0 / \partial \xi = I_d$. This gives the bounds for $|\beta| = 1$ and $|t| \lesssim 1$.

\medskip

For the induction step $|\beta|>1$ we find by taking $|\beta|-1$ derivatives of \eqref{eq:VariationalEquationI}:
\begin{equation*}
\frac{d}{dt} \begin{pmatrix}
\partial_{\xi}^{\beta} x^t \\ \partial_{\xi}^{\beta} \xi^t
\end{pmatrix}
= A(t) 
\begin{pmatrix}
\partial_{\xi}^{\beta} x^t \\ \partial_{\xi}^{\beta} \xi^t
\end{pmatrix}
+ \eta.
\end{equation*}
The inhomogeneity consists of lower order terms
\begin{equation*}
\eta = \sum_{\beta_1,\beta_2,\zeta} \mu_{\beta_1,\beta_2,\zeta} \partial_{\xi}^{\beta_1} (\partial^2_{\kappa_1 \kappa_2} p)(x^t,\xi^t) \partial_{\xi}^{\beta_2} \partial_{\xi} \zeta^t,
\end{equation*}
with $\kappa_i, \zeta \in \{ x; \xi \}$, $|\beta|-1 = |\beta_1|+|\beta_2|$, $|\beta_2| < |\beta|-1$. $\mu_{\beta_1,\beta_2,\zeta} \in \N$ denote coefficients. The first condition is a consequence of the Leibniz's rule, and the second condition reflects that in the inhomogeneity we are collecting the terms for which not all derivatives fall on the second term in \eqref{eq:VariationalEquationI}.

Furthermore, computing $\partial_{\xi}^{\beta_1} \partial^2_{\kappa_1 \kappa_2} a$ by the chain rule, we find in one case expressions of the form
\begin{equation}
\label{eq:InhomogeneityI}
\begin{split}
&\quad \partial_{\xi}^{\beta_1} \partial^2_{\kappa_1,\kappa_2} p \\
&= \sum_{\beta_{11},\beta_{12}} \mu_{\beta_{11},\beta_{12}} \big[ (\partial^{\beta_{11}}_x \partial^{\beta_{12}}_{\xi} \partial^2_{xx} p)(x^t,\xi^t) \prod_{k=1}^{|\beta_{11}|} (\partial_{\xi}^{\beta_{11}^k} x^t) \prod_{k=1}^{|\beta_{12}|} (\partial_{\xi}^{\beta_{12}^k} \xi^t) \big] \; \cdot \; (\partial_{\xi}^{\beta_2} \partial_{\xi} x^t),
\end{split}
\end{equation}
where 
\begin{itemize}
\item $1 \leq |\beta_{11}| + |\beta_{12}| \leq |\beta_1| \leq |\beta_1|+|\beta_2| = |\beta|-1$, $\quad$ $|\beta_2| < |\beta|-1$,
\item  $1 \leq |\beta_{11}^{k_1}|,|\beta_{12}^{k_2}| \leq |\beta| - 1$, 
\item $\sum_{k_1} |\beta_{11}^{k_1}| + \sum_{k_2} |\beta_{12}^{k_2}| = |\beta_1|$.
\end{itemize} 

There are more terms, involving different second order derivatives of $p$ in the first factor, but these are better behaved since $p$ is smooth in $\xi$. On the other hand, the second factor might read $\partial_{\xi} \xi^t$, for which we have inferior estimates than for $\partial_{\xi} x^t$. We estimate the other terms in \eqref{eq:InhomogeneityII}.

\smallskip

 We have the bound \eqref{eq:BoundHamiltonSymbol}. Plugging in the induction assumption and integrating we finish the proof of \eqref{eq:RegularityHamiltonFlowFrequencies}:
\begin{equation*}
\begin{split}
&\quad \int_0^t \big| \eqref{eq:InhomogeneityI} \big| ds \\
&\lesssim |t| \cdot  R^{\frac{|\beta_{11}|}{2}} \big( \prod_{k=1}^{|\beta_{11}|}  |t| (1+R^{\frac{1}{2}} |t|)^{|\beta_{11}^{k}| - 1} \big) \, \cdot \, \big( \prod_{k=1}^{|\beta_{12}|} (1+ R^{\frac{1}{2}} |t|)^{|\beta_{12}^k| - 1} \big) \\ 
&\quad \times \big( |t| (1+R^{\frac{1}{2}} |t|)^{|\beta_2| } \big) \\
&\lesssim |t| R^{\frac{|\beta_{11}|}{2}} |t|^{|\beta_{11}|} (1+ R^{\frac{1}{2}}|t|)^{\sum_k |\beta_{11}^k| - |\beta_{11}|} (1+ R^{\frac{1}{2}} |t|)^{\sum_k |\beta_{12}^k| - |\beta_{12}|} \\
&\lesssim |t|^2 (1+ R^{\frac{1}{2}} |t|)^{|\beta| -1 - |\beta_{12}|},
\end{split}
\end{equation*}
which in the case $|\beta_{12}| = 0$ is trivially estimated by the above.

A second term we need to estimate is given by
\begin{equation}
\label{eq:InhomogeneityII}
\big[ (\partial^{\beta_{11}}_x \partial^{\beta_{12}}_{\xi} \partial^2_{x \xi} p)(x^t,\xi^t) \prod_{k=1}^{|\beta_{11}|} (\partial_{\xi}^{\beta_{11}^k} x^t) \prod_{k=1}^{|\beta_{12}|} (\partial_{\xi}^{\beta_{12}^k} \xi^t) \big] \; \cdot \; (\partial_{\xi}^{\beta_2} \partial_{\xi} \xi^t),
\end{equation}
for which we find the estimate for $|\beta_2| > 0$:
 \begin{equation*}
 \begin{split}
	&\quad \int_0^t \big| \eqref{eq:InhomogeneityII} \big| ds \\
 	 &\lesssim |t| R^{\frac{(|\beta_{11}|-1)_+}{2}} \big( \prod_{k=1}^{|\beta_{11}|}  |t| (1+R^{\frac{1}{2}} |t|)^{|\beta_{11}^{k}| - 1} \big) \, \cdot \, \big( \prod_{k=1}^{|\beta_{12}|} (1+ R^{\frac{1}{2}} |t|)^{|\beta_{12}^k| - 1} \big) \\ 	 
 &\quad \times \big( (1+R^{\frac{1}{2}} |t|)^{|\beta_2| } \big) \\
 	&\lesssim |t| R^{\frac{(|\beta_{11}|-1)_+}{2}} |t|^{|\beta_{11}|} (1+ R^{\frac{1}{2}} |t|)^{\sum_k |\beta_{11}^k| + \sum_k |\beta_{12}^k| - |\beta_{11}| - |\beta_{12}| + |\beta_2|} \\
 	&\lesssim |t| (1+ R^{\frac{1}{2}} |t|)^{|\beta_1|} (1+R^{\frac{1}{2}} |t|)^{|\beta| -1 - |\beta_1| - |\beta_2|} \\
 	&\lesssim |t| (1+R^{\frac{1}{2}} |t|)^{|\beta|-1}.
 	\end{split}
 \end{equation*}
 
For $\beta_2 = 0$ we find 
\begin{equation*}
\begin{split}
\int_0^t \big| \eqref{eq:InhomogeneityII} \big| ds &\lesssim |t|  R^{\frac{(|\beta_{11}|-1)_+}{2}} \big( \prod_{k=1}^{|\beta_{11}|}  |t| (1+R^{\frac{1}{2}} |t|)^{|\beta_{11}^{k}| - 1} \big) \, \cdot \, \big( \prod_{k=1}^{|\beta_{12}|} (1+ R^{\frac{1}{2}} |t|)^{|\beta_{12}^k| - 1} \big) \\
&\lesssim |t|  R^{\frac{(|\beta_{11}|-1)_+}{2}} |t|^{|\beta_{11}|} (1+ R^{\frac{1}{2}} |t|)^{\sum_k |\beta_{11}^k| + \sum_k |\beta_{12}^k| - |\beta_{11}| - |\beta_{12}|} \\
&\lesssim |t| (1+ R^{\frac{1}{2}} |t|)^{|\beta|-1}.
\end{split}
\end{equation*}
The proof is complete.
\end{proof}

\begin{lemma}[Regularity~of~the~Hamiltonian~flow~in~position~variables]
\label{lem:HamiltonFlowPositionRegularity}
The Hamiltonian flow \eqref{eq:RescaledHamiltonFlow} satisfies the following estimates:
\begin{equation}
\label{eq:RegularityHamiltonFlowSpatialI}
\left\{ \begin{array}{cl}
|\partial_x x^t| &\lesssim 1, \\
|\partial_x \xi^t| &\lesssim |t|,
\end{array} \right.
\end{equation}
and for $|\alpha| \geq 2$, we find
\begin{equation}
\label{eq:RegularityHamiltonFlowSpatialII}
\left\{ \begin{array}{cl}
|\partial_x^{\alpha} x^t| &\lesssim R^{\frac{|\alpha|-1}{2}} |t|, \\
|\partial_x^{\alpha} \xi^t| &\lesssim R^{\frac{|\alpha|-1}{2}} |t|.
\end{array} \right.
\end{equation}
\end{lemma}
\begin{proof}
The argument is similar to the preceding proof. For $|\alpha| = 1$ we find
\begin{equation}
\label{eq:VariationalEquationII}
\frac{d}{dt} \begin{pmatrix}
\partial x^t / \partial x \\ \partial \xi^t / \partial x
\end{pmatrix}
= \underbrace{\begin{pmatrix}
\partial_{x \xi}^2 p & \partial^2_{\xi \xi} p \\
- \partial_{x x}^2 p  & - \partial^2_{x \xi} p
\end{pmatrix}}_{B(t)}
 \begin{pmatrix}
\partial x^t / \partial x \\
\partial \xi^t / \partial x
\end{pmatrix}
\end{equation}
with initial data
\begin{equation*}
\begin{pmatrix}
\partial x^0 / \partial x \\
\partial \xi^0 / \partial x
\end{pmatrix}
=
\begin{pmatrix}
I_d \\ 0
\end{pmatrix}
\end{equation*}
and bounded $B(t)$. This yields \eqref{eq:RegularityHamiltonFlowSpatialI}. 

We turn to the induction step $|\alpha| > 1$, for which by taking $|\alpha|-1$ derivatives of \eqref{eq:VariationalEquationII} we obtain
\begin{equation*}
\frac{d}{dt}
\begin{pmatrix}
\partial_x^{\alpha} x^t \\
\partial_x^{\alpha} \xi^t
\end{pmatrix}
= B(t)
\begin{pmatrix}
\partial_x^{\alpha} x^t \\
\partial_x^{\alpha} \xi^t
\end{pmatrix}
+ \eta
\end{equation*}
with the inhomogeneity consisting of lower order terms. In the ``worst" case we have terms of the form
\begin{equation*}
\begin{split}
&\quad \partial_x^{\alpha_1} (\partial^2_{xx} p(x^t,\xi^t)) \partial_x^{\alpha_2} (\partial x^t / \partial x) \\
&= \sum_{\alpha_{11},\alpha_{12}} (\partial_{x}^{\alpha_{11}} \partial_{\xi}^{\alpha_{12}} \partial^2_{xx} p)(x^t,\xi^t) \prod_{k=1}^{|\alpha_{11}|} (\partial_{x}^{\alpha_{11}^k} x^t) \prod_{k=1}^{|\alpha_{12}|} (\partial_x^{\alpha_{12}^k} \xi^t) \cdot( \partial_x^{\alpha_2} \partial x^t / \partial x).
\end{split}
\end{equation*}
The reason for the other terms being better behaved is that for the second derivative $\partial^2_{xx} p$ any further derivative in $x$ gives rise to factors of $R^{\frac{1}{2}}$, which makes the above term the worst case scenario. The conditions on the indices are given by
\begin{equation*}
\sum_k |\alpha_{11}^k| + \sum_k |\alpha_{12}^k| = |\alpha_1|, \quad 1 \leq |\alpha_{1i}^k| \leq |\alpha_1| \leq |\alpha| - 1, \quad |\alpha_1| + |\alpha_2| = |\alpha| - 1.
\end{equation*}

The induction assumption is applicable and gives
\begin{equation*}
\begin{split}
|\cdot|_{\alpha} &\lesssim |\partial_{x}^{\alpha_{11}} \partial_{\xi}^{\alpha_{12}} \partial^2_{xx} p| \prod_{k=1}^{|\alpha_{11}|} |\partial_x^{\alpha_{11}^k} x^t| \prod_{k=1}^{|\alpha_{12}^k|} |\partial_x^{\alpha_{12}^k} \xi^t| |\partial_x^{\alpha_2} \partial_x x^t| \\
&\lesssim R^{\frac{|\alpha_{11}|}{2}} \prod_{k=1}^{|\alpha_{11}|} R^{\frac{|\alpha_{11}^k| - 1}{2}} \prod_{k=1}^{|\alpha_{12}|} (R^{\frac{|\alpha_{12}^k| - 1}{2}} |t| ) \cdot R^{\frac{|\alpha_2|}{2}} \\
&\lesssim R^{\frac{|\alpha_1|+|\alpha_2|}{2}}.
\end{split}
\end{equation*}
Integration in $t$ shows the claim.
\end{proof}
Next, we turn to the regularity of the Hamiltonian flow in position and frequency. We start with the following:

\begin{lemma}
\label{lem:HamiltonFlowMixedRegularityI}
The Hamiltonian flow \eqref{eq:RescaledHamiltonFlow} satisfies the following estimates for $|\beta| \geq 1$:
\begin{equation}
\label{eq:MixedRegularityI}
\left\{ \begin{array}{cl}
|\partial_{\xi}^{\beta} \partial_x x^t | &\lesssim |t| (1+ R^{\frac{1}{2}} |t|)^{|\beta|}, \\
|\partial_{\xi}^{\beta} \partial_x \xi^t| &\lesssim |t| (1+ R^{\frac{1}{2}} |t|)^{|\beta|}.
\end{array} \right.
\end{equation}
\end{lemma}
\begin{proof}
The proof is again carried out by induction on $|\beta|$. Here the base case is the one for mixed derivatives of first order:
\begin{equation*}
\begin{split}
\frac{d}{dt}  \begin{pmatrix}
\partial_{\xi_i} (\partial x^t / \partial x) \\
\partial_{\xi_i} (\partial x^t / \partial x)
\end{pmatrix} 
&=
\partial_{\xi_i} \frac{d}{dt} \big[ \begin{pmatrix}
\partial_{x} x^t \\ \partial_{x} \xi^t
\end{pmatrix}
\big] = \partial_{\xi_i} ( B(t) \begin{pmatrix}
\partial x^t / \partial x \\ \partial \xi^t / \partial x
\end{pmatrix}
) \\
&= (\partial_{\xi_i} B)
\begin{pmatrix}
\partial x^t / \partial x \\
\partial \xi^t / \partial x
\end{pmatrix}
+ B(t)
\begin{pmatrix}
\partial^2 x^t / (\partial \xi_i \partial x) \\
\partial^2 \xi^t / (\partial \xi_i \partial x)
\end{pmatrix}
.
\end{split}
\end{equation*}
We compute the coefficients of
\begin{equation*}
\partial_{\xi_i} B = \begin{pmatrix}
\frac{\partial^3 p}{\partial x \partial \xi \partial x} \partial_{\xi_i} x + \frac{\partial^3 p}{\partial x (\partial \xi)^2} \partial_{\xi_i} \xi & \frac{\partial^3 p}{(\partial \xi)^2 \partial x} \partial_{\xi_i} x + \frac{\partial^3 p}{(\partial \xi)^3} \partial_{\xi_i} \xi \\
- \frac{\partial^3 p}{(\partial x)^3} \partial_{\xi_i} x - \frac{\partial^3 p}{(\partial x)^2 \partial \xi} \partial_{\xi_i} \xi & - \frac{\partial^3 p}{\partial x \partial \xi \partial x} \partial_{\xi_i} x - \frac{\partial^3 p}{\partial x (\partial \xi)^2} \partial_{\xi_i} \xi
\end{pmatrix}
.
\end{equation*}
Notably, $|\partial^3 p / (\partial x)^3| \lesssim R^{\frac{1}{2}}$. The other derivatives are uniformly bounded. Since $(\partial_{\xi_i} B)_{21}$ acts on $\partial_x x^t$, which is bounded by $1$, we obtain after integration
\begin{equation*}
\int_0^t \big| ( \partial_{\xi_i} B) \begin{pmatrix}
\partial x^t / \partial x \\ \partial \xi^t / \partial x
\end{pmatrix} \big| ds \lesssim \int_0^t (1+ R^{\frac{1}{2}}|s|) ds \lesssim |t| (1+ R^{\frac{1}{2}} |t|).
\end{equation*}
This proves \eqref{eq:MixedRegularityI} for $|\beta| = 1$ because of vanishing initial data.

Next, we prove \eqref{eq:MixedRegularityI} for higher values of $|\beta|$ by induction. To this end, we apply $\partial_{\xi}^{\beta}$ to \eqref{eq:VariationalEquationII} which leads us to
\begin{equation*}
\frac{d}{dt} \begin{pmatrix}
\partial_{\xi}^{\beta} \partial_x  x^t \\ \partial_{\xi}^{\beta} \partial_x \xi^t
\end{pmatrix}
= B(t) 
\begin{pmatrix}
\partial_{\xi}^{\beta} \partial_x x^t \\ \partial_{\xi}^{\beta} \partial_x \xi^t
\end{pmatrix}
+ \eta
\end{equation*}
with inhomogeneity $\eta$ also consists of terms of the form
\begin{equation*}
\partial_{\xi}^{\beta_1} (\partial^2_{xx} p)(x^t, \xi^t) \cdot \partial_{\xi}^{\beta_2} \partial_x x^t, \quad |\beta_1| + |\beta_2| = |\beta|, \quad |\beta_2| \leq |\beta|-1.
\end{equation*}
These terms again give rise to the worst case because we cannot afford anymore $x$-derivatives on $\partial^2_{xx} p$ without losing factors of $R^{\frac{1}{2}}$. The inhomogeneous terms involving derivatives $\partial^2_{x \xi} p$ or $\partial^2_{\xi \xi} p$ are better behaved.

We expand the first factor by the chain rule as above:
\begin{equation}
\label{eq:MixedDerivativesAuxiliary}
= \big( \sum_{\beta_{11},\beta_{12}} \big[ (\partial_x^{\beta_{11}} \partial_{\xi}^{\beta_{12}} \partial^2_{xx} p)(x^t,\xi^t) \prod_{k=1}^{|\beta_{11}^k|} \big( \partial_{\xi}^{\beta_{11}^k} x^t \big) \prod_{k=1}^{|\beta_{12}^k|} \big( \partial_{\xi}^{\beta_{12}^k} \xi^t \big) \big] \big) \; \cdot \; \big( \partial_{\xi}^{\beta_2} \partial_x x^t \big)
\end{equation}
with indices satisfying
\begin{enumerate}
\item $|\beta_1| + |\beta_2| = |\beta|$, $|\beta_2| \leq |\beta|-1$,
\item $\sum_{k=1}^{|\beta_{11}|} |\beta_{11}^k| + \sum_{k=1}^{|\beta_{12}^k|} |\beta_{12}^k| = |\beta_1|, \quad |\beta_{1i}^k| \geq 1$.
\end{enumerate}
For $|\beta_2| = 0$ we find the following bound for \eqref{eq:MixedDerivativesAuxiliary} for fixed $\beta$-indices:
\begin{equation*}
\begin{split}
| \cdot |_{\beta} &\lesssim R^{\frac{|\beta_{11}|}{2}} \prod_{k=1}^{|\beta_{11}|} |t| (1+ R^{\frac{1}{2}} |t|)^{|\beta_{11}^k| - 1} \prod_{k=1}^{|\beta_{12}|} (1+ R^{\frac{1}{2}} |t|)^{|\beta_{12}^k| - 1} \\
&\lesssim R^{\frac{|\beta_{11}|}{2}} |t|^{|\beta_{11}|} (1+R^{\frac{1}{2}} |t|)^{|\beta_1|-|\beta_{11}| -|\beta_{12}|} \\
&\lesssim (1+ R^{\frac{1}{2}} |t|)^{|\beta_1| - |\beta_{12}|},
\end{split}
\end{equation*}
which upon integration gives the claim.

For $|\beta_2| \geq 1$ we obtain
\begin{equation*}
\begin{split}
&\lesssim R^{\frac{|\beta_{11}|}{2}} \prod_{k=1}^{|\beta_{11}|} |t| ( 1+ R^{\frac{1}{2}} |t|)^{|\beta_{11}^k| - 1} \prod_{k=1}^{|\beta_{12}|} (1+R^{\frac{1}{2}} |t|)^{|\beta_{12}^k| - 1} (1+ R^{\frac{1}{2}} |t|)^{|\beta_2|} |t|^{|\beta_2|} \\
&\lesssim (1+ R^{\frac{1}{2}} |t|)^{|\beta| - |\beta_{12}|} |t|^{|\beta_2|},
\end{split}
\end{equation*}
which is better than the first bound. The proof is complete.
\end{proof}

Finally, we formulate the estimates for the joint regularity:
\begin{lemma}
\label{lem:HamiltonFlowMixedRegularityII}
Let $|\alpha| \geq 1$ and $|\beta| \geq 1$. The Hamiltonian flow \eqref{eq:RescaledHamiltonFlow} satisfies the following estimates:
\begin{equation}
\label{eq:MixedRegularityHigherOrder}
\left\{ \begin{array}{cl}
| \partial_x^{\alpha} \partial_{\xi}^{\beta} x^t| &\lesssim_{\alpha,\beta} R^{\frac{|\alpha|-1}{2}} |t| (1+R^{\frac{1}{2}} |t|)^{|\beta|}, \\
| \partial_x^{\alpha} \partial_{\xi}^{\beta} \xi^t| &\lesssim_{\alpha,\beta} R^{\frac{|\alpha|-1}{2}} |t| (1+R^{\frac{1}{2}} |t|)^{|\beta|}.
\end{array} \right.
\end{equation}
\end{lemma}
\begin{proof}
We prove the claim by induction on $|\alpha| + |\beta|$.
The base case $|\alpha| = 1$ and $|\beta| = 1$ has been treated above. The starting point is \eqref{eq:VariationalEquationII} to which we apply $\partial_x^{\alpha} \partial_{\xi}^{\beta}$. We arrive at
\begin{equation*}
\frac{d}{dt} \begin{pmatrix}
\partial_x^{\alpha} \partial_{\xi}^{\beta} \partial_x x^t \\ \partial_x^{\alpha} \partial_{\xi}^{\beta} \partial_x \xi^t
\end{pmatrix}
= B(t) \begin{pmatrix}
\partial_x^{\alpha} \partial_{\xi}^{\beta} \partial_x x^t \\ \partial_x^{\alpha} \partial_{\xi}^{\beta} \partial_x \xi^t
\end{pmatrix}
+ \eta.
\end{equation*}
The ``worst" terms in $\eta$ stem from $\partial_x^{\alpha} \partial_{\xi}^{\beta} ( \partial^2_{xx} p \cdot \partial_x x^t)$, where not all derivatives $\partial_x^{\alpha} \partial_{\xi}^{\beta}$ hit $\partial_x x^t$. After distributing the derivatives on both factors we find terms
\begin{equation*}
\partial_x^{\alpha_1} \partial_{\xi}^{\beta_1} (\partial^2_{xx} p(x^t,\xi^t)) \cdot (\partial_x^{\alpha_2} \partial_{\xi}^{\beta_2} \partial_x x^t)
\end{equation*}
with conditions
\begin{equation*}
|\alpha_1| + |\alpha_2| = |\alpha|, \quad |\beta| = |\beta_1| + |\beta_2|, \quad |\alpha_2| + |\beta_2| < |\alpha| + |\beta|.
\end{equation*}
The first two constraints are due to Leibniz's rule, the last one reflects that not all derivatives act on $\partial_x x^t$ in the inhomogeneity $\eta$.

We expand the first factor to invoke the induction hypothesis: after applying $\partial_{\xi}^{\beta_1} (\partial^2_{xx} p(x^t,\xi^t))$ we find as above terms of the form
\begin{equation}
\label{eq:IntermediateMixedTerm}
\partial_x^{\beta_{11}} \partial_{\xi}^{\beta_{12}} \partial^2_{xx} p(x^t,\xi^t) \prod_{k=1}^{|\beta_{11}|} \partial_{\xi}^{\beta_{11}^k} x^t \prod_{k=1}^{|\beta_{12}|} \partial_{\xi}^{\beta_{12}^k} \xi^t.
\end{equation}
The conditions on $\beta_{ij}^k$ read
\begin{equation*}
1 \leq |\beta_{1j}^k| \leq |\beta_1|, \quad \sum_{k=1}^{|\beta_{11}|} |\beta_{11}^k| + \sum_{k=1}^{|\beta_{12}|} |\beta_{12}^k| =  |\beta_1|.
\end{equation*}
We apply $\partial_x^{\alpha_1}$ to \eqref{eq:IntermediateMixedTerm} to find
\begin{equation*}
\begin{split}
&\quad \sum_{\alpha_{11},\alpha_{12}} \partial_x^{\alpha_{11}} (\partial_x^{\beta_{11}} \partial_{\xi}^{\beta_{12}} \partial^2_{xx} p(x^t,\xi^t)) \partial_x^{\alpha_{12}} \big( \prod_{k=1}^{|\beta_{11}|} \partial_{\xi}^{\beta_{11}^k} x^t \prod_{k=1}^{|\beta_{12}|} \partial_{\xi}^{\beta_{12}^k} \xi^t \big) \\
&= \sum_{\substack{\alpha_{111},\alpha_{112},\alpha_{111}^k, \\ \alpha_{112}^k, \alpha_{121}^k, \alpha_{122}^k}} (\partial_x^{\alpha_{111}} \partial_{\xi}^{\alpha_{112}} \partial_x^{\beta_{11}} \partial_{\xi}^{\beta_{12}} \partial^2_{xx} p) ( \prod_{k=1}^{|\alpha_{111}|} \partial_x^{\alpha_{111}^k} x^t ) (\prod_{k=1}^{|\alpha_{112}|} \partial_x^{\alpha_{112}^k} \xi^t) \\
&\quad \quad \times \big( \prod_{k=1}^{|\beta_{11}|} \partial_x^{\alpha_{121}^k} \partial_{\xi}^{\beta_{11}^k} x^t \prod_{k=1}^{|\beta_{12}|} \partial_x^{\alpha_{122}^k} \partial_{\xi}^{\beta_{12}^k} \xi^t \big).
\end{split}
\end{equation*}
The additional conditions read
\begin{equation*}
\sum_{k=1}^{|\beta_{11}|} |\alpha_{121}^k| + \sum_{k=1}^{|\beta_{12}|} |\alpha_{122}^k| = |\alpha_{12}|, \, |\alpha_{11}| + |\alpha_{12}| = |\alpha_1|, \, \sum_{k=1}^{|\alpha_{111}|} |\alpha_{111}^k| + \sum_{k=1}^{|\alpha_{112}|} |\alpha_{112}^k| = |\alpha_{11}|.
\end{equation*}

Having collected all conditions on indices we can invoke the induction hypothesis to estimate the terms as
\begin{equation}
\label{eq:AllTermsMixedRegularity}
\begin{split}
&\quad |\partial_x^{\alpha_{111}} \partial_{\xi}^{\alpha_{112}} \partial_x^{\beta_{11}} \partial_{\xi}^{\beta_{12}} \partial^2_{xx} p | \prod_{k=1}^{|\alpha_{111}|} |\partial_x^{\alpha_{111}^k} x^t | \prod_{k=1}^{|\alpha_{112}|} |\partial_x^{\alpha_{112}^k} \xi^t | \\
&\quad \quad \times \prod_{k=1}^{|\beta_{11}|} |\partial_x^{\alpha_{121}^k} \partial_{\xi}^{\beta_{11}^k} x^t | \prod_{k=1}^{|\beta_{12}|} |\partial_x^{\alpha_{122}^k} \partial_{\xi}^{\beta_{12}^k} \xi^t | \\
&\lesssim R^{\frac{|\alpha_{111}|}{2}+\frac{|\beta_{11}|}{2}} \prod_{k=1}^{|\alpha_{111}|} R^{\frac{|\alpha_{111}^k|-1}{2}} \prod_{k=1}^{|\alpha_{112}|} |t| R^{\frac{|\alpha_{112}^k|-1}{2}} \prod_{k=1}^{|\beta_{11}|} |\partial_x^{\alpha_{121}^k} \partial_{\xi}^{\beta_{11}^k} x^t | \prod_{k=1}^{|\beta_{12}|} |\partial_x^{\alpha_{122}^k} \partial_{\xi}^{\beta_{12}^k} \xi^t |.
\end{split}
\end{equation}

By the summation properties of the $\alpha$-indices we can estimate the first three factors in \eqref{eq:AllTermsMixedRegularity} by
\begin{equation}
\label{eq:JointRegularityI}
\begin{split}
&\quad R^{\frac{|\alpha_{111}|}{2}+\frac{|\beta_{11}|}{2}} \prod_{k=1}^{|\alpha_{111}|} R^{\frac{|\alpha_{111}^k|-1}{2}} \prod_{k=1}^{|\alpha_{112}|} |t| R^{\frac{|\alpha_{112}^k|-1}{2}} \\
&\lesssim R^{\frac{|\beta_{11}|}{2} + \frac{|\alpha_{11}|}{2} - \frac{|\alpha_{112}|}{2}} |t|^{|\alpha_{112}|}.
\end{split}
\end{equation}

To estimate the fourth factor efficiently, we partition $ \{ 1, \ldots, |\beta_{11}| \}$ into sets $(a)$ and $(b)$:
\begin{equation*}
k \in (a) \Leftrightarrow \alpha_{121}^k = 0, \quad k \in (b) \Leftrightarrow \alpha_{121}^k \geq 1. 
\end{equation*}
We can now effectively estimate
\begin{equation}
\label{eq:AuxiliaryEstimateAllTerms}
\begin{split}
(IV) \text{ in } \eqref{eq:AllTermsMixedRegularity} &\lesssim \prod_{k \in (a)} |t| (1+R^{\frac{1}{2}} |t|)^{|\beta_{11}^k| - 1} \cdot \prod_{k \in (b)} |t| R^{\frac{|\alpha_{121}^k|-1}{2}} (1+ R^{\frac{1}{2}} |t|)^{|\beta_{11}^k|} \\
&\lesssim |t|^{|\beta_{11}|} R^{\sum_{k \in (b)} \frac{|\alpha_{121}^k| - 1}{2}} (1+ R^{\frac{1}{2}} |t|)^{\sum_{k \in (b)} |\beta_{12}^k| + \sum_{k \in (a)} |\beta_{11}^k| - \# (a)}.
\end{split}
\end{equation}
And moreover,
\begin{equation}
\label{eq:JointRegularityII}
\begin{split}
R^{\frac{|\beta_{11}|}{2}} \eqref{eq:AuxiliaryEstimateAllTerms} &\lesssim R^{\frac{|\beta_{11}|}{2}} |t|^{|\beta_{11}|} (1 + R^{\frac{1}{2}} |t|)^{\sum_k |\beta^k_{11}| - \# (a)} R^{\sum_k \frac{|\alpha_{121}^k|}{2} - \# (b)} \\
&\lesssim R^{\sum_k \frac{|\alpha_{121}^k|}{2}} (1+R^{\frac{1}{2}} |t|)^{\sum_k |\beta_{11}^k|}.
\end{split}
\end{equation}

Taking \eqref{eq:JointRegularityI}, \eqref{eq:JointRegularityII}, and the fifth term in \eqref{eq:AllTermsMixedRegularity} together we find
\begin{equation*}
\begin{split}
\eqref{eq:AllTermsMixedRegularity} &\lesssim R^{\frac{|\alpha_{11}|}{2}} R^{\sum_k \frac{|\alpha_{121}^k|}{2}} (1+R^{\frac{1}{2}} |t|)^{\sum_{k} |\beta_{11}^k|} \prod_{k=1}^{|\beta_{12}|} R^{\frac{(|\alpha_{122}^k|-1)_+}{2}} (1+ R^{\frac{1}{2}} |t|)^{|\beta_{12}^k|} \\
&\lesssim R^{\frac{|\alpha_{11}|}{2}} R^{\sum_k \frac{|\alpha_{121}^k|}{2} + \sum_k \frac{|\alpha_{122}^k|}{2}} (1+ R^{\frac{1}{2}} |t|)^{\sum_k |\beta_{11}^k| + \sum_k |\beta_{12}^k|} \\
&\lesssim R^{\frac{|\alpha|}{2}} (1+ R^{\frac{1}{2}}|t|)^{|\beta|}.
\end{split}
\end{equation*}
The estimates for the fifth term are crude, but since possibly $|\beta_{12}| = 0$ there is no overall improvement carrying out a refined analysis.
\end{proof}

We state the essence of the above lemma; on the unit time scale any derivative loses factors of $R^{\frac{1}{2}}$:
\begin{corollary}
Let $|\alpha| \geq 1$ and $|\beta| \geq 1$. Then the following estimate holds for $|t| \lesssim 1$:
\begin{equation*}
|\partial_x^{\alpha} \partial_{\xi}^{\beta} x^t| + |\partial_x^{\alpha} \partial_{\xi}^{\beta} \xi^t| \lesssim_{\alpha,\beta} R^{\frac{|\alpha|+|\beta|-1}{2}} |t|.
\end{equation*}
\end{corollary}

Finally, we consider time regularity:
\begin{lemma}
\label{lem:TimeRegularityHamiltonFlow}
For $|\alpha| + |\beta| + |\sigma| \geq 1$ the following holds:
\begin{equation}
\label{eq:TimeRegularityHamiltonFlow}
|\partial_t^{\sigma} \partial_x^{\alpha} \partial_{\xi}^{\beta} x^t | + |\partial_t^{\sigma} \partial_x^{\alpha} \partial_{\xi}^{\beta} \xi^t | \lesssim_{\alpha,\beta,\sigma} R^{\frac{|\alpha| + |\beta| + |\sigma| -1}{2}}.
\end{equation}
\end{lemma}
Here we do not track the time-dependence anymore. Indeed, we have \\
 $\partial_t^{\sigma}(x^t,\xi^t) \big\vert_{t = 0} \neq 0$. This leads to slightly inferior bounds compared to exclusive space or frequency derivatives because the initial data is non-vanishing in general.
\begin{proof}
We begin by deriving the governing equations for the higher time derivatives. Starting from the Hamiltonian flow
\begin{equation}
\label{eq:HamiltonFlowTimeDerivatives}
\left\{ \begin{array}{cl}
\dot{x}^t &= \partial_{\xi} p(x^t,t,\xi^t), \\
\dot{\xi}^t &= - \partial_x p(x^t,t,\xi^t),
\end{array} \right.
\end{equation}
we find for the second time derivatives:
\begin{equation}
\label{eq:HamiltonFlowTimeI}
\begin{pmatrix}
\ddot{x}^t \\
\ddot{\xi}^t
\end{pmatrix}
=
\underbrace{\begin{pmatrix}
\partial^2_{x \xi} p & \partial^2_{\xi \xi} p \\
- \partial^2_{x x} p & - \partial^2_{x \xi} p
\end{pmatrix}}_{C(t)}
\begin{pmatrix}
\dot{x}^t \\ \dot{\xi}^t
\end{pmatrix}
+
\begin{pmatrix}
\partial^2_{t \xi} p(x^t,t,\xi^t) \\
- \partial^2_{t x} p(x^t,t,\xi^t)
\end{pmatrix}
.
\end{equation}
$C(t)$ is bounded like before $A(t)$ and $B(t)$, and so is the inhomogeneity by the regularity of $a$. This yields boundedness of $(\dot{x}^t,\dot{\xi}^t)$.

We denote the $k$th time derivatives by $x^{(k)}(t)$, $\xi^{(k)}(t)$. By applying $d^k / (dt)^k$ to \eqref{eq:HamiltonFlowTimeI} we find
\begin{equation}
\label{eq:HamiltonFlowHigherOrderTimeDerivative}
\begin{pmatrix}
(\dot{x}^t)^{(k+1)} \\ (\dot{\xi}^t)^{(k+1)}
\end{pmatrix}
= C(t)
\begin{pmatrix}
x^{(k+1)}(t) \\ \xi^{(k+1)}(t)
\end{pmatrix}
+ \eta.
\end{equation}
The terms from the inhomogeneity $\eta$ take the form
\begin{equation*}
\frac{d^{k_1}}{dt^{k_1}} (\partial^2_{\zeta_1 \zeta_2} a ) \cdot \frac{d^{k_2}}{(dt)^{k_2}} \dot{\zeta}, \quad \zeta_i, \zeta \in \{ x ; \xi \}; \quad k_1 + k_2 = k, \; k_2 < k,
\end{equation*}
or
\begin{equation*}
\frac{d^k}{dt^k} \partial^2_{t \zeta_i} p, \quad \zeta_i \in \{x,\xi\}.
\end{equation*}

For the terms of the first kind we compute by the chain and product rule:
\begin{equation*}
\sum_{\kappa_1, \kappa_2, \kappa_3} \partial_x^{\kappa_1} \partial_{\xi}^{\kappa_2} \partial_t^{\kappa_3} \partial^2_{\zeta_1 \zeta_2} p(x^t,t,\xi^t) \prod_{k=1}^{|\kappa_1|} (x^t)^{(\kappa_1^k)} \prod_{k=1}^{|\kappa_2|} (\xi^t)^{(\kappa_2^k)} \cdot \zeta^{(k_2+1)}
\end{equation*}
with
\begin{equation*}
0 \leq |\kappa_1|, |\kappa_2|, |\kappa_3| \leq k_1, \quad \sum_{k=1}^{|\kappa_1|} \kappa_1^{k} + \sum_{k=1}^{|\kappa_2|} \kappa_2^k + \kappa_3 = k_1, \quad 1 \leq |\kappa_1| + |\kappa_2| + |\kappa_3|.
\end{equation*}
For terms of the second kind we find
\begin{equation*}
\sum_{\kappa_1, \kappa_2, \kappa_3} \partial_x^{\kappa_1} \partial_{\xi}^{\kappa_2} \partial_t^{\kappa_3} \partial^2_{\zeta_1 \zeta_2} p(x^t,t,\xi^t) \prod_{k=1}^{|\kappa_1|} (x^t)^{(\kappa_1^k)} \prod_{k=1}^{|\kappa_2|} (\xi^t)^{(\kappa_2^k)} 
\end{equation*}
with
\begin{equation*}
0 \leq |\kappa_1|, |\kappa_2|, |\kappa_3| \leq k_1, \quad \sum_{k=1}^{|\kappa_1|} \kappa_1^{k} + \sum_{k=1}^{|\kappa_2|} \kappa_2^k + \kappa_3 = k, \quad 1 \leq |\kappa_1| + |\kappa_2| + |\kappa_3|.
\end{equation*}

We claim that $|(x^{(k+1)}(t),\xi^{(k+1)}(t))| \lesssim R^{\frac{k}{2}}$  for $k \geq 1$. This is proved by induction. We find for the inhomogeneity, taking the estimate for $l \leq k$ for granted:
\begin{equation*}
\begin{split}
&\quad \big| (\partial_x^{\kappa_1} \partial_{\xi}^{\kappa_2} \partial_t^{\kappa_3} \partial^2_{\zeta_1 \zeta_2} a)(x^t,t,\xi^t) \big| \prod_{k=1}^{|\kappa_1|} x^{(\kappa_1^k)}(t) \big| \big| \prod_{k=1}^{|\kappa_2|} \xi^{(\kappa_2^k)}(t) \big| \cdot \big| \zeta^{(k_2+1)}(t) \big| \\
&\lesssim R^{\frac{|\kappa_1|+|\kappa_3|}{2}} \prod_{k=1}^{\kappa_1} R^{\frac{\kappa_1^k-1}{2}} \prod_{k=1}^{\kappa_2} R^{\frac{\kappa_2^k-1}{2}} \cdot R^{\frac{k_2}{2}} \\
&\lesssim R^{(\sum_k \kappa_1^k + \sum_k \kappa_2^k)/2 + \kappa_3 /2 +  k_2 / 2} = R^{\frac{k_1+k_2}{2}} = R^{\frac{k}{2}}.
\end{split}
\end{equation*}

By the estimate for the inhomogeneity, evaluating \eqref{eq:HamiltonFlowHigherOrderTimeDerivative} at $t=0$ and induction, we find
\begin{equation*}
|(x^{(k)}(0),\xi^{(k)}(0))| \lesssim R^{\frac{(k-2)_+}{2}} \text{ for } k \geq 1.
\end{equation*}
We remark that this yields the refined estimate for $k \geq 1$:
\begin{equation*}
|(x^{(k)}(t),\xi^{(k)}(t))| \lesssim R^{\frac{(k-2)_+}{2}} + |t| R^{\frac{(k-1)_+}{2}}.
\end{equation*}

Next we turn to the estimate of mixed derivatives. The starting point is \eqref{eq:HamiltonFlowHigherOrderTimeDerivative}. We apply $\partial_x^{\alpha}$. We obtain
\begin{equation}
\label{eq:MixedDerivativesTimeRegularityI}
\begin{pmatrix}
\partial_x^{\alpha} \partial_t^k \ddot{x}^t \\ \partial_x^{\alpha} \partial_t^k \ddot{\xi}^t
\end{pmatrix}
= C(t) 
\begin{pmatrix}
\partial_x^{\alpha} \partial_t^k \dot{x}(t) \\ \partial_x^{\alpha} \partial_t^k \dot{\xi}(t)
\end{pmatrix}
+ \underbrace{\sum_{\alpha_1,\alpha_2} \mu_{\alpha_1,\alpha_2} \partial_x^{\alpha_1} (\partial^2_{\zeta_1 \zeta_2} p) \partial_x^{\alpha_2} \zeta^{(k+1)}}_{(I)} + \underbrace{\partial_x^{\alpha} \eta}_{(II)}.
\end{equation}
The inhomogeneities of the first kind are readily estimated as above by expanding
\begin{equation*}
(\partial_x^{\alpha_{11}} \partial_{\xi}^{\alpha_{12}} \partial^2_{\zeta_1 \zeta_2} p) \prod_{k=1}^{|\alpha_{11}|} \partial_x^{\alpha_{11}^k} x^t \prod_{k=1}^{|\alpha_{12}|} \partial_x^{\alpha_{12}^k} \xi^t \cdot \partial_x^{\alpha_2} \zeta^{(k+1)}(t).
\end{equation*}
with
\begin{equation*}
\sum_k |\alpha_{11}^k| + \sum_k |\alpha_{12}^k| = |\alpha_1|, \quad |\alpha_2| <k.
\end{equation*}
We obtain by the previous Lemmas and the induction assumption that
\begin{equation*}
\begin{split}
| \cdot | &\lesssim R^{\frac{|\alpha_{11}|}{2}} R^{\frac{\sum_k |\alpha_{11}^k| - 1}{2}} R^{\frac{\sum_k |\alpha_{12}^k| - 1}{2}} R^{\frac{|\alpha_2|+k-1}{2}} \\
&\lesssim R^{\frac{|\alpha_{11}^k| + \sum_k |\alpha_{12}^k|}{2} + \frac{|\alpha_2| + k-1}{2}} = R^{\frac{|\alpha_1| + |\alpha_2| + k-1}{2}} = R^{\frac{|\alpha|+k-1}{2}}. 
\end{split}
\end{equation*}

For the second term we find in the ``worst" case:
\begin{equation*}
\begin{split}
&\quad \partial_x^{\alpha} (\partial_x^{\kappa_1} \partial_{\xi}^{\kappa_2} \partial_t^{\kappa_3} \partial^2_{xx} p(x^t,t,\xi^t) \prod_{k=1}^{|\kappa_1|} (x^t)^{\kappa_1^k} \prod_{k=1}^{|\kappa_2|} (\xi^t)^{\kappa_2^k} \cdot \zeta^{(k_2+1)} ) \\
&= \sum_{\alpha_1,\alpha_2} \partial_x^{\alpha_1} \big[ \partial_x^{\kappa_1} \partial_{\xi}^{\kappa_2} \partial_t^{\kappa_3} \partial^2_{xx} p(t,x^t,\xi^t) \prod (x^t)^{\kappa_1^k} \prod (\xi^t)^{\kappa_2^k} \big] \cdot \partial_x^{\alpha_2} \zeta^{(k_2+1)} \\
&= \sum_{\alpha_{11},\alpha_{12}} \partial_x^{\alpha_{111}} \partial_{\xi}^{\alpha_{112}} \partial_x^{\kappa_1} \partial_{\xi}^{\kappa_2} \partial_t^{\kappa_3} \partial^2_{xx} p(x^t,t,\xi^t) \prod_{k=1}^{|\alpha_{111}|} \partial_x^{\alpha_{111}^k} x^t \prod_{k=1}^{|\alpha_{112}|} \partial_x^{\alpha_{112}^k} \xi^t \\
&\quad \quad \times \partial_x^{\alpha_{12}} \big[ \big( \prod_{k=1}^{|\kappa_1|} (x^t)^{\kappa_1^k} \big( \prod_{k=1}^{|\kappa_2|} (\xi^t)^{\kappa_2^k} \big) \big] \big( \partial_x^{\alpha_2} \zeta^{(k_2+1)} \big) \\
&= \sum_{\alpha_{11},\alpha_{12}} \partial_x^{\alpha_{111}} \partial_{\xi}^{\alpha_{112}} \partial_x^{\kappa_1} \partial_{\xi}^{\kappa_2} \partial_t^{\kappa_3} \partial^2_{xx} p(x^t,t,\xi^t) \prod_{k=1}^{|\alpha_{111}|} \partial_x^{\alpha_{111}^k} x^t \prod_{k=1}^{|\alpha_{112}|} \partial_x^{\alpha_{112}^k} \xi^t \\
&\quad \quad \times \big[ \big( \prod_{k=1}^{|\kappa_1|} \partial_x^{\alpha_{121}^{k}} (x^t)^{(\kappa_1^k)} \big( \prod_{k=1}^{|\kappa_2|} \partial_x^{\alpha_{122}^{k}} (\xi^t)^{(\kappa_2^k)} \big) \big] \big( \partial_x^{\alpha_2} \zeta^{(k_2+1)} \big).
\end{split}
\end{equation*}
with
\begin{itemize}
\item $\sum_{k=1}^{|\alpha_{111}|} |\alpha_{111}^k| + \sum_{k=1}^{|\alpha_{112}|} |\alpha_{112}^k| = |\alpha_{11}|$,
\item $\sum_{k=1}^{|\kappa_1|} |\alpha_{121}^k| + \sum_{k=1}^{|\kappa_2|} |\alpha_{122}^k| = |\alpha_{12}|$, 
\item $|\alpha_{11}| + |\alpha_{12}| + |\alpha_2| = |\alpha|$.
\end{itemize}

By the induction assumption and the previously established lemmas we find
\begin{equation*}
\begin{split}
| \cdot | &\lesssim R^{\frac{|\alpha_{111}| + |\kappa_1| + |\kappa_3|}{2}} R^{\sum_{k=1}^{|\alpha_{111}|} \frac{|\alpha_{111}|-1}{2}} R^{\sum_{k=1}^{|\alpha_{112}|} \frac{|\alpha_{112}^k|-1}{2} } \\
&\quad \quad \times R^{\sum_{k=1}^{|\kappa_1|} \frac{|\alpha_{121}^k| + |\kappa_1^k| - 1}{2}} R^{\sum_{k=1}^{|\kappa_2|} \frac{|\alpha_{122}^k| + |\kappa_2^k| - 1}{2}} R^{\frac{|\alpha_2| + k_2}{2}} \\
&\lesssim R^{\frac{|\alpha_{11}|}{2} + \frac{|\kappa_1|+ |\kappa_3|}{2}} R^{\frac{|\alpha_{12}|}{2}} R^{\frac{k_1}{2}-\frac{|\kappa_1|}{2}} R^{\frac{|\alpha_2|}{2}+\frac{k_2}{2}} \\
&\lesssim R^{\frac{|\alpha_1| + |\alpha_2| + k_1 + k_2}{2}} = R^{\frac{|\alpha| + k}{2}},
\end{split}
\end{equation*}
which establishes the claim \eqref{eq:TimeRegularityHamiltonFlow} for $\beta  =0$.

To show the claim for mixed derivatives in $x,\xi,t$ we apply $\partial_{\xi}^{\beta}$ to \eqref{eq:MixedDerivativesTimeRegularityI} to find
\begin{equation*}
\begin{split}
\begin{pmatrix}
\partial_{\xi}^{\beta} \partial_x^{\alpha} \partial_t^k \ddot{x}^t \\ \partial_{\xi}^{\beta} \partial_x^{\alpha} \partial_t^k \ddot{\xi}^t
\end{pmatrix}
&= C(t)
\begin{pmatrix}
\partial_{\xi}^{\beta} \partial_x^{\alpha} \partial_t^k \dot{x}^t \\ \partial_{\xi}^{\beta} \partial_x^{\alpha} \partial_t^k \dot{\xi}^t
\end{pmatrix}
 + \underbrace{\sum_{\beta_1,\beta_2} \mu_{\alpha,\beta_1,\beta_2} \partial_{\xi}^{\beta_1} (\partial^2_{\zeta_1 \zeta_2} p) \partial_{\xi}^{\beta_2} \partial_x^{\alpha} \zeta^{(k+1)}}_{(I)}
\\
&\quad + \underbrace{\sum_{\beta_1,\beta_2} \nu_{\alpha,\beta_1,\beta_2} \partial_{\xi}^{\beta_1} \partial_x^{\alpha_1} (\partial^2_{\zeta_1 \zeta_2} p) \partial_{\xi}^{\beta_2} \partial_x^{\alpha_2} \zeta^{(k+1)}}_{(II)} + \underbrace{\partial_{\xi}^{\beta} \partial_x^{\alpha} \eta}_{(III)}.
\end{split}
\end{equation*}
For the estimate of the first term we expand
\begin{equation*}
(I) = \sum_{\substack{\beta_{11},\beta_{12}, \\ \beta_{11}^k, \beta_{12}^k}} \mu_{\beta_1,\beta_2,\alpha} \partial_x^{\beta_{11}} \partial_{\xi}^{\beta_{12}} \partial^2_{\zeta_1 \zeta_2} p \prod_{k=1}^{|\beta_{11}|} \partial_{\xi}^{\beta_{11}^k} x^t \prod_{k=1}^{|\beta_{12}|} \partial_{\xi}^{\beta_2^k} \partial_{\xi}^{\beta_2} \partial_x^{\alpha} \zeta^{(k+1)}
\end{equation*}
with $\beta = \beta_1 + \beta_2$, $\sum_{k=1}^{|\beta_{11}|} |\beta_{11}^k| + \sum_{k=1}^{|\beta_{12}|} |\beta_{12}^k| = |\beta_1| \geq 1$. We find
\begin{equation*}
\begin{split}
|(I)|_{\beta} &\lesssim R^{\frac{|\beta_{11}|}{2}} R^{\sum_{k=1}^{|\beta_{11}|} \frac{|\beta_{11}^k| - 1}{2}} R^{\sum_{k=1}^{|\beta_{12}|} \frac{|\beta_{12}^k|-1}{2}} R^{\frac{|\beta_2|+|\alpha|+k}{2}} \\
&\lesssim R^{\frac{|\beta_1|+|\beta_2|+|\alpha|+k}{2}} = R^{\frac{|\beta|+|\alpha|+k}{2}}.
\end{split}
\end{equation*}

For the estimate of the second term we expand the first factor as
\begin{equation*}
\begin{split}
\partial_{\xi}^{\beta_1} \partial_x^{\alpha_1} \partial^2_{\zeta_1 \zeta_2} p 
&= \sum_{\alpha_{11},\alpha_{12}} \partial_{\xi}^{\beta_1} \big[ \partial_x^{\alpha_{11}} \partial_{\xi}^{\alpha_{12}} \partial^2_{\zeta_1 \zeta_2} p \prod_{k=1}^{|\alpha_{11}|} \partial_x^{\alpha_{11}^k} x^t \prod_{k=1}^{|\alpha_{12}|} \partial_x^{\alpha_{12}^k} \xi^t \big] \\
&= \sum_{\substack{\alpha_{11},\alpha_{12}, \\ \beta_{111},\beta_{112}}} \partial_x^{\beta_{111}} \partial_{\xi}^{\beta_{112}} \partial_x^{\alpha_{11}} \partial_{\xi}^{\alpha_{12}} \partial^2_{\zeta_1 \zeta_2} p \prod_{k=1}^{|\beta_{111}|} \partial_{\xi}^{\beta_{111}^k} x^t \prod_{k=1}^{|\beta_{112}|} \partial_{\xi}^{\beta_{112}^k} \xi^t \\
&\quad \times \partial_{\xi}^{\beta_{12}} \big[ \prod_{k=1}^{|\alpha_{11}|} \partial_x^{\alpha_{11}^k} x^t \prod_{k=1}^{|\alpha_{12}|} \partial_x^{\alpha_{12}^k} \xi^t \big] \\
&= \sum_{\substack{\alpha_{11},\alpha_{12}, \\ \beta_{111},\beta_{112}}} \partial_x^{\beta_{111}} \partial_{\xi}^{\beta_{112}} \partial_x^{\alpha_{11}} \partial_{\xi}^{\alpha_{12}} \partial^2_{\zeta_1 \zeta_2} p \prod_{k=1}^{|\beta_{111}|} \partial_{\xi}^{\beta_{111}^k} x^t \prod_{k=1}^{|\beta_{112}|} \partial_{\xi}^{\beta_{112}^k} \xi^t \\
&\quad \quad \times \prod_{k=1}^{|\alpha_{11}|} \partial_{\xi}^{\beta_{121}^k} \partial_x^{\alpha_{11}^k} x^t \prod_{k=1}^{|\alpha_{12}|} \partial_{\xi}^{\beta_{122}^k} \partial_x^{\alpha_{12}^k} \xi^t.
\end{split}
\end{equation*}

The conditions on the indices are
\begin{itemize}
\item $\sum_{k=1}^{|\alpha_{11}|} |\alpha_{11}^k| + \sum_{k=1}^{|\alpha_{12}|} |\alpha_{12}^k| = |\alpha_1|$,
\item $\sum_{k=1}^{|\beta_{111}|} |\beta_{111}^k| + \sum_{k=1}^{|\beta_{112}|} |\beta_{112}^k| = |\beta_{11}|$,
\item $\sum_{k=1}^{|\alpha_{111}|} |\beta_{121}^k| + \sum_{k=1}^{|\alpha_{12}|} |\beta_{122}^k| = |\beta_{112}|$.
\end{itemize}
With these at hand, we can estimate
\begin{equation*}
\begin{split}
| \cdot | &\lesssim R^{\frac{|\beta_{111}|}{2} + \frac{|\alpha_{111}|}{2}} R^{\sum_{k=1}^{|\beta_{111}|} \frac{|\beta_{111}^k| - 1}{2}} R^{\sum_{k=1}^{|\beta_{112}|} \frac{|\beta_{112}^k|-1}{2}} R^{\sum_{k=1}^{|\alpha_{11}|} \frac{|\beta_{121}^k| + |\alpha_{11}^k| - 1}{2}} \\
&\quad \quad \times R^{\sum_{k=1}^{|\alpha_{12}|} \frac{|\beta_{122}^k| + |\alpha_{12}^k| - 1}{2}} \\
&\lesssim R^{\frac{|\beta_{11}|}{2} + \frac{|\alpha_{11}|}{2}} R^{\sum_{k=1}^{|\alpha_{11}|} \frac{|\beta_{121}^k| + |\alpha_{11}^k| - 1}{2}} R^{\sum_{k=1}^{|\alpha_{12}|} \frac{|\beta_{122}^k| + |\alpha_{12}^k| - 1}{2}} \\
&\lesssim R^{\frac{|\beta_{11}|}{2} + \frac{|\alpha_{11}|}{2}} R^{\frac{|\alpha_1|}{2} - \frac{|\alpha_{11}|}{2}} R^{\frac{|\beta_{12}|}{2}} = R^{\frac{|\alpha_1|+|\beta_1|}{2}}.
\end{split}
\end{equation*}
For the second factor we find by induction assumption
\begin{equation*}
| \partial_{\xi}^{\beta_2} \partial_x^{\alpha_2} \zeta^{(k+1)} | \lesssim R^{\frac{|\beta_2| + |\alpha_2| + k}{2}}.
\end{equation*}
Together with the previous bound we conclude the estimate of $(II)$.

We turn to the estimate of $(III)$. To this end, we write for $\beta_1 = \beta_{11} + \beta_{12} + \beta_{13}$, $\beta = \beta_1 + \beta_2$:
\begin{equation*}
\begin{split}
&\sum \partial_{\xi}^{\beta_{11}} (\partial_x^{\alpha_{111}} \partial_{\xi}^{\alpha_{112}} \partial_x^{\kappa_1} \partial_{\xi}^{\kappa_2} \partial_t^{\kappa_3} \partial^2_{\zeta_1 \zeta_2} p) \partial_{\xi}^{\beta_{12}} \big( \prod_{k=1}^{|\alpha_{111}|} \partial_x^{\alpha_{111}^k} x^t \prod_{k=1}^{|\alpha_{112}|} \partial_x^{\alpha_{112}^k} \xi^t \big) \\
&\quad \quad \times \partial_{\xi}^{\beta_{13}} \big( \prod_{k=1}^{|\kappa_1|} \partial_x^{\alpha_{121}^k} \partial_t^{\kappa_1^k} x^t \prod_{k=1}^{|\kappa_2|} \partial_x^{\alpha_{122}^k} \partial_t^{\kappa_2^k} \xi^t \big) \partial_{\xi}^{\beta_2} \partial_x^{\alpha_2} \zeta^{(k_2+1)} \\
&= \sum \partial_x^{\beta_{111}} \partial_{\xi}^{\beta_{112}} \partial_x^{\alpha_{111}} \partial_{\xi}^{\alpha_{112}} \partial_x^{\kappa_1} \partial_{\xi}^{\kappa_2} \partial_t^{\kappa_3} \partial^2_{\zeta_1 \zeta_2} p \prod_{k=1}^{|\beta_{111}|} \partial_{\xi}^{\beta_{111}^k} x^t \prod_{k=1}^{|\beta_{112}|} \partial_{\xi}^{\beta_{112}^k} \xi^t \\
&\quad \quad \times \prod_{k=1}^{|\alpha_{111}|} \partial_{\xi}^{\beta_{121}^k} \partial_x^{\alpha_{111}^k} x^t \prod_{k=1}^{|\alpha_{112}|} \partial_{\xi}^{\beta_{122}^k} \partial_x^{\alpha_{112}^k} \xi^t \\
&\quad \quad \times \prod_{k=1}^{|\kappa_1|} \partial_{\xi}^{\beta_{131}^k} \partial_x^{\alpha_{121}^k} \partial_t^{\kappa_1^k} x^t \prod_{k=1}^{|\kappa_2|} \partial_{\xi}^{\beta_{132}^k} \partial_x^{\alpha_{122}^k} \partial_t^{\kappa_2^k} \xi^t \; \times \; \partial_{\xi}^{\beta_2} \partial_x^{\alpha_2} \zeta^{(k_2+1)}.
\end{split}
\end{equation*}
In addition to the conditions on $\alpha_{ij}$ from above, we have
\begin{itemize}
\item $\sum_{k=1}^{|\beta_{111}|} |\beta_{111}^k| + \sum_{k=1}^{|\beta_{112}|} |\beta_{112}^k| = |\beta_{11}|$, 
\item $\sum_{k=1}^{|\alpha_{111}|} |\beta_{121}^k| + \sum_{k=1}^{|\alpha_{112}|} |\beta_{122}^k| = |\beta_{12}|$,
\item $\sum_{k=1}^{|\kappa_1|} |\beta_{131}^k| + \sum_{k=1}^{|\kappa_2|} |\beta_{132}^k| = |\beta_{13}|$.
\end{itemize}

We can estimate the single terms as follows:
\begin{equation*}
\begin{split}
| \cdot | &\lesssim R^{\frac{|\beta_{111}|}{2} + \frac{|\kappa_1|+ |\kappa_3|}{2}} R^{\sum_{k=1}^{|\beta_{111}|} \frac{|\beta_{111}^k|-1}{2}} R^{\sum_{k=1}^{|\beta_{112}|} \frac{|\beta_{112}^k| - 1}{2}} \\
&\quad  \times R^{\frac{|\alpha_{111}|}{2}} R^{\sum_{k=1}^{|\alpha_{111}|} \frac{|\beta_{121}^k| + |\alpha_{111}^k| - 1}{2}} R^{\sum_{k=1}^{|\alpha_{112}|} \frac{|\beta_{122}^k| + |\alpha_{112}^k| - 1}{2}} \\
&\quad \times R^{\frac{|\kappa_1|}{2}} R^{\sum_{k=1}^{|\kappa_1|} \frac{|\beta_{131}^k| + |\alpha_{121}^k| + |\kappa_1^k| - 1}{2}} R^{\sum_{k=1}^{|\kappa_2|} \frac{|\beta_{132}^k| + |\alpha_{122}^k| + |\kappa_2^k| - 1}{2}} R^{\frac{|\beta_2| + |\alpha_2| + k_2}{2}} \\
&\lesssim R^{\frac{|\beta_{11}|}{2}} R^{\frac{|\alpha_{11}|}{2} + \frac{|\beta_{12}|}{2}} R^{\frac{|\beta_{13}|}{2} + \frac{|\alpha_{12}|}{2} + \frac{k_1}{2}} R^{\frac{|\beta_2| + |\alpha_2| + k_2}{2}} \\
&\lesssim R^{\frac{|\beta| + |\alpha| + k}{2}}.
\end{split}
\end{equation*}
The claim follows from integration in time.
\end{proof}

Let $(x'_{t'},\xi'_{t'})$ denote a solution to the time-normalized Hamiltonian flow \eqref{eq:RescaledHamiltonFlow}. We observe that $(x^t,\xi^t) = (R x'_{R^{-1} t}, \xi'_{R^{-1} t})$ with $|t| \lesssim R$ yields solutions to the original Hamiltonian flow
\begin{equation*}
\left\{ \begin{array}{cl}
\dot{x}^t &= \partial_{\xi} p(x^t,t, \xi^t), \\
\dot{\xi}^t &= - \partial_x p( x^t,t,\xi^t),
\end{array} \right.
\end{equation*}
which was considered in Assumption \ref{ass:HamiltonianFlowRegularity}. We record the bounds for the original solution:
\begin{corollary}
Let $(x^t,\xi^t)$ be a solution to the Hamiltonian flow governed by $p$ satisfying Assumption \ref{ass:HamiltonianFlowRegularity}. The following estimate holds for $|\alpha| + |\beta| + \sigma \geq 1$:
\begin{equation}
\label{eq:RegularityOriginalHamiltonFlow}
\begin{split}
|\partial_t^{\sigma} \partial_x^{\alpha} \partial_{\xi}^{\beta} x^t|  &\lesssim_{\alpha,\beta,\sigma} R^{\frac{-|\alpha| + |\beta| - |\sigma| + 1}{2}}, \\
|\partial_t^{\sigma} \partial_x^{\alpha} \partial_{\xi}^{\beta} \xi^t| &\lesssim_{\alpha,\beta,\sigma} R^{\frac{-|\alpha|+|\beta|-|\sigma|-1}{2}}.
\end{split}
\end{equation}
For $|\alpha| + |\beta| + \sigma \leq m$ the implicit constants depend only on the constants $C_{\gamma,\delta}$ \eqref{eq:RegularityEstimatesHamiltonianFlowAssumption} for $|\gamma| + |\delta| \leq M(m)$.
\end{corollary}

We use the following to center the (modified) Green function $G$:
\begin{lemma}
\label{lem:RegularityPhaseFunction}
Assume that $p$ satisfies Assumption \ref{ass:HamiltonianFlowRegularity}. Then for $|\alpha|+|\beta| + |\sigma| + |\nu| \geq 1$ we have
\begin{equation}
\label{eq:RegularityHamiltonII}
\begin{split}
&\quad |\partial_x^{\alpha} \partial_y^{\nu} \partial_{\xi}^{\beta} \partial_t^{\sigma} [ \xi^t (y-x^t) +  \xi_0 (x-x_0) \\
&\quad - \psi(x_0,t,\xi_0) + \psi(x,t,\xi) - \xi_0^t (y-x_0^t)] |_{x=x_0,\xi=\xi_0} \\
&\leq c_{\alpha \beta \sigma \nu} R^{\frac{-|\alpha|+|\beta|-|\sigma|}{2}} (1+R^{-\frac{1}{2}} |y-x_0^t| ).
\end{split}
\end{equation}
\end{lemma}
\begin{proof}
We note that derivatives in $y$ can be handled by linearity. We suppose in the following that $\nu = 0$. Moreover, for $|\alpha|=|\nu|=|\beta|=0$, that is to say, considering only time derivatives, we see that the expression vanishes.

\smallskip

In the following we turn to the main estimate for derivatives in $x$, $\xi$, and $t$ such that $|\alpha|+|\beta| \geq 1$. Let $|\alpha| = 1$ and $|\beta| = \sigma = 0$. We have
\begin{equation*}
\partial_x [ \xi^t (y-x^t) + \psi(t,x,\xi) + \xi_0 (x-x_0)] = \partial_x \xi^t (y-x^t) + \partial_x \psi(t,x,\xi) - \xi^t \partial_x x^t + \xi_0.
\end{equation*}
The first term is estimated by Lemma \ref{lem:HamiltonFlowFrequencyRegularity}:
\begin{equation*}
| [ \partial_x \xi^t (y-x^t) ]_{x=x_0,\xi=\xi_0} | \lesssim R^{-\frac{1}{2}} \cdot R^{-\frac{1}{2}} |y-x_0^t|.
\end{equation*}
For the second term let $F(t) = [ \xi^t \partial_x x^t - \partial_x \psi(x,t,\xi) - \xi_0 ]_{x = x_0,\xi = \xi_0}$. Clearly, $F(0) = 0$ and we have for the time derivative
\begin{equation*}
\begin{split}
&\, \frac{d [\xi^t \partial_x x^t - \partial_x \psi(x,t,\xi) - \xi_0]}{dt} \\
 &= -  \partial_x p(x^t,\xi^t) \partial_x x^t + \xi^t \partial_x ( \partial_{\xi} p(x^t, \xi^t)) \\
&\quad - \partial_x [ \xi^t \partial_{\xi} p(x^t,t;\xi^t) - p(x^t,t;\xi^t)] \\
&= - \partial_x p(x^t,t,\xi^t) \partial_x x^t - \partial_x \xi^t \partial_{\xi} p(x^t,t,\xi^t) + \partial_x p(x^t,t,\xi^t) \partial_x x^t +  \partial_{\xi} p(x^t,t;\xi^t) \partial_x \xi^t = 0.
\end{split}
\end{equation*}

\medskip

For other combinations of $\alpha,\beta,\sigma$ we have $\partial_x^{\alpha} \partial_{\xi}^{\beta} \partial_t^{\sigma} [ \xi_0 (x-x_0)] = 0$. We recast the remaining terms as
\begin{equation*}
\begin{split}
&\quad \partial_x^{\alpha} \partial_{\xi}^{\beta} \partial_t^{\sigma} [  \xi^t (y-x^t) +  \psi(t,x,\xi)] \\
 &=  \partial_x^{\alpha} \partial_{\xi}^{\beta} \partial_{t}^{\sigma} \xi^t \cdot (y-x^t) +  \xi^t \partial_x^{\alpha} \partial_{\xi}^{\beta} \partial_t^{\sigma} (y-x^t) + (\text{mixed derivatives}) + \partial_x^{\alpha} \partial_{\xi}^{\beta} \partial_t^{\sigma}  \psi(t,x,\xi).
\end{split}
\end{equation*}
The first term is simply estimated by Lemma \ref{lem:TimeRegularityHamiltonFlow}:
\begin{equation*}
|\partial_x^{\alpha} \partial_{\xi}^{\beta} \partial_t^{\sigma} \xi^t \cdot (y-x^t)| \big\vert_{x=x_0,\xi=\xi_0} \lesssim R^{\frac{-|\alpha|+|\beta|-|\sigma|}{2}} R^{-\frac{1}{2}} |y-x_0^t|.
\end{equation*}

The terms with derivatives as well on $\xi^t$ as on $y-x^t$ are estimated by Leibniz's rule and Lemma \ref{lem:TimeRegularityHamiltonFlow}. Estimating the second and fourth term is most involved. Let
\begin{equation*}
F(t) = [ \xi^t \partial_x^{\alpha} \partial_{\xi}^{\beta} \partial_t^{\sigma} x^t - \partial_x^{\alpha} \partial_{\xi}^{\beta} \partial_t^{\sigma} \psi(x,t,\xi) ]_{x=x_0,\xi=\xi_0}.
\end{equation*}

We handle the case $\sigma \geq 1$ first. We compute
\begin{equation*}
\begin{split}
&\quad \xi^t \partial_x^{\alpha} \partial_{\xi}^{\beta} \partial_t^{\sigma -1 } (\partial_{\xi} p(x^t,t;\xi^t)) - \partial_x^{\alpha} \partial_{\xi}^{\beta} \partial_t^{\sigma -1 } [ \xi^t \partial_{\xi} p(x^t,t;\xi^t) - p ] \\
&= - (\partial_x^{\alpha} \partial_{\xi}^{\beta} \partial_t^{\sigma - 1} \xi^t) \partial_{\xi} p(x^t,t;\xi^t) - \text{mixed der.} + \partial_x^{\alpha} \partial_{\xi}^{\beta} \partial_t^{\sigma - 1} p.
\end{split}
\end{equation*}
For the mixed terms we have
\begin{equation*}
\text{mixed der.} = \sum_{\substack{\alpha_1,\alpha_2, \beta_1, \\ \beta_2, \sigma_1, \sigma_2 }} (\partial_x^{\alpha_1} \partial_{\xi}^{\beta_1} \partial_t^{\sigma_1} \xi^t) (\partial_x^{\alpha_2} \partial_{\xi}^{\beta_2} \partial_t^{\sigma_2} \partial_{\xi} a)
\end{equation*}
with the constraints $\sigma_1 + \sigma_2 = \sigma - 1$, $\alpha_1 + \alpha_2 = \alpha$, $\beta_1 + \beta_2 = \beta$, $|\alpha_1|+|\beta_1|+|\sigma_1| \geq 1$, and $|\alpha_2|+|\beta_2|+|\sigma_2| \geq 1$. Noting that $\partial_{\xi} p = \partial_t x^t$ and applying Lemma \ref{lem:TimeRegularityHamiltonFlow} yields
\begin{equation*}
|\text{mixed der.}| \lesssim R^{\frac{-|\alpha_1|+|\beta_1|-\sigma_1 - 1}{2}} R^{\frac{- |\alpha_2| + |\beta_2| - (\sigma_2+1)+1}{2}} = R^{\frac{-|\alpha|+|\beta|-\sigma}{2}}.
\end{equation*}
This makes the contribution of the mixed terms admissible.
For the contribution of the first term we can apply Lemma \ref{lem:TimeRegularityHamiltonFlow} and symbol regularity:
\begin{equation*}
|\partial_x^{\alpha} \partial_{\xi}^{\beta} \partial_t^{\sigma - 1} \xi^t \partial_{\xi} p| \lesssim R^{\frac{-|\alpha|+|\beta|-(\sigma -1) -1}{2}} = R^{\frac{-|\alpha|+|\beta|-\sigma}{2}}.
\end{equation*}

It remains to treat the final term.
The expansion is reminiscent of the proof of Lemma \ref{lem:HamiltonFlowMixedRegularityII}. We obtain for $\partial_t^{\sigma-1} p$ by the chain rule:
\begin{equation*}
\begin{split}
&\sum_{\kappa_1,\kappa_2, \kappa_3} \big( \partial_x^{\kappa_1} \partial_{\xi}^{\kappa_2} \partial_t^{\kappa_3} p \big) \big[ \prod_{k=1}^{|\kappa_1|} \partial_t^{\kappa_1^k} x^t \prod_{k=1}^{|\kappa_2|} \partial_t^{\kappa_2^k} \xi^t \big]; \quad 0 \leq |\kappa_1|, |\kappa_2|,|\kappa_3| \leq \sigma-1, \\ 
&\quad \quad  \sum_{k=1}^{|\kappa_1|} \kappa_1^k + \sum_{k=1}^{|\kappa_2|} \kappa_2^k + \kappa_3  = \sigma-1.
\end{split}
\end{equation*}
Next, we apply $\partial_x^{\alpha}$, which leads to terms
\begin{equation*}
\begin{split}
&\quad \sum_{\substack{\alpha_1,\alpha_2, \\ \alpha_{11}, \alpha_{12}}} (\partial_x^{\alpha_{11}} \partial_{\xi}^{\alpha_{12}} \partial_x^{\kappa_1} \partial_{\xi}^{\kappa_2} \partial_t^{\kappa_3} p) \cdot \big[ \prod_{k=1}^{|\alpha_{11}|} \partial_x^{\alpha_{11}^k} x^t \prod_{k=1}^{|\alpha_{12}|} \partial_x^{\alpha_{12}^k} \xi^t  \big] \cdot \partial_x^{\alpha_{2}} \big[ \prod_{k=1}^{|\kappa_1|} \partial_t^{\kappa_1^k} x^t \prod_{k=1}^{|\kappa_2|} \partial_t^{\kappa_2^k} \xi^t \big] \\
&= \sum_{\substack{\alpha_1,\alpha_2, \\ \alpha_{11}, \alpha_{12}}} (\partial_x^{\alpha_{11}} \partial_{\xi}^{\alpha_{12}} \partial_x^{\kappa_1} \partial_{\xi}^{\kappa_2} \partial_t^{\kappa_3} p) \cdot \big[ \prod_{k=1}^{|\alpha_{11}|} \partial_x^{\alpha_{11}^k} x^t \prod_{k=1}^{|\alpha_{12}|} \partial_x^{\alpha_{12}^k} \xi^t \big] \\
&\quad \quad \cdot \big[ \prod_{k=1}^{|\kappa_1|} \partial_t^{\kappa_1^k} \partial_x^{\alpha_{21}^k} x^t \prod_{k=1}^{|\kappa_2|} \partial_t^{\kappa_2^k} \partial_x^{\alpha_{22}^k} \xi^t \big]
\end{split}
\end{equation*}
with
\begin{equation*}
    \alpha_1 + \alpha_2 = \alpha, \, \sum_{k=1}^{|\alpha_{11}|} |\alpha_{11}^k| + \sum_{k=1}^{|\alpha_{12}|} |\alpha_{12}^k| = |\alpha_{1}|, \, \sum_{k=1}^{|\kappa_1|} |\alpha_{21}^k | + \sum_{k=1}^{|\kappa_2|} |\alpha_{22}^k| = |\alpha_2|.
\end{equation*}

Finally, we apply $\partial_{\xi}^{\beta}$ to the above terms to find
\begin{equation*}
\begin{split}
&\partial_{\xi}^{\beta_1} (\partial_x^{\alpha_{11}} \partial_{\xi}^{\beta_{12}} \partial_x^{\kappa_1} \partial_{\xi}^{\kappa_2} \partial_t^{\kappa_3}  p) \cdot \partial_{\xi}^{\beta_2} \big[ \prod_{k=1}^{|\alpha_{11}|} \partial_x^{\alpha_{11}^k} x^t \prod_{k=1}^{|\alpha_{12}|} \partial_x^{\alpha_{12}^k} \xi^t \big] \\
&\quad \quad \cdot \partial_{\xi}^{\beta_3} \big[ \prod_{k=1}^{|\kappa_1|} \partial_t^{\kappa_1^k} \partial_x^{\alpha_{21}^k} x^t \prod_{k=1}^{|\kappa_2|} \partial_t^{\kappa_2^k} \partial_x^{\alpha_{22}^k} \xi^t \big] \\
&= \sum (\partial_x^{\beta_{11}} \partial_{\xi}^{\beta_{12}} \partial_x^{\alpha_{11}} \partial_{\xi}^{\alpha_{12}} \partial_x^{\kappa_1} \partial_{\xi}^{\kappa_2} \partial_t^{\kappa_3} p) \cdot \big( \prod_{k=1}^{|\beta_{11}|} \partial_{\xi}^{\beta_{11}^k} x^t \prod_{k=1}^{|\beta_{12}|} \partial_{\xi}^{\beta_{12}^k} \xi^t \big) \\
&\quad \cdot \big( \prod_{k=1}^{|\alpha_{11}|} \partial_{\xi}^{\beta_{21}^k} \partial_{x}^{\alpha_{11}^k} x^t \prod_{k=1}^{|\alpha_{12}|} \partial_{\xi}^{\beta_{22}^k} \partial_x^{\alpha_{12}^k} \xi^t \big) \cdot \big( \prod_{k=1}^{|\kappa_1|} \partial_{\xi}^{\beta_{31}^k} \partial_t^{\kappa_1^k} \partial_x^{\alpha_{21}^k} x^t \prod_{k=1}^{|\kappa_2|} \partial_{\xi}^{\beta_{32}^k} \partial_t^{\kappa_2^k} \partial_x^{\alpha_{22}^k} \xi^t \big).
\end{split}
\end{equation*}
The conditions on $\beta$ read
\begin{itemize}
\item $\beta = \beta_1 + \beta_2 + \beta_3$, 
\item $\sum_{k=1}^{|\beta_{11}|} |\beta_{11}^k| + \sum_{k=1}^{|\beta_{12}|} |\beta_{12}^k| = |\beta_1|$, 
\item $\sum_{k=1}^{|\alpha_{11}|} |\beta_{21}^k| + \sum_{k=1}^{|\alpha_{12}|} |\beta_{22}^k| = |\beta_2|$, 
\item $\sum_{k=1}^{|\kappa_1|} |\beta_{31}^k| + \sum_{k=1}^{|\kappa_2|} |\beta_{32}^k| = |\beta_{3}|$.
\end{itemize}

We obtain the $R$-dependent bounds:
\begin{equation}
\label{eq:PhaseFunctionBound}
\begin{split}
|\cdot|_{\alpha,\beta,\sigma} &\lesssim R^{-|\beta_{11}| - |\alpha_{11}| - |\kappa_1| - |\kappa_3|}  R^{\frac{(|\alpha_{11}| + |\beta_{11}| + |\kappa_1| + |\kappa_3| - 2)_+}{2}} \\
&\quad \times R^{\sum_{k=1}^{|\beta_{11}|} \frac{|\beta_{11}^k| - 1}{2}} R^{\sum_{k=1}^{|\beta_{12}|} \frac{|\beta_{12}^k|-1}{2}} \\
&\quad \times R^{\sum_{k=1}^{|\alpha_{11}|} \frac{|\beta_{21}^k| - |\alpha_{11}^k| + 1}{2}} R^{\sum_{k=1}^{|\alpha_{12}|} \frac{|\beta_{22}^k|- |\alpha_{12}^k| - 1}{2}} \\
&\quad \times R^{\sum_{k=1}^{|\kappa_1|} \frac{|\beta_{31}^k| - |\kappa_1^k| - |\alpha_{12}^k| + 1}{2}} 
 R^{\sum_{k=1}^{|\kappa_2|} \frac{|\beta_{32}^k| - |\kappa_2^k| - |\alpha_{22}^k| - 1}{2}} \\
&\lesssim R^{-|\beta_{11}| - |\alpha_{11}| - |\kappa_1| - |\kappa_3|} R^{\frac{(|\alpha_{11}| + |\beta_{11}| + |\kappa_1| + |\kappa_3| - 2)_+}{2}} R^{\frac{|\beta_1|}{2} + \frac{|\beta_{11}|}{2} - \frac{|\beta_{12}|}{2}}  \\
&\quad \times R^{\frac{|\beta_2|}{2} - \frac{|\alpha_1|}{2} + \frac{|\alpha_{11}|}{2} - \frac{|\alpha_{12}|}{2}} R^{\frac{|\beta_3|}{2}-\frac{|\alpha_2|}{2} - \frac{\sum_k |\kappa_1^k|}{2} - \frac{\sum |\kappa_2^k|}{2} +  \frac{|\kappa_1|}{2}-\frac{|\kappa_2|}{2}}.
\end{split}
\end{equation}

In case $|\alpha_{11}|+|\beta_{11}| + |\kappa_1| + |\kappa_3| \geq 2$, the expression simplifies to
\begin{equation*}
| \cdot |_{\alpha,\beta,\sigma} \lesssim R^{\frac{-|\alpha|+|\beta|-\sigma-1}{2}},
\end{equation*}
which yields an acceptable contribution.

\smallskip

\textbf{Case A:} $|\beta_{11}|=|\alpha_{11}|=|\kappa_1| = |\kappa_3|= 0$.

The above simplifies to
\begin{equation*}
| \cdot |_{\alpha,\beta,\sigma} \lesssim R^{\frac{-|\alpha|+|\beta|-|\sigma|+1}{2} - \frac{|\kappa_2|}{2} - \frac{|\alpha_{12}|}{2} - \frac{|\beta_{12}|}{2}}.
\end{equation*}

\textbf{Case A.I:} For $|\kappa_2| + |\alpha_{12}| + |\beta_{12}| \geq 1$, the estimate suffices.

\textbf{Case A.II:} $|\kappa_2| = 0$, $|\alpha_{12}| = |\beta_{12}| = 0$. In this case we find $\sigma = 1$ and $|\alpha| = |\beta| = 0$, which is not considered here as this amounts to only time derivatives, which was handled above.
\smallskip

Finally, let us turn to $\sigma = 0$ and $(|\alpha|,|\beta|) \neq (1,0)$ as the case $(|\alpha|,|\beta|) = (1,0)$ was handled above.
We consider again
\begin{equation*}
F(t) = [ \xi^t \partial_x^{\alpha} \partial_{\xi}^{\beta} x^t - \partial_x^{\alpha} \partial_{\xi}^{\beta} \psi(x,t,\xi) ]_{x=x_0,\xi = \xi_0}.
\end{equation*}
It seems difficult to estimate this directly. So, we observe that $F(0)= 0$ and estimate the time-derivative by $|F'(t)| \lesssim R^{\frac{-|\alpha|+|\beta|-2}{2}}$.
We compute (and omit the evaluation to ease notation)
\begin{equation}
\label{eq:FPrime}
\begin{split}
F'(t) &= - \partial_x p \cdot \partial_x^{\alpha} \partial_{\xi}^{\beta} x^t + \xi^t \partial_x^{\alpha} \partial_{\xi}^{\beta} \partial_t x^t - \partial_x^{\alpha} \partial_{\xi}^{\beta} [ \xi^t \partial_{\xi} p] + \partial_x^{\alpha} \partial_{\xi}^{\beta} p \\
&= - \partial_x p \partial_x^{\alpha} \partial_{\xi}^{\beta} x^t - \sum_{(*)} (\partial_x^{\alpha_1} \partial_{\xi}^{\beta_1} \xi^t) (\partial_x^{\alpha_2} \partial_{\xi}^{\beta_2} (\partial_{\xi} p)) + \partial_x^{\alpha} \partial_{\xi}^{\beta} p.
\end{split}
\end{equation}
The summation restricted to $(*)$ refers to $\alpha_1,\beta_1,\alpha_2,\beta_2$ with $|\alpha_1|+|\beta_1| \neq 0$, noting that the complementary term cancels the second term of the first line.

In the following we expand $\partial_x^{\alpha} \partial_{\xi}^{\beta} p$ and see that the terms, which do not cancel with the first and second term, contribute only to order $R^{\frac{-|\alpha|+|\beta|-2}{2}}$.

We treat $\beta = 0$ and $|\beta| \geq 1$ separately. We first consider $\beta = 0$. In this case we expand
\begin{equation*}
\begin{split}
    \partial_x^{\alpha} \partial_{\xi}^{\beta} p &= \partial_x^{\alpha} p = \sum_{\alpha_1,\alpha_2} \partial_x^{\alpha_1} \partial_{\xi}^{\alpha_2} \prod_{k=1}^{|\alpha_1|} \partial_x^{\alpha_1^k} x^t \prod_{k=1}^{|\alpha_2|} \partial_x^{\alpha_2^k} \xi^t \\
    &= \partial_x p \partial_x^{\alpha} x^t + \partial_{\xi} p \partial_x^{\alpha} \xi^t + \sum_{\substack{\alpha_1,\alpha_2,\\|\alpha_1|+|\alpha_2| \geq 2}} (\partial_x^{\alpha_1} \partial_{\xi}^{\alpha_2} p) \prod_{k=1}^{|\alpha_1|} \partial_x^{\alpha_1^k } x^t  \prod_{k=1}^{|\alpha_2|} \partial_x^{\alpha_2^k} \xi^t.
\end{split}    
\end{equation*}
The first and second expression cancels with the first and second term in \eqref{eq:FPrime}. For the third expression we estimate by the symbol regularity and Lemma \ref{lem:TimeRegularityHamiltonFlow}:
\begin{equation*}
        \lesssim R^{-|\alpha_1|} R^{\frac{(|\alpha_1|-2)_+}{2}} R^{\sum_{k=1}^{|\alpha_1|} \frac{-|\alpha_1^k|}{2}+\frac{1}{2}} R^{\sum_{k=1}^{|\alpha_2|} \frac{-|\alpha_2^k|}{2} - \frac{1}{2}}.
\end{equation*}
For $|\alpha_1| \geq 2$ this is acceptable by symbol regularity. For $\alpha_1 = 0$ this is $\lesssim R^{-\frac{|\alpha|}{2}-\frac{|\alpha_2|}{2}}$, which is acceptable because $|\alpha_2| \geq 2$. For $|\alpha_1| = 1$ we find $\lesssim R^{-1} R^{-\frac{|\alpha|}{2}} R^{\frac{1}{2}} R^{-\frac{|\alpha_2|}{2}}$, again acceptable by $|\alpha_2| \geq 2$. 

\smallskip

We turn to the case $|\beta| \geq 1$, for which we expand 
\begin{equation*}
\partial_{\xi}^{\beta} p = \sum_{\beta_1,\beta_2} \partial_x^{\beta_1} \partial_{\xi}^{\beta_2} p \cdot \prod_{i=1}^{|\beta_1|} \partial_{\xi}^{\beta_1^k} x^t \prod_{k=1}^{|\beta_2|} \partial_{\xi}^{\beta_2} \xi^t \text{ with } \sum_k|\beta_1^k| + \sum_j |\beta_2^j| = |\beta|.
\end{equation*}

This is split as
\begin{equation}
\label{eq:FPrimeExpAux}
    (\partial_{\xi} p) \partial_{\xi}^{\beta} \xi^t + (\partial_x p) \partial_{\xi}^{\beta} x^t + \sum_{|\beta_1|+|\beta_2| \geq 2} (\partial_{x}^{\beta_1} \partial_{\xi}^{\beta_2} p) \big( \prod_{k=1}^{|\beta_1|} x^t \big) \big( \prod_{k=1}^{|\beta_2|} \partial_{\xi}^{\beta_2} \xi^t \big).
\end{equation}
It remains to apply $\partial_x^{\alpha}$: For the first two terms we find
\begin{equation}
\label{eq:FPrimeAux}
    \sum_{\alpha_1,\alpha_2} \partial_x^{\alpha_1} (\partial_\xi p) \partial_x^{\alpha_2} \partial_{\xi}^{\beta} \xi^t + (\partial_x^{\alpha_1} (\partial_x p)) \partial_{x}^{\alpha_2} \partial_{\xi}^{\beta} x^t.
\end{equation}
The terms with $|\alpha_1| = 0$ in \eqref{eq:FPrimeAux} are cancelled by the first term in \eqref{eq:FPrime} and the second term in \eqref{eq:FPrime} with the condition on the summation $|\alpha_2| + |\beta_2| = 0$.

\textbf{First expression in \eqref{eq:FPrimeAux}:} The single terms are given by
\begin{equation*}
    \partial_x^{\alpha_1} (\partial_{\xi} p) (\partial_x^{\alpha_2} \partial_{\xi}^{\beta} \xi^t) = \sum_{\alpha_{11},\alpha_{12}} (\partial_{x}^{\alpha_{11}} \partial_{\xi}^{\alpha_{12}} \partial_{\xi} p) \big( \prod_{k=1}^{|\alpha_{11}} \partial_x^{\alpha_{11}^k} x^t \big) \big( \prod_{k=1}^{|\alpha_{12}|} (\partial_x^{\alpha_{12}^k} \xi^t) \big( \partial_x^{\alpha_2} \partial_{\xi}^{\beta} \xi^t \big).
\end{equation*}
We have the conditions $|\alpha_{11}| + |\alpha_{12}| \geq 1$, $\sum_{k=1}^{|\alpha_{11}|} |\alpha_{11}^k| + \sum_{k=1}^{|\alpha_{12}^k|} |\alpha_{12}^k| = |\alpha_1|$. We estimate a single term $|\cdot|_{\alpha,\beta}$ by symbol regularity and Lemma \ref{lem:TimeRegularityHamiltonFlow}:
\begin{equation*}
\begin{split}
|\cdot|_{\alpha,\beta}| &\lesssim R^{-|\alpha_{11}|} R^{\frac{(|\alpha_{11}|-2)_+}{2}} R^{\sum_{k=1}^{|\alpha_{11}|} (- \frac{|\alpha_{11}^k|}{2} + \frac{1}{2})} R^{\sum_{k=1}^{|\alpha_{12}|} (- \frac{|\alpha_{12}^k|}{2} - \frac{1}{2} )} R^{\frac{-|\alpha_2|+|\beta|-1}{2}} \\
&\lesssim R^{-|\alpha_{11}|} R^{\frac{(|\alpha_{11}|-2)_+}{2}} R^{-\frac{|\alpha_1|}{2}+\frac{|\alpha_{11}|}{2}-\frac{|\alpha_{12}|}{2}} R^{\frac{-|\alpha_2|+|\beta|-1}{2}}.
\end{split}
\end{equation*}
Now we can conclude with a case-by-case analysis. For $|\alpha_{11}| \geq 2$ the expression is found to be $|\cdot|_{\alpha,\beta} \lesssim R^{\frac{-|\alpha|+|\beta|-2}{2}}$. For $|\alpha_{11}| = 0$ we find for a single term
\begin{equation*}
   |\cdot|_{\alpha,\beta} \lesssim R^{\sum_{k=1}^{|\alpha_{12}|} (- \frac{|\alpha_{12}^k|}{2} - \frac{1}{2})} R^{\frac{-|\alpha_2|+|\beta|-1}{2}} \lesssim R^{\frac{-|\alpha|+|\beta|-2}{2}}.
\end{equation*}
In the ultimate estimate we used $|\alpha_1| \geq 1$ by assumption such that $|\alpha_{12}| \geq 1$.
Similarly, the term with $|\alpha_{11}| = 1$ can be estimated:
\begin{equation*}
    \begin{split}
    |\cdot|_{\alpha,\beta} &\lesssim R^{-1} R^{-\frac{|\alpha_1|}{2}} R^{\frac{1}{2}} R^{-\frac{|\alpha_{12}|}{2}} R^{\frac{-|\alpha_2|+|\beta|-1}{2}} \\
    &\lesssim R^{\frac{-|\alpha|+|\beta|-2}{2}}.
    \end{split}
\end{equation*}

\textbf{Second expression in \eqref{eq:FPrimeAux}}: we have, recalling that $|\alpha_1| \geq 1$:
\begin{equation*}
    \partial_x^{\alpha_1} (\partial_x p) \partial_{x_2}^{\alpha_2} \partial_{\xi}^{\beta} x^t = \sum_{\alpha_{11},\alpha_{12}} (\partial_{x}^{\alpha_{11}} \partial_{\xi}^{\alpha_{12}} \partial_x p) \big( \prod_{k=1}^{|\alpha_{11}|} \partial_{x}^{\alpha_{11}^k} x^t \big) \big( \prod_{k=1}^{|\alpha_{12}|} \partial_x^{\alpha_{12}^k} \xi^t \big) \big( \partial_x^{\alpha_2} \partial_{\xi}^{\beta} x^t \big).
\end{equation*}
The indices satisfy $\sum_{k=1}^{|\alpha_{11}|} |\alpha_{11}^k| + \sum_{k=1}^{|\alpha_{12}|} |\alpha_{12}^k| = |\alpha_1| \geq 1$. 
We obtain for a single term the estimate:
\begin{equation*}
\begin{split}
 |\cdot|_{\alpha,\beta} &\lesssim R^{-|\alpha_{11}|-1} R^{\frac{(|\alpha_{11}|-1)_+}{2}} R^{\frac{-|\alpha_1|+|\alpha_{11}|-|\alpha_{12}|}{2}} R^{-\frac{|\alpha_2|}{2}+\frac{|\beta|}{2}-\frac{1}{2}} \\
 &\lesssim R^{-|\alpha_{11}|-1} R^{\frac{(|\alpha_{11}|-1)_+}{2}} R^{-\frac{|\alpha|}{2}+\frac{|\alpha_{11}|}{2}-\frac{|\alpha_{12}|}{2}+\frac{|\beta|}{2}-\frac{1}{2}}.
 \end{split}
\end{equation*}
It is straightforward that $|\alpha_{11}| \geq 1$ gives an acceptable contribution, and we check that $\alpha_{11} = 0$ likewise.

We have taken care of the first and second term in \eqref{eq:FPrimeExpAux}. It remains to estimate the third term, for which holds $|\beta_1|+|\beta_2| \geq 2$:
\begin{equation*}
    \begin{split}
        &\quad \partial_x^{\alpha} \big[ (\partial_x^{\beta_1} \partial_{\xi}^{\beta_2} p) \big( \prod_{k=1}^{|\beta_1|} \partial_{\xi}^{\beta_1^k} x^t \big) \big( \prod_{k=1}^{|\beta_2|} \partial_{\xi}^{\beta_2^k} \xi^t \big) \big] \\
        &= \sum_{\alpha_{11},\alpha_{12},\ldots} (\partial_x^{\alpha_{11}} \partial_{\xi}^{\alpha_{12}} \partial_x^{\beta_1} \partial_{\xi}^{\beta_2} p) \prod_{k=1}^{|\alpha_{11}|} \partial_x^{\alpha_{11}^k} x^t \prod_{k=1}^{|\alpha_{12}|} \partial_x^{\alpha_{12}^k} \xi^t \prod_{k=1}^{|\beta_1|} (\partial_x^{\alpha_{21}^k} \partial_{\xi}^{\beta_1^k} x^t \big) \prod_{k=1}^{|\beta_2|} \big( \partial_x^{\alpha_{22}^k} \partial_{\xi}^{\beta_2^k} \xi^t \big)
    \end{split}
\end{equation*}
and additionally,
\begin{equation*}
|\alpha_1|+|\alpha_2| = |\alpha|, \, \sum_k|\alpha_{11}^k| + \sum_k |\alpha_{12}^k| = |\alpha_1|, \, \sum_{k=1}^{|\beta_1|} |\alpha_{21}^k| + \sum_{k=1}^{|\beta_2|} |\alpha_{22}^k| = |\alpha_2|.
\end{equation*}
This allows us to estimate a single term as
\begin{equation*}
    |\cdot|_{\alpha,\beta} \lesssim R^{-|\alpha_{11}|-|\beta_1|} R^{\frac{(|\alpha_{11}|+|\beta_1|-2)_+}{2}} R^{-\frac{|\alpha_1|}{2}} R^{\frac{|\alpha_{11}|-|\alpha_{12}|}{2}} R^{-\frac{|\alpha_2|}{2}} R^{\frac{|\beta|}{2}} R^{\frac{|\beta_1|}{2}} R^{-\frac{|\beta_2|}{2}}.
\end{equation*}
For $|\beta_1| \geq 2$ the expression is acceptable by the symbol regularity, and more generally for $|\beta_{11}|+|\alpha_{11}| \geq 2$. In case $|\beta_1| = 1$ and $|\alpha_{11}| = 0$ the estimate
\begin{equation*}
     |\cdot|_{\alpha,\beta} \lesssim R^{-1} R^{-\frac{|\alpha_1|}{2}} R^{-\frac{|\alpha_{12}|}{2}} R^{-\frac{|\alpha_2|}{2}} R^{\frac{|\beta|}{2}} R^{\frac{|\beta_1|}{2}} R^{-\frac{|\beta_2|}{2}} \lesssim R^{\frac{-|\alpha|+|\beta|}{2}} R^{-1}.
\end{equation*}
Finally, we estimate for $|\beta_1|=0$, $|\beta_2| \geq 2$
\begin{equation*}
    |\cdot|_{\alpha,\beta} \lesssim R^{-|\alpha_{11}|} R^{\frac{(|\alpha_{11}|-2)_+}{2}} R^{-\frac{|\alpha_1|}{2}} R^{\frac{|\alpha_{11}|-|\alpha_{12}|}{2}} R^{-\frac{|\alpha_2|}{2}} R^{\frac{|\beta|}{2}} R^{-1}.
\end{equation*}
From Case-By-Case analysis for $|\alpha_{11}|=0$, $|\alpha_{11}| = 1$, $|\alpha_{11}| \geq 2$, we see that the expression is always estimated $\lesssim R^{\frac{-|\alpha|}{2}+\frac{|\beta|}{2}-1}$.

\smallskip

We have proved that $|F'(t)| \lesssim R^{\frac{-|\alpha| + |\beta| + |\sigma|-2}{2}}$,
which completes the proof by integrating over $|t| \lesssim R$.
\end{proof}

\begin{proof}[Proof~of~Proposition~\ref{prop:ParametrixConstruction}]
Suppose that we want to prove \eqref{eq:RegularityGreen} for $(x,\xi) = (x_0,\xi_0)$. With Lemma \ref{lem:RegularityPhaseFunction} at hand, it suffices to show regularity and decay for the modified function at $(x_0,\xi_0)$:
\begin{equation*}
G_1(t,x,\xi,y) = C_R^2 e^{-i \xi_0^t (y-x_0^t) } e^{-i \psi(x_0,t,\xi_0)} (S(t,0) (e^{i \xi_0 (x-x_0)} \phi_{x,\xi}))(y).
\end{equation*}
We translate $G_1$ to the origin by
\begin{equation*}
\begin{split}
G_2(t,x,\xi,y) &= G_1(t,x_0+x,\xi_0+\xi,x_0^t + y) \\
&= C_R^2 e^{-i \xi_0^t \cdot y} e^{-i \psi(x_0,t,\xi_0)} (S(t,0) (e^{i \xi_0 \cdot x} \phi_{x+x_0,\xi+\xi_0}))(y+x_0^t)
\end{split}
\end{equation*}
and we shall verify \eqref{eq:RegularityGreen} for $G_2$ at the origin. We have the following lemma:
\begin{lemma}
$G_2$ solves
\begin{equation*}
(D_t + p_2^w(y,t,D_y)) G_2 = 0, \quad G_2(0) = C_{R}^2 \phi_{x,\xi} = C_R^2 e^{i \xi (x-y)} e^{- \rho(y-x)^2},
\end{equation*}
where
\begin{equation*}
p_2(y,t,\eta) = p(x_0^t+y,t, \xi_0^t + \eta) - p(x_0^t,t,\xi_0^t) - y p_x(x_0^t,t,\xi_0^t) - \eta p_\xi(x_0^t,t,\xi_0^t).
\end{equation*}    
\end{lemma}
\begin{proof}
    Verifying the claimed initial value is straightforward. Next, we compute
    \begin{equation}
    \label{eq:ExpansionTimeDerivativeModGreen}
        \begin{split}
            (-i \partial_t G_2) &= C_R^2 e^{-i \xi_0^t \cdot y} e^{-i \psi(x_0,t,\xi_0)} \big( \big[ (\partial_x p(x_0^t,t,\xi_0^t) \cdot y - [ \xi_0^t \partial_{\xi} p(x_0^t,t,\xi_0^t) - p(x_0^t,t,\xi_0^t)]) \\
            &\quad \times (S(t,0) (e^{i \xi_0 \cdot x} \phi_{x+x_0,\xi+\xi_0}))(y+x_0^t) \big] \\
            \quad &+ (-p^w (y+x_0^t,t,D_y))[S(t,0) e^{i \xi_0 \cdot x} \phi_{x+x_0,\xi+\xi_0}](y+x_0^t) \\
            \quad &+ (-i \partial_y)[S(t,0) e^{i \xi_0 \cdot x} \phi_{x+x_0,\xi+\xi_0}](y+x_0^t) \partial_{\xi} p(x_0^t,t,\xi_0^t) \big).
        \end{split}
    \end{equation}
    For the second summand on the rhs we find
    \begin{equation}
        \label{eq:ModGreenII}
        \begin{split}
        &\quad C_R^2 e^{-i \xi_0^t \cdot y} e^{-i \psi(x_0,t,\xi_0)} (-p^w(y+x_0^t,t,D_y)[S(t,0) e^{i \xi_0 \cdot x} \phi_{x+x_0,\xi+\xi_0}](y+x_0^t) \\
        &= - p^w(y+x_0^t,t,\eta+\xi_0^t)[ C_R^2 e^{-i \xi_0^t \cdot y} e^{-i \psi(x_0,t,\xi_0)} S(t,0) (e^{i \xi_0 \cdot x} \phi_{x+x_0,\xi+\xi_0})](y+x_0^t).
        \end{split}
    \end{equation}
    For the third summand on the rhs we obtain
    \begin{equation}
        \label{eq:ModGreenIII}
        \begin{split}
            &\quad C_R^2 e^{-i \xi_0^t \cdot y} e^{-i \psi(x_0,t,\xi_0)} (-i \partial_y) [S(t,0) e^{i \xi_0 \cdot x} \phi_{x+x_0,\xi+\xi_0}](y+x_0^t) \partial_{\xi} a(x_0^t,t,\xi_0^t) \\
            &= \partial_{\xi} a(x_0^t,t,\xi_0^t) (-i \partial_y) [C_R^2 e^{-i \xi_0^t \cdot y} e^{-i \psi(x_0,t,\xi_0)} S(t,0) \big[ e^{i \xi_0 \cdot x} \phi_{x+x_0,\xi+\xi_0}](y+x_0^t)]
        \end{split}
    \end{equation}
    Taking the above three equations together yields that $G_2$ satisfies the claimed evolution equation satisfied by $G_2$.
\end{proof}

For this function it suffices to prove for that for times $|t| \lesssim R$ we have
\begin{equation}
\label{eq:ModifiedGreenBounds}
\| | (R^{-\frac{1}{2}} y)^{\gamma} \partial_x^{\alpha} \partial_t^{\sigma} \partial_{\xi}^{\beta} \partial_y^{\nu} G_2(t,x,\xi,y)| \big|_{x=0,\xi = 0} \|_2 \leq c_{\gamma,\alpha,\beta,\nu,\sigma} R^{\frac{-|\alpha| + |\beta| - |\nu| - |\sigma|}{2}} C_R^2 \| \phi_{x,\xi} \|_2. 
\end{equation}
The notation $c_{\gamma,\alpha,\beta,\nu,\sigma}$ indicates that for any $N$ such that $|\gamma|+|\alpha|+|\beta|+|\nu|+|\sigma| \leq N$ we can find a constant $c$ such that \eqref{eq:ModifiedGreenBounds} holds uniformly in $R \geq 1$ with constant depending on seminorm quantities $C_{\gamma,\delta}$ from \eqref{eq:RegularityEstimatesHamiltonianFlowAssumption} with $|\gamma| + |\delta| \leq M(N)$. 
The proof of Proposition \ref{prop:ParametrixConstruction} will be complete, once we have established \eqref{eq:ModifiedGreenBounds}.
\end{proof}

\begin{proof}[Proof~of~\eqref{eq:ModifiedGreenBounds}]
We note that $\phi_{x,\xi}$ is a Schwartz function with derivatives in $x$ and $\xi$ at the origin satisfying the bounds
\begin{equation*}
\| \partial_x^{\alpha} \partial_{\xi}^{\beta} \phi_{x,\xi} \vert_{x=0,\xi=0} \|_{L^2} \lesssim R^{(-|\alpha|+|\beta|)/2} \| \phi_{0,0} \|_2.
\end{equation*}
Here it is important to have centered phase space variables. This handles the derivatives in $x$ and $\xi$. It remains to check derivatives in $y$, $t$, and spatial confinement.

To show \eqref{eq:ModifiedGreenBounds}, we use energy estimates for $L^2$-normalized Schwartz functions
\begin{equation*}
(D_t + p_2^w(y,D)) v = 0, \quad v(0) = v_0,
\end{equation*}
and prove
\begin{equation*}
\| (R^{\frac{1}{2}} \partial_t)^{\gamma} (R^{-\frac{1}{2}} y)^{\alpha} (R^{\frac{1}{2}} \partial_y)^{\beta} v(t) \|_2 \lesssim \sum_{|\alpha'| + |\beta'| \leq |\alpha| + |\beta| + 2 |\gamma|} \| (R^{-\frac{1}{2}} y)^{\alpha'} (R^{\frac{1}{2}} \partial_y)^{\beta'} v(0) \|_2.
\end{equation*}

Recall that $|T| \leq R$.
In the first step we show
\begin{equation}
\label{eq:EnergyArgumentsI}
\| (R^{-\frac{1}{2}} y)^{\alpha} (R^{\frac{1}{2}} \partial_{y})^{\beta} v(t) \|_2 \lesssim \sum_{|\bar{\alpha}| + |\bar{\beta}| \leq |\alpha| + |\beta|} \| (R^{-\frac{1}{2}} y)^{\bar{\alpha}} (R^{\frac{1}{2}} \partial_{y})^{\bar{\beta}} v_0 \|_2.
\end{equation}
Time regularity is considered further below. We prove the estimate by induction on $|\alpha| + |\beta|$. For $|\alpha| + |\beta| = 1$ we shall prove that for 
\begin{equation*}
A(t) = \| R^{-\frac{1}{2}} y v(t) \|_2^2 + \| R^{\frac{1}{2}} \partial_{\eta} v \|_2^2 + \| v(t) \|_2^2 \text{ it holds } \frac{d}{dt} A(t) \lesssim R^{-1} A(t).
\end{equation*}
and the claim follows from Gr\o nwall's inequality.
The above will follow from bounds
\begin{align}
\label{eq:EnergyBoundI}
\| (D_t + p_2^w) (R^{-\frac{1}{2}} y v) \|_2 &\lesssim R^{-1} (\| v(t) \|_2 + \| R^{-\frac{1}{2}} y v(t) \|_2 + \| R^{\frac{1}{2}} \partial_{\eta} v(t) \|_2), \\
\label{eq:EnergyBoundII}
\| (D_t + p_2^w) (R^{\frac{1}{2}} \partial_{\eta} v)(t) \|_2 &\lesssim R^{-1} ( \| v(t) \|_2 + \| R^{-\frac{1}{2}} y v(t) \|_2  + \| R^{\frac{1}{2}} \partial_{\eta} v(t) \|_2).
\end{align}
Indeed, the first bound implies
\begin{equation*}
\begin{split}
\frac{d}{dt} \| R^{-\frac{1}{2}} y v(t) \|_2^2 &= 2 \Re ( \langle R^{-\frac{1}{2}} y v(t) , \big( \frac{d}{dt} + i p_2^w \big) (R^{- \frac{1}{2}} y v(t)) \rangle ) \\
&\lesssim \| R^{-\frac{1}{2}} y v(t) \|_2 (R^{-1} ( \| v(t) \|_2 + \| R^{-\frac{1}{2}} y v(t) \|_2 + \| R^{\frac{1}{2}} \partial_{\eta} v(t) \|_2 ) ) \\
&\lesssim R^{-1} A(t).
\end{split}
\end{equation*}
The first equation follows from self-adjointness of $p_2^w$, the second one from Cauchy-Schwarz inequality and \eqref{eq:EnergyBoundI}. Similarly, the second bound implies an estimate $\partial_t \| R^{\frac{1}{2}} \partial_{\eta} v(t) \|_2^2 \lesssim R^{-1} A(t)$.

To verify \eqref{eq:EnergyBoundI} and \eqref{eq:EnergyBoundII}, we compute the equations satisfied by $R^{-\frac{1}{2}} y v$ and $R^{\frac{1}{2}} \partial_y v$. By taking commutators for the Weyl quantization, like in \cite[p.~238]{KochTataru05} we have
\begin{equation}
\label{eq:EnergyCommutator}
\left\{ \begin{array}{cl}
(D_t + p_{2}^w) (R^{-\frac{1}{2}} y v) &= - i (\partial_{\eta} p_{2})^w(y,D_y) R^{-\frac{1}{2}} v, \\
(D_t + p_{2}^w)(R^{\frac{1}{2}} \partial_y v) &= i R^{\frac{1}{2}} (\partial_y p_{2})^w(y,D_y) v.
\end{array} \right.
\end{equation}

We compute by the fundamental theorem of calculus 
\begin{equation}
\label{eq:SymbolExpansionSecondOrderEta}
\begin{split}
(\partial_{\eta} p_2)(y,\eta) &= p_{\eta}(t,x_0^t + y, \xi_0^t + \eta) - p_{\eta}(t,x_0^t, \xi_0^t) \\
&= \int_0^1 p_{\eta y}(t,x_0^t + sy, \xi_0^t + s \eta) y + p_{\eta \eta}(t,x_0^t + s y, \xi_0^t + s \eta) \cdot \eta ds
\end{split}
\end{equation}
and note that we can equivalently prove bounds for the standard quantization $\| \partial_{\eta} p_2(y,D) R^{-\frac{1}{2}} v \|_2 + \| R^{\frac{1}{2}} \partial_y p_2(y,D) v\|_2$ because the difference to Weyl quantization is estimated by
\begin{equation*}
R^{-\frac{1}{2}} \| (\langle D_y, D_\eta \rangle \partial_{\eta} p_2) v \|_{L^2} 
\lesssim R^{-1} ( \|v(t) \|_2 + \| R^{-\frac{1}{2}} y v(t) \|_2 + \| R^{\frac{1}{2}} \partial_{\eta} v(t) \|_2),
\end{equation*}
and similarly for the second term.

We plug in the expansion in standard quantization and obtain
\begin{equation}
\label{eq:StandardQuantizationReduction}
\begin{split}
(\partial_{\eta} p_2)(y,D_y) R^{-\frac{1}{2}} v &= C_d \int_{\R^d} e^{-i y \eta} \int_0^1 p_{\eta y}(x_0^t + sy,t, \xi_0^t + s \eta) y R^{-\frac{1}{2}} \hat{v}(\eta) ds d\eta \\
&\quad + C_d \int_{\R^d} e^{-i y \eta} \int_0^1 p_{\eta y} (x_0^t + sy,t, \xi_0^t + s \eta) \eta R^{-\frac{1}{2}} \hat{v}(\eta) ds d\eta.
\end{split}
\end{equation}
The first term in \eqref{eq:StandardQuantizationReduction} we write as
\begin{equation*}
\begin{split}
&\quad \int_0^1 \int_{\R^d} e^{-i y \eta} p_{\eta y}(x_0^t + sy,t, \xi_0^t + s \eta) (y R^{-\frac{1}{2}} \hat{v}(\eta)) d\eta ds \\
&= \int_0^1 \big( p_{\eta y}(x_0^t + sy,t, \xi_0^t +sD_y) R^{-\frac{1}{2}} y \big)(y,D_y) v ds.
\end{split}
\end{equation*}
We use symbol composition to write
\begin{equation*}
\begin{split}
&\quad \| ( R^{-\frac{1}{2}} y p_{\eta y}(x_0^t + sy,t,\xi_0^t + s D_y)) v \|_2   \\
&\leq \| p_{\eta y}(x_0^t + s y,t, \xi_0^t + s D_y) \circ (R^{-\frac{1}{2}} y) v \|_2 + R^{-\frac{1}{2}} \mathcal{O}( \| p_{\eta y y }(x_0^t + sy,t,\xi_0^t + s D_y) v \|_2) \\
&\lesssim R^{-1} ( \| R^{-\frac{1}{2}} y v \|_2 + \| v(t) \|_2).
\end{split}
\end{equation*}

The second term in \eqref{eq:StandardQuantizationReduction} we write as
\begin{equation*}
\begin{split}
&\quad C_d \int_{\R^d} e^{-i y \eta} \int_0^1 p_{\eta \eta}(x_0^t + s y,t, \xi_0^t + s \eta) (R^{-\frac{1}{2}} \eta \hat{v}(\eta)) ds d\eta \\
&= i C_d \int_0^1 \int_{\R^d} e^{-i y \eta} p_{\eta \eta}(x_0^t +sy,t, \xi_0^t + s \eta) \big[ \widehat{\partial_y v} R^{-\frac{1}{2}} \big] d\eta dy \\
&= R^{-1} i C_d \int_0^1 \int_{\R^d} e^{-i y \eta} p_{\eta \eta}(x_0^t + sy, t, \xi_0^t+s\eta) \widehat{R^{\frac{1}{2}} \partial_y v}(\eta) d\eta ds.
\end{split}
\end{equation*}
Since $\| p_{\eta \eta}(x_0^t + s y ,t, \xi_0^t + s D_y) \|_{L^2 \to L^2} \lesssim 1$, we obtain
\begin{equation*}
\| C_d \int_{\R^d} e^{-i y \eta} \int_0^1 p_{\eta \eta}(x_0^t + s y,t, \xi_0^t + s \eta) (R^{-\frac{1}{2}} \eta \hat{v}(\eta)) ds d\eta \|_2 \lesssim R^{-1} \| R^{\frac{1}{2}} \partial_y v \|_2.
\end{equation*}
This completes the proof of \eqref{eq:EnergyBoundI}.

We turn to the proof of \eqref{eq:EnergyBoundII}. By the fundamental theorem we have
\begin{equation*}
\begin{split}
\partial_y p_2(y,\eta) &= p_y(x_0^t + y,t, \xi_0^t + \eta) - p_y(x_0^t,t,\xi_0^t) \\
&= \int_0^1 p_{yy}(x_0^t + s y,t, \xi_0^t + s \eta) y + p_{y \eta}(x_0^t+sy,t, \xi_0^t + s \eta) \eta ds.
\end{split}
\end{equation*}
Next, we follow along the above argument to show \eqref{eq:EnergyBoundII} with the obvious changes. We note that it is admissible to change to standard quantization. Next, we write
\begin{equation}
\label{eq:StandardQuantizationReductionII}
\begin{split}
R^{\frac{1}{2}} \partial_y p_2(y,D_y) v &= C_d \int_0^1 \int_{\R^d} e^{i y \cdot \eta} p_{yy}(x_0^t+sy,t,\xi_0^t+s\eta) (R^{\frac{1}{2}} y) \hat{v}(\eta) d\eta ds \\
&\quad + C_d \int_0^1 \int_{\R^d} e^{i y \cdot \eta} p_{y \eta}(x_0^t + s y,t,\xi_0^t + s \eta) (R^{\frac{1}{2}} \eta) \hat{v}(\eta) d\eta ds.
\end{split}
\end{equation}
By symbol composition, as above we find
\begin{equation*}
\begin{split}
\| R^{\frac{1}{2}} y p_{yy}(x_0^t + s y, t, \xi_0^t + s D_y) v \|_2 &\lesssim \| p_{yy}(t,x_0^t + sy,\xi_0^t+s D_y) (R^{\frac{1}{2}} y v) \|_2 \\
&\quad + R^{\frac{1}{2}} \| p_{y \eta \eta}(t,x_0^t+sy,\xi_0^t+sD_y) v \|_2 \\
&\lesssim R^{-1} (\| R^{-\frac{1}{2}} y v \|_2 + \| v \|_2 ).
\end{split}
\end{equation*}
This gives an acceptable estimate for the first term in \eqref{eq:StandardQuantizationReductionII}.

\smallskip

For the second term in \eqref{eq:StandardQuantizationReductionII} we obtain
\begin{equation*}
\begin{split}
&\quad C_d \int_{\R^d} e^{i y \eta} \int_0^1 p_{\eta y}(t,x_0^t+sy,\xi_0^t+s\eta) (R^{\frac{1}{2}} \eta \hat{v}(\eta)) ds d\eta \\
&= i C_d \int_0^1 \int_{\R^d} e^{i y \eta} p_{\eta y}(t,x_0^t+sy,\xi_0^t+s\eta) \widehat{R^{\frac{1}{2}} \partial_y v}(\eta) d\eta ds.
\end{split}
\end{equation*}
The $L^2$-norm of $p_{\eta y}(t,x_0^t+sy,\xi_0^t + s D_y)$ is bounded by $R^{-1}$, which gives the estimate of the above by $R^{-1} \| R^{\frac{1}{2}} \partial_y v \|_{L^2}$. The proof of \eqref{eq:EnergyBoundII} is complete.
\smallskip

This proves \eqref{eq:EnergyArgumentsI} for $|\alpha| + |\beta| \leq 1$. We turn to the induction step: Here we expand $(D_t + p_{2}^w) (R^{-\frac{1}{2}} y)^{\alpha} (R^{\frac{1}{2}} \partial_y)^{\beta} v$ as commutators. We commute $(D_t + p_{2}^w)$ to act on $v$ taking advantage of $(D_t + p_{2}^w)v = 0$.

In the first step by the commutator of $p_{2}^w$ and $y$ we find
\begin{equation*}
\begin{split}
&\quad \| (D_t + p_{2}^w) [(R^{-\frac{1}{2}} y) (R^{-\frac{1}{2}} y)^{\alpha-1} (R^{\frac{1}{2}} \partial_y)^{\beta} v] \|_2 \\ &\lesssim \| (R^{-\frac{1}{2}} y) (D_t + p_{2}^w) (R^{-\frac{1}{2}} y)^{\alpha-1} (R^{\frac{1}{2}} \partial_y)^{\beta} v \|_2 \\
&\quad + \| R^{-\frac{1}{2}} (\partial_{\eta} p_{2})^w (R^{-\frac{1}{2}} y)^{\alpha-1} (R^{-\frac{1}{2}} \partial_y)^{\beta} v \|_2.
\end{split}
\end{equation*}
Here we abuse notation and denote multiindices $\alpha'$ with $|\alpha'| = |\alpha|- 1$ by $\alpha -1 $.

The second term is readily estimated by the same expansion as above:
\begin{equation*}
\begin{split}
&\quad \| R^{-\frac{1}{2}} (\partial_{\eta} p_2)^w (R^{-\frac{1}{2}} y)^{\alpha'} (R^{-\frac{1}{2}} \partial)^{\beta} v \|_2 \\
&\lesssim R^{-1} ( \| (R^{-\frac{1}{2}} y) (R^{-\frac{1}{2}} y)^{\alpha-1} (R^{\frac{1}{2}} \partial_y)^{\beta} v \|_2 + \| (R^{\frac{1}{2}} \partial_y ) (R^{-\frac{1}{2}} y)^{\alpha-1} (R^{\frac{1}{2}} \partial_y)^{\beta} v \|_2 \\
&\quad + \| (R^{-\frac{1}{2}} y)^{\alpha-1} (R^{\frac{1}{2}} \partial_y)^{\beta} v \|_2) \text{ and } |\alpha| = |\alpha'| + |\alpha''|.
\end{split}
\end{equation*}

Next, we shall prove that for $\alpha'$, $\alpha''$ with $|\alpha'| \cdot |\alpha''| > 0$ we find
\begin{equation}
\label{eq:AuxEnergyEst}
\begin{split}
&\quad \| (R^{-\frac{1}{2}} y)^{\alpha'} (D_t + p_{2}^w) (R^{-\frac{1}{2}} y)^{\alpha''} (R^{\frac{1}{2}} \partial_y)^{\beta} v \|_2 \\ &\lesssim \| (R^{-\frac{1}{2}} y)^{\alpha'+1} (D_t + p_{2}^w) (R^{-\frac{1}{2}} y)^{\alpha''-1} (R^{\frac{1}{2}} \partial_y)^{\beta} v \|_2 \\
&\quad + R^{-1} \sum_{\substack{\bar{\alpha},\bar{\beta}, \\ |\bar{\alpha}| + |\bar{\beta}| \leq |\alpha| + |\beta|}} \| (R^{-\frac{1}{2}} y)^{\bar{\alpha}} (R^{\frac{1}{2}} \partial_y)^{\bar{\beta}} v \|_2. 
\end{split}
\end{equation}

To this end we induct on $|\alpha''|$ and plug in the commutator:
\begin{equation*}
[ p_2^w, R^{-\frac{1}{2}} y] = R^{-\frac{1}{2}} (\partial_{\eta} p_2)^w,
\end{equation*}
which reduces us to estimate
\begin{equation*}
\| (R^{-\frac{1}{2}} y)^{\alpha'} R^{-\frac{1}{2}} (\partial_{\eta} p_2)^w (R^{-\frac{1}{2}} y)^{\alpha''-1} (R^{\frac{1}{2}} \partial_y)^{\beta} v \|_2.
\end{equation*}
We would now like to use the above argument, but before that we commute \\ $R^{-\frac{1}{2}} (\partial_{\eta} p_2^w)$ to the end: By another commutator estimate
\begin{equation*}
[R^{-\frac{1}{2}} y, R^{-\frac{1}{2}} \partial_{\eta} p_2^w] = R^{-1} (\partial_{\eta \eta} p_2)^w,
\end{equation*}
which is $R \cdot L^2$-bounded. Further commutators with $R^{-\frac{1}{2}} y$ are better behaved. We can estimate
\begin{equation*}
\begin{split}
&\quad \| (R^{-\frac{1}{2}} y)^{\alpha'} (R^{-\frac{1}{2}} \partial_{\eta} p_2)^w (R^{-\frac{1}{2}} y)^{\alpha''_{-1}} (R^{\frac{1}{2}} \partial_y)^{\beta} v \|_2 \\ &\lesssim \| R^{-\frac{1}{2}} (\partial_{\eta} p_2)^w (R^{-\frac{1}{2}} y)^{\alpha_{-1}} (R^{\frac{1}{2}} \partial_y)^{\beta} v \|_2 + R^{-1} \sum_{\substack{\bar{\alpha},\bar{\beta}, \\ |\bar{\alpha}| + |\bar{\beta}| \leq |\alpha| + |\beta|}} \| (R^{-\frac{1}{2}} y)^{\bar{\alpha}} (R^{\frac{1}{2}} \partial_y)^{\bar{\beta}} v \|_2.
\end{split}
\end{equation*}
The first term can be estimated by the second line by means of the expansion \eqref{eq:SymbolExpansionSecondOrderEta}.

By the above, we can reduce the left hand side \eqref{eq:AuxEnergyEst} to
\begin{equation*}
\| (R^{-\frac{1}{2}} y)^{\alpha} (D_t + p_{2}^w) (R^{\frac{1}{2}} \partial_y)^{\beta} v \|_2 + \text{l.o.t.}
\end{equation*}
The strategy is to commute $(D_t + p_{2}^w)$ to act on $v$, which is annihilated. We need to compute the commutator with $R^{\frac{1}{2}} \partial_y$:
\begin{equation*}
[p_2^w, R^{\frac{1}{2}} \partial_y] = R^{\frac{1}{2}} (\partial_y p_2)^w.
\end{equation*}
We want to use the estimate from above, but this requires to commute $R^{\frac{1}{2}} (\partial_y p_2)^w$ after $(R^{-\frac{1}{2}} y)^{\alpha}$.
We have the commutators
\begin{equation*}
[ (R^{-\frac{1}{2}} y), R^{\frac{1}{2}} (\partial_y p_2)^w] =  (\partial_{y \eta} p_2)^w,
\end{equation*}
which is $R \cdot L^2$-bounded, and moreover
\begin{equation*}
[ R^{\frac{1}{2}} \partial_y, R^{\frac{1}{2}} (\partial_y p_2)^w] = R (\partial_{yy} p_2)^w,
\end{equation*}
which is $R \cdot L^2$-bounded as well, and so are further commutators with $R^{-\frac{1}{2}} y$ and $R^{\frac{1}{2}} \partial_y$. We obtain for $|\beta|= |\beta_1| + |\beta_2|$
\begin{equation*}
\begin{split}
&\quad \| ( R^{-\frac{1}{2}} y)^{\alpha} (R^{\frac{1}{2}} \partial_y)^{\beta_1} (D_t + p_{2}^w) (R^{\frac{1}{2}} \partial_y)^{\beta_2} v \|_2 \\ 
&\lesssim R^{-1} \sum_{\substack{\bar{\alpha},\bar{\beta}, \\ |\bar{\alpha}| + |\bar{\beta}| \leq |\alpha| + |\beta|}} \| (R^{-\frac{1}{2}} y)^{\bar{\alpha}} (R^{\frac{1}{2}} \partial_y)^{\bar{\beta}} v(t) \|_2.
\end{split}
\end{equation*}
This suffices for \eqref{eq:EnergyArgumentsI} by Gr\o nwall's inequality.

Finally, we turn to the time regularity. We will successively use the equation to find the estimate
\begin{equation}
\label{eq:TimeDerivativeTrade}
\| (R^{-\frac{1}{2}} y)^{\alpha} (R^{\frac{1}{2}} \partial_y)^{\beta} (R^{\frac{1}{2}} \partial_t)^{\gamma} v(t) \|_2 \lesssim \sum_{|\alpha'| + |\beta'| \leq |\alpha| + |\beta| + 2 |\gamma|} \| ( R^{-\frac{1}{2}} y)^{\alpha'} (R^{\frac{1}{2}} \partial_y)^{\beta'} v(t) \|_2.
\end{equation}
The right hand side can be handled by the energy arguments from above.
We begin with the case $\gamma =1$, $\alpha = \beta = 0$: Since $v$ solves $(D_t + p_2^w) v = 0$, we find $\| R^{\frac{1}{2}} \partial_t v \|_2 = R^{\frac{1}{2}} \| p_2^w v \|_2$. As above we can change to standard quantization and by Taylor expansion we have
\begin{equation*}
p_2(t,y,\eta) = \mathcal{O}(y^2 \partial_y^2 p + y \eta \partial^2_{y \eta} p + \eta^2 \partial_{\eta}^2 p).
\end{equation*}
Now the estimate follows from symbol composition
\begin{equation*}
\| p_2(y,D_y) v \|_2 \lesssim R^{-\frac{1}{2}} \sum_{|\alpha| + |\beta| \leq 2} \| (R^{-\frac{1}{2}} y)^{\alpha} (R^{\frac{1}{2}} \partial_y)^{\beta} v(t) \|_2.
\end{equation*}

To evaluate higher orders, we plug in the equation to find
\begin{equation}
\label{eq:HigherOrderSimplification}
\| (R^{-\frac{1}{2}} y)^{\alpha} (R^{\frac{1}{2}} \partial_y)^{\beta} (R^{\frac{1}{2}} \partial_t)^{\gamma} v \|_2 = R^{\frac{1}{2}} \| ((R^{-\frac{1}{2}} y)^{\alpha} (R^{\frac{1}{2}} \partial_y)^{\beta} (R^{\frac{1}{2}} \partial_t)^{\gamma-1} p_2^w v \|_2.
\end{equation}
Now we have to commute $p_2^w$ to the end and need to consider commutators \\ $[R^{\frac{1}{2}} \partial_t, p_2^w] = R^{\frac{1}{2}} (\partial_t p_2)^w$. $\partial_t p_2$ is composed of three terms, which share the same structure:
\begin{equation}
\label{eq:TimeDerivativea2Expansion}
\begin{split}
&\quad \partial_t p_2(t,y,\eta) \\
&= \big( (\partial_t p)(t,x_0^t+y,\xi_0^t + \eta) - (\partial_t p)(t,x_0^t,\xi_0^t) - y \partial^2_{tx} p)t,x_0^t, \xi_0^t) - \eta \partial^2_{t \eta} p)t,x_0^t,\xi_0^t) \big) \\
&\quad + \big( \partial_x p)t,x_0^t+y, \xi_0^t + \eta) - \partial_x p)t,x_0^t, \xi_0^t) \\
&\quad \quad - y \partial^2_{xx} p)t,x_0^t,\xi_0^t) - \eta \partial^2_{x \eta} p)t,x_0^t,\xi_0^t) \big) p_{\eta}(t,x_0^t,\xi_0^t) \\
&\quad + \big( \partial_{\eta} p)t,x_0^t+y,\xi_0^t+\eta) - \partial_{\eta} p)t,x_0^t,\xi_0^t) \\
&\quad \quad - y \partial^2_{x \eta} p)t,x_0^t,\xi_0^t) - \eta \partial^2_{\eta \eta} p(t,x_0^t,\xi_0^t) \big) (- p_x(t,x_0^t,\xi_0^t)).
\end{split}
\end{equation}
Note that for all three terms the symbols vanish of second order in $(y,\eta)$; here it suffices to use the first order Taylor expansion.
By the fundamental theorem we have the identity for the symbols:
\begin{equation*}
\begin{split}
(A) &= \partial_t p(t,x_0^t+y, \xi_0^t + \eta) - \partial_t p(t,x_0^t, \xi_0^t) \\
&= \int_0^1 \partial^2_{tx} p(t,x_0^t+sy,\xi_0^t + s \eta) \cdot y + \partial^2_{t \eta} p(t,x_0^t+sy,\xi_0^t+s \eta) \cdot \eta ds.
\end{split}
\end{equation*}

By changing as above to standard quantization and symbol composition, this gives an estimate:
\begin{equation*}
\| (A)^w R^{\frac{1}{2}} v \|_2 \lesssim R^{-1} ( \| R^{-\frac{1}{2}} y  v \|_{L^2} + \| (R^{\frac{1}{2}} \partial_y ) v \|_{L^2} + \| v \|_{L^2}).
\end{equation*}
The third and fourth term in the first line of the expansion in \eqref{eq:TimeDerivativea2Expansion} can be estimated directly by the regularity of $a$:
\begin{equation*}
 \| \big[ R^{\frac{1}{2}} y \partial^2_{tx} p(t,x_0^t,\xi_0^t) - (R^{\frac{1}{2}} \eta) (\partial^2_{t \xi} p)(t,x_0^t,\xi_0^t) \big]^w v \|_{L^2} \lesssim R^{-1}( \| (R^{-\frac{1}{2}} y) v \|_2 + \| (R^{\frac{1}{2}} \partial_y) v \|_2).
\end{equation*}

The second line in the expansion is estimated like $(A)$. Similarly, for the third line in the expansion we can argue like for the third and fourth term in the first line. For the penultimate line we write by the mean value theorem:
\begin{equation*}
\begin{split}
&\quad \partial_{\eta} p(t,x_0^t+y,\xi_0^t + \eta) - \partial_{\eta} p(t,x_0^t,\xi_0^t) \\
&= \int_0^1 \partial^2_{x \eta} p(t,x_0^t+sy,\xi_0^t + s \eta) \cdot y + \partial^2_{\eta \eta} p(t,x_0^t+sy,\xi_0^t + s \eta) \cdot \eta ds.
\end{split}
\end{equation*}
Then we change to standard quantization and argue as above by symbol composition
\begin{equation*}
\begin{split}
&\quad \| R^{\frac{1}{2}} (\partial_{\eta} p(t,x_0^t + y, \xi_0^t + \eta) - \partial_{\eta} p(t,x_0^t,\xi_0^t) (a_x(t,x_0^t,\xi_0^t)))^w v \|_2 \\
&\lesssim R^{-1} ( \| R^{-\frac{1}{2}} y v \|_2 + \| R^{\frac{1}{2}} \partial_y v \|_2 + \| v \|_2).
\end{split}
\end{equation*} 
The last line in the expansion \eqref{eq:TimeDerivativea2Expansion} can be estimated directly again. 

Coming back to \eqref{eq:HigherOrderSimplification}, the commutator terms are apparently better behaved. We find that after expanding the right hand side
\begin{equation*}
\begin{split}
&\quad R^{\frac{1}{2}} \| ((R^{-\frac{1}{2}} y)^{\alpha} (R^{\frac{1}{2}} \partial_y)^{\beta} (R^{\frac{1}{2}} \partial_t)^{\gamma-1} p_2^w v \|_2 \\
&= R^{\frac{1}{2}} \| p_2^w ((R^{-\frac{1}{2}} y)^{\alpha} (R^{\frac{1}{2}} \partial_y)^{\beta} (R^{\frac{1}{2}} \partial_t)^{\gamma-1}  v \|_2 + \text{l.o.t.}
\end{split}
\end{equation*}
We use the expansion for $p_2^w$ like in case $\gamma = 1$, $\alpha = \beta = 0$ to find
\begin{equation*}
\begin{split}
&\quad R^{\frac{1}{2}} \| p_2^w ((R^{-\frac{1}{2}} y)^{\alpha} (R^{\frac{1}{2}} \partial_y)^{\beta} (R^{\frac{1}{2}} \partial_t)^{\gamma-1}  v \|_2 \\
 &\lesssim \sum_{|\alpha'| + |\beta'| \leq |\alpha| + |\beta| + 2} \| (R^{-\frac{1}{2}} \partial_y)^{\alpha'} (R^{\frac{1}{2}} \partial_y)^{\beta'} (R^{\frac{1}{2}} \partial_t)^{\gamma-1} v \|_2
 \end{split}
\end{equation*}
and can now conclude the reduction \eqref{eq:TimeDerivativeTrade} by induction on $\gamma$. With \eqref{eq:TimeDerivativeTrade} at hand, the claim \eqref{eq:ModifiedGreenBounds} follows from \eqref{eq:EnergyArgumentsI}.
\end{proof}

\subsection{Localization properties of wave packets}
\label{section:Localization}
By the regularity properties of the kernel we can now state the localization properties of the wave packets.
Recall that we let
\begin{equation*}
u_{x_0,\xi_0}(y,t) = \int_{\R^d} K_{x_0,\xi_0}(y,t,\tilde{y},0) u_0(\tilde{y}) d\tilde{y}
\end{equation*}
with
\begin{equation*}
\begin{split}
K_{x_0,\xi_0}(y,t,\tilde{y},0) &= \int_{\R^{2d}} e^{- \frac{\rho}{2}(\tilde{y}-x)^2} e^{-i \xi (\tilde{y}-x)} e^{i \xi^t (y-x^t)} e^{i (\psi(x,t,\xi) - \psi(x,s,\xi))} \psi_{x_0,\xi_0}(x,\xi) \\
&\quad \quad \times G(t,s,x,\xi,y) dx d\xi.
\end{split}
\end{equation*}

\begin{lemma}
The following estimate holds for any $N \geq 1$:
\begin{equation*}
|u_{x_0,\xi_0}(t,y)| \lesssim_N R^{-\frac{d}{4}} (1+R^{-\frac{1}{2}}|x_0^t - y|)^{-N} \big( \int |\psi_{(x_0,\xi_0)}(x,\xi) T_R u_0(x,\xi)|^2 dx d\xi \big)^{\frac{1}{2}}
\end{equation*}
with implicit constant depending on the seminorm quantities $C_{\gamma,\delta}$ from \eqref{eq:RegularityEstimatesHamiltonianFlowAssumption} for $|\gamma|+|\delta| \leq M(N)$.
\end{lemma}
\begin{proof}
By the regularity properties of $G$ we find the estimate
\begin{equation*}
(1+ R^{-\frac{1}{2}} |y-x^t|)^N |G(t,s,x,\xi,y)| \lesssim_N C_R^2.
 \end{equation*}
This follows from Sobolev embedding, the frequency localization of $G$, and its regularity properties shown in Proposition \ref{prop:ParametrixConstruction}. Indeed, the rescaling $G_2'(t',x',\xi',y') = G_2(R^{\frac{1}{2}} t', R^{\frac{1}{2}} x', R^{-\frac{1}{2}} \xi', R^{\frac{1}{2}} y')$ yields centered Schwartz functions $G_2'$ by \eqref{eq:ModifiedGreenBounds} with Schwartz norms dependence on the symbol bounds as claimed.

This implies
\begin{equation*}
    |S(t,0) \phi_{x,\xi}(y)| \lesssim (1+ R^{-\frac{1}{2}} |x^t - y|)^{-N}.
\end{equation*}

By the localization property of $\psi_{x_0,\xi_0}$ and the bi-Lipschitz property of the Hamiltonian flow Lemma \ref{lem:GeometryWavepackets} we can estimate
\begin{equation*}
    \begin{split}
  |u_{x_0,\xi_0}(y) | &\leq  C_R \int_{\R^{2d}} | S(t,0) \phi_{x,\xi}(y) \psi_{x_0,\xi_0}(x,\xi) T_R u_0(x,\xi) | dx d\xi \\
  &\lesssim C_R \int (1+ R^{-\frac{1}{2}}|x^t-y|)^{-N} \psi_{x_0,\xi_0}(x,\xi) |T_R u_0(x,\xi)| dx d\xi \\
  &\lesssim R^{-\frac{d}{4}} (1+ R^{-\frac{1}{2}} |x_0^t - y|)^{-N} \| \psi_{x_0,\xi_0} T_R u_0 \|_{L^2}.
  \end{split}
\end{equation*}
\end{proof}

In the following let $\chi : \R^d \to \R_{\geq 0}$ denote a radially decreasing bump function with $\chi \equiv 1$ on $B(0,1)$ and $\text{supp}(\chi) \subseteq B(0,2)$. Denote $1 - \chi = \chi_{\geq 1}$. We have the following spatial localization.
\begin{corollary}[Spatial~localization~of~wave~packets]
\label{cor:SpatialLocalization}
Let $\delta > 0$. The following holds for any $m \geq 1$:
\begin{equation*}
|u_{x_0,\xi_0}(y,t) - \chi(R^{-\frac{1}{2}-\delta}(x_0^t - y)) u_{x_0,\xi_0} | \leq C_{m,\delta} R^{-m} \| u_0 \|_2
\end{equation*}
with constant $C_{m,\delta}$ depending on the seminorm quantities $C_{\alpha,\beta}$ from \eqref{eq:RegularityEstimatesHamiltonianFlowAssumption} for $|\alpha| + |\beta| \leq M(m,\delta)$.
\end{corollary}

We show the following regarding frequency localization:
\begin{lemma}
Let $\hat{u}_{x_0,\xi_0}(\eta,t) = \mathcal{F}_y \big[ u_{x_0,\xi_0}(y,t) \big](\eta )$. The following localization of the spatial Fourier transform holds
\begin{equation}
\label{eq:FrequencyLocalizationWavePacketMaximal}
|\hat{u}_{x_0,\xi_0}(\eta,t)| \lesssim_N R^{\frac{d}{4}} (1 + R^{\frac{1}{2}} | \xi_0^t - \eta |)^{-N} \big( \int |\psi_{x_0,\xi_0}|^2 |T_R u_0(x,\xi)|^2 dx d\xi \big)^{\frac{1}{2}}
\end{equation}
with implicit constant depending on seminorm quantities $C_{\gamma,\delta}$ from \eqref{eq:RegularityEstimatesHamiltonianFlowAssumption} for $|\gamma|+|\delta| \leq M(N)$.
\end{lemma}
\begin{proof}
We write the Fourier transform by Fubini's theorem as
\begin{equation*}
\begin{split}
&\mathcal{F}_y \big[ u_{x_0,\xi_0}(y,t) \big](\eta ) \\
&= C_R \iint \int S(t,0) \phi_{x,\xi}(\tilde{y}) e^{-i \eta \tilde{y}} d\tilde{y} \psi_{x_0,\xi_0}(x,\xi) T_R u_0(x,\xi)  dx d\xi.
\end{split}
\end{equation*}
Below we obtain from the regularity properties of $G$ the estimate
\begin{equation}
\label{eq:FTLocalizationGreen}
    |\widehat{S(t,0) \phi_{x,\xi}}(\eta)| \lesssim_N R^{\frac{d}{2}} (1+ R^{\frac{1}{2}} |\xi^t - \eta|)^{-N}.
\end{equation}
With \eqref{eq:FTLocalizationGreen} at disposal, we can plug this into the preceding display and conclude \eqref{eq:FrequencyLocalizationWavePacketMaximal} like in the previous lemma.

We turn to \eqref{eq:FTLocalizationGreen}, to which end we handle first the case $|\xi^t - \eta | \lesssim R^{-\frac{1}{2}}$.
The estimate is immediate from the pointwise estimate of $G$ and localization to a $R^{\frac{1}{2}}$-ball such that
\begin{equation}
\label{eq:FTCoherentStateI}
\big| \int_{\R^d} dy e^{i ( \xi^t - \eta) y} G(t,s,x,\xi,y) \big| \lesssim R^{\frac{d}{2}} C_R^2 \lesssim 1.
\end{equation}

Next, we turn to the case $|\xi^t - \eta| \gg R^{-\frac{1}{2}}$. In this case we can integrate by parts in $y$. Let $j$ denote the index such that $| \xi^t - \eta| \sim |(\xi^t - \eta)_j|$. We find
\begin{equation}
\label{eq:FTCoherentStateII}
\begin{split}
&\quad \big| \int_{\R^d} e^{i ( \xi^t - \eta) \cdot y} G(t,s,x,\xi,y) dy \big| \\
&= \big| \int_{\R^d} \big[ i( \xi^t - \eta)_j \big]^{-N} \partial_j^N e^{i(\xi^t - \eta) \cdot y} e^{- i \xi^t \cdot x^t} G(t,s,x,\xi,y) dy \big| \\
&= \int_{\R^d} \big| \frac{R^{-\frac{N}{2}}}{\big[ i (\xi^t - \eta)_j \big]^N} e^{i(\xi^t - \eta) \cdot y} e^{-i \xi^t \cdot x^t} (R^{\frac{1}{2}} \partial_{y_j})^N G(t,s,x,\xi,y) \big| dy \\
&\lesssim \frac{R^{-\frac{N}{2}}}{|\xi^t - \eta|^N}.
\end{split}
\end{equation}
In the ultimate step we used that rescaled $G$ is a normalized Schwartz function as proved in Proposition \ref{prop:ParametrixConstruction}. From \eqref{eq:FTCoherentStateI} and \eqref{eq:FTCoherentStateII} the bounds \eqref{eq:FTLocalizationGreen} are immediate.
\end{proof}

We record the following corollary:
\begin{corollary}
\label{cor:FrequencyLocalization}
Let $\delta > 0$ and $\hat{u}_{x_0,\xi_0}(\eta,t) = \mathcal{F}_y \big[ u_{x_0,\xi_0}(y,t) \big](\eta )$. The following estimate holds for any $m \geq 1$
\begin{equation}
\label{eq:FrequencyLocalizationWavePacket}
|\hat{u}_{x_0,\xi_0}(\eta,t) - \chi(R^{\frac{1}{2} - \delta} (\xi_0^t - \eta)) \hat{u}_{x_0,\xi_0}(\eta)| \leq C_{m,\delta} R^{-m} \| u_0 \|_2
\end{equation}
with constant $C_{m,\delta}$ depending on the seminorm quantities $C_{\alpha,\beta}$ from \eqref{eq:RegularityEstimatesHamiltonianFlowAssumption} for $|\alpha|+|\beta| \leq M(m,\delta)$. 
\end{corollary}

Finally, we show the following localization of the time frequencies:
\begin{lemma}
The following estimate holds for the windowed Fourier transform
\begin{equation}
\label{eq:TimeFrequencyLocalization}
|\mathcal{F}_t[ u_{x_0,\xi_0} \chi(R^{-\frac{1}{2}}(t-t_0))](\tau ) | \lesssim_N R^{\frac{1}{2}} (1 + R^{\frac{1}{2}} |p(x^{t_0}_0,t_0, \xi^{t_0}_0) + \tau |)^{-N} R^{-\frac{d}{4}} \| \psi_{x_0,\xi_0} T_R u_0 \|_{L^2(T^* \R^d)}
\end{equation}
with implicit constant depending on seminorm quantities $C_{\gamma,\delta}$ from \eqref{eq:RegularityEstimatesHamiltonianFlowAssumption} for $|\gamma|+|\delta| \leq M(N)$.
\end{lemma}
\begin{proof}
We obtain for the windowed Fourier transform by Fubini's theorem:
\begin{equation*}
\begin{split}
&\quad \mathcal{F}_t [ \chi(R^{-\frac{1}{2}}(t-t_0)) u_{x_0,\xi_0}(y,t) ] \\
&= \int_{\R^d} d \tilde{y} u_0(\tilde{y}) \int_{\R^{2d}} dx d\xi e^{-\frac{\rho}{2}(\tilde{y}-x)^2} e^{-i \xi (\tilde{y}-x)} \psi_{x_0,\xi_0}(x,\xi) \\
&\quad \quad \times \int_{\R} dt \chi(R^{-\frac{1}{2}}(t-t_0)) e^{-it \tau} e^{i \xi^t (y-x^t)} e^{i \psi(x,t,\xi)} G(t,s,x,\xi,y).
\end{split}
\end{equation*}

The estimate will again be a consequence of non-stationary phase; in this case we integrate by parts in time. To this end, we linearize the Hamiltonian flow in time. As a consequence of Lemma \ref{lem:HamiltonFlowMixedRegularityII} we have the expansion
\begin{equation*}
\left\{ \begin{array}{cl}
x^t &= x^{t_0} + (t-t_0) \partial_{\xi} p(x^{t_0},t_0,\xi^{t_0}) + G_{x,\xi}(t), \\
\xi^t &= \xi^{t_0} + F_{x,\xi}(t),
\end{array} \right.
\end{equation*}
which satisfies for $|t-t_0| \lesssim R^{\frac{1}{2}}$ the bounds
\begin{equation}
\label{eq:TimeRegularityHamiltonian}
|\partial_t^{\sigma} F_{x,\xi}(t) | \lesssim_{\sigma} R^{\frac{-\sigma-1}{2}}, \quad |\partial_t^{\sigma} G_{x,\xi}(t)| \lesssim_{\sigma} R^{\frac{-\sigma}{2}}.
\end{equation}
By the Taylor expansion we can rewrite the $t$-dependent phase function as
\begin{equation*}
e^{i \xi^t(y-x^t)} = e^{i (\xi^{t_0} + F_{x,\xi}(t))(y-x^t)} = e^{i \xi^{t_0} (y-x^t)} e^{i F_{x,\xi}(t) (y-x^t)}.
\end{equation*}
We record the following estimate, which is a consequence of the time regularity of the Hamilton flow (Lemma \ref{lem:HamiltonFlowMixedRegularityII}) and the Leibniz rule:
\begin{equation}
\label{eq:TimeRegularityPhaseFunctionI}
| \partial_t^{\sigma} ( F_{x,\xi}(t) (y-x^t))| \lesssim_{\sigma} R^{-\frac{\sigma}{2}} (1+R^{-\frac{1}{2}} |y-x^t|)
\end{equation}
provided that $\sigma \geq 0$ and $|t-t_0| \lesssim R^{\frac{1}{2}}$. Indeed, this follows from distributing the time derivatives and \eqref{eq:TimeRegularityHamiltonian}. We continue by writing
\begin{equation*}
\begin{split}
&\quad e^{i \xi^{t_0} (y-x^{t_0}) - i \xi^{t_0} ((t-t_0) \partial_{\xi} p(x^{t_0},\xi^{t_0}) - G_{x,\xi}(t))} \\
&= e^{i \xi^{t_0} (y- x^{t_0})} e^{-i (t-t_0) \xi^{t_0} \partial_{\xi} p(x^{t_0}, \xi^{t_0})} e^{i \xi^{t_0} G_{x,\xi}(t)}.
\end{split}
\end{equation*}
For the last term it is easy to see by \eqref{eq:TimeRegularityHamiltonian} that
\begin{equation}
\label{eq:TimeRegularityPhaseFunctionII}
\big| \partial_t^{\sigma} (e^{i \xi^{t_0} G_{x,\xi}(t)}) \big| \lesssim R^{-\frac{\sigma}{2}} \text{ for } |t-t_0| \lesssim R^{\frac{1}{2}}.
\end{equation}
Secondly, we expand
\begin{equation*}
    \psi(x,t,\xi) = \psi(x,t_0,\xi) + (t-t_0)[-p(x^{t_0},t_0,\xi^{t_0}) + \xi^{t_0} \partial_{\xi} p(x^{t_0},t_0,\xi^{t_0})] + H_{x,\xi}(t)
\end{equation*}
with
\begin{equation*}
    |\partial_t^{\sigma} H_{x,\xi}(t)| \lesssim R^{-\frac{\sigma}{2}} \text{ for } |t-t_0| \lesssim R^{\frac{1}{2}}.
\end{equation*}

We rewrite the time-integral in the second line by the expansion of the Hamiltonian flow as
\begin{equation*}
\begin{split}
&\quad \int_{\R} dt \chi(R^{-\frac{1}{2}}(t-t_0)) e^{-it \tau} e^{i \xi^t (y-x^t)} G(t,s,x,\xi,y) \\
&= e^{i \xi^{t_0} (y-x^{t_0})} e^{i t_0 p(x^{t_0},\xi^{t_0})} \int_{\R} dt \chi(R^{-\frac{1}{2}}(t-t_0)) e^{-i t (p(x^{t_0}, \xi^{t_0}) + \tau)} \\
&\quad \quad \times e^{i \xi^{t_0} G_{x,\xi}(t)} e^{i F_{x,\xi}(t)(y-x^t)} e^{i H_{x,\xi}(t)} G(t,s,x,\xi,y).
\end{split}
\end{equation*}
By Taylor expansion and Lemma \ref{lem:GeometryWavepackets} we see that for $(x,\xi) \in \text{supp}(\psi_{x_0,\xi_0})$ and $|\tau + p(x_0^{t_0},t_0,\xi_0^{t_0})| \lesssim R^{-\frac{1}{2}}$ that
\begin{equation*}
| \tau + p(x^{t_0},t_0,\xi^{t_0})| \lesssim R^{-\frac{1}{2}}.
\end{equation*}
In this case we can use the size of $|G| \lesssim R^{-\frac{d}{2}}$ to estimate
\begin{equation*}
\begin{split}
|\mathcal{F}_t[\chi(R^{-\frac{1}{2}}(t-t_0)) u_{x_0,\xi_0}(t,y)] | &\lesssim R^{\frac{1}{2}} R^{-\frac{d}{2}} \int_{\R^d} d\tilde{y} u_0(\tilde{y}) \int_{\R^{2d}} e^{-\frac{\rho}{2}(\tilde{y}-x)^2} \psi_{x_0,\xi_0}(x,\xi) \\
&\lesssim R^{\frac{1}{2}} R^{-\frac{d}{4}} \| \psi_{x_0,\xi_0} T_R u_0 \|_{L^2(T^* \R^d)}.
\end{split}
\end{equation*}

On the other hand, for $|\tau + p(x_0^{t_0},t_0,\xi_0^{t_0})| \geq 2 C R^{-\frac{1}{2}}$ for a constant $C \geq 1$ depending on the $C^{1,1}_{z,\xi}$-norm of $a$, we have by Lemma \ref{lem:GeometryWavepackets} that
\begin{equation*}
| \tau + p(x^{t_0},t_0,\xi^{t_0})| \geq C R^{-\frac{1}{2}} \text{ for } (x,\xi) \in \text{supp}(\psi_{x_0,\xi_0}).
\end{equation*}
We suppress the explicit dependence on $a$ in the following to ease notation.
The above allows us to rewrite the $t$-integral by integration by parts as
\begin{equation*}
\begin{split}
&\quad \int_{\R} dt \chi(R^{-\frac{1}{2}}(t-t_0)) \partial_t^{\alpha} \big( \frac{ e^{-i t (p(x^{t_0},\xi^{t_0}) + \tau)}}{[-i (p(x^{t_0},\xi^{t_0}) + \tau)]^{\alpha}} \big) e^{i \xi^{t_0} G_{x,\xi}(t)} e^{i F_{x,\xi}(t)(y-x^t)} G(t,x,\xi,y) \\
&= (-1)^{\alpha} \int_{\R} dt \frac{ e^{-i t (p(x^{t_0},\xi^{t_0}) + \tau)}}{[-i (p(x^{t_0},\xi^{t_0}) + \tau)]^{\alpha}} \\
&\quad \partial_t^{\alpha} \big( \chi(R^{-\frac{1}{2}}(t-t_0)) e^{i \xi^{t_0} G_{x,\xi}(t)} e^{i F_{x,\xi}(t)(y-x^t)} G(t,s,x,\xi,y) \big).
\end{split}
\end{equation*}
By the time regularity of $G$ and \eqref{eq:TimeRegularityPhaseFunctionI} and \eqref{eq:TimeRegularityPhaseFunctionII}, we obtain the pointwise estimate
\begin{equation*}
| \partial_t^{\alpha} \big( \chi(R^{-\frac{1}{2}}(t-t_0)) e^{i \xi^{t_0} G_{x,\xi}(t)} e^{i F_{x,\xi}(t)(y-x^t)} G(t,x,\xi,y) \big) | \lesssim R^{-\frac{|\alpha|}{2}} R^{-\frac{d}{2}} \tilde{\chi}(R^{-\frac{1}{2}}(t-t_0))
\end{equation*}
for $\tilde{\chi}$ a mild enlargement of $\chi$. We can conclude
\begin{equation*}
\begin{split}
&\quad |\mathcal{F}_t[\chi(R^{-\frac{1}{2}}(t-t_0)) u_{x_0,\xi_0}(y,t)](\tau) | \\
 &\lesssim_\alpha R^{\frac{1}{2}} R^{-\frac{d}{2}} \int_{\R^d} d\tilde{y} |u_0(\tilde{y})| \int_{\R^{2d}} e^{-\frac{\rho}{2}(\tilde{y}-x)^2} \psi_{x_0,\xi_0}(x,\xi) (1+ R^{\frac{1}{2}} |p(x^{t_0}_0,\xi^{t_0}_0)|)^{-|\alpha|} \\
&\lesssim_\alpha R^{\frac{1}{2}} (1+ R^{\frac{1}{2}} |p(x^{t_0}_0,\xi^{t_0}_0)|)^{-|\alpha|} R^{-\frac{d}{4}} \| \psi_{x_0,\xi_0} T_R u_0 \|_{L^2(T^* \R^d)}.
\end{split}
\end{equation*}
\end{proof}

\begin{corollary}
\label{cor:TimeFrequencyLocalization}
Let $\delta > 0$. For any $m \geq 1$, the following estimate of the localized Fourier transform in time $ \tilde{u}_{x_0,\xi_0}(\tau) =  \mathcal{F}_t[ u_{x_0,\xi_0} \chi(R^{-\frac{1}{2}}(t-t_0))](\tau )$ holds:
\begin{equation}
\label{eq:FrequencyLocalizationWavePacketTime}
| \tilde{u}_{x_0,\xi_0}(\tau) - \chi(R^{\frac{1}{2} - \delta} (\tau + p(x_0^{t_0},\xi_0^{t_0})) \tilde{u}_{x_0,\xi_0}(\tau)| \leq  C_{m,\delta} R^{-m} \| u_0 \|_2
\end{equation}
with constant $C_{m,\delta}$ depending on the seminorm quantities $C_{\alpha,\beta}$ from \eqref{eq:RegularityEstimatesHamiltonianFlowAssumption} for $|\alpha|+|\beta| \leq M(m,\delta)$. 
\end{corollary} 

Let $\alpha_{(x_0,\xi_0)} = \| \psi_{x_0,\xi_0} T_r u_0 \|_{L^2(T^* \R^d)}$. We obtain normalized wave packets letting $\phi_{(x_0,\xi_0)} = u_{(x_0,\xi_0)} / \alpha_{(x_0,\xi_0)}$ and in conclusion,
\begin{equation*}
    u = \sum_{\substack{(x_0,\xi_0) \in \Lambda_r, \\ \alpha_{(x_0,\xi_0)} \neq 0}} \alpha_{(x_0,\xi_0)} \phi_{(x_0,\xi_0)}
\end{equation*}
satisfies the wave packet axioms from Assumption \ref{ass:WP}. The localization properties were checked above, the property 
\begin{equation*}
    \sum_{(x_0,\xi_0) \in \Lambda_r} |\alpha_{(x_0,\xi_0)}|^2 \lesssim \| u_0 \|_{L^2}^2
\end{equation*}
is a consequence of the isometry of $T_r$.

As a further consequence we obtain for $r_1 \leq r$ the pointwise estimates in phase space:
\begin{equation*}
|T_{r_1} \phi_T (x,\xi)| \leq C_N (1 + r^{-\frac{1}{2}} |x_T - x| + r_1^{\frac{1}{2}} |\xi_T - \xi |)^{-N}.
\end{equation*}
From the kernel estimates obtained above it becomes clear that the implicit constant depends on seminorm quantities $C_{\alpha,\beta}$ for $|\alpha|+|\beta| \leq M(N)$ from \eqref{eq:RegularityEstimatesHamiltonianFlowAssumption}. 

\begin{proof}[Conclusion~of~the~Proof~of~Theorem~\ref{thm:WavePacketDecompositions}]
    We need to define the pseudo-differential operators $\chi_{\overline{\mathcal{Y}}_r}(x,D)$, for which the wave packet expansion holds. Recall the thickenings of $\mathcal{Y}$ given in Definition \ref{def:Thickening}. We define the symbol $\chi_{\overline{\mathcal{Y}}_r}$ as identically one on $\mathcal{Y}_{r}$, which is the $(R r^{-\frac12 + \delta_0},r^{-\frac12 + \delta_0})$-thickening of $\mathcal{Y}_{\nu^{-2}}$, being supported on the $(C R r^{-\frac12 + \delta_0},Cr^{-\frac12 + \delta_0})/2$-thickening and vanishing off $\overline{\mathcal{Y}}_{r}$, that is the \\ $(C R r^{-\frac12 + \delta_0},C r^{-\frac12 + \delta_0})$-thickening of $\mathcal{Y}_{\nu^{-2}}$. More precisely, we apply Urysohn's lemma and mollification to construct symbols $\chi_r = \chi_{\overline{\mathcal{Y}}_r}$ with the claimed support properties, which satisfy the regularity estimates:
    \begin{equation*}
        |\partial_x^{\alpha} \partial_{\xi}^{\beta} \chi_r(x,\xi)| \lesssim_{\alpha,\beta} (R r^{-\frac{1}{2}+\delta_0})^{-|\alpha|} r^{|\beta| (\frac{1}{2}-\delta_0)}.
    \end{equation*}
    
    The $L^2$-boundedness is a consequence of scaling. We obtain by change of variables $\xi = r^{-\frac{1}{2}} \xi'$ and $x = r^{\frac{1}{2}} x'$
    \begin{equation*}
        \int \big| \int \chi_r(x,\xi) \hat{f}(\xi) e^{i x \cdot \xi} d \xi \big|^2 dx = r^{-\frac{d}{2}} \int \big| \int \chi( r^{\frac{1}{2}} x', r^{-\frac{1}{2}} \xi') \hat{f}(r^{-\frac{1}{2}} \xi') e^{i x' \cdot \xi'} d\xi' \big|^2 dx'.
    \end{equation*}
    The resulting symbol is in $S_{00}^0$. Indeed, we have
    \begin{equation*}
        | \partial_x^{\alpha} \partial_{\xi}^{\beta} \chi( r^{\frac{1}{2}-\delta_0} x, r^{-\frac{1}{2}+\delta_0} \xi)| \lesssim_{\alpha,\beta} (R r^{-1+\delta_0})^{-|\alpha|} r^{-\delta_0 |\beta|}.
    \end{equation*}
    It follows from  the Calderon--Vaillancourt estimate
    \begin{equation*}
        \int \big| \int \chi( r^{\frac{1}{2}} x', r^{-\frac{1}{2}} \xi') \hat{f}(r^{-\frac{1}{2}} \xi') e^{i x' \cdot \xi'} d\xi' \big|^2 dx' \lesssim \int |\hat{f}(r^{-\frac{1}{2}} \xi')|^2 d\xi',
    \end{equation*}
    which yields the claim.

    We turn to establishing the wave packet decomposition by separating the contribution from $\overline{\mathcal{Y}}_r$ and $\overline{\mathcal{Y}}_r^c$:  
    \begin{equation*}
    \begin{split}
    		&\quad S(t,0) T_r^* T_r (\chi_{\overline{\mathcal{Y}}_r}(x,D) f) \\
    	 &= S(t,0) T_r^* \big( \sum_{(x_0,\xi_0) \in \Lambda_r \cap \overline{\mathcal{Y}}_r} \psi_{x_0,\xi_0} +  \sum_{(x_0,\xi_0) \notin \Lambda_r \cap \overline{\mathcal{Y}}_r^c} \psi_{x_0,\xi_0} \big) T_r (\chi_{\overline{\mathcal{Y}}_r}(x,D) f).
    	\end{split}
    \end{equation*}
    Recall that $\psi_{x_0,\xi_0}$ is a smooth bump function adapted to $B(x_0,r^{\frac{1}{2}}) \times B(\xi_0,r^{-\frac{1}{2}})$.
    
	The first sum gives rise to the wave packet decomposition
	\begin{equation*}
	\begin{split}
		&\quad S(t,0) T_r^* \sum_{(x_0,\xi_0) \in \Lambda_r \cap \overline{\mathcal{Y}}_r} \psi_{x_0,\xi_0} (T_r \chi_{\overline{\mathcal{Y}}_r}(x,D) f) \\
		&= \sum_{(x_0,\xi_0) \in \Lambda_r \cap \overline{\mathcal{Y}}_r} \int (S(t,0) \phi_{x,\xi}) \psi_{x_0,\xi_0} T_r(\chi_{\overline{\mathcal{Y}}_r}(x,D) f)(x,\xi) dx d\xi,
	\end{split}
	\end{equation*}	    
    which has been analyzed in Section \ref{section:Localization}. The veracity of the wave packet axioms in Assumption~\ref{ass:WP} follows from Corollaries \ref{cor:SpatialLocalization}, \ref{cor:FrequencyLocalization}, and \ref{cor:TimeFrequencyLocalization}.
    
\smallskip    
    
    It remains to prove the error estimate for the second term. First, we handle the claimed $L^2$-estimate. In the following we apply the $L^2$-boundedness of the phase space transform and the $L^2$-boundedness of the evolution to find
    \begin{equation*}
    \begin{split}
    		&\; \| S(t,0) T_r^* \sum_{(x_0,\xi_0) \notin \Lambda_r \cap \overline{\mathcal{Y}}_r^c} \psi_{x_0,\xi_0} T_r (\chi_{\overline{\mathcal{Y}}_r}(x,D) f) \|_{L^2(\R^d)} \\
    		&\lesssim \big\| \sum_{(x_0,\xi_0) \notin \Lambda_r \cap \overline{\mathcal{Y}}_r^c} \psi_{x_0,\xi_0} T_r (\chi_{\overline{\mathcal{Y}}_r}(x,D) f) \big\|_{L^2(T^* \R^d)}.
    	\end{split}
    \end{equation*}
   The reason why we can show a favorable estimate is that the support of $\chi_{\overline{\mathcal{Y}}_r}$ and $\text{supp}(\sum_{(x_0,\xi_0) \notin \Lambda_r \cap \overline{\mathcal{Y}}^c} \psi_{(x_0,\xi_0)} )$ are separated by $C R r^{-\frac{1}{2}+\delta_0}$ in space and $C r^{-\frac{1}{2}+\delta_0} $ in frequency, and the FBI transform is applied on a finer scale. This will be used in the following.

\smallskip

    We recall the conjugation of pseudo-differential operators (see \cite[p.~351]{Tataru2000}) with the FBI transform given by
    \begin{equation*}
        T_\lambda f(z) = \lambda^{\frac{3d}{4}} 2^{-\frac{d}{2}} \pi^{-\frac{3d}{4}} \int e^{- \frac{\lambda}{2} (z-y)^2} f(y) dy, \quad z = x-i\xi \in T^* \R^d.
    \end{equation*}
    For a compactly supported symbol let $a_\lambda(x,\xi) = a(x,\xi/\lambda)$ be a rescaling supported at frequencies $\lambda$. By formally conjugating position and momentum operator we obtain the formal asymptotic
    \begin{equation*}
        T_\lambda a_\lambda(x,D) \approx \sum_{\alpha,\beta}  \frac{\partial_x^{\alpha} \partial_{\xi}^{\beta} a(x,\xi)}{\alpha! \beta! (-i \lambda)^{|\alpha|} \lambda^{|\beta|}} (\partial_{\xi} - \lambda \xi)^{\alpha} ( \frac{1}{i} \partial_x - \lambda \xi)^{\beta} T_\lambda.
    \end{equation*}
    For symbols with uniformly bounded $\xi$-derivatives we have error estimates in terms of the $C^s$-norm for conjugation with the truncated series defined by
    \begin{equation*}
        \tilde{a}^s_\lambda = \sum_{|\alpha|+|\beta|<s}  \frac{\partial_x^{\alpha} \partial_{\xi}^{\beta} a(x,\xi)}{\alpha! \beta! (-i \lambda)^{|\alpha|} \lambda^{|\beta|}} (\partial_{\xi} - \lambda \xi)^{\alpha} \big( \frac{1}{i} \partial_x - \lambda \xi \big)^{\beta}.
    \end{equation*}
    It holds for the remainder $R^s_{\lambda,a} = T_\lambda A_\lambda - \tilde{a}^s_\lambda T_\lambda$.
    \begin{theorem}[{\cite[Theorem~1]{Tataru2001}}]
    \label{thm:SymbolConjugation}
        Assume that $a \in C^s_x C^\infty_c$. Then
        \begin{equation*}
            \| R^s_{\lambda,a} \|_{L^2 \to L^2_{\Phi}} \leq c \lambda^{-\frac{s}{2}}.
        \end{equation*}
    \end{theorem}
    We aim to apply the theorem at scale $\lambda = r^{\frac{1}{2}}$. We can change from the currently used phase space transform $T_r$ to $T_\lambda$ by the change of variables $\xi = r^{-\frac{1}{2}} \xi'$ and $x = r^{\frac{1}{2}} x'$. The $C^s_x$-bounds of $\chi(x',\xi') = \chi_{\overline{\mathcal{Y}}_r}(r^{\frac{1}{2}} x',r^{-\frac{1}{2}} \xi')$ are decaying in $s$. 

    We turn to the details. First, we suppose that $\chi'(x,\xi) = b(x) c(\xi)$ is of product form by a Fourier series expansion:
    \begin{equation*}
        \chi'(x,\xi) = \frac{1}{r^{\frac{d}{2}}} \sum_{k \in r^{-\frac{1}{2}} \Z^d} e^{i \langle k , \xi \rangle} \hat{\chi}(x,k) \beta(\xi).
    \end{equation*}
    The expansion is clearly rigorous as for $|k| \geq 2$ the integration by parts gives rapid decay in $k$. By linearity the below error terms expansions extend to the Fourier series.
    
    After the phase space rescaling from above we find
    \begin{equation*}
        T f = C_r \int e^{i \xi'(x'-y')} e^{-\frac{(x'-y')^2}{2}} (\chi'_{\mathcal{Y}}(y',D_{y'}) f') dy'
    \end{equation*}
    and the derivatives of $\chi'_{\mathcal{Y}'}$ decay as
    \begin{equation}
    \label{eq:}
        \begin{split}
            |\partial_{x'}^{\alpha} \chi_{\mathcal{Y}}'| &\lesssim_{\alpha} (R r^{-1+\delta_0})^{-|\alpha|} \lesssim_{\alpha} r^{-|\alpha| \delta_0}, \\
            |\partial_{\xi'}^{\beta} \chi'_{\mathcal{Y}}| &\lesssim_{\beta} (r^{\delta_0})^{-|\beta|}.
        \end{split}
    \end{equation}

    We compute like in \cite[p.~355]{Tataru2000} for the normalized phase space transform $\lambda = 1$:
    \begin{equation*}
        T(b'(y') c'(D_y') f') = \tilde{b}^s T(c'(D_y) f') + R_b c'(D_y) f'
    \end{equation*}
    with
    \begin{equation*}
        \tilde{b}^s = \sum_{|\alpha| < s} \frac{\partial_x^{\alpha} b' (\partial_{\xi} - \xi)^{\alpha}}{\alpha!}.
    \end{equation*}
    For the first term we conjugate $c'(D_y)$ to find
    \begin{equation*}
        T(c'(D_y) f') = \sum_{|\beta| < s} \frac{\partial_{\xi}^{\beta} c}{\beta!} (\partial_x - \xi)^{\beta} Tf' + R^s_c f'.      
    \end{equation*}
    Taking the conjugated expressions $\tilde{b}^s$ and $\tilde{c}^s$ together we find an expression, which vanishes on the phase space region we are considering.
    Decay for the mixed term $\tilde{b}^s R^s_c f$ follows like in \cite[Eq.~(2.10)]{Tataru2000}. Consequently, we have
    \begin{equation*}
    \begin{split}
        T(b(y) c(D_y) f) &= \tilde{b}^s T(c(D) f) + R^s_b (c(D) f) \\
        &= \tilde{b}^s \tilde{c}^s Tf + \tilde{b}^s R^s_c + R^s_b (c(D) f).
        \end{split}
    \end{equation*}
    The first expression is local and confined to the support of $a =b \cdot c$ in phase space. Hence, it is not contributing. The second and third term can be estimated like in \cite{Tataru2000} and give contributions of $\text{RapDec}(r) \| f \|_2$
    This finishes the proof.
\end{proof}

\subsection{Wave packet decompositions at different scales and Assumption \ref{ass:Frames}}
\label{subsection:WavePacketDecompositions}

With the basic wave packet decomposition established in the preceding subsection we comment on the frame property for approximate solutions in Assumption \ref{ass:Frames}. We take the opportunity to enlist the assumptions, summarize the results, and give an extended version of Theorem \ref{thm:WavePacketDecompositions}.

\medskip

\textbf{Assumptions on the Hamiltonian:} Let $R \geq 1$, $0 < \delta_0 < \frac{1}{2}$, $\nu \in [R^{-\frac{1}{2}+\delta_0},1]$, and $p \in C^{\infty}(\R^{d+1} \times \R^d)$. We assume that for any $(\alpha,\beta) \in \N_0^d \times \N_0^d$:
\begin{equation}
\label{eq:RegularityHamiltonianSummary}
|\partial_z^{\alpha} \partial_{\xi}^{\beta} p| \leq C_{\alpha,\beta} R^{-|\alpha| + \frac{(|\alpha|-2)_+}{2}}.
\end{equation}
Let $0 < c < 1$ be a small universal constant (in applications of the wave packet decomposition in Theorems \ref{thm:BilinearSelfinteractionsSEQ}, \ref{thm:BilinearSelfInteractionsWave} this relates $\epsilon_0$ from Assumption \ref{ass:RegularityFlow} to $\nu$).

\smallskip

\textbf{Assumptions on the phase space region:} Let $\emptyset \neq \mathcal{Y} \subseteq B_d(0,10 R) \times \R^d$ be an open set, and $C=c^{-\frac{1}{2}}$ with $c$ from above. We require that $\mathcal{Y}$ satisfies the following size condition:
\begin{equation}
\label{eq:SizeConditionPhaseSpaceSummary}
\pi_{\xi}(\mathcal{Y}) \subseteq B_d(\xi^*,C \nu).
\end{equation}

With $\overline{\mathcal{Y}}$ defined in 
\begin{equation*}
    \overline{\mathcal{Y}} = \{ (y,\eta) \in T^* \R^d : \; \exists (x,\xi) \in \mathcal{Y} : |y-x| < C R \nu \text{ and } |\eta - \xi | < C \nu \},
\end{equation*}
we require the flow-out condition \eqref{eq:FlowOutCondition}. In case $p$ is $1$-homogeneous in $\xi$, we can weaken \eqref{eq:SizeConditionPhaseSpaceSummary} to 
\begin{equation*}
\{ \xi / |\xi| : \xi \in \pi_{\xi}(\mathcal{Y}) \} \subseteq B_d(\xi^*,C\nu).
\end{equation*}

\smallskip

\textbf{Phase space localizations:}
Let $r \in [\nu^{-2-\delta_0},R]$ and $\mathcal{Y}_r$, $\overline{\mathcal{Y}}_r$ be defined as $(R/r^{\frac{1}{2}-\delta_0},r^{-\frac{1}{2}+\delta_0})$-thickenings of $\mathcal{Y}_{\nu^{-2}}$ (see Definition \ref{def:Thickening}) or the $(CR/r^{\frac{1}{2}-\delta_0}, Cr^{-\frac{1}{2}+\delta_0})$-thickening, resp. Let $\chi_t:T^* \R^d \to T^* \R^d$ be the symplectic flow generated by $p$ and let $\mathcal{Y}_{r,t} = \chi_t(\mathcal{Y}_r)$. We consider the thickenings $(CR r^{-\frac{1}{2}+\delta_0},C r^{-\frac{1}{2}+\delta_0})$ of $\mathcal{Y}_{r,t}$ resulting in the sets $\overline{\mathcal{Y}}_{r,t}$. We construct smooth functions $\chi_{\overline{\mathcal{Y}}_{r,t}}(x,\xi)$ by invoking Urysohn's lemma and mollification. The smooth functions are identically one on $\overline{\mathcal{Y}}_{r,t}$ and vanish off the $(2 CR r^{-\frac{1}{2}+\delta_0},2C r^{-\frac{1}{2}+\delta_0})$ thickening. They moreover satisfy the estimates
\begin{equation*}
|\partial_x^{\alpha} \partial_{\xi}^{\beta} \chi_{\overline{\mathcal{Y}}_{r,t}}(x,\xi) | \lesssim_{\alpha,\beta} (R r^{-\frac{1}{2}+\delta_0})^{-|\alpha|} (r^{\frac{1}{2}-\delta_0})^{-|\beta|}.
\end{equation*}

The relevance of considering nested sets for different scales $r_1 \leq r_2^{1-\delta}$ is that upon evolving the essential part as $S(t_f,t_i) \chi_{\overline{\mathcal{Y}}_{r_2,t_i}} f= \sum_{T \in \T_{r_2} \subseteq \Lambda_{r_2} \cap \overline{\mathcal{Y}}_{r_2,t_i}} a_T \phi_T + g(t)$ for $|t_i-t_f| \leq r_2$ the final state at $t_f$ can be regarded as
\begin{equation}
    \label{eq:ExpansionDifferentScales}
    \sum_{T \in \T_{r_2} \subseteq \Lambda_{r_2} \cap \overline{\mathcal{Y}}_{r_2,t_i}} a_T \phi_T(t_f) = \chi_{\overline{\mathcal{Y}}_{r_1,t_f}} \sum_{T \in \T_{r_2} \subseteq \Lambda_{r_2} \cap \overline{\mathcal{Y}}_{r_2,t_i}} a_T \phi_T + \text{RapDec}(r_1) \| \chi_{\overline{\mathcal{Y}}_{r_2,t_i}} f \|_2.
\end{equation}
This will allow us to regard the solution on a smaller space-time scale as microlocalized on a coarser scale. This is required in the proofs of Theorems \ref{thm:GeneralizationBilinearParaboloid} and \ref{thm:GeneralizationBilinearCone} to employ the induction hypothesis.

\smallskip
With the above settings, we can formulate an extended version of Theorem \ref{thm:WavePacketDecompositions}:
\begin{theorem}
Let $R,\nu,\delta_0,$ and $p \in C^\infty$ be like in the \textbf{assumptions on the Hamiltonian.} Let $\emptyset \neq \mathcal{Y} \subseteq T^* \R^d$ satisfy the \textbf{assumptions on the phase space region}. Let $(\mathcal{Y}_{r,t})_{r \in [\nu^{-2-\delta_0},R]}$ be given in the \textbf{phase space localizations} such that
\begin{equation*}
\overline{\mathcal{Y}}_{R,t} \subseteq \overline{\mathcal{Y}}_{r_2,t} \subseteq \overline{\mathcal{Y}}_{r,t} \subseteq \overline{\mathcal{Y}} \subseteq T^* \R^d
\end{equation*}
with $R \geq r_2 \geq r_1 \geq \nu^{-2-\delta_0}$. 

Let $S(t,t_i):L^2 \to L^2$ be the time-evolution generated by $-i p^w$. Then, for any $t_i \in [-R,R]$ and $t \in (t_i + [-r,r]) \cap [-R,R]$ with $Q_r \subseteq Q_R$ denoting a space-time cube of size $r$ with time-center $t_i$, we have
\begin{equation}
\label{eq:LocalizedWavePacketSummary}
\begin{split}
S(t,t_i) (\chi_{\overline{\mathcal{Y}}_{r,t_i}} f) &= S(t,t_i) f_0 + S(t,t_i) f_1 \\
&= \sum_{T \in \mathbb{T} \subseteq \Lambda_r \cap \overline{\mathcal{Y}}_{r,t_i}} \alpha_T \phi_T + g
\end{split}
\end{equation}
with $\sum_T |\alpha_T|^2 \lesssim \| f \|_2^2$, $\# \T \leq r^{10d}$ and 
\begin{equation*}
\| g \|_{L^p_{x,t}(Q_r)} \leq C_m \nu^{d(\frac{1}{2}-\frac{1}{p})} (1+r \nu^2)^{-m} \| f \|_2. 
\end{equation*}
The $(\phi_T)$ satisfy the wave packet axioms with implicit constants in \eqref{eq:LocPos}, \eqref{eq:LocFrequ}, and \eqref{eq:LocTimeFrequ} for fixed $m$ and $\delta$, which are depending only on constants $C_{\alpha,\beta}$ from \eqref{eq:RegularityHamiltonianSummary} for $|\alpha|+|\beta| \leq M(m,\delta)$.

Let $\delta \geq \delta_0$. For a wave packet $\phi_T$ from above at scale $r_2$ with $r_1 \leq r_2^{1-\delta}$ it holds
\begin{equation*}
\| (1- \chi_{\overline{\mathcal{Y}}_{r,t_f}}(x,D)) \phi_T(t_f) \|_{L^2} \lesssim_\delta \text{RapDec}(r_1)
\end{equation*}
provided that $(x_T(t_f),\xi_T(t_f))$ is in the $(CRr_2^{-\frac{1}{2}+\delta_0},r_2^{-\frac{1}{2}+\delta_0})$-thickening of $\overline{\mathcal{Y}}_{r_2,t_f}$.
\end{theorem}
\begin{proof}
In the preceding we have established a wave packet decomposition for $\chi_{\overline{\mathcal{Y}}_{r,t_i}} f$. It remains to further localize this with a smooth cutoff $\chi_{Q_r}$.
More precisely, for given $Q_r$ we can further localize $\chi_{\overline{\mathcal{Y}}_{r,t_i}}(x,D) f$ to a spatial cube $Q_{r,x}$ (centered around the spatial center of $Q_r$) of size $Cr$ such that
\begin{equation*}
    S(t,t_i) \chi_{\overline{\mathcal{Y}}_{r,t_i}} f = S(t,t_i) [\chi_{Q_{r,x}} \chi_{\overline{\mathcal{Y}}_{r,t_i}} f] + g(t)
\end{equation*}
such that
\begin{equation*}
    \| g \|_{L^p_{x,t}(Q_r)} \lesssim \nu^{d \big( \frac12 - \frac1p \big)} \text{RapDec}(r) \|f \|_{L^2}.
\end{equation*}
$\chi_{Q_{r,x}}$ denotes a smooth function, which is identically one on $Q_{r,x}$, satisfies derivative bounds $|\partial_x^{\alpha} \chi_{Q_{r,x}}| \lesssim_{\alpha} r^{-|\alpha|}$ and is supported on $2 Q_{r,x}$.
For $\chi_{Q_{r,x}} \chi_{\overline{\mathcal{Y}}_{r,t_i}} f$ we carry out the wave packet construction from the previous paragraph, noting that the symbol $\chi_{Q_{r,x}} \chi_{\overline{\mathcal{Y}}_{r,t_i}}$ can still be handled like above and taking into account the additional spatial localization allows us to choose $\# \T \lesssim r^{10d}$ as the conjugation argument now restricts to $\Lambda_r \cap \overline{\mathcal{Y}}_{r,t_i} \cap Q_{r,x}$.
The bound for $g$ follows from a kernel estimate:
\begin{equation*}
\| 1_{Q_{r,x}} S(t,t_i) (1- \chi_{Q_{r,x}}) h \|_{L^p_x(\R^d)} \lesssim (1+r)^{-m} \| h \|_{L^p}.
\end{equation*}
This in turn is immediate from Proposition \ref{prop:ParametrixConstruction}, as the Hamiltonian flowout of $(1-\chi_{Q_{r,x}})$ is separated from $Q_{r,x}$ on the scale $r$. The factor $\nu^{d \big( \frac{1}{2} - \frac{1}{p} \big)}$ follows from applying Bernstein's inequality.

We turn to the proof of \eqref{eq:WavePacketLocalizationFinal}, that is for $\phi_T$ a normalized wave packet at scale $r_2$, having the localization 
\begin{equation}
\label{eq:WavePacketLocalizationFinalProof}
    \| (1-\chi_{\overline{\mathcal{Y}}_{r_1,t_f}} (x,D)) \phi_T(t) \|_{L^2} \lesssim_{\delta} \text{RapDec}(r_1)
\end{equation}
provided that $(x_T(t),\xi_T(t)) \in \overline{\mathcal{Y}}_{r_2,t_f}$ and $r_1 \leq r_2^{1-\delta}$.
First recall that $\overline{\mathcal{Y}}_{r_1,t_f}$ is the $C( R r_1^{-\frac{1}{2}+\delta_0}, r_1^{-\frac{1}{2}+\delta_0})$-thickening and we have the phase space localization
\begin{equation*}
    |T_{r_1} \phi_T(t,x,\xi)| \lesssim_N (1+ r_2^{-\frac{1}{2}} |x_T(t) - x| + r^{\frac{1}{2}}_1 |\xi_T(t) - \xi|)^{-N}.
\end{equation*}
Since $r_2^{\frac{1}{2}} \ll R r_2^{-\frac{1}{2}+\delta_0} \ll R r_1^{-\frac{1}{2}+\delta_0}$, and $r_2^{-\frac{1}{2}} \ll r_1^{-\frac{1}{2}+\delta_0}$, the essential support of $T_{r_1} \phi_T$ is separated by a margin of $(r_2^{\delta},r_1^{-\delta})$ and the rapid decay \eqref{eq:WavePacketLocalizationFinalProof} follows from conjugating $(1-\chi_{\overline{\mathcal{Y}}_{r_1,t_f}})$ with $T_{r_1}$ to phase space like in the previous paragraph. We omit the details to avoid repetition. 

Finally, we observe that evolving $\overline{\mathcal{Y}}_{r_1,t_i}$ under $p$ yields a thickening of size \\
$C(R r_1^{-\frac{1}{2}+\delta_0},r_1^{-\frac{1}{2}+\delta_0})$ of $\mathcal{Y}_{r_1,t_f}$ by the bi-Lipschitz property of the flow, Lemma \ref{lem:GeometryWavepackets}.
Consequently, for $\phi_T$ a wave packet with $T \in \Lambda_{r_1} \cap \overline{\mathcal{Y}}_{r_1,t_i}$ we have $(x_T(t_f),\xi_T(t_f)) \in \mathcal{Y}_{r_2,t_f}$ and can arrange for the support of $1- \chi_{\overline{\mathcal{Y}}_{r_1,t_f}}$ to be separated by a margin $C(R r_1^{-\frac{1}{2}+\delta_0},r_1^{-\frac{1}{2}+\delta_0})$.
\end{proof}

\section{Applications to linear and bilinear estimates}
\label{section:Applications}
In this section we show how the wave packet decomposition can be utilized to prove Strichartz and bilinear estimates. In the Appendix we show in Propositions \ref{prop:DispersiveEstimateSEQ} and \ref{prop:DispersiveEstimateWave} how to use wave packet decompositions and non-degeneracy properties of the Hamiltonian flow to obtain dispersive estimates.

\subsection{Estimates for wave equations with rough coefficients}
We focus on wave equations governed by a time-independent metric with coefficients $g^{ij} \in C^{1,s}(\R^{d+1})$, $s \in [0,1]$ for definiteness. The ellipticity condition reads
\begin{equation}
\label{eq:EllipticityLinearBilinear}
\exists \lambda, \Lambda > 0: \, \forall (x,t) \in \R^{d+1}: \; \forall \xi \in \R^d: \; \lambda |\xi|^2 \leq g^{ij}(x,t) \xi_i \xi_j \leq \Lambda |\xi|^2.
\end{equation}
The Cauchy problem for the wave equation is given by
\begin{equation}
\label{eq:CauchyProblemRoughWaveEquation}
\left\{ \begin{array}{cl}
\partial_t^2 u &= \partial_i g^{ij} \partial_j u, \quad (t,x) \in \R \times \R^d, \\
u(0) &= u_0, \quad \dot{u}(0) = u_1.
\end{array} \right.
\end{equation}

First, we outline the proof of Strichartz estimates. Here we recover the sharp estimates due to the second author \cite{Tataru2001}. The key argument to descend from $C^{1,1}$-coefficients to lower regularity is as before (\cite{BahouriChemin1999,Tataru2001,Tataru2002}) a paradifferential decomposition for the coefficients of the metric: For $C^{1,1}$-coefficients, and solutions at frequency $N$, a frequency truncation of the metric coefficients at frequency $N^\frac12$
 allows us to obtain a $C^{1,1}$-wave packet decomposition on the unit time scale with an error term that can be treated perturbatively invoking Duhamel's formula. In turn, it is well-known that a wave packet decomposition, which satisfies $C^{1,1}$-regularity bounds, can be used to recover Euclidean Strichartz estimates (e.g. Smith \cite{Smith1998}, Smith--Tataru \cite{SmithTataru2005}, Candy--Herr \cite{CandyHerr2018}).

\smallskip

For rougher coefficients, when attempting the same paradifferential decomposition the contribution from the Duhamel term becomes large. 
We may truncate the coefficients at higher frequencies to ameliorate its contribution.
But when truncating the coefficients at frequencies $N^{\sigma}$, $\sigma > \frac{1}{2}$, we can only obtain a wave packet decomposition on a frequency-dependent time scale, and a loss is incurred when summing over the frequency-dependent time intervals. Our choice of $\sigma$ is dictated by balancing these two contributions.

Since we obtain the sharp results \cite{Tataru2001,SmithTataru2002}, this shows that our regularity requirements on the wave packet decomposition are sharp on the $C^{s}$-scale.

Secondly, we show Theorem \ref{thm:RoughWaveEquationsIntro}, which is concerned with bilinear estimates. The proof uses similar reductions and decompositions like in the linear case: after recovering the wave packet decomposition on a frequency-dependent time scale, we can use the wave packet decomposition and invoke Theorem \ref{thm:GeneralizationBilinearCone}.

\subsubsection{Strichartz estimates}

In the following we recover \cite[Theorem~4]{Tataru2001}; we consider only $d \geq 4$ because the endpoint Strichartz space $L_t^2 L_x^{\frac{2(d-1)}{d-3}}$ becomes admissible and moreover, in order to simplify the exposition, we assume that the coefficients are  time-independent. Sharpness of the result was proved by Smith--Tataru \cite{SmithTataru2002}. We consider frequency-localized solutions $u_N$, $N \geq N_0$ 
and solve
\begin{equation}
\label{eq:FrequencyLocalizedWaveEquation}
    \partial_t^2 u_N = \partial_i g^{ij}_{\leq N/16} \partial_j u_N + f_N.
\end{equation}
For the reduction of the full equation \eqref{eq:CauchyProblemRoughWaveEquation} to \eqref{eq:FrequencyLocalizedWaveEquation} we refer to \cite[pp.~415--417]{Tataru2001}.

\begin{theorem}[Endpoint~Strichartz~estimates~via~wave~packet~decompositions]
\label{thm:EndpointStrichartzWave}
Let $s \in [0,1]$, and $(g^{ij})_{ij} \subseteq C^{1,s}(\R^{d}) $, which satisfies \eqref{eq:EllipticityLinearBilinear}.
Let $d \geq 4$, $q = \frac{2(d-1)}{d-3}$, $\kappa_0 = \frac{d-1}{2} - \frac{d}{q}$, and $\kappa_1 = \frac{1-s}{2(3+s)}$. The following estimate for solutions to \eqref{eq:FrequencyLocalizedWaveEquation} holds:
\begin{equation}
\label{eq:EndpointStrichartzWave}
N^{1-\kappa_0 - \kappa_1} \| u_N \|_{L_t^2([0,1],L^q_x(\R^d))} \lesssim N \| u_N(0) \|_{L^2} + \| \partial_t u_N(0) \|_{L^2} + \| f_N \|_{L^1_t L^2_x}.
\end{equation}
\end{theorem}
\begin{proof}
By a finite partition and mild dilation we can additionally suppose that $\| g^{ij} \|_{\dot{C}^{1,s}} \ll 1$.
By the energy estimate $\| \nabla_{t,x} u_N \|_{L_t^\infty L_x^2} \lesssim \| \nabla_{t,x} u_N(0) \|_{L^2} + \| f_N \|_{L_t^1 L_x^2}$ it suffices to show
\begin{equation}
\label{eq:FrequencyLocalizedStrichartz}
N^{1-\kappa_0-\kappa_1} \| u_N \|_{L_t^2([0,1];L^q_x)} \lesssim \| \nabla_{t,x} u_N \|_{L^{\infty} L^2} + \| f_N \|_{L^1 L^2}.
\end{equation}

For the proof of \eqref{eq:FrequencyLocalizedStrichartz} we carry out a paradifferential decomposition
\begin{equation*}
    \partial_i g^{ij}_{\ll N} \partial_j u_N = \partial_i g^{ij}_{\ll N^{\sigma}} \partial_j u_N + \partial_i g^{ij}_{N^{\sigma} \lesssim \cdot \ll N } \partial_j u_N
\end{equation*}
and want to treat the second term as part of the Duhamel term.

We shall argue below that truncating the coefficients at $N^{\sigma}$ allows us to use Euclidean Strichartz estimates on a time interval $I$ of length $N^{1-2\sigma}$. We show how the proof is concluded with this information at hand.

Let $L_N = - \partial_i g^{ij}_{\lesssim N^{\sigma}} \partial_j$. 
We have for the error term from truncating the coefficients at frequencies $N^{\sigma}$:
\begin{equation*}
\begin{split}
\big\| \frac{1}{\sqrt{L_N}} \partial_i g^{ij}_{N^{\sigma} \lesssim \cdot \ll N} \partial_j u_N \big\|_{L^2_x} &\lesssim N \| g^{ij}_{\gtrsim N^{\sigma}} \|_{L_x^\infty} \| u_N \|_{L^2_x} \\
&= N^{1-\sigma(1+s)} N^{\sigma(1+s)} \| g^{ij}_{\gtrsim N^{\sigma}} \|_{L_x^\infty} \| u_N \|_{L^2_x} \\
&\lesssim N^{1-\sigma(1+s)} \| g^{ij} \|_{C^{1,s}} \| u_N \|_{L^2_x}.
\end{split}
\end{equation*}
Applying the Euclidean Strichartz estimates and bearing in mind the estimate for the error term, we find
\begin{equation*}
\begin{split}
    N^{1-\kappa_0} \| u_N \|_{L_t^2(I;L_x^q)} &\lesssim \| u_N \|_{L^{\infty}_t(I;L_x^2)} + N^{1-\sigma(1+s)} \| u_N \|_{L_t^1(I;L_x^2)} + \| f \|_{L_I^1 L_x^2} \\
    &\lesssim (1+ N^{2-3 \sigma -\sigma s}) \| u_N \|_{L_I^{\infty} L_x^2} + \| f \|_{L_I^1 L_x^2}.
    \end{split}
\end{equation*}
The last estimate follows from H\"older's inequality on the interval $I$.
The contribution for the two terms is balanced for $\sigma = \frac{2}{3+s}$. The additional loss of carrying out an $\ell^2_I$-sum over the frequency-dependent time intervals incurs the additional loss $N^{\kappa_1}$. To finish the proof, it remains to argue that we have the Euclidean Strichartz estimates on the intervals $I$.

It suffices to estimate the half-wave evolution for an operator with truncated frequencies. It is a consequence of Seeley's theorem (see e.g. \cite[Theorem~3.3.1]{Sogge2017}) that
\begin{equation}
\label{eq:ApproximationPseudo}
\sqrt{L_N} f = p^w(x,t;D) f + p_0^w f, \quad a(x;\xi) = (g_{\leq N^{\sigma}}^{ij}(x,t) \xi_i \xi_j)^{\frac{1}{2}}
\end{equation}
with $p^w_0 : L^2 \to L^2$ uniformly bounded. Then, by a similar Duhamel argument by which we replace the coefficients with their truncated versions it suffices to show Strichartz estimates for $S^{\pm}_N(t)$, which denotes the propagator for $D_t \pm p^w$. 
After rescaling to unit frequencies $|\xi| \sim 1$, the symbol reads
\begin{equation*}
p_N(t,x,\xi) = ( g^{ij}_{\leq N^{\sigma}}(N^{-1} x,N^{-1} t) \xi_i \xi_j)^{\frac{1}{2}} \chi_1(\xi).
\end{equation*}
We check that the symbol regularity
    \begin{equation*}
        |\partial_z^{\alpha} \partial_{\xi}^{\beta} p(x,t;\xi)| \lesssim_{\alpha,\beta} \begin{cases}
            R^{-|\alpha|}, &\quad 0 \leq |\alpha| \leq 2, \\
            R^{- \frac{(|\alpha|-2)_+}{2}-2}, &\quad |\alpha| \geq 2.
        \end{cases}
    \end{equation*}
for Theorem \ref{thm:WavePacketDecompositions} is satisfied for $R=N^{2(1-\sigma)}$. Hence, we recover the wave packet decomposition by Theorem \ref{thm:WavePacketDecompositions}. Second, we obtain the dispersive estimate by Proposition \ref{prop:DispersiveEstimateWave} by the non-degeneracy of the flow (here we need $\| g^{ij} \|_{\dot{C}^{1,s}} \leq \epsilon \ll 1$). And consequently, we obtain Euclidean Strichartz estimates by the Keel--Tao argument \cite{KeelTao1998} on intervals of length $|I| = N^{1-2 \sigma}$. The proof is complete.
\end{proof}

\subsubsection{Bilinear estimates}

We turn to the proof of Theorem \ref{thm:RoughWaveEquationsIntro}, which is concerned with bilinear estimates for solutions $u_i$ to
\begin{equation*}
\left\{ \begin{array}{cl}
\partial_t^2 u_{1} &= \partial_i g^{ij}_1 \partial_j u_{1} + f_1, \\
\partial_t^2 u_{2} &= \partial_i g^{ij}_2 \partial_j u_{2} + f_2, \quad (t,x) \in [0,1] \times \R^d
\end{array} \right.
\end{equation*}
with $f_i \in L_t^1([0,1];L_x^2)$. We assume that $u_{i}$ satisfy the frequency localizations
\begin{equation*}
\text{supp}(\hat{u}_{i}) \subseteq \{ \xi \in \R^d : |\xi| \sim N, \quad \big| \frac{\xi}{|\xi|} - \xi_{*} \big| \leq \epsilon^* \ll 1\}.
\end{equation*}
Assume that $g^{ij}_k \in C^{1,s}$, $k=1,2$, $s \in [0,1]$ and $(g^{ij}_k)_{i,j}$ satisfy \eqref{eq:EllipticityLinearBilinear} and are frequency-truncated at frequencies $\ll N$. Additionally, assume that
\begin{equation*}
|g^{ij}_k - \nu_k \delta_{ij} | \leq \epsilon^* \ll 1, \qquad |\nu_1-\nu_2| \sim 1.
\end{equation*}
and recall that $p = p_d = \frac{d+3}{d+1}$.

\begin{proof}[Proof~of~Theorem~\ref{thm:RoughWaveEquationsIntro}]
We adjust the arguments from the previous paragraph.
as above we use the energy estimate to reduce to
\begin{equation}
\label{eq:ClaimEnergyBilinearReduction}
    N^2 \| u_1 u_2 \|_{L^p_{t,x}([0,1] \times \R^d) } \lesssim_\varepsilon N^{\frac{4}{3+s} \frac{d-1}{d+3} + \varepsilon} \prod_{i=1}^2 \big( \| \nabla_{t,x} u_i \|_{L^\infty L^2} + \| f_i \|_{L^1 L^2} \big).
\end{equation}
We consider equations with truncated coefficients
\begin{equation*}
    \partial^2_t u_k = \partial_i g^{ij}_{k,\leq N^{\sigma}} \partial_j u_k + \partial_i g^{ij}_{k,N^{\sigma} \lesssim \cdot \ll N} \partial_j u_k + f_k
\end{equation*}
with $f_k \in L_t^1 L_x^2$. In the following it will be important that we have for the second term approximately the same frequency localization of $u_k$. Let $L_{N,k} = - \partial_i g_{k, \leq N^{\sigma}}^{ij} \partial_j$.

By the same argument as in the previous theorem, we may replace the half-wave propagators $e^{\pm i t \sqrt{L_{N,k}}} $ with $S_{N,k}^{\pm}(t)$, which is the evolution generated by $D_t \pm p^w$, for $p(x,t;\xi) = \big( g^{ij}_{k,\leq N^{\sigma}} \xi_i \xi_j \big)^{\frac{1}{2}}$. We have already seen above that we obtain a wave packet decomposition on intervals of length $N^{1-2 \sigma}$ (after plugging in the finite speed of propagation). When estimating products
\begin{equation*}
    \| S_{N,1}^{\pm} g_1 S_{N,2}^{\pm} g_2 \|_{L^p_{t,x}(I \times B_d(x,|I|))}
\end{equation*}
by Theorem \ref{thm:GeneralizationBilinearCone} we need to check the transversality assumptions on the flows, with the initial data $g_i \in L^2(\R^d)$ having the frequency support assumptions from the outset. For products $S_{N,1}^+(t) g_1 S^+_{N,2}(t) g_2$ note that
\begin{equation*}
| \nabla_{\xi} p_1(z,\xi^1) - \nabla_{\xi} p_2(z,\xi^2) | = \big| \big( \frac{g^{ij}_{1,\leq N^{\sigma}} \xi^1_j }{(g^{ij}_{1,\leq N^{\sigma}} \xi^1_i \xi^1_j)^{\frac{1}{2}}} - \frac{g^{ij}_{2,\leq N^{\sigma}} \xi^2_j }{(g^{ij}_{2,\leq N^{\sigma}} \xi^2_i \xi^2_j)^{\frac{1}{2}}} \big) \mathbf{e}_i \big|.
\end{equation*}
$\mathbf{e}_i$ denotes the $i$th unit vector.
The error incurred replacing $g^{ij}_k$ with $\nu_k$ is $\mathcal{O}(\epsilon^*)$, and similarly, when replacing $\xi^i$ with $\xi_*$. Consequently,
\begin{equation*}
|\nabla_{\xi} p_1(z,\xi^1) - \nabla_{\xi} p_2(z,\xi^2)| = \big| \frac{\nu_1^{\frac{1}{2}} \xi_*}{|\xi_*|} - \frac{\nu_2^{\frac{1}{2}} \xi_*}{|\xi_*|} \big| + \mathcal{O}(\epsilon^*).
\end{equation*}
We conclude the transversality estimate by noting that
\begin{equation*}
\big| \nu_1^{\frac{1}{2}} - \nu_2^{\frac{1}{2}} \big| \sim |\nu_1 - \nu_2| \sim 1.
\end{equation*}

Next, we check the second transversality condition, which reads in the present context:
\begin{equation*}
\big| \langle \xi_*, \frac{g^{ij}_1 \xi^1_j \mathbf{e}_i}{(g^{ij}_{1,\leq N^{\sigma}} \xi^1_i \xi^1_j)^{\frac{1}{2}}} - \frac{g^{ij}_2 \xi^2_j \mathbf{e}_i}{(g^{ij}_{2,\leq N^{\sigma}} \xi^2_i \xi^2_j)^{\frac{1}{2}}} \rangle \big| \gtrsim 1.
\end{equation*}
Replacing $\xi^i$ with $\xi_*$ and $g^{ij}_{k, \leq N^{\sigma}}$ with $\nu_k \delta_{ij}$ incurs like above a loss of $\mathcal{O}(\epsilon^*)$. We find
\begin{equation*}
\big| \langle \xi_*, \nu_1^{\frac{1}{2}}  \frac{\xi_*}{|\xi_*|} - \nu_2^{\frac{1}{2}} \frac{\xi_*}{|\xi_*|} \rangle \big| \gtrsim | \nu_1 - \nu_2 | \gtrsim 1.
\end{equation*}

Upon replacing $p_2 \to - p_2$ we see that the transversality expressions are actually better behaved and we can apply Theorem \ref{thm:GeneralizationBilinearCone} likewise when estimating the mixed expressions $S^+_{N,1}(t) g_1$, $S^-_{N,2}(t) g_2$.

By the $L^2$-bound for the Duhamel term
\begin{equation*}
    \big\| \frac{1}{\sqrt{L_{k,N}}} \partial_i g^{ij}_{k, N^{\sigma} \lesssim \cdot \ll N} \partial_j u_k(t) \big\|_{L^2_x} \lesssim N^{1-\sigma(1+s)} \| u_k(t) \|_{L^2_x}
\end{equation*}
we recover on time intervals $I$, $|I|=N^{1-2\sigma}$ the Euclidean estimate
\begin{equation*}
\begin{split}
    &\quad N^2 \| u_1 u_2 \|_{L^{p_d}_{t,x}(I \times \R^d)} \\
    &\lesssim_\varepsilon N^{\frac{d-1}{d+3}+\varepsilon} \prod_{i=1}^2 \big( \| \nabla_{t,x} u_i \|_{L^\infty_I L^2_x} + N^{1- \sigma(1+s)} \| u_i \|_{L_I^1 L_x^2} + N \| f_i \|_{L_I^1 L_x^2} \big) \\
    &\lesssim_\varepsilon N^{\frac{d-1}{d+3}+\varepsilon} \prod_{i=1}^2 \big( (1+ N^{1- \sigma(1+s) + 1 - 2\sigma} ) \| u_i \|_{L^{\infty}_I L_x^2} + N \| f_i \|_{L_I^1 L_x^2} \big).
    \end{split}
\end{equation*}
Choosing $\sigma = \frac{2}{3+s}$ we simplify
\begin{equation*}
    N^2 \| u_1 u_2 \|_{L^{p_d}_{t,x}(I \times \R^d)} \lesssim_\varepsilon N^{\frac{d-1}{d+3}+\varepsilon} \prod_{i=1}^2 \big( \| \nabla_{t,x} u_i \|_{L^\infty_I L^2_x} + N \| f_i \|_{L_I^1 L_x^2} \big).
\end{equation*}
It remains to carry out the $\ell^{p_d}_I$-sum over intervals of length $|I| = N^{1-2 \sigma}$, which yields the claimed estimate. The proof is complete.
\end{proof}

\subsection{Estimates for Schr\"odinger equations with rough coefficients}

We consider a solution to the Schr\"odinger equation
\begin{equation}
\label{eq:FrequencyDependentSchroedinger}
i \partial_t u_N + \partial_i a^{ij} \partial_j u_N = f_N \in L^1 L^2
\end{equation}
on the frequency-dependent time interval $T=[0,N^{-1}]$ with $a^{ij} \in C^1_t C_x^{1,s}$, $s \in [0,1]$, and $\text{supp}(\hat{u}_N) \subseteq \{ |\xi| \sim N \}$. 

\subsubsection{Strichartz estimates}
We require non-degeneracy of $a^{ij}$ given by
\begin{equation}
\label{eq:UniformNondegeneracy}
\exists c, C > 0: \forall \xi \in \R^d: \forall (t,x) \in \R^{d+1}: c |\xi| \leq |a^{ij}(t,x) \xi_j| \leq C |\xi|.
\end{equation}
Schr\"odinger equations with operators satisfying this condition can be classified
as either \emph{elliptic}, where the matrix $a^{ij}$ is of definite (positive or negative) 
sign, or  as \emph{ultrahyperbolic}, where it is of indefinite sign. 

\normalsize

Here we recover the semiclassical Strichartz estimates without derivative loss for $s=1$, first proved in the elliptic case by Staffilani--Tataru \cite{StaffilaniTataru2002}, see also \cite{Tataru2008}.  Burq--G\'erard--Tzvetkov \cite{BurqGerardTzvetkov2004} provided a semi-classical interpretation and showed that there are compact Riemannian manifolds $(M,g)$, e.g., the sphere, for which the Strichartz estimate for the Laplace--Beltrami operator on a finite time interval
\begin{equation*}
\| P_N e^{it \Delta_g} f \|_{L_t^p([0,1], L^q_x(M))} \lesssim N^{\frac{1}{p}} \| f \|_{L^2(M)}
\end{equation*}
cannot be improved for some sharp Schr\"odinger admissible pairs $(p,q,d)$ which satisfy
\begin{equation*}
\frac{2}{p} + \frac{d}{q} = \frac{d}{2}, \quad (p,q,d) \neq (2,\infty,2).
\end{equation*}
The semiclassical Strichartz estimates are  completely similar in the ultrahyperbolic case, as only non-degeneracy is used in the proof, see the Appendix.

For comparison recall that sharp linear Strichartz estimates on Euclidean space are given by
\begin{equation*}
\| e^{it \Delta} f \|_{L_t^p(\R;L^q_x(\R^d))} \lesssim \| f \|_{L^2}.
\end{equation*}

In the following we illustrate how our wave packet decomposition yields lossless estimates for rough coefficients on frequency-dependent time scales. The argument is similar to the above: we employ a paradifferential decomposition, which provides us with a wave packet decomposition on a smaller time scale than $N^{-1}$. We balance the error from additional decomposition in time and the error from the Duhamel term. After the parabolic rescaling, reducing to unit frequencies, the argument becomes  a reprise of the proof of Theorem \ref{thm:EndpointStrichartzWave}. The assumption $d \geq 3$ is a simplification to work with the endpoint.

\begin{theorem}[Endpoint~Strichartz~estimates~for~rough~coefficients]
\label{thm:LinearStrichartzSEQ}
Let $d \geq 3$, $s \in [0,1]$, and $(a^{ij})_{i,j} \in C^{0,1}_t C^{1,s}_x$ be uniformly non-degenerate \eqref{eq:UniformNondegeneracy}. Let $q = \frac{2d}{d-2}$ and $N \gg 1$. Let $u$ be a solution to
\begin{equation*}
i \partial_t u + \partial_i a^{ij}_{\ll N} \partial_j u = f
\end{equation*}
with $\text{supp}(\hat{u}) \subseteq \{ \xi \in \R^d: |\xi| \sim N \}$. Then the following estimate holds:
\begin{equation}\label{Str-Schrodinger}
N^{-\kappa} \| u \|_{L_t^2([0,N^{-1}],L^q_x)} \lesssim \| u(0) \|_{L^2} + \| f \|_{L^1 L^2}
\end{equation}
with $\kappa = \frac{1-s}{2(3+s)}$.
\end{theorem}
\begin{proof}
We begin with a straightforward energy estimate
\[
\|u\|_{L^\infty L^2} \lesssim \| u(0) \|_{L^2} + \| f \|_{L^1 L^2}
\]
which allows us to reduce \eqref{Str-Schrodinger} to 
\begin{equation}\label{Str-Schrodinger1}
N^{-\kappa} \| u \|_{L_t^2([0,N^{-1}],L^q_x)} \lesssim \| u \|_{L^\infty L^2} + \| f \|_{L^1 L^2}
\end{equation}

We employ a paradifferential decomposition of the coefficients $a_{ij}$ at spatial frequencies $N^{\sigma}$, $\sigma \in [\frac12,1]$ to be chosen depending on $s$, to 
rewrite the equation for $u$ in the form
\begin{equation*}
i \partial_t u + \partial_i a^{ij}_{\ll N, \leq N^{\sigma}} \partial_j u = g:= - \partial_i a^{ij}_{c} \partial_j u + f.
\end{equation*}
Above $a^{ij}_{\ll N, \leq N^{\sigma}}$ indicates that the temporal frequencies are truncated at $\ll N$ and the spatial frequencies at size $N^{\sigma}$, while $a^{ij}_c$ denotes the complement.

By finite speed of propagation on $[0,N^{-1}]$, we can confine $u$ spatially to a unit ball. We carry out the estimate on $[0,N^{-1}] \times B_d(0,1)$ by rescaling 
\[
t \to N^2 t, \qquad x \to Nx, \qquad \xi \to \xi/N, \qquad
u_N(t,x) = N^{-\frac{d}2} u(N^{-1} x,N^{-2} t),
\]
noting that  the Strichartz norms rest unchanged
\begin{equation*}
\| u \|_{L_t^2([0,N^{-1}],L^q_x(B_{d}(0,1))}  =  \| u_N \|_{L_t^2([0,N],L^q_x(B_d(0,N))}.
\end{equation*}
So do the energy norms, therefore the bound \eqref{Str-Schrodinger1} also remains the same.

The functions $u_N$ are now localized at spatial frequency $1$, and the rescaled symbols are given by
\begin{equation*}
a(x,t;\xi)= a_{\ll N, \leq N^{\sigma}}^{ij}(N^{-1} x, N^{-2}t) \xi_i \xi_j.
\end{equation*}
For an initial regularity $a^{ij} \in C^{0,1}_t C^{1,s}_x$ we note that \eqref{eq:RegularityEstimatesHamiltonianFlowAssumption} is satisfied for $R=N^{2-2 \sigma}$. As above the rescaled time interval $[0,N]$ has to be split into $N^{2\sigma - 1}$ intervals of length $R$. On the small intervals we can use the wave packet decomposition and find Euclidean Strichartz estimates to hold; see Proposition \ref{prop:DispersiveEstimateSEQ} for the dispersive estimate. Correspondingly, the interval $[0,N^{-1}]$ has to be split into intervals $I'$ of length $|I'| = N^{-2\sigma}$ to apply Euclidean Strichartz estimates,
\[
 \| u\|_{L_t^2(I',L^q_x)} \lesssim \|u\|_{L^\infty(I', L^2)}
 + \|g\|_{L^1(I', L^2)}
\]
We can estimate the source term on the short intervals by 
\[
\begin{aligned}
\|g\|_{L^1(I', L^2)} \lesssim & \  |I'| \|a^{ij}_{c}\|_{L^\infty}
\| u \|_{L^\infty(I', L^2)} + \|f\| _{L^1(I', L^2)}
\\
\lesssim  & \  N^{-2\sigma} N^{2-\sigma(1+ s)} \| a^{ij} \|_{C^1_t C_x^{1,s}}\| u \|_{L^\infty L^2} + \|f\| _{L^1(I', L^2)}
\end{aligned}
\]
This leads to
\[
 \| u\|_{L_t^2(I',L^q_x)} \lesssim
 (1+ N^{2-\sigma(3+ s)}) \|u\|_{L^\infty L^2}
 + \|f\|_{L^1(I', L^2)}
\]
The two contributions balanced for $\sigma = \frac{2}{3+s}$, which gives
\[
 \| u\|_{L_t^2(I',L^q_x)} \lesssim
  \|u\|_{L^\infty L^2}
 + \|f\|_{L^1(I', L^2)}
\]
Since the number of intervals $I'$ is $N^{2\sigma -1}$,
summing over all $I' \subset [0,N^{-1}]$ yields 
derivative loss $\kappa = \sigma - \frac12 =  \frac{1-s}{2(3+s)}$ as claimed.

\end{proof}

\subsubsection{Bilinear estimates}

Next, we show a bilinear smoothing estimate on frequency-dependent times for Schr\"odinger equations with rough coefficients. Consider solutions to the system
\begin{equation*}
\left\{ \begin{array}{cl}
i \partial_t u_1 + \partial_i g^{ij}_{1, \ll N} \partial_j u_1 &= f_1, \\
i \partial_t u_2 + \partial_i g^{ij}_{2, \ll K} \partial_j u_2 &= f_2
\end{array} \right.
\end{equation*}
with
\begin{equation*}
\text{supp}(\hat{u}_1) \subseteq B(0,N) \backslash B(0,N/4), \quad \text{supp}(\hat{u}_2) \subseteq B(0,K)
\end{equation*}
and $K \leq c N$ with $c$ chosen small enough depending on the ellipticity constants. 

\begin{theorem}[Bilinear~smoothing~estimate]
\label{thm:BilinearSmoothingEstimate}
Let $d \geq 1$, and $u_i$ as above. We assume that $g^{ij}_k \in C^{0,1}_t C^{1,s}_x$ is uniformly elliptic as in \eqref{eq:EllipticityLinearBilinear}. Then the following estimate holds:
\begin{equation*}
\| u_1 u_2 \|_{L^{\frac{d+3}{d+1}}_{t,x}([0,N^{-1}] \times \R^d)} \lesssim_\varepsilon N^{\kappa_0+\kappa_1 +\varepsilon}  \prod_{i=1}^2 ( \| u_i(0) \|_{L^2} + \| f_i \|_{L^1 L^2})
\end{equation*}
with $\kappa_0 = -\frac{2}{d+3}$ and $\kappa_1 = \frac{(d+1)(1-s)}{(3+s)(d+3)}$.
\end{theorem}
\begin{proof}
An energy estimate reduces the claim to
\begin{equation}
    \label{eq:EnergyReductionBilinearStrichartz}
    \| u_1 u_2 \|_{L_{t,x}^{\frac{d+3}{d+1}}([0,N^{-1}] \times \R^d)} \lesssim_\varepsilon N^{\kappa_0 + \kappa_1 + \varepsilon} \prod_{i=1}^2 \big( \| u_i \|_{L_t^{\infty} L_x^2} + \| f_i \|_{L_t^1 L_x^2} \big).
\end{equation}
We use the same paradifferential truncation as before:
\begin{equation*}
i \partial_t u_k + \partial_i g^{ij}_{k, \leq N, \leq N^{\sigma}} \partial_j u_k = - \partial_i g^{ij}_{k, c} \partial_j u_k + f_k
\end{equation*}
and note that we have wave packet decompositions available on times $N^{-2 \sigma}$, which was verified in the proof of Theorem \ref{thm:LinearStrichartzSEQ}. Now we check the transversality assumptions. After rescaling to unit frequencies, we have $\text{supp}(\hat{u}_{1N}) \subseteq B(0,2) \backslash B(0,1/2)$ and $\text{supp}(\hat{u}_{2N}) \subseteq B(0,K/N)$. Regarding the difference of the group velocities we see that
\begin{equation*}
\begin{split}
| \nabla_{\xi} p_1(z,\xi_1) - \nabla_{\xi} p_2(z,\xi_2) | &= |(g^{ij}_1 \xi_{1j} - g^{ij}_2 \xi_{2j}) \mathbf{e}_i | \\
&\geq | |g^{ij}_1 \xi_{1j}| - |g^{ij}_2 \xi_{2j}| | \gtrsim 1- K/N \sim 1.
\end{split}
\end{equation*}
Next, we turn to the second order transversality:
Here we use the ellipticity of $g^{ij}_k$, for $k=1,2$ to estimate
\begin{equation*}
\begin{split}
&\quad (\partial_{\xi_i} p_1(z,\xi_1) - \partial_{\xi_i} p_2(z,\xi_2)) g^{ij}_k (\partial_{\xi_j} p_1(z,\xi_1) - \partial_{\xi_j} p_2(z,\xi_2))  \\
 &\sim | \nabla_{\xi} p_1(z,\xi_1) - \nabla_{\xi} p_2(z,\xi_2)| \sim 1.
\end{split}
\end{equation*}
Consequently, after the above reductions using finite speed of propagation, we can invoke Theorem \ref{thm:GeneralizationBilinearParaboloid} on intervals $I$ of length $N^{-2\sigma}$ and obtain on the time scale $N^{-2 \sigma}$ the Euclidean estimate:
\begin{equation*}
    \| u_1 u_2 \|_{L^{\frac{d+3}{d+1}}_{t,x}(I \times \R^d)} \lesssim_\varepsilon N^{-\frac{2}{d+3}+\varepsilon} \prod_{k=1}^2 \big( \| u_k \|_{L^\infty_I L^2_x} + \| \partial_i g^{ij}_{k,c} \partial_j u_k \|_{L_I^1 L^2_x} + \| f_k \|_{L_I^1 L_x^2} \big).
\end{equation*}
We balance the estimate as above for $\sigma = \frac{2}{3+s}$
\begin{equation*}
    \| u_k \|_{L^\infty_I L^2_x} + \| \partial_i g^{ij}_{k,c} \partial_j u_k \|_{L^1_I L^2_x} \lesssim (1 + N^{1-2 \sigma} N^{1- \sigma(1+s)}) \| u_k \|_{L^\infty_I L^2_x}.
\end{equation*}
The proof of \eqref{eq:EnergyReductionBilinearStrichartz} is concluded by carrying out the $\ell^{\frac{d+3}{d+1}}_I$-summation over intervals of length $|I| = N^{-2 \sigma}$.  This finishes the proof of Theorem \ref{thm:BilinearSmoothingEstimate}.
\end{proof}

\subsection{Self-interactions with small angle}
\label{subsection:SelfInteractions}
In our main bilinear estimates in Theorem~\ref{thm:GeneralizationBilinearParaboloid} and Theorem~\ref{thm:GeneralizationBilinearCone} we consider the small angle interaction of two different flows, 
but with a flatness condition on the two symbols, captured by the parameter $\epsilon_0 \ll \nu$ in \eqref{eq:MixedRegularity}. The main result of this section asserts that we can remove the flatness condition \eqref{eq:MixedRegularity} of the metric in the case of small transversality when we are dealing with self-interactions, i.e., the two interacting flows are the same. 

To describe our strategy, we first remark 
that we can recenter the coordinates and rescale the symbol to obtain self-interactions with unit transversality. However, the rescaling disrupts the time-regularity of the symbol. In order to cope with this, we need to establish versions of Theorems \ref{thm:GeneralizationBilinearParaboloid} and \ref{thm:GeneralizationBilinearCone} for $\nu \sim 1$ with strongly reduced time-regularity.

\subsubsection{Small angle self-interactions for Schr\"odinger symbols}
In the following we consider a symbol
\begin{equation*}
p(x,t,\xi) = g^{ij}(x,t) \xi_i \xi_j 
\end{equation*}
with $g^{ij}$ being uniformly elliptic and satisfying the space-time regularity
\begin{equation}
\label{eq:OriginalSpaceTimeRegularity}
|\partial_z^{\alpha} g^{ij}(x,t)| \lesssim_{\alpha}
R^{-|\alpha| + \frac{(|\alpha|-2)_+}{2}}.
\end{equation}
Denote the evolution governed by $-ia^w$ with $S(t,s)$. We show the following:
\begin{theorem}[Bilinear~self~interactions~for~Schr\"odinger~equations]
\label{thm:BilinearSelfinteractionsSEQ}
Assume the metric $g$ satisfies the regularity property \eqref{eq:OriginalSpaceTimeRegularity}. Let $\epsilon_0 = c_0 \nu$ for a fixed small constant $0<c_0<1$.
Suppose that $\text{supp}(\hat{u}_i) \subseteq B_d(\xi_i^*, \epsilon_0)$ and $|\xi_1^* - \xi_2^*| \sim \nu \gg \epsilon_0 \gtrsim R^{-\frac{1}{2}}$. Then the following estimate  holds for the evolution governed by $-i a^w$:
\begin{equation*}
\| S(t,0) u_1 S(t,0) u_2 \|_{L^{\frac{d+3}{d+1}}_{t,x}(B_{d+1}(0,R))} \lesssim_{\varepsilon} R^{\varepsilon} \nu^{-\frac{2}{d+3}} \| u_1 \|_{L^2} \| u_2 \|_{L^2}.
\end{equation*}
\end{theorem}
A word of clarification here: we are considering frequency supports, which are of diameter that is smaller, but still comparable to the frequency support separation.

The main improvement over Theorem \ref{thm:GeneralizationBilinearParaboloid} in the case of self-interactions is that we can remove the asymptotic flatness of the coefficients (i.e. the $\epsilon_0$ factor in \eqref{eq:MixedRegularity}) as $\epsilon_0 \to 0$. The case $\epsilon_0 \sim R^{-\frac{1}{2}}$ is straightforward by Bernstein's inequality. In the following we suppose that $R^{-\frac{1}{2}} \ll \epsilon_0 \ll 1$. We organize the proof in several steps.

\bigskip

\paragraph{\textbf{(i) Almost orthogonality}}

Let $p = \frac{d+3}{d+1}$ for brevity. Below we treat the case $\nu \leq R^{-\frac{1}{2}+\delta}$ directly. We have the following almost orthogonality by virtue of the bi-Lipschitz property of the Hamiltonian flow:
\begin{lemma}
Let $\delta > 0$, $\nu \geq R^{-\frac{1}{2}+\delta}$ and suppose that $f_i$ has the frequency-support properties from above. The following holds:
\begin{equation*}
\begin{split}
\| S(t,0) f_1 S(t,0) f_2 \|_{L^p_{t,x}(B_{d+1}(0,R))} &\lesssim \big( \sum_{|x_1-x_2| \lesssim R \nu } \| S(t,0) f_{1 x_1} S(t,0) f_{2 x_2} \|_{L^p_{t,x}}^p \big)^{\frac{1}{p}} \\
&\quad + \text{RapDec}(R) \| f_1 \|_2 \| f_2 \|_2.
\end{split}
\end{equation*}
In the above display $f_{i x_i} = \chi_{x_i} f_i$ denotes spatially smoothly localized functions to balls of size $R \nu$.
\end{lemma}
\begin{proof}
For $x \in R \nu \Z^d$ let $\chi_x$ be rapidly decaying off $B(x,R \nu)$ with Fourier support $\text{supp}(\hat{\chi}_x) \subseteq B(0,100 (R \nu)^{-1}) \subseteq B(0,\epsilon_0/100)$ such that
\begin{equation*}
f_i = \sum_{x \in R \nu \Z^d} \chi_x f_i.
\end{equation*}
The wave packet decomposition provided by Theorem \ref{thm:WavePacketDecompositions} yields (up to a rapidly decaying error term):
\begin{equation*}
S(t,0) (\chi_x f_i) = \sum_{(x,\xi) \in (N^{\frac{1}{2}} \Z^d \cap B(x,10 N \nu)) \times B(\xi_i^*, 10 \epsilon_0)} u_{x,\xi}.
\end{equation*}
We will argue that the contribution $S(t,0) (\chi_{x_1} f_1) S(t,0) (\chi_{x_2} f_2)$ is negligible for $|x_1-x_2| \gg R \nu$. To this end, we consider central bicharacteristics $(x_i^t,\xi_i^t)$ with $(x_i^0,\xi_i^0) = (x_i,\xi^*)$ and $B(\xi_i^*,3 \epsilon_0) \subseteq B(\xi^*,\nu)$. For $|x_1-x_2| \gg R \nu$ we can argue that the phase space distance between the bicharacteristics $d_R((x_1^t,\xi^t_1),(x_2^t,\xi_2^t)) = R^{-\frac{1}{2}} |x_1^t - x_2^t| + R^{\frac{1}{2}} |\xi_1^t - \xi_2^t|$ is always dominated by the spatial contribution:
\begin{equation*}
d_R((x_1^t,\xi^t_1),(x_2^t,\xi_2^t)) \sim R^{-\frac{1}{2}} |x_1^t - x_2^t| \gg R^{\frac{1}{2}} \nu \gtrsim R^{\frac{1}{2}+\delta}.
\end{equation*}
Next, we take two wave packets $(y_i^t,\eta_i^t)$, which are $R^{\frac{1}{2}} \nu$-close in phase space to the central bicharacteristics $(x_i^t,\xi_i^t)$. By the bi-Lipchitz property we see that the phase space distance is essentially unchanged: $d_R((y_1^t,\eta_1^t),(y_2^t,\eta_2^t)) \sim d_R((x_1^t,\xi_1^t),(x_2^t,\xi_2^t))$ and the spatial distance is still dominant $\sim R^{\frac{1}{2} + \delta}.$

For this reason, the associated wave packets are not essentially interacting spatially and the almost orthogonal decomposition established.
\end{proof}

Using simpler versions of the above arguments, we treat the case of very small transversality directly:
\begin{lemma}
For any $\epsilon > 0$ there is $\delta > 0$ such that whenever $R^{-\frac{1}{2}} \leq \nu \leq R^{-\frac{1}{2}+\delta}$ and $\text{supp}(\hat{f}_i) \subseteq B(\xi^*,\nu)$, the following estimate holds:
\begin{equation*}
\| S(t,0) f_1 S(t,0) f_2 \|_{L^p_{t,x}} \lesssim_\varepsilon R^{\varepsilon} \nu^{-\frac{2}{d+3}} \prod_{i=1}^2 \| f_i \|_2.
\end{equation*}
\end{lemma}
\begin{proof}
 By H\"older's inequality it will suffice to show 
\begin{equation*}
\| S(t,0) f_i \|_{L^{2p}_{t,x}} \lesssim_\varepsilon R^{\varepsilon} \nu^{-\frac{1}{d+3}} \| f_i \|_2.
\end{equation*}
The wave packet decomposition gives
\begin{equation*}
S(t,0) f_i = \sum_{(x,\xi) \in R^{\frac{1}{2}} \Z^d \times B(\xi^*,R^{-\frac{1}{2}+2 \delta})} u_{x,\xi} 
\end{equation*}
up to a rapidly decaying error term. There are $\mathcal{O}(R^{10 d \delta})$ initial frequencies $\xi \in R^{-\frac{1}{2}} \Z^d \times B(\xi^*,R^{-\frac{1}{2}+2\delta})$.
Adjusting $\varepsilon$, we have reduced to
\begin{equation*}
\big\| \sum_{x \in R^{\frac{1}{2}} \Z^d} u_{x,\xi} \big\|_{L^{2p}_{t,x}} \lesssim_{\varepsilon} R^{\varepsilon} \nu^{-\frac{1}{d+3}} \| f \|_{L^2}.
\end{equation*}
Here we use an almost orthogonal argument from above (actually, a simpler version): At fixed times, it holds for any $\kappa > 0$
\begin{equation}
\label{eq:PointwiseAlmostOrthogonality}
\big\| \sum_x u_{x,\xi} \big\|_{L^{2p}_x} \lesssim_{\kappa} R^{\kappa} \big( \sum_x \|u_{x,\xi} \|_{L^{2p}_x}^{2p} \big)^{\frac{1}{2p}}.
\end{equation}
The reason this estimate holds is again that for bicharacteristics $(x_i^t,\xi_i^t)$ with $ |x_1^0 - x_2^0| \gg R^{\frac{1}{2}+\kappa}, \; \xi_1(0) = \xi_2(0) $, the spatial distance dominates the phase space distance, and we find that  for $|x_1^0 - x_2^0| \gg R^{\frac{1}{2}+\delta}$ there is  essentially no spatial overlap. Integrating \eqref{eq:PointwiseAlmostOrthogonality} in $t$ yields
\begin{equation*}
\big\| \sum_x u_{x,\xi} \big\|_{L^{2p}_{t,x}} \lesssim_{\kappa} R^{\kappa} \big( \sum_x \|u_{x,\xi} \|_{L^{2p}_{t,x}}^{2p} \big)^{\frac{1}{2p}}.
\end{equation*}
Next, $u_{x,\xi}(t)$ is essentially frequency supported in an $R^{-\frac{1}{2}}$ ball. For this reason, by Bernstein's inequality and integration in time, 
\begin{equation*}
\| u_{x,\xi} \|_{L^{2p}_{t,x}} \lesssim R^{\frac{d+1}{2(d+3)}} \big( R^{-\frac{d}{2}} \big)^{\frac{1}{2}-\frac{d+1}{2(d+3)}} \| u_{x,\xi}(0) \|_2.
\end{equation*}
The claim follows from almost orthogonality of the wave packets in $L^2$.
\end{proof}

Consequently, it will suffice to prove an estimate
\begin{equation}
\label{eq:LocalizedExpression}
\| S(t,0) f_{1 z} S(t,0) f_{2 z} \|_{L^p_{t,x}} \lesssim_\varepsilon R^{\varepsilon} \nu^{-\frac{2}{d+3}} \prod_{i=1}^2 \| f_{i z} \|_2
\end{equation}
with $f_{i z}$ being frequency localized like at the outset and spatially localized to a ball of size $R \nu$ centered at $z \in R \nu \Z^d$, $\nu \in [R^{-\frac{1}{2}+\delta},1]$. Indeed,
\begin{equation*}
\begin{split}
\| S(t,0) f_1 S(t,0) f_2 \|_{L^{p}_{t,x}(B_{d+1}(0,R))} &\lesssim_\varepsilon R^{\varepsilon} \nu^{-\frac{2}{d+3}} \big( \sum_x \| f_{1x} \|_2^p \| f_{2x} \|_2^p \big)^{\frac{1}{p}} \\
&\lesssim_{\varepsilon} R^{\varepsilon} \nu^{-\frac{2}{d+3}} \big( \sum_x \| f_{1x} \|_2^{2p} \big)^{\frac{1}{2p}} \big( \sum_x \| f_{2x} \|_2^{2p} \big)^{\frac{1}{2p}} \\
&\lesssim_{\varepsilon} R^{\varepsilon} \nu^{-\frac{2}{d+3}} \| f_1 \|_2 \| f_2 \|_2.
\end{split}
\end{equation*}
The final estimate follows from the embedding $\ell^2 \hookrightarrow \ell^{2p}$.

For the localized expression in \eqref{eq:LocalizedExpression} we center the coordinates in phase space. For this we note that $S(t,0) f_{iz}$ are localized in the phase space around a central bi-characteristic $(\bar{x}^t,\bar{\xi}^t)$ on the scale $\epsilon_0 R \times \epsilon_0$. To take advantage, we center $S(t,0) f_{iz}$ around the central bicharacteristic with $(\bar{x}(0),\bar{\xi}(0)) = (z,\xi^*)$. In the first step we center $u_i(t,x) = S(t,0) f_{i z}$ in frequencies by considering $v_i(t,x) = e^{-i \bar{\xi}(t) x} u_i(t,x)$, or equivalently, $\hat{v}_i(t,\xi) = \hat{u}_i(t,\xi + \bar{\xi}(t))$.
\begin{lemma}
$v_i$ solves the equation $-i \partial_t v_i + \bar{p}^w(x,D) v_i = 0$ with 
\begin{equation*}
\begin{split}
\bar{p}(x,t,\xi) &= p(x,t,\xi + \bar{\xi}(t)) + x \partial_t \bar{\xi}(t) \\
&= p(x,t,\xi + \bar{\xi}(t)) - x \nabla_x p(\bar{x},t,\bar{\xi}(t)).
\end{split}
\end{equation*}
\end{lemma}
This is the same argument as in \eqref{eq:ModifiedGreenBounds}, for which reason the proof is omitted.
Next, we center in space and consider $w_i(x,t) = v_i(x+\bar{x}(t),t)$.
\begin{lemma}
$w_i$ solves the equation $-i \partial_t w_i + \bar{p}^w(x,D) w_i = 0$ with
\begin{equation*}
\begin{split}
\bar{p}(x,t,\xi) &= p(x,t,\xi + \bar{\xi}(t)) - (x+\bar{x}(t)) \nabla_x p(\bar{x}(t),t,\bar{\xi}(t)) \\
&\quad - \xi \partial_t \bar{x}(t).
\end{split}
\end{equation*}
\end{lemma}
Since the solution $u_i(t)$ was spatially essentially supported in $\bar{x}(t) + B_d(0,5 \epsilon_0 R)$, we have
\begin{equation*}
\| u_1 u_2 \|_{L^p_{t,x}(B_{d+1}(0,R)} = \| w_1 w_2 \|_{L^p_{t,x}([-R,R] \times B_d(0, \epsilon_0 R)},
\end{equation*}
up to rapidly decaying error terms with $w_i$ governed by
\begin{equation*}
\begin{split}
\tilde{p}(x,t,\xi) &= p(x+\bar{x}(t),t,\xi + \bar{\xi}(t)) - x \nabla_x p(\bar{x}(t),t,\bar{\xi}(t)) \\
&\quad - \xi \nabla_{\xi} p(\bar{x}(t),t,\bar{\xi}(t)) - p(\bar{x}(t),t,\bar{\xi}(t)).
\end{split}
\end{equation*}
Note that we have changed the purely time-dependent summand of $\tilde{p}$ by the gauge transformation $w \to e^{i \int_0^t f(s) ds} w$ for suitable $f$.

\bigskip

\paragraph{\textbf{(ii) Rescaling}}

Next, we carry out a parabolic rescaling. Time is rescaled by $t \to \epsilon_0^2 t$ and space by $x \to \epsilon_0 x$. Note that the relevant space-time scale is now $|(x,t)| \lesssim \epsilon_0^2 R$, and the transversality parameter becomes $\nu \sim 1$.

The resulting symbol reads
\begin{equation*}
\begin{split}
p'(x,t,\xi) &= \epsilon_0^{-2} (p(\epsilon^{-1} x+ \bar{x}(\epsilon_0^{-2} t),\epsilon_0^{-2}t,\xi + \bar{\xi}(\epsilon_0^{-2} t) - \epsilon_0^{-1} x \nabla_x p(\bar{x}(\epsilon_0^{-2} t, \epsilon_0^{-2} t, \bar{\xi}(\epsilon_0^{-2} t)) \\
&\quad - \epsilon_0 \xi \nabla_{\xi} p(\bar{x}(\epsilon_0^{-2} t),\epsilon_0^{-2} t, \bar{\xi}(\epsilon_0^{-2} t)) - p(\bar{x}(\epsilon_0^{-2} t),\epsilon_0^{-2} t, \bar{\xi}(\epsilon_0^{-2} t))).
\end{split}
\end{equation*}

Based on the size and regularity estimates for the original symbol $|\partial_x^{\alpha} \partial_{\xi}^{\beta} p| \lesssim_{\alpha,\beta} R^{-|\alpha| + \frac{(|\alpha|-2)_+}{2}}$, we find:
\begin{lemma}
$p'$ satisfies the size and regularity estimates:
\begin{equation}
\label{eq:SizeRegularityRescaled}
|\partial_x^{\alpha} \partial_{\xi}^{\beta} p'| \lesssim_{\alpha,\beta} (\epsilon_0^2 R)^{-|\alpha|+\frac{(|\alpha|-2)_+}{2}}.
\end{equation}
\end{lemma}
\begin{proof}
A straightforward computation gives the separation into powers of $\xi$:
\begin{equation}
\label{eq:DecompositionSymbol}
\begin{split}
p'(x,t,\xi) &= g^{ij}(\epsilon_0^{-1} x + \bar{x}(\epsilon_0^{-2} t), \epsilon_0^{-2} t) \xi_i \xi_j \\
&\quad + 2 \epsilon_0^{-1} (g^{ij}(\epsilon_0^{-1} x+ \bar{x}(\epsilon_0^{-2} t),\epsilon_0^{-2} t) - g^{ij}(\bar{x}(\epsilon_0^{-2} t),\epsilon_0^{-2} t)) \xi_i \bar{\xi}_j(\epsilon_0^{-2} t) \\
&\quad + \epsilon_0^{-2} (g^{ij}(\epsilon_0^{-1} x + \bar{x}(\epsilon_0^{-2} t),\epsilon_0^{-2} t) - \epsilon_0^{-1} x \partial_x g^{ij}(\bar{x}(\epsilon_0^{-2} t, \epsilon_0^{-2} t) \\
&\quad \quad - g^{ij}(\bar{x}(\epsilon_0^{-2} t),\epsilon_0^{-2} t)) \bar{\xi}_i(\epsilon_0^{-2} t) \bar{\xi}_j(\epsilon_0^{-2} t)) \\
&= I + II + III.
\end{split}
\end{equation}
We analyze $I$, $II$, and $III$ separately. The size and regularity estimates for $I$ are straightforward. For $II$ the mean-value theorem yields
\begin{equation*}
\begin{split}
&\quad 2 \epsilon_0^{-1} (g^{ij}(\epsilon_0^{-1} x+ \bar{x}(\epsilon_0^{-2} t),\epsilon_0^{-2} t) - g^{ij}(\bar{x}(\epsilon_0^{-2} t),\epsilon_0^{-2} t)) \xi_i \bar{\xi}_j(\epsilon_0^{-2} t) \\
&= 2 \epsilon_0^{-2} \partial_x g^{ij}(x',\epsilon_0^{-2} t) \xi_i \bar{\xi}_j(\epsilon_0^{-2} t).
\end{split}
\end{equation*}
Now, by size and regularity of $a$ and $|x| \lesssim \epsilon_0^2 R$ we find $|II| \lesssim 1$. The size estimates of $II$ for derivatives in $x$ are straightforward. 

Finally, for $III$ we use Taylor's theorem to find
\begin{equation*}
III = \epsilon_0^{-4} x^2 \partial^2_x g^{ij}(x',\epsilon_0^{-2}t ) \bar{\xi}_i(\epsilon_0^{-2} t) \bar{\xi}_j(\epsilon_0^{-2} t),
\end{equation*}
which makes the estimate $|III| \lesssim 1$ immediate. Next,
\begin{equation*}
\partial_x III = \epsilon_0^{-3} ( \partial_x g^{ij}(\epsilon_0^{-1} x + \bar{x}(\epsilon_0^{-2} t,\epsilon_0^{-2} t) - \partial_x g^{ij}(\bar{x}(\epsilon_0^{-2} t),\epsilon_0^{-2} t)) \bar{\xi}_i(\epsilon_0^{-2} t) \bar{\xi}_j(\epsilon_0^{-2} t),
\end{equation*}
which is estimated again by the mean-value theorem and size estimates of $g$. Higher derivative estimates of $III$ follow more directly from size and regularity of $a$.
\end{proof}

\smallskip

\paragraph{\textbf{(iii) A bilinear estimate for the rescaled evolution}}

The previous lemma provides sufficient spatial regularity for 
the symbol $a'$ relative to the new space-time scale  $\epsilon_0^2 R$. However,
 Theorem \ref{thm:GeneralizationBilinearParaboloid} is not directly applicable because the time-regularity of the symbol is worse here. Instead, we modify the argument of the proof for transversality $\nu \sim 1$.
 There are two steps in the proof of Theorem \ref{thm:GeneralizationBilinearParaboloid} which depend on the time-regularity of the symbol:
\begin{enumerate}
\item \emph{joints} lying on energy shells, i.e., the intersection of four wave packets leading to a localization of the difference of energies, see \eqref{eq:EnergyConservation}. 
\item when checking for multiple intersections between $\Gamma$ and $T_2$ in the proof of Lemma \ref{l:transverse}.
\end{enumerate}

We begin with the analysis of joints. Note that the regularity in $x$ and $\xi$ already gives the sharp localization of wave packets along the Hamiltonian flow
\begin{equation*}
\left\{ \begin{array}{cl}
\dot{x}^t &= \partial_{\xi} p'(x^t,t,\xi^t), \\
\dot{\xi}^t &= - \partial_x p'(x^t,t,\xi^t).
\end{array} \right.
\end{equation*}

Recall that we let
\begin{equation*}
u_{x_0,\xi_0}(y,t) = \int_{\R^d} K_{x_0,\xi_0}(y,t,\tilde{y},0) u_0(\tilde{y}) d\tilde{y}
\end{equation*}
with (as above letting $\rho = (\epsilon_0^2 R)^{-1}$)
\begin{equation*}
\begin{split}
K_{x_0,\xi_0}(y,t,\tilde{y},0) &= \int_{\R^{2d}} e^{- \frac{\rho}{2}(\tilde{y}-x)^2} e^{-i \xi (\tilde{y}-x)} e^{i \xi^t (y-x^t)} e^{i (\psi(x,t,\xi) - \psi(x,0,\xi))} \psi_{x_0,\xi_0}(x,\xi) \\
&\quad \quad \times G(t,0,x,\xi,y) dx d\xi.
\end{split}
\end{equation*}
We find
\begin{equation*}
u_{x_0,\xi_0}(y,t) = \int_{\R^{2d}} e^{i \xi^t (y-x^t)} e^{i \psi(x,t,\xi) } \psi_{x_0,\xi_0}(x,\xi) T_{\epsilon_0^2 R} u_0(x,\xi) G(t,s,x,\xi,y) dx d\xi
\end{equation*}
and after normalization, we can suppose $\| \psi_{x_0,\xi_0} T_{\epsilon_0^2 R} u_0 \|_{L^2(T^* \R^d)} = 1$.
Recall that the bounds on $G$ follow from the regularity properties of the Hamiltonian flow. 
First note that the bi-Lipschitz property from Lemma \ref{lem:GeometryWavepackets} still applies as it does not depend on the time-regularity.
 In particular, the localization in $x$ and $\xi$ follow from the regularity properties of the Hamiltonian flow in $x$ and $\xi$, but $a'$ satisfies the required regularity on the scale $\epsilon_0^2 R$. For this reason, the localization in space and frequency of the wave packets is retained.

It remains to prove an acceptable time-frequency localization of wave packets on time-scales $(\epsilon_0^2 R)^{\frac{1}{2}}$. It will be useful that one time-derivative of $G$ can be expressed in terms of its spatial regularity. We can no longer prove the rapid decay of the time-frequency away from $p'$, since $p'(x^t,t,\xi^t)$ is no longer approximately constant for $\Delta t \sim (\epsilon_0^2 R)^{\frac{1}{2}}$. The problematic term is given by $III$. We define the renormalized Hamiltonian $p''(x,t,\xi) = I + II$.

 We consider the time-regularity for joints which intersect with transversality $\nu \sim 1$ at $t=t_0$ on a time scale $\Delta t \sim (\epsilon_0^2 R)^{\frac{1}{2}}$. Let $\chi$ be a bump function on the unit scale and $\chi_{\mathfrak{q}}(x,t) = \chi((\epsilon_0^2 R)^{-\frac{1}{2}}(t-t_0)) \chi( (\epsilon_0^2 R)^{-1} | x-x_0|^2)$ be a function localizing to a space-time cube of size $(\varepsilon_0^2 R)^{\frac{1}{2}}$ - the scale on which the interaction of four transversely interacting wave packets takes place. A mild enlargement is denoted by $\tilde{\chi}_{\mathfrak{q}}$. We consider the localized quadrilinear expression
\begin{equation}\label{quad}
\iint \chi_{\mathfrak{q}}(y,t) u_{x_1,\xi_1}(y,t) u_{x_3,\xi_3}(y,t) \overline{u}_{x_2,\xi_2}(y,t) \overline{u}_{x_4,\xi_4}(y,t) dy dt.
\end{equation}

For each of the four packets we use the representation above. Since the $G$ factors already are smooth on the appropriate scales, it suffices to analyze the time-regularity of the symmetrized phase function. One integration by parts will already yield acceptable bounds in order to allow us to repeat the argument in the proof of Theorem~\ref{thm:GeneralizationBilinearParaboloid}. Moreover, we use the fact that  $x_j^{t_0}$ coincide modulo $\mathcal{O}((\epsilon_0^2 R)^{\frac{1}{2}})$ errors. 
Let $\phi^*_{x,\xi}(y,t) = C_N (\epsilon_0^2 R)^{-\frac{d}{4}} (1+ (\epsilon_0^2 R)^{-\frac{1}{2}} |x^t - y|)^{-N}$, $N=10d$ denote an amplitude majorization of $u_{x,\xi}(y,t)$. We show the following: 
\begin{proposition}
\label{prop:TimeRegularityJoints}
Suppose that $\epsilon_0 \gg R^{-\frac{1}{2}}$. Let 
\begin{equation*}
\bar{p}^s = p''(x_1^{t_0},t_0,\xi_1^{t_0}) + p''(x_3^{t_0},t_0,\xi_3^{t_0}) - p''(x_2^{t_0},t_0,\xi_2^{t_0}) - p''(x_4^{t_0},t_0,\xi_4^{t_0}).
\end{equation*}
Then it holds for $|\bar{p}^s| \gg (\epsilon_0^2 R)^{-\frac{1}{2}}$:
\begin{equation*}
\begin{split}
&\quad \big| \iint \chi_{\mathfrak{q}}(y,t) u_{x_1,\xi_1}(y,t) u_{x_3,\xi_3}(y,t) \overline{u}_{x_2,\xi_2}(y,t) \overline{u}_{x_4,\xi_4}(y,t) dy dt \big| \\
 &\lesssim \frac{1}{(\epsilon_0^2 R)^{\frac{1}{2}} |\bar{p}^s|} \iint \tilde{\chi}_{\mathfrak{q}}(y,t) \phi^*_{x_1,\xi_1}(y,t) \phi^*_{x_3,\xi_3}(y,t) \phi^*_{x_2,\xi_2}(y,t) \phi^*_{x_4,\xi_4}(y,t) dy dt.
 \end{split}
\end{equation*}
\end{proposition}
\begin{proof}
First, recall the majorization
\begin{equation*}
    |u_{x_i,\xi_i}(y,t)| \lesssim (\epsilon_0^2 R)^{-\frac{d}{4}} (1+ (\epsilon_0^2 R)^{-\frac{1}{2}} |x_i^t - y|)^{-N}
\end{equation*}
as a consequence of the pointwise bound for $G$, which was derived in Section \ref{section:RegularityFlow}:
\begin{equation}
\label{eq:MajoriizationG}
    |G(t,x,\xi,y)| \lesssim (\epsilon_0^2 R)^{-\frac{d}{4}} (1+ (\epsilon_0^2 R)^{-\frac{1}{2}} |x^t - y|)^{-N}.
\end{equation}

Next, we consider the time dependence of the phase function of a single wave packet relative to the center time $t_0$,
\begin{equation}
\label{eq:ExpPhase}
\begin{aligned}
\xi^t \cdot (y-x^t) + \psi(x,t,\xi) = & \  \xi^{t_0} \cdot (y-x^{t_0}) + \psi(x,t_0,\xi)  \\
& + \int_{t_0}^t \!\!\!\! - \partial_x p'(x^s,s,\xi^s) (y-x^s) ds - \int_{t_0}^t \!\!\! p'(x^s,s,\xi^s) ds.
\end{aligned}
\end{equation}
We recast the second integral as
\begin{equation*}
\begin{split}
\int_{t_0}^t p'(x^s,s,\xi^s) ds &= \int_{t_0}^t \frac{d}{ds} (s-t) p'(x^s,s,\xi^s) ds \\
&= (t-t_0) p'(x^{t_0},t_0,\xi^{t_0}) - \int_{t_0}^t (s-t) \frac{d}{ds} p'(x^s,s,\xi^s) ds.
\end{split}
\end{equation*}
We find by the Hamiltonian equations:
\begin{equation*}
\int_{t_0}^t (s-t) \frac{d}{ds} p'(x^s,s,\xi^s) ds = \int_{t_0}^t (s-t) (\partial_t p')(x^s,s,\xi^s) ds.
\end{equation*}
We have obtained the representation of a single wave packet
\begin{equation*}
\begin{split}
    u_{x_i,\xi_i}(y,t) &= \int e^{-i (t-t_0) p'(x^{t_0},t_0,\xi^{t_0}) + i \phi(x,t_0,\xi) -i \int_{t_0}^t \partial_x p'(x^s,s,\xi^s)(y-x^s) ds} \\
    &\quad \times e^{i \int_{t_0}^t (s-t) (\partial_t p')(x^s,s,\xi^s)} \psi_{x_i,\xi_i}(x,\xi) T_{\epsilon_0^2 R} u_0(x,\xi) G(t,x,\xi,y) dx d\xi.
    \end{split}
\end{equation*}
We want to integrate by parts the quadrilinear expression \eqref{quad} with respect to time in order to take advantage of the phase variation.
Recall again by the estimates from Section \ref{section:RegularityFlow} that
the kernels $G(t,x,\xi,y)$ 
and spatially smooth and  localized on the wave packet scales, and that we also control their time derivative
\begin{equation} 
\label{eq:PointwiseBoundG}
    |\partial_t G(t,x,\xi,y)| \lesssim (\epsilon_0^2 R)^{-\frac{1}{2}} (\epsilon_0^2 R)^{-\frac{d}{4}} (1+ (\epsilon_0^2 R)^{-\frac{1}{2}} |x^t - y|)^{-N}.
\end{equation}
Notably, for the first time-derivative we rely on the governing equation, see the proof of \eqref{eq:ModifiedGreenBounds}. So, by recasting the first time-derivative as spatial derivatives, we do not need the time-regularity of the symbol.

For phase space variables $(x_\alpha,\xi_\alpha), \ldots, (x_\delta,\xi_\delta) \in \R^{2d}$ and an observable $O(x,t,\xi)$ we denote
\begin{equation*}
    O^s(\bar{x},t,\bar{\xi}) = O(x_\alpha,t,\xi_\alpha) - O(x_{\beta},t,\xi_{\beta}) + O(x_{\gamma},t,\xi_{\gamma}) - O(x_{\delta},t,\xi_{\delta}).
\end{equation*}
Below we denote the integration variables for assembling the wave packets $u_{x_i,\xi_i}$ with $(x_{\alpha},\xi_{\alpha}), \ldots, (x_{\delta},\xi_{\delta})$, we let $\epsilon \in \mathcal{E} = \{ \alpha, \beta, \gamma, \delta\}$ and define 
\begin{equation*}
    \text{sgn}(\epsilon) = 
    \begin{cases}
        1, &\quad \epsilon \in \{ \alpha, \gamma \}, \\
        -1, &\quad \epsilon \in \{\beta, \delta\}.
    \end{cases}
    \qquad 
        b^{(\epsilon)} = \begin{cases}
        b, &\quad \epsilon \in \{ \alpha, \gamma \}, \\
        \overline{b}, &\quad \epsilon \in \{\beta, \delta \}.
    \end{cases}
\end{equation*}

We recast the quadrilinear expression \eqref{quad} as
\begin{equation*}
    \begin{split}
        &\quad \iint \chi_{\mathfrak{q}}(y,t) u_{x_1,\xi_1}(y,t) \overline{u}_{x_2,\xi_2}(y,t) u_{x_3,\xi_3}(y,t) \overline{u}_{x_4,\xi_4}(y,t) dy dt \\
        &= \iint dt dy \chi((\epsilon_0^2 R)^{-\frac{1}{2}} (t-t_0)) \chi(\epsilon_0^2 R | y- x_0|^2)  \iint \prod_{\epsilon \in \mathcal{E}} d x_{\epsilon} d\xi_{\epsilon} e^{-i(t-t_0)(p')^s(\bar{x}^{t_0},t_0,\bar{\xi}^{t_0}) }  \\
        &\; \times e^{i \phi^s(\bar{x},t_0,\bar{\xi})} \prod_{\epsilon \in \mathcal{E}} e^{i \text{sgn}(\epsilon) \int_{t_0}^t (\partial_x p')(x_{\epsilon},s,\xi_{\epsilon})(y-x_{\epsilon}^s) ds }
        e^{i \int_{t_0}^t (s-t) (\partial_t p')^s(\bar{x}^s,s,\bar{\xi}^s) ds} \\
        &\; \times \prod_{\epsilon \in \mathcal{E}} \big( (T_{\epsilon_0^2 R}u_0)(x_{\epsilon},\xi_{\epsilon}) \psi_{x_i,\xi_i}(x_{\epsilon},\xi_{\epsilon}) G(t,x_{\epsilon},\xi_{\epsilon},y) \big)^{(\epsilon)}.
    \end{split}
\end{equation*}
In the above display we plug in the identity 
\begin{equation*}
\begin{split}
     &\quad e^{-i(t-t_0) (p'(x_{\alpha}^{t_0},t_0,\xi_{\alpha}^{t_0}) - p'(x_{\beta}^{t_0},t_0,\xi_{\beta}^{t_0}) + p'(x_{\gamma}^{t_0},t_0,\xi_{\gamma}^{t_0}) - p'(x_{\delta}^{t_0},t_0,\xi_{\delta}^{t_0}))} \\
     &= \frac{d}{dt} \frac{e^{-i(t-t_0) (p'(x_{\alpha}^{t_0},t_0,\xi_{\alpha}^{t_0}) - p'(x_{\beta}^{t_0},t_0,\xi_{\beta}^{t_0}) + p'(x_{\gamma}^{t_0},t_0,\xi_{\gamma}^{t_0}) - p'(x_{\delta}^{t_0},t_0,\xi_{\delta}^{t_0}))}}{-i(p'(x_{\alpha}^{t_0},t_0,\xi_{\alpha}^{t_0}) - p'(x_{\beta}^{t_0},t_0,\xi_{\beta}^{t_0}) + p'(x_{\gamma}^{t_0},t_0,\xi_{\gamma}^{t_0}) - p'(x_{\delta}^{t_0},t_0,\xi_{\delta}^{t_0}))}
     \end{split}
\end{equation*}
and integrate by parts in time. By the support conditions of $(x_{\epsilon},\xi_{\epsilon})$ and the hypothesis $|\bar{a}^s| \gg (\epsilon_0^2 R)^{-\frac{1}{2}}$, we have
\begin{equation*}
    | p'(x_{\alpha}^{t_0},t_0,\xi_{\alpha}^{t_0}) - p'(x_{\beta}^{t_0},t_0,\xi_{\beta}^{t_0}) + p'(x_{\gamma}^{t_0},t_0,\xi_{\gamma}^{t_0}) - p'(x_{\delta}^{t_0},t_0,\xi_{\delta}^{t_0}) | \sim |\bar{a}^s|.
\end{equation*}
Indeed, here we use the fact   that 
\[
|III(t)| \lesssim \epsilon_0^{-4} R^{-2} \ll \epsilon_0^{-1} R^{-\frac12},
\]
which is easily seen by Taylor expansion.
This makes the integration by parts legitimate due to the minimum size of the denominator.

\smallskip

After integration by parts in time, the derivative can hit
\begin{itemize}
    \item $\chi((\epsilon_0^2 R)^{-\frac{1}{2}}(t-t_0))$, 
    \item $\exp(i \text{sgn}(\epsilon) \int_{t_0}^t (\partial_x p')(x_\epsilon,s,\xi_\epsilon)(y-x_{\epsilon}^s) ds) $, 
    \item $\exp( i \text{sgn}(\epsilon) \int_{t_0}^t (s-t) (\partial_t p')(x_{\epsilon}^s,s,\xi_{\epsilon}^s) ds)$, 
    \item $G(t,x_{\epsilon},\xi_{\epsilon},y)$.
\end{itemize}
Suppose that the derivative hits the first term. This yields a factor of $(\epsilon_0^2 R)^{-\frac{1}{2}}$. By taking absolute values and plugging in \eqref{eq:MajoriizationG}, we can obtain the claimed bound. Similarly, the claimed bound follows when the derivative hits the Green's function, which corresponds to the fourth term. For the second term we note by the localization property of the Green's function
\begin{equation*}
    \big| \frac{d}{dt} \exp(i \text{sgn}(\epsilon) \int_{t_0}^t (\partial_x p')(x^s_{\epsilon},s,\xi_{\epsilon}^s) (y-x_{\epsilon}^s) ds ) \big| \lesssim (\epsilon_0^2 R)^{-\frac{1}{2}}.
\end{equation*}
For handling the third term in the above list, suppose that the following estimate were true:
\begin{equation*}
| \partial_t p'(x^s,s,\xi^s) | \lesssim (\epsilon_0^2 R)^{-1}.
\end{equation*}
Then we could conclude like in the remaining cases. But in general the expression is not good enough by a factor of $\epsilon_0^{-1}$. 
Recall that
\begin{equation*}
\begin{split}
p'(x,t,\xi) = I(t) + II(t) + III(t)
\end{split}
\end{equation*}
with $I$, $II$, $III$ defined in \eqref{eq:DecompositionSymbol}.

The contributions of the first and second term are acceptable as we have an additional $\epsilon_0$ factor to spare, which compensates for the time derivative of $\bar{x}(\epsilon_0^{-2} t)$.
Indeed, for the first term we find
\begin{equation*}
\partial_t I(t) = \partial_t ( g^{ij}(\epsilon_0^{-1} x + \bar{x}(\epsilon_0^{-2} t)) \xi_i \xi_j) = \partial_x g^{ij}(\epsilon_0^{-1} x+ \bar{x}(\epsilon_0^{-2} t)) \xi_i \xi_j \epsilon_0^{-2} \partial_t \bar{x}.
\end{equation*}
We obtain by the size and regularity of $g^{ij}$:
\begin{equation*}
| \partial_t I(t) | \lesssim \epsilon_0^{-2} R^{-1}.
\end{equation*}
Similarly, the second term is acceptable.

The third term is problematic. A direct estimate gives $|\partial_t III| \lesssim \epsilon^{-3} R^{-1}$, which is not sufficient.

\medskip

We observe that, if we could regard the oscillation caused by $III$ to be independent of $(x_j,\xi_j)$, it would cancel in the quadrilinear form describing a joint. This can be expected by the space-time localization of joints and will be made rigid by applying a gauge transform. Let
\begin{equation*}
\alpha(x,t) = \epsilon_0^{-2} ( g^{ij}(\epsilon_0^{-1} x + \bar{x}(\epsilon_0^{-2} t))-  g^{ij}(\bar{x}(\epsilon_0^{-2}t))) \bar{\xi}_i \bar{\xi}_j
\end{equation*}
and consider the gauge transform
\begin{equation*}
v(x,t) = e^{i \int \alpha(x,t) dt} u(x,t).
\end{equation*}
When analyzing the size and regularity properties of the transformed symbol, we emphasize that this is a local gauge transformation on the time scale $|\Delta t| \lesssim \epsilon_0 R^{\frac{1}{2}}$. 

\smallskip

With the original symbol given by
\begin{equation*}
\begin{split}
p' &= g^{ij}(\epsilon_0^{-1} x + \bar{x}(t \epsilon_0^{-2})) \xi_i \xi_j + 2 \epsilon_0^{-1} (g^{ij}(\epsilon_0^{-1} x + \bar{x}(t \epsilon_0^{-2})) - g^{ij}(\bar{x}(t\epsilon_0^{-2}))) \xi_i \bar{\xi}_j \\
&\quad + \epsilon_0^{-2} (g^{ij}(\epsilon_0^{-1} x + \bar{x}(t \epsilon_0^{-2})) - g^{ij}(\bar{x}(t\epsilon_0^{-2}))) \bar{\xi}_i \bar{\xi}_j,
\end{split}
\end{equation*}
after that gauge transformation we arrive at
\begin{equation*}
p^c = (p'(x,\xi) - \alpha(x)) + \int \nabla_x \alpha(x,s) ds \nabla_{\xi} p'(x,\xi) + \big( \big( \int \nabla_x \alpha \big)^2 + \big( \int \partial_x^2 \alpha \big) \big) \partial^2_{\xi} p'.
\end{equation*} 
The time-regularity of the conjugated symbol is acceptable and so is the $\xi$-regularity. Above we have already seen that
\begin{equation*}
|\partial_x^{\beta} \alpha| \lesssim 
\begin{cases}
(\epsilon_0^2 R)^{-|\beta|}, \quad 0 \leq |\beta| \leq 2, \\
(\epsilon_0^2 R)^{-2 - \frac{|\beta|-2}{2}}, \quad 2 \leq |\beta|.
\end{cases}
\end{equation*}
With the time integral over an interval of length $\epsilon_0 R^{\frac{1}{2}}$, the time-integrals of derivatives of $\alpha$ are bounded in modulus.
Consequently, $|\partial_x p^c| \lesssim (\epsilon_0^2 R)^{-1}$. However, when estimating $\partial_x^2 p^c$ the critical terms to consider are $ \big( \int \partial_x^3 \alpha \big) \mathcal{O}(1))$ and $\int \partial_x^4 \alpha$. The latter is acceptable since $|\partial_x^4 \alpha| \lesssim (\epsilon_0^2 R)^{-3}$. For the former we compute using \eqref{eq:OriginalSpaceTimeRegularity} 
\begin{equation*}
\int \big( \partial_x^3 \alpha \big) = \epsilon_0^{-5} \int \partial_x^3 g^{ij} dt \lesssim \epsilon_0^{-3} R^{- \frac{3}{2}},
\end{equation*}
which is acceptable on the time scale $|\Delta t| \lesssim (\epsilon_0^2 R)^{\frac{1}{2}}$. The higher derivatives are likewise acceptable. We carry out the local gauge transform
\begin{equation*}
\iint u_{x_1,\xi_1} u_{x_3,\xi_3} \overline{u}_{x_2,\xi_2} \overline{u}_{x_4,\xi_4} dx dt = \iint v_{x_1,\xi_1} v_{x_3,\xi_3} \overline{v}_{x_2,\xi_2} \overline{v}_{x_4,\xi_4} dx dt.
\end{equation*}
The modified Hamiltonian is given by $p^c$, and we note that by the assumption
$|p''| \gg (\epsilon_0^2 R)^{-\frac{1}{2}}$ and $|\int \nabla_x \alpha ds| \lesssim (\epsilon_0^2 R)^{-\frac{1}{2}}$ we have $p^c = p'' + \mathcal{O}((\epsilon_0^2 R)^{-\frac{1}{2}})$.
Consequently, expanding the phase function like in \eqref{eq:ExpPhase} and integrating by parts the joint phase for $v$, the claim follows.
\end{proof}

The estimate in the above proposition allows us to consider the energy difference function
\begin{equation*}
F^z_{\xi_1,\xi_2'}(\eta) =p''(z,\xi_1)+p''(z,\eta+\xi_{2}'-\xi_{1})-p''(z,\eta) - p''(z,\xi_{2}')
\end{equation*}
for the shifted energy shells
\begin{equation*}
\mathcal{E}^{q,k}_{\xi_1,\xi_2'} = \{ \eta_{T_1}(t_q) : T_1 \in \T_1, \quad |F^{z}_{\xi_1,\xi_2'}(\eta_{T_1}(t_q)) - k (\epsilon_0^2 R)^{-\frac{1}{2}}| \lesssim (\epsilon_0^2 R)^{-\frac{1}{2}} \}.
\end{equation*}
Note that $|k| \lesssim (\epsilon_0^2 R)^{\frac{1}{2}}$ as $|p''| \lesssim 1$.
We carry out the argument for all shifted energy shells separately. Since we have obtained that the contribution is summable in $k \in \Z$ up to a logarithm in $R$, this allows us to carry out the argument like in Section \ref{section:ProofBilinear}.
To be precise, we rewrite \eqref{eq:AuxEstimateEnergyI} as
\begin{equation*}
\begin{split}
    &\quad \sum_{\substack{q \in \mathfrak{q}[\mu_1,\mu_2] \\  q \subseteq S}} \sum_{\substack{T_i,T_j',\\ T_1 \not\sim S, T_1' \not\sim S}} \iint \chi_q^4 \phi_{T_1} \overline{\phi_{T_1'}} \phi_{T_2} \overline{\phi_{T_2'}} dz \\
    &\lesssim \sum_{\substack{q \in \mathfrak{q}[\mu_1,\mu_2] \\  q \subseteq S}} \sum_{|k| \lesssim (\epsilon_0^2 R)^{\frac{1}{2}}} \sum_{\substack{T_i,T_j',\\ T_1 \not\sim S, T_1' \not\sim S, \\ \eta_{T_1'}(t_q) \in \mathcal{E}^{q,k}_{\xi_1,\xi_2'}}} \iint \chi_q^4 \phi_{T_1} \overline{\phi_{T_1'}} \phi_{T_2} \overline{\phi_{T_2'}} dz
\end{split}
\end{equation*}
We split the contributions of $k=0$ and $0 < |k| \lesssim (\epsilon_0^2 R)^{\frac{1}{2}}$. The contribution of $k=0$ can be handled exactly like in Section \ref{section:ProofBilinear}. For the contribution of $|k| \neq 0$ we integrate by parts in time to find by Proposition \ref{prop:TimeRegularityJoints}
\begin{equation*}
   \big| \iint \chi_q^4 \phi_{T_1} \overline{\phi}_{T_1'} \phi_{T_2} \overline{\phi}_{T_2'} dz \big| \lesssim \frac{1}{|k|} \iint \chi_q^4 \big| \phi_{T_1} \overline{\phi}_{T_1'} \phi_{T_2} \overline{\phi}_{T_2'} \big| dz.
\end{equation*}
Then, carrying out the summation like in \eqref{eq:AuxEnergyEst}, we find
\begin{equation*}
\begin{split}
    \sum_{\substack{q \in \mathfrak{q}[\mu_1,\mu_2] \\  q \subseteq S}} \sum_{\substack{T_i,T_j',\\ T_1 \not\sim S, T_1' \not\sim S, \\ \eta_{T_1'}(t_q) \in \mathcal{E}^{q,k}_{\xi_1,\xi_2'}}} \iint \chi_q^4 \phi_{T_1} \overline{\phi_{T_1'}} \phi_{T_2} \overline{\phi_{T_2'}} dz &\lesssim \frac{(\epsilon_0^2 R)^{-\frac{d-1}{2} + c \delta}}{|k|} \sum_q | \T_1^{\not\sim S}(q,\lambda_1,\mu_1,\mu_2)| \\
    &\; \times |\T_2(q)| \, \big( \sup_{\xi_1,\xi_2'} \big| \T_1^{\not\sim S}(\mathcal{E}^{q,k}_{\xi_1,\xi_2'},\lambda_1,\mu_1,\mu_2) \big| \big).
\end{split}
\end{equation*}
With the crucial combinatorial estimate for shifted energy shells at hand,
\begin{equation}
\label{eq:ModCombEst}
|\T_1^{\not\sim S}(\mathcal{E}^{q,k}_{\xi_1,\xi_2'},\lambda_1,\mu_1,\mu_2)| \lesssim (\epsilon_0^2 R)^{c \delta} \nu^{-1} \frac{\# \T_2}{\lambda_1 \mu_2},
\end{equation}
the claim follows like in Section \ref{section:ProofBilinear}. Here we need to exclude multiple intersections between the cone $\Gamma$ and $T_2$, following along the proof of \eqref{eq:CombEst} in Section \ref{section:ProofBilinear}.

Multiple intersections can be ruled out following the argument in the proof of Lemma \ref{l:transverse} and noting that $\nabla_{\xi} p'$ and $\partial^2_{\xi} p'$ are small perturbations of the flat case $\bar{p}(\xi) = |\xi|^2$. Indeed,
\begin{equation*}
\partial_{\xi_j} p'(x,t,\xi)= 2g^{ij}(\epsilon_0^{-1} x + \bar{x}(t \epsilon_0^{-2})) \xi_i + 2 \epsilon_0^{-1} (g^{ij}(\epsilon_0^{-1} x + \bar{x}(t \epsilon_0^{-2})) - \delta^{ij}) \bar{\xi}_j = 2 \xi_j + \mathcal{O}(\epsilon_0).
\end{equation*}
Consequently, we have
\[
|\nabla_{\xi} p'(x,t,\xi_1) - \nabla_{\xi} p'(x,t,\xi_2)| \sim |\xi_1 - \xi_2| \sim 1, 
\]
as well as 
\[
\partial_{\xi}^2 p = 2 g^{ij}(\epsilon_0^{-1} x+  \bar{x}(t \epsilon_0^{-2})) = 2 \delta^{ij} + \mathcal{O}(\epsilon_0^2).
\]
This excludes multiple intersections and finishes the proof of Theorem \ref{thm:BilinearSelfinteractionsSEQ}. \hfill $\Box$

\subsubsection{Small angle self-interactions for $1$-homogeneous symbols}

In this section we consider a small angle self-interaction for flows associated with  $1$-homo\-geneous symbols. The analysis is similar in spirit, though the null direction requires an anisotropic rescaling. We turn to the details. We consider a symbol in the form
\begin{equation*}
p(x,\xi) = \big( g^{ij}(x,t) \xi_i \xi_j \big)^{\frac{1}{2}}
\end{equation*}
with $g^{ij}$ being uniformly elliptic and satisfying size and regularity as in \eqref{eq:OriginalSpaceTimeRegularity}. The evolution governed by $-ip^w$ is denoted by $S(t,s)$, and we use the notation $\xi = (\xi',\xi_d) \in \R^{d-1} \times \R$. Many of the arguments are the same as in the previous section, so we shall be brief.
\begin{theorem}[Small~angle~self-interactions~for~wave~equations]
\label{thm:BilinearSelfInteractionsWave}
Let $\varepsilon_0 = c_0 \nu$ for a fixed small $0<c_0<1$. Assume that 
\begin{equation*}
\text{supp}(\hat{u}_i) \subseteq \{ \xi \in \R^d : \xi_d \in (1/2,3/2), \; \big| \frac{\xi'}{\xi_d} - \xi_i^* \big| \leq \epsilon_0 \} 
\end{equation*}
and moreover $|\xi_1^* - \xi_2^*| \sim \nu \gg \epsilon_0 \gg R^{-\frac{1}{2}}$. Then it holds:
\begin{equation*}
\| S(t,0) u_1 S(t,0) u_2 \|_{L^{\frac{d+3}{d+1}}_{t,x}(B_{d+1}(0,R))} \lesssim_{\varepsilon} R^{\varepsilon} \nu^{-\frac{4}{d+3}} \| u_1 \|_{L^2} \| u_2 \|_{L^2}.
\end{equation*}
\end{theorem}

\textbf{(i) Almost orthogonality.}

Without loss of generality we suppose that, with $\xi' = (\xi_1,\ldots,\xi_{d-1})$, we have  $|\xi'| \lesssim \epsilon_0$, $\xi_d \sim 1$. 
We treat the case of tiny transversality directly:
\begin{lemma}
Let $\varepsilon > 0$ and $\delta = C_d \varepsilon$ for $C_d$ a dimensional constant. Under the assumptions of Theorem \ref{thm:BilinearSelfInteractionsWave}, for $\nu \leq R^{-\frac{1}{2}+\delta}$, it holds
\begin{equation*}
\| S(t,0) u_1 S(t,0) u_2 \|_{L^{\frac{d+3}{d+1}}_{t,x}(B_{d+1}(0,R))} \lesssim_\varepsilon R^{\varepsilon} \nu^{-\frac{4}{d+3}} \| u_1 \|_2 \| u_2 \|_2.
\end{equation*}
\end{lemma}
\begin{proof}
We decompose $u_i$ in into sectors of aperture $R^{-\frac{1}{2}}$ in frequencies:
\begin{equation*}
\| S(t,0) u_1 S(t,0) u_2 \|_{L^p} \leq \sum_{\theta_i: R^{-\frac{1}{2}}-\text{sector}} \| S(t,0) u_{1 \theta_1} \|_{L^{2p}_{t,x}} \| S(t,0) u_{2 \theta_2} \|_{L^{2p}_{t,x}}.
\end{equation*}
Then we can use the wave packet decomposition by Geba--Tataru \cite[Proposition~5.2]{GebaTataru2007} in which was proved that solutions to rough wave equations at frequencies $|\xi| \sim R$, which are localized in frequencies to $R^{\frac{1}{2}} \times \ldots \times R^{\frac{1}{2}} \times R$ (angular $\times$ radial frequencies) and spatially localized dually to $R^{-\frac{1}{2}} \times \ldots \times R^{-\frac{1}{2}} \times 1$ retain their phase space localization on times $\mathcal{O}(1)$.

Moreover, the evolved wave packets upon time evolution maintain their spatial spearation, which is a consequence of the bi-Lipschitz continuity of the Hamiltonian flow, with respect to Smith's pseudometric on the co-sphere bundle:
\begin{equation}
\label{eq:SmithPseudo}
d((x,\omega),(y,\eta)) \approx |\langle \omega,x-y \rangle | + \min(|x-y|,|x-y|^2) + |\omega - \eta|^2.
\end{equation}
As a consequence,
\begin{equation*}
\| S(t,0) f_{1 \theta} \|_{L^{2p}_x(B_d(0,R))} \lesssim \big( \sum_x \| S(t,0) f_{1 \theta,x} \|^{2p}_{L^{2p}_x(B_d(0,R))} \big)^{\frac{1}{2p}}.
\end{equation*}
Since the essential frequency support of $S(t,0) f_{1 \theta,x}$ is contained in a sector of aperture $R^{-\frac{1}{2}}$, applying Bernstein's inequality gives
\begin{equation*}
\| S(t,0) f_{1 \theta,x} \|_{L^{2p}_{t,x}(B_{d+1}(0,R))} \lesssim R^{\frac{1}{2p}} R^{-\frac{d-1}{2} \big( \frac{1}{2} - \frac{1}{2p} \big)} \| f_{1 \theta,x} \|_{L^2}
\end{equation*}
and by the embedding $\ell^{2p} \hookrightarrow \ell^2$ we find (choosing $\delta$ small enough)
\begin{equation*}
\| S(t,0) f_{1 \theta} \|_{L^{2p}_{t,x}(B_{d+1}(0,R))} \lesssim R^{\varepsilon/2} \nu^{-\frac{2}{d+3}} \| f_{1 \theta} \|_{L^2}.
\end{equation*}
Consequently, by summing over $R^{-\frac{1}{2}}$-sectors contained in the $\nu$-sector, we find
\begin{equation*}
\| S(t,0) f_1 S(t,0) f_2 \|_{L^p_{t,x}(B_{d+1}(0,R))} \leq C_\varepsilon R^{\varepsilon} \nu^{-\frac{4}{d+3}} \| f_1 \|_2 \| f_2 \|_2.
\qedhere
\end{equation*}
\end{proof}

In the following we suppose that $\nu \geq R^{-\frac{1}{2}+\delta}$, $\delta = \delta(\epsilon)>0$ chosen as in the isotropic case. First, we localize $x$ to $\epsilon_0 R$-balls. Here we use that we can always suppose by the $1$-homogeneity of the symbol that $\xi_{d}(0) = 1$. 
Consequently, $\xi$ can effectively be localized to a $\epsilon_0$-ball, which leads us to an almost orthogonal decomposition:
\begin{equation*}
\begin{split}
&\quad \| S(t,0) f_1 S(t,0) f_2 \|_{L^p_{t,x}(B_{d+1}(0,R))} \\
&\lesssim \big( \sum_{|x_1 - x_2| \leq C \epsilon_0 R} \| S(t,0) f_{1x_1} S(t,0) f_{2x_2} \|^p_{L^p_{t,x}(B_{d+1}(0,R))} \big)^{\frac{1}{p}} + \text{RapDec}(R) \| f_1 \|_2 \| f_2 \|_2.
\end{split}
\end{equation*}
Here the supports of the functions $f_{ix}$ are again localized to balls of size $\epsilon_0 R$. To accomplish an anisotropic decomposition, which will match the anisotropic rescaling, we analyze the flow of the angular frequencies. Let $\omega(t)$ denote the ratio $\xi'(t) / \xi_d(t)$. Here $\bar{\omega}(t) = \frac{\bar{\xi}'(t)}{\bar{\xi}_d(t)}$ denotes the ratio of the $(\xi',\xi_d)$-component of the central bicharacteristic. We have the following lemma:
\begin{lemma}
\label{lem:AnisotropicBiLipschitz}
The evolution equations for $(x^t, \omega^t, \xi_d^t)$ read:
\begin{equation}
\label{eq:ModifiedHamiltonian}
\left\{ \begin{array}{cl}
\dot{x}^t &= \partial_{\xi} p(x^t, \omega^t, 1), \\
\dot{\omega}^t &= - \partial_{x'} p(x^t,\omega^t,1) + \partial_{x_d} p(x^t,\omega^t,1) \omega^t, \\
\partial_t \log(\xi_d^t) &= - \partial_{x_d} p(x^t,\omega^t,1).
\end{array} \right.
\end{equation}
In particular $(x^t,\omega^t)$ are governed by an ODE:
\begin{equation*}
\frac{d}{dt} (x^t, \omega^t) = F(x^t,t,\omega^t)
\end{equation*}
with $F \in C_t C^{0,1}_{x,\omega}$, and we have the bi-Lipschitz dependence: For $(x^0,\omega^0)$, $(\bar{x}^0,\bar{\omega}^0)$ with 
\begin{equation*}
\{ |x^0 - \bar{x}^0| \leq \epsilon_0 R, \quad |\omega^0 - \bar{\omega}^0| \leq \epsilon_0 \},
\end{equation*}
it holds
\begin{equation*}
\{ |x^t - \bar{x}^t| \lesssim \epsilon_0 R, \quad |\omega^t - \bar{\omega}^t| \lesssim \epsilon_0 \}.
\end{equation*}
\end{lemma}
\begin{proof}
Since for $|\xi_d^0| \sim 1$ we still have by Hamiltonian equations $|\xi_d^t| \sim 1$ by $|\dot{\xi}_d^t| \leq c/N$, the expressions are well-defined, and we readily derive from the Hamiltonian equations:
\begin{equation*}
\left\{ \begin{array}{cl}
\dot{x}^t &= \partial_{\xi} p(x^t,\xi^t), \\
\dot{\xi}^t &= - \partial_x p(x^t,\xi^t)
\end{array} \right.
\end{equation*}
and $1$-homogeneity of the flow:
\begin{equation*}
\left\{ \begin{array}{cl}
\dot{x}^t &= \partial_{\xi} p(x^t,\omega^t,1), \\
\frac{d}{dt} \frac{(\xi')^t}{\xi_d^t} &= -\partial_{x'} p(x^t,\omega^t,1) - \frac{(\xi')^t}{(\xi_d^t)^2} (\dot{\xi}_d)^t \\
&= - \partial_{x'} p(x^t,\omega^t,1) + \omega^t \partial_{x_d} p(x^t,\omega^t,1), \\
\frac{d}{dt} \log(\xi_d^t) &= - \partial_{x_d} p(x^t,\omega^t,1).
 \end{array} \right.
\end{equation*}
The bi-Lipschitz property is proved by a bootstrap argument as in Section \ref{section:BilinearRoughEstimates}.
\end{proof}
Consequently, the central angular frequency in the isotropic expression
\begin{equation*}
\| S(t,0) f_{1x} S(t,0) f_{2x} \|_{L^p_{t,x}(B_{d+1}(0,R))}
\end{equation*}
is determined up to order $\epsilon_0$. We use the anisotropic wave packet decomposition provided by Geba--Tataru \cite{GebaTataru2007} to obtain an anisotropic refinement:
\begin{equation}
\label{eq:AnisotropicAlmostOrthogonality}
\begin{split}
&\quad \| S(t,0) f_{1x} S(t,0) f_{2x} \|_{L^p_{t,x}(B_{d+1}(0,R))} \\
 &\lesssim \big( \sum_{z_i: R \epsilon_0 \times R \epsilon_0^2 - \text{ellipsoid}} \| S(t,0) f_{1z_1} S(t,0) f_{2z_2} \|_{L^p_{t,x}(B_{d+1}(0,R))}^p \big)^{\frac{1}{p}} + \text{RapDec}(R) \| f_1 \|_2 \| f_2 \|_2,
 \end{split}
\end{equation}
where the sum is over localization to ellipsoids centered at $z_i$, which are contained in a joint ellipsoid of size $CR \epsilon_0 \times C R \epsilon_0^2$.
\begin{proof}[Proof~of~\eqref{eq:AnisotropicAlmostOrthogonality}]
The anisotropic wave packets take into consideration the null direction of the cone and allow for smaller spatial localization into the dual direction. We use the joint localization to ellipsoids of size $R^{\frac{1}{2}} \times \ldots \times R^{\frac{1}{2}} \times R$ in space and dually in frequency.
The bi-Lipschitz continuity of the pseudo-metric \eqref{eq:SmithPseudo} holds on the unit time interval, which makes it natural to apply the pseudo-metric to the rescaled quantities $x \to x /R$, $\xi \to R \xi$.

We additionally localize the (rescaled) $e_d$-direction into $ \epsilon_0^2$-intervals, and consider centers of anisotropic wave packets $(x,\omega)$ and $(y,\eta)$ from separated $\epsilon_0^2$-intervals. We see that the pseudo-metric for the rescaled variables satisfies
\begin{equation*}
d((x,\omega),(y,\eta)) \approx | \langle \omega, x- y\rangle| \gg \epsilon_0^2.
\end{equation*}
Consequently, for any $t \in [0,1]$, we have
\begin{equation*}
d((x^t,\omega^t),(y^t,\eta^t)) \gg \epsilon_0^2 \gg R^{-1+\delta}.
\end{equation*}
For this reason, the anisotropic wave packets from separated $\epsilon_0^2$-intervals are not essentially interacting.
Indeed, we compute the projection to $\bar{\omega}^t$ - letting $(\bar{x}^t,\bar{\omega}^t)$ denote the central bicharacteristic of the rescaled $\epsilon_0$-ball $\times$ $\epsilon_0$-sector in phase space. To this end, we note that $|\omega^t - \bar{\omega}^t| \lesssim \epsilon_0$ by the bi-Lipschitz property. Now, the packet $(x^t,\omega^t)$ has length $1$ in $\omega^t$ direction and $R^{-\frac{1}{2}}$ in the orthogonal directions. This shows that the projection of the wave packet to $\bar{\omega}^t$ is contained in an interval of length $R^{-\frac{1}{2}} \epsilon_0$. Similarly, for $(y^t,\eta^t)$ and since $|\langle x^t - y^t, \bar{\omega}^t \rangle| \gg \epsilon_0^2 \gg R^{-\frac{1}{2}+\delta} \epsilon_0$, we have established the separation of the wave packets. 
\end{proof}

We turn to estimating 
\begin{equation*}
\| S(t,0) f_{1z} S(t,0) f_{2z} \|_{L^p_{t,x}(B_{d+1}(0,R))}
\end{equation*}
with $f_{iz}$ localized to the same $N\epsilon_0 \times N \epsilon_0^2$-ellipsoid (up to mild dilation).

We cannot exactly center around a bicharacteristic like in the isotropic case, as the $\xi_d$-frequency is not localized. Instead we center the localized solutions around a bicharacteristic ray in phase space:
\begin{equation*}
\{ (\bar{x}(t), \xi_d \bar{\omega}(t), \xi_d) \, : \, \xi_d \in (1/2,3/2) \}.
\end{equation*}
As a consequence of Lemma \ref{lem:AnisotropicBiLipschitz}, the solutions are localized in frequencies $\{ \xi \in \R^d: \xi_d \sim 1, \; \big| \frac{\xi'}{\xi_d} - \bar{\omega}(t) | \leq \epsilon_0 \}$. For this reason, we center $v_i$ around the ray in frequencies by setting
\begin{equation*}
\hat{v}_i(t,\xi) = \hat{u}_i(t,\xi'+\xi_d \bar{\omega}(t) , \xi_d),
\end{equation*}
or, equivalently, $v_i(x,t) = u_i(x',x_d-x'\cdot \bar{\omega}(t),t)$.
We have the following:
\begin{lemma}
$v_i$ satisfy the equation
\begin{equation*}
-i \partial_t v_i + \tilde{p}^w(x,D) v_i = 0
\end{equation*}
with 
\begin{equation*}
\tilde{p}(x,\xi) = p(x',x_d-x' \cdot \bar{\omega}(t), t,\xi'+\xi_d \bar{\omega}(t),\xi_d) - x' \xi_d \dot{\bar{\omega}}(t).
\end{equation*}
\end{lemma}

Next, we center around the spatial part of the bicharacteristic by considering $w_i(x,t) = v_i(x+\bar{x}(t),t)$. We have the following:
\begin{lemma}
The functions $w_i$ solve the equation $-i \partial_t w_i + \tilde{p}^w(x,D) w_i = 0$ governed by the symbol
\begin{equation*}
\begin{split}
\tilde{a}(x,\xi) &= \xi_d (p(x'+\bar{x}'(t),x_d+\bar{x}_d(t),\frac{\xi'}{\xi_d}+\bar{\omega}(t),1) - (x'+\bar{x}'(t)) \nabla_{x'} p(\bar{x}(t),\bar{\omega}(t),1)) \\
&\, - \xi' \nabla_{\xi'} p(\bar{x}(t),\bar{\omega}(t),1) - \xi_d (p(\bar{x}(t),\bar{\omega}(t),1) - \bar{x}'(t) \nabla_{x'} p(\bar{x}(t),\bar{\omega}(t),1)] \\
&= \xi_d [ p(x'+\bar{x}'(t),x_d+\bar{x}_d(t),\frac{\xi'}{\xi_d}+\bar{\omega}(t),1) - x' \nabla_{x'} p(\bar{x}(t),\bar{\omega}(t),1) \\
&\quad - \frac{\xi'}{\xi_d} \nabla_{\xi'} p(\bar{x}(t),\bar{\omega}(t),1) - p(\bar{x}(t),\bar{\omega}(t),1))].
\end{split}
\end{equation*}
Moreover, $w_i$ is essentially localized to $|x'| \lesssim \epsilon_0 R$, $|x_d| \lesssim \epsilon_0^2 R$.
\end{lemma}
Consequently, we have (up to rapidly decaying error terms)
\begin{equation*}
\| S(t,0) f_{1z } S(t,0) f_{2z} \|_{L^p_{t,x}(B_{d+1}(0,R))} \lesssim \| w_1 \cdot w_2 \|_{L^p_{t,x}([-R,R] \times B_d(0,5 \epsilon_0 R))}.
\end{equation*}

\textbf{(ii) Rescaling.} We carry out the anisotropic rescaling, which is given by
\begin{equation*}
\xi' \to \epsilon_0 \xi', \xi_d \to \xi_d, \quad x' \to \epsilon_0^{-1} x', \; x_d \to x_d, \; t \to \epsilon_0^{-2} t.
\end{equation*}
We abuse notation and let $u_i(x,t) = w_i(\epsilon_0^{-1} x',x_d,\epsilon_0^{-2} t)$. We find the transversality parameter to be normalized upon rescaling: $\nu \sim 1$. We summarize the above in the following:
\begin{proposition}
The estimate 
\begin{equation*}
\| w_1 \cdot w_2 \|_{L^p_{t,x}([-R,R] \times B_d(0,5 \epsilon_0 R)} \lesssim_\varepsilon R^\varepsilon \nu^{-\frac{4}{d+3}} \prod_{i=1}^2 \| f_{i z} \|_{L^2},
\end{equation*}
is implied by
\begin{equation}
\label{eq:RescaledUnitTransversality}
\| u_1 \cdot u_2 \|_{L^p_{t,x}(B_{d+1}(0,\epsilon_0^2 R))} \lesssim_\varepsilon (\epsilon_0^2 R)^{\varepsilon} \prod_{i=1}^2 \| u_i(0) \|_2.
\end{equation}
\end{proposition}
The governing symbol of the rescaled evolution reads
\begin{equation*}
\begin{split}
p'(x,t,\xi) &= \epsilon_0^{-2} \xi_d (p(\epsilon_0^{-1} x' + \bar{x}'(\epsilon_0^{-2} t),x_d + \bar{x}_d(\epsilon_0^{-2} t), \epsilon_0^{-2} t, \frac{\epsilon_0 \xi'}{\xi_d} + \bar{\omega}(\epsilon_0^{-2} t),1) \\
&\; - \epsilon_0^{-1} x' \nabla_{x'} p(\bar{x}(\epsilon_0^{-2} t),\epsilon_0^{-2} t, \bar{\omega}(\epsilon_0^{-2} t),1)  - \frac{\epsilon_0 \xi'}{\xi_d} \nabla_{\xi'} p(\bar{x}(\epsilon_0^{-2} t),\epsilon_0^{-2} t,\bar{\omega}(\epsilon_0^{-2} t),1) \\
&\; - p(\bar{x}(\epsilon_0^{-2} t),\epsilon_0^{-2} t, \bar{\omega}(\epsilon_0^{-2} t),1)).
\end{split}
\end{equation*}
We carry out a Taylor expansion in $\xi'/\xi_d$ of the first term to obtain an expansion reminiscent of the isotropic case:
\begin{equation*}
\begin{split}
&\quad p(\epsilon_0^{-1} x'+\bar{x}'(\epsilon_0^{-2} t),x_d+\bar{x}_d(\epsilon_0^{-2} t), \epsilon_0^{-2} t, \frac{\epsilon_0 \xi'}{\xi_d} + \bar{\omega}(\epsilon_0^{-2} t), 1) \\
&= p(\epsilon_0^{-1} x'+\bar{x}'(\epsilon_0^{-2} t),x_d+\bar{x}_d(\epsilon_0^{-2} t), \epsilon_0^{-2} t, \bar{\omega}(\epsilon_0^{-2} t), 1) \\
&\quad + \frac{\epsilon_0 \xi'}{\xi_d} \nabla_{\xi'} p(\epsilon_0^{-1} x'+\bar{x}'(\epsilon_0^{-2} t),x_d+\bar{x}_d(\epsilon_0^{-2} t), \epsilon_0^{-2} t, \bar{\omega}(\epsilon_0^{-2} t), 1) \\
&\quad + \frac{\epsilon_0^2}{2 \xi_d^2} \langle \xi', \partial^2_{\xi' \xi'} p(\epsilon_0^{-1} x'+\bar{x}'(\epsilon_0^{-2} t),x_d+\bar{x}_d(\epsilon_0^{-2} t), \epsilon_0^{-2} t, \tilde{\xi}', 1) \xi' \rangle
\end{split}
\end{equation*}
and sort the expressions into powers of $\xi'/\xi_d$.
\begin{lemma}
We have the following representation of $p'$:
\begin{equation*}
p'(x,t,\xi) = I + II + III 
\end{equation*}
with
\begin{equation*}
\begin{split}
I &= \frac{1}{2 \xi_d^2} \langle \xi', \partial^2_{\xi' \xi'} p(\epsilon_0^{-1} x'+\bar{x}'(\epsilon_0^{-2} t),x_d+\bar{x}_d(\epsilon_0^{-2} t), \epsilon_0^{-2} t, \tilde{\xi}', 1) \xi' \rangle, \\
II &= \epsilon_0^{-1} \xi' (\nabla_{\xi'} p(\epsilon_0^{-1} x'+\bar{x}'(\epsilon_0^{-2} t),x_d+\bar{x}_d(\epsilon_0^{-2} t),\epsilon_0^{-2} t, \bar{\omega}(\epsilon_0^{-2} t),1) \\
&\quad - \nabla_{\xi'} p(\bar{x}(\epsilon_0^{-2} t),\epsilon_0^{-2} t,\bar{\omega}(\epsilon_0^{-2} t),1)),
\end{split}
\end{equation*}
and
\begin{equation*}
\begin{split}
III &= \epsilon_0^{-2} \xi_d (p(\epsilon_0^{-1} x' + \bar{x}'(\epsilon_0^{-2} t), x_d + \bar{x}_d(\epsilon_0^{-2} t), \epsilon_0^{-2} t, \bar{\omega}(\epsilon_0^{-2} t),1) \\
&\quad - \epsilon_0^{-1} x' \nabla_{x'} p(\bar{x}'(\epsilon_0^{-2} t),\bar{x}_d(\epsilon_0^{-2} t),\epsilon_0^{-2} t,1) - p(\bar{x}(\epsilon_0^{-2} t),\epsilon_0^{-2} t, \bar{\omega}(\epsilon_0^{-2} t),1)) \\
&= \xi_d (V(x',t) + W(x,t)).
\end{split}
\end{equation*}
$p'$ satisfies the size and regularity estimates:
\begin{equation*}
|\partial_x^{\alpha} \partial_{\xi}^{\beta} p'| \lesssim_{\alpha,\beta} (\epsilon_0^2 R)^{-|\alpha| + \frac{(|\alpha|-2)_+}{2}}.
\end{equation*}
\end{lemma}
\begin{proof}
The size and regularity estimates for $I$ are immediate. For $II$ we use the mean-value theorem to write
\begin{equation*}
II = \langle \epsilon_0^{-1} \xi', \partial^2_{x \xi'} p(\tilde{x}',\tilde{x}_d,\epsilon_0^{-2} t, \bar{\omega}(\epsilon_0^{-2} t),1) \cdot (\epsilon_0^{-1} x',x_d)\rangle,
\end{equation*}
from which $|II| \lesssim 1$ is immediate by $|(x',x_d)| \lesssim \epsilon_0^2 R$ and size and regularity of $a$. Bounds for derivatives are easier to verify.

For $III$ we carry out a Taylor expansion in $x_d$:
\begin{equation*}
\begin{split}
&\quad \epsilon_0^{-2} (p(\epsilon_0^{-1} x'+\bar{x}'(\epsilon_0^{-2} t),x_d + \bar{x}_d(\epsilon_0^{-2} t), \epsilon_0^{-2} t,\bar{\omega}(\epsilon_0^{-2} t),1)) \\
&= \epsilon_0^{-2} (p(\epsilon_0^{-1} x' + \bar{x}'(\epsilon_0^{-2} t),\bar{x}_d(\epsilon_0^{-2} t),\epsilon_0^{-2} t, \bar{\omega}(\epsilon_0^{-2} t),1)) \\
&\quad + x_d \partial_{x_d} p(\epsilon_0^{-1} x' + \bar{x}'(\epsilon_0^{-2} t),\tilde{x}_d,\epsilon_0^{-2} t, \bar{\omega}(\epsilon_0^{-2} t),1).
\end{split}
\end{equation*}
The first term, which is purely $x'$-dependent is ameliorated by the second and third defining terms of $III$:
\begin{equation*}
\begin{split}
\xi_d V(x',t) &= \xi_d \epsilon_0^{-2} ( p(\epsilon_0^{-1} x' + \bar{x}'(\epsilon_0^{-2} t),\bar{x}_d(\epsilon_0^{-2} t),\epsilon_0^{-2} t, \bar{\omega}(\epsilon_0^{-2} t),1)) \\
&\quad - \epsilon_0^{-1} x' \nabla_{x'} p(\bar{x}'(\epsilon_0^{-2} t),\bar{x}_d(\epsilon_0^{-2} t),\epsilon_0^{-2} t,1) - p(\bar{x}(\epsilon_0^{-2} t),\epsilon_0^{-2} t, \bar{\omega}(\epsilon_0^{-2} t),1)).
\end{split}
\end{equation*}
It is readily verified by repeated applications of the mean-value theorem that
\begin{equation*}
|\partial_{x'}^{\alpha} V(x',t)| \lesssim_\alpha (\epsilon_0^2 R)^{-|\alpha| + \frac{(|\alpha|-2)_+}{2}}.
\end{equation*}
For the term
\begin{equation*}
W(x,t) = \epsilon_0^{-2} x_d \partial_{x_d} p(\epsilon_0^{-1} x' + \bar{x}'(\epsilon_0^{-2} t), \tilde{x}_d, \epsilon_0^{-2} t, \bar{\omega}(\epsilon_0^{-2} t),1))
\end{equation*}
the size and regularity estimates are easier to verify.
\end{proof}

\paragraph{\textbf{(iii) A bilinear estimate for the rescaled evolution}}

Theorem \ref{thm:GeneralizationBilinearCone} is again not applicable due to the insufficient time regularity of the symbol $p'$.

If we had the time-regularity
\begin{equation*}
|\partial_t p'(x,t,\xi)| \lesssim (\epsilon_0^2 R)^{-1},
\end{equation*}
we could proceed like in the isotropic case.
This would allow us to analyze joints (of isotropic wave packets), which intersect with transversality $\nu \sim 1$ at $t=t_0$ on times $\Delta t \sim (\epsilon_0^2 R)^{\frac{1}{2}}$:
\begin{equation*}
\iint u_{x_1,\xi_1} u_{x_3,\xi_3} \overline{u}_{x_2,\xi_2} \overline{u}_{x_4,\xi_4} dx dt.
\end{equation*}
The problematic term is localized in the expansion
\begin{equation*}
p'(x,t,\xi) = \xi_d V(x',t) + p''(x,t,\xi)
\end{equation*}
to be $\xi_d V(x',t)$. Indeed, the bounds
\begin{equation*}
|\partial_t W| \lesssim (\epsilon_0^2 R)^{-1}, \quad |\partial_t I| + |\partial_t II| \lesssim (\epsilon_0^2 R)^{-1}
\end{equation*}
are straight-forward.

We consider as above the localized quadrilinear expression
\begin{equation*}
 \iint \chi_\mathfrak{q}(y,t) u_{x_1,\xi_1}(y,t) u_{x_3,\xi_3}(y,t) \bar{u}_{x_2,\xi_2}(y,t) \bar{u}_{x_4,\xi_4}(y,t) dy dt 
\end{equation*}
on cubes of size $(\epsilon_0^2 R)^{\frac{1}{2}}$ - the size being determined by the unit transversality. We decompose $p'$ into singular structured and regular part:
\begin{equation*}
p'(x,t,\xi) = \xi_d V(x',t) + p''(x,t,\xi)
\end{equation*}
and consider the renormalized energy difference function:
\begin{equation*}
F^z_{\xi_1,\xi_2'}(\eta) = p''(z,\xi_1) + p''(z,\eta + \xi_2'-\xi_1) - p''(z,\eta) - p''(z,\xi_1').
\end{equation*}
The contribution of $|F^z_{\xi_1,\xi_2'}(\eta)| \lesssim (\epsilon_0^2 R)^{-\frac{1}{2}}$ can be handled exactly like in Section \ref{section:ProofBilinear}. To handle the contribution of large renormalized energy difference, we show the following analog of Proposition \ref{prop:TimeRegularityJoints}:
\begin{proposition}
Suppose that $\epsilon_0 \gg R^{-\frac{1}{2}}$. With the notations from the $1$-homogeneous case introduced above, let
\begin{equation*}
\bar{p}^s = p''(x_1^{t_0},t_0,\xi_1^{t_0}) + p''(x_3^{t_0},t_0,\xi_3^{t_0}) - p''(x_2^{t_0},t_0,\xi_2^{t_0}) - p''(x_4^{t_0},t_0,\xi_4^{t_0}).
\end{equation*}
Then it holds for $|\bar{p}^s| \gg (\epsilon_0^2 R)^{-\frac{1}{2}}$:
\begin{equation*}
\begin{split}
&\quad \big| \iint \chi_{\mathfrak{q}}(y,t) u_{x_1,\xi_1}(y,t) u_{x_3,\xi_3}(y,t) \bar{u}_{x_2,\xi_2}(y,t) \bar{u}_{x_4,\xi_4}(y,t) dy dt \big| \\
 &\lesssim \frac{1}{(\epsilon_0^2 R)^{\frac{1}{2}} |\bar{p}^s|} \iint \tilde{\chi}_{\mathfrak{q}}(y,t) \phi^*_{x_1,\xi_1}(y,t) \phi^*_{x_3,\xi_3}(y,t) \phi^*_{x_2,\xi_2}(y,t) \phi^*_{x_4,\xi_4}(y,t) dy dt.
 \end{split}
\end{equation*}
\end{proposition}
\begin{proof}
Many steps are like in the proof of Proposition \ref{prop:TimeRegularityJoints}: plugging in the expansion of $u_{x_i,\xi_i}$ in terms of Green's functions $G(t,x,y,\xi)$ and oscillatory factors, which can be integrated by parts. We recall that in the proof of Proposition \ref{prop:TimeRegularityJoints} the regular part of the Hamiltonian was recovered after a local gauge transform. Presently, we carry out the change of variables $x_d \to x_d + \int_0^t V(x',s) ds$, which clearly incurs no Jacobian factor. Consider a function $-i \partial_t u + (p')^w(x,D) u = 0$. We ameliorate the singular part by considering $w(x',x_d,t) = u(x',x_d-\int_0^t V(x',s) ds, t)$. We compute
\begin{equation*}
\begin{split}
&\quad \nabla_{x'} w(x',x_d,t)  \\
&= (\nabla_{x'} u)(x',x_d + \int V(x',s) ds, t) + \big( \int \nabla_{x'} V \big) \partial_{x_d} u(x',x_d+\int V(x',s) ds, t) \\
&= (\nabla_{x'} w)(x',x_d , t) + \big( \int \nabla_{x'} V \big) \partial_{x_d} w(x',x_d, t) 
\end{split}
\end{equation*}
which gives the transformation rule:
\begin{equation*}
\begin{split}
-i \partial_t w + \tilde{a}''(x,D)^w w &= 0, \\ 
\tilde{p}'' &= p''(x',x_d-\int_0^t V(x',s),\xi'+ \xi_d \int_0^t \nabla_{x'} V(x',s) ds, \xi_d) +l.o.t.
\end{split}
\end{equation*}
up to lower order terms. Indeed, the full transformation rule is given in \cite[Theorem~9.3]{ZworskiSemiclassical} and by stationary phase, the first order correction is given by
\begin{equation*}
\big(\int \partial^2_{y'} V \big) \partial_{y_d} \partial^2_{\eta'} p'' + \big( \int \partial^3_{y'} V \big) \partial^3_{\eta'} p'',
\end{equation*}
which is clearly of lower order and does not essentially change the Hamiltonian flow.

We record by a simple Taylor expansion:
\begin{equation*}
\tilde{p}''(x,t,\xi) = p''(x,t,\xi) + \mathcal{O}((\epsilon_0^2 R)^{-\frac{1}{2}})
\end{equation*}
and the size and regularity estimates
\begin{equation*}
|\partial_x^{\alpha} \partial_{\xi}^{\beta} \tilde{p}''| \lesssim_{\alpha,\beta} (\epsilon_0^2 R)^{-|\alpha| + \frac{(|\alpha|-2)_+}{2}}
\end{equation*}
by similar considerations like in Proposition \ref{prop:TimeRegularityJoints}, in particular, by the locality of the transformation $\Delta t \lesssim (\epsilon_0^2 R)^{\frac{1}{2}}$.

The time regularity $|\partial_t \tilde{p}''| \lesssim (\epsilon_0^2 R)^{-1}$ is straightforward. Following along the lines of Proposition \ref{prop:TimeRegularityJoints}, the proposition follows.
\end{proof}
At this point we can argue like in the preceding section that we can consider localizations of the energy shell to regions around $k (\epsilon_0^2 R)^{-\frac{1}{2}}$, $k \in \Z$, over which we can sum. Multiple intersections are ruled out observing that
\begin{equation*}
\begin{split}
\nabla_{\xi'} p'(x,t,\xi_1) - \nabla_{\xi'} p'(x,t,\xi_2) &= \nabla_{\xi'} I(x,t,\xi_1) - \nabla_{\xi'} I(x,t,\xi_2) + \mathcal{O}(\epsilon_0) \\
&= |\xi_1'-\xi_2'| + \mathcal{O}(\epsilon_0) = \mathcal{O}(1),
\end{split}
\end{equation*}
and by the argument of Lemma \ref{l:transverse}. The proof is complete. \hfill $\Box$

\appendix

\section{Dispersive properties via wave packet decompositions}

In the Appendix we obtain dispersive properties via wave packet decompositions. To prove the sharp decay for Schr\"odinger equations with uniformly non-degenerate metric, we use the non-degeneracy of the Hamiltonian flow in all directions.
We show the following, which quantifies the absence of caustics:
\begin{proposition}
\label{prop:NonDegeneracyFlowSEQ}
Let $p(x,t,\xi) = g^{ij}(x,t) \xi_i \xi_j$ with the metric satisfying the regularity assumption
\begin{equation}
\label{eq:RegularityMetricSEQ}
|\partial_z^{\alpha} g^{ij}(x,t)| \leq 
\epsilon R^{-|\alpha|}, \quad 1 \leq |\alpha| \leq 2, 
\end{equation}
and uniform non-degeneracy
\begin{equation}
\label{eq:NonDegeneracySEQ} 
\exists c, C > 0: c |\xi| \leq |g^{ij} \xi_i| \leq C |\xi|
\end{equation}
with $\epsilon \leq c' c$, for a $c' < 1$ chosen depending only on $C$. Then it holds for two bicharacteristics with $|\xi_i| \leq \frac{1}{2}$ with $x_0^1 = x_0^2 = x_0$ and $\xi_0^i = \xi_i$ for $0 \leq |t| \leq R$ with $\Delta \xi = \xi_1 - \xi_2$:
\begin{equation*}
|x_t^1 - x_t^2| \sim |t| |\Delta \xi|, \quad |\xi_t^1 - \xi_t^2| \sim |\Delta \xi|.
\end{equation*}
\end{proposition}
\begin{proof}
In the first step we show upper bounds
\begin{equation*}
|x_t^1 - x_t^2| \lesssim |t| |\Delta \xi|, \quad |\xi_t^1 - \xi_t^2| \lesssim |\Delta \xi|.
\end{equation*}
This will be proved by bootstrap: $|\xi_t^1 - \xi_t^2| \lesssim |\Delta \xi|$ is clearly true by
\begin{equation*}
\xi_t^1 - \xi_t^2 = \Delta \xi - \int_0^t \frac{\partial p}{\partial x}(x_s^1,\xi_s^1) - \frac{\partial p}{\partial x}(x_s^2,\xi_s^2) ds
\end{equation*}
for $t_0 \lesssim |\Delta \xi|$.
Next, we expand
\begin{equation}
\label{eq:xtExpand}
\begin{split}
x_t^1 - x_t^2 &= \int_0^t \frac{\partial p}{\partial \xi}(x_s^1,\xi_s^1) - \frac{\partial p}{\partial \xi}(x_s^2,\xi_s^2) ds \\
&= \int_0^t \frac{\partial^2 p}{\partial x \partial \xi}(x_s^1-x_s^2) + \frac{\partial^2 p }{(\partial \xi)^2} (\xi_s^1 - \xi_s^2) ds.
\end{split}
\end{equation}
Clearly, by the first identity $|x_t^1 - x_t^2| \leq C |t|$ and by the second,
\begin{equation*}
|x_t^1 - x_t^2| \leq \frac{C_1 \epsilon}{R} |t|^2 + C_2 |t| |\Delta \xi|.
\end{equation*}
Choosing $|t|$ small enough we have for $0 \leq |t| \leq t_0$ the estimates $|x_t^1 - x_t^2| \leq C |t| |\Delta \xi|$, $|\xi_t^1 - \xi_t^2| \leq C |\Delta \xi|$. 

We proceed with the bootstrap. Plugging these estimates into the second identity we find
\begin{equation*}
|x_t^1 - x_t^2| \leq \frac{C \epsilon |t|^2}{2R} |\Delta \xi| + C |t| |\Delta \xi|,
\end{equation*}
and for the frequencies
\begin{equation*}
\begin{split}
|\xi_t^1 - \xi_t^2| &\leq |\Delta \xi| + \int_0^t \big| \frac{\partial^2 p}{(\partial x)^2}(x_s^1 - x_s^2) \big| + \big| \frac{\partial^2 p}{\partial x \partial \xi} (\xi_s^1 - \xi_s^2) \big| ds \\
&\leq |\Delta \xi| + \frac{C |t|^2 \epsilon |\Delta \xi|}{R^2} + \frac{C \epsilon |t|}{R} |\Delta \xi|.
\end{split}
\end{equation*}
This allows us to carry on the bootstrap up to $t_0 = R$. For $\epsilon \ll c$, depending on $C$, the bounds can be bootstrapped and then, for $\epsilon \ll 1$ the estimate shows that $|\xi_t^1 - \xi_t^2| \sim |\Delta \xi|$.

Next we estimate $|x_t^1 - x_t^2| \sim |t| |\Delta \xi|$. In \eqref{eq:xtExpand} we estimate the first term by
\begin{equation*}
\big| \int_0^t \frac{\partial^2 p}{\partial x \partial \xi} (x_s^1 - x_s^2) ds \big| \leq \frac{|t|^2 \epsilon |\Delta \xi|}{R}
\end{equation*}
and the second term by the  
\begin{equation*}
\big| \int_0^t \frac{\partial^2 p}{(\partial \xi)^2}(\xi_s^1- \xi_s^2) \big| = |t| \big| \frac{\partial^2 p}{(\partial \xi)^2} (\xi_{t_*}^1 - \xi_{t_*}^2) \big| \geq c |t| |\Delta \xi|.
\end{equation*}
Consequently, for $\epsilon$ small compared to $c$ from the uniform non-degeneracy assumption, we find $|\Delta x_t| \sim |t| |\Delta \xi|$.
\end{proof}

Let $\mathcal{Y} = B_d(0,R) \times B_d(0,1)$ and
 $\chi_{\mathcal{Y}}(x,\xi)$ be a symbol which is identical one on $\mathcal{Y}$, which is vanishing off $(2 \mathcal{Y})^c$, and satisfies the bounds
 \begin{equation*}
 | \partial_x^{\alpha} \partial_{\xi}^{\beta} \chi_{\mathcal{Y}} | \lesssim_{\alpha,\beta} R^{-|\alpha|}.
 \end{equation*}

We are ready for the proof of dispersive estimates for non-degenerate Schr\"odinger equations:
\begin{proposition}
\label{prop:DispersiveEstimateSEQ}
Let $p(x,t,\xi) = g^{ij}(x,t) \xi_i \xi_j \chi(\xi)$ with uniformly non-degenerate metric satisfying the estimates
\begin{equation*}
|\partial_z^{\alpha} g^{ij}(x,t)| \leq 
\begin{cases}
\epsilon R^{-|\alpha|}, &\quad 1 \leq |\alpha| \leq 2, \\
C_\alpha R^{-2 - \frac{(|\alpha|-2)_+}{2}}, &\quad |\alpha| \geq 3.
\end{cases}
\end{equation*}
Let $u$ be a solution to
\begin{equation*}
\left\{ \begin{array}{cl}
D_t u + p^w(x,t;D) u &= 0, \quad (t,x) \in \R \times \R^d, \\
u(0) &= \chi_{\mathcal{Y}} f.
\end{array} \right.
\end{equation*}
Then we have the dispersive estimate: 
\begin{equation*}
\| u(t) \|_{L^\infty(\R^d)} \lesssim (1+|t|)^{-\frac{d}{2}} \| f \|_{L^1(\R^d)}.
\end{equation*}
\end{proposition}
\begin{remark}
We recover the above regularity of a flow as follows: Consider $g^{ij} \in C^{0,1}_t C^{1,1}_x$ with solutions localized at frequencies $N$, the coefficients being paradifferentially truncated at $\leq N$ in time and $\leq N^{\frac{1}{2}}$ in spatial frequencies on unit time-scales. On times $T=T(N)=N^{-1}$ we obtain \eqref{eq:RegularityMetricSEQ} after parabolic rescaling $t \to N^2 t$ and $x \to N x$ with $R=N$.
\end{remark}

\begin{proof}
We plug in the wave packet decomposition at scale $|t| \sim r \gg 1$:
\begin{equation*}
u(t) = \sum_{T \in \Lambda_r \cap \overline{\mathcal{Y}}} \phi_T + \text{RapDec}(r) \| f \|_{L^2}.
\end{equation*}
To ease notation, we let again $R=r$. Since we assumed $f$ to be concentrated on the $R$-ball in space and on unit frequencies, we further have $\| f \|_{L^2} \lesssim \| f \|_{L^1}$. It remains to understand $\big\| \sum_{T \in \Lambda_R \cap \overline{\mathcal{Y}}} \phi_T \big\|_{L^\infty(Q_R)}$. To this end, we use the localization properties to $R^{\frac{1}{2}}$-squares; it suffices to consider an $R^{\frac{1}{2}}$-square:
\begin{equation*}
\begin{split}
&\quad \big\| \sum_{T \in \Lambda_R \cap \overline{\mathcal{Y}}} \phi_T \big\|_{L^\infty(Q_{R^{\frac{1}{2}}})} \\
&\lesssim \sum_{v \in R^{\frac{1}{2}} \Z^d \cap Q_R} \sum_{\theta \in R^{-\frac{1}{2}} \Z^d \cap B_d(0,1)} R^{-\frac{d}{2}} \int_{\R^d} e^{- \frac{\rho (y-v)^2}{2}} |f| dy \\
&\quad \quad \times (1+ R^{-\frac{1}{2}}|x_q - x_{\theta,v}(R)|)^{-(d+1)} \\
&\lesssim R^{-\frac{d}{2}} \sum_{v \in R^{\frac{1}{2}} \Z^d} \| \chi_v \chi_{\mathcal{Y}} f \|_{L^1} \sum_{\theta \in R^{-\frac{1}{2}} \Z^d} (1 + R^{-\frac{1}{2}} |x_q - x_{\theta,v}(R)|)^{-(d+1)} \\
&\lesssim R^{-\frac{d}{2}} \| f \|_{L^1}.
\end{split}
\end{equation*}
In the final line we used that for fixed $v \in R^{\frac{1}{2}} \Z^d$ we have $|x_{(\theta,v)}(R) - x_{(\theta',v)}(R)| \sim R|\theta - \theta'|$ by Proposition \ref{prop:NonDegeneracyFlowSEQ}.
\end{proof}
To obtain an analogous result for wave equations we show the separation property of the Hamiltonian flow in the angular direction.
\begin{proposition}
\label{prop:NonDegeneracyFlowWave}
Let $p(x,t,\xi) = (g^{ij}(x,t,\xi) \xi_i \xi_j)^{\frac{1}{2}} \chi_1(\xi)$ with uniformly elliptic metric and the coefficients satisfying
\begin{equation*}
|\partial_z^{\alpha} g^{ij}(x,t)| \leq 
\epsilon R^{-|\alpha|}, \quad 1 \leq |\alpha| \leq 2.
\end{equation*}
$\epsilon \leq \kappa(c,C)$ is chosen small depending on the ellipticity constants. Then it holds for bicharacteristics with $|\xi_i| \sim 1$ and $x_0 = x_0^1 = x_0^2$ and $\xi_0^i = \xi_i$ for $0 \leq |t| \leq R$:
\begin{equation*}
|x_t^1 - x_t^2| \sim |t| |\angle (\xi_1,\xi_2)|, \quad |\angle(\xi_t^1,\xi_t^2)| \sim |\angle(\xi_1,\xi_2)|.
\end{equation*} 
\end{proposition}

With this property of the flow at hand, it is not difficult to show the dispersive property. Let $\mathcal{Y} = B_d(0,R) \times A_d$ and $\chi_{\mathcal{Y}} \in C^\infty(\R^d \times \R^d)$ which is identical to one on $\mathcal{Y}$ and $\chi_{\mathcal{Y}}(x,\xi) = 0$ for $x \not\in B_d(0,2R)$ or $|\xi| \not\sim 1$. 

\begin{proposition}
\label{prop:DispersiveEstimateWave}
Let $p$ be as in Proposition \ref{prop:NonDegeneracyFlowWave} with coefficients
\begin{equation*}
|\partial_z^{\alpha} g^{ij}(x,t)| \leq 
\begin{cases}
\epsilon R^{-|\alpha|}, &\quad 1 \leq |\alpha| \leq 2, \\
C_\alpha R^{-2 - \frac{(|\alpha|-2)_+}{2}}, &\quad |\alpha| \geq 3,
\end{cases}
\end{equation*}
and let $u$ be a solution to
\begin{equation*}
\left\{ \begin{array}{cl}
i \partial_t u + p^w(x,t;D) u &= 0, \quad (t,x) \in \R \times \R^d, \\
u(0) &= \chi_{\mathcal{Y}} f.
\end{array} \right.
\end{equation*}
Then the following dispersive property holds:
\begin{equation*}
\| u(t) \|_{L^\infty(\R^d)} \lesssim (1+|t|)^{-\frac{d-1}{2}} \| f \|_{L^1(\R^d)}.
\end{equation*}
\end{proposition}
\begin{remark}
This corresponds to metric coefficients $g^{ij} \in C^{1,1}_{t,x}$ on a unit time scale and frequency-localized solutions at frequency $N$ after paradifferential decomposition at frequencies $\leq N^{\frac{1}{2}}$ and rescaling $(x,t) \to N(x,t)$.
\end{remark}

\begin{proof}
As in Proposition \ref{prop:DispersiveEstimateSEQ} we plug in the wave packet decomposition to find
\begin{equation*}
u(t) = \sum_{T \in \Lambda_R \cap \overline{\mathcal{Y}}} \phi_T + \text{RapDec}(R) \| f \|_{L^2}.
\end{equation*}
For the main term from the wave packet decomposition we have
\begin{equation*}
\begin{split}
&\quad \big\| \sum_{T \in \Lambda_R \cap \overline{\mathcal{Y}}} \phi_T \big\|_{L^\infty(Q_{R^{\frac{1}{2}}})} \\
 &\lesssim R^{-\frac{d}{2}} \sum_{v \in R^{\frac{1}{2}} \Z^d} \| f \phi_v \|_{L^1} \sum_{\theta \in R^{-\frac{1}{2}} \Z^d \cap \bar{A}_d} (1 + R^{-\frac{1}{2}} | x_q - x_{(\theta,v)}(R)|)^{-(d+1)} \\
&\lesssim R^{-\frac{d-1}{2}} \| f \|_{L^1}.
\end{split}
\end{equation*}
\end{proof}

\section*{Acknowledgements}

Robert Schippa gratefully acknowledges financial support by the Humboldt Foundation (Feo\-dor-Lynen Fellowship) and partial support by the NSF grant DMS-2054975.
Daniel Tataru was supported by the NSF grant DMS-2054975, by a Simons Investigator grant from the Simons Foundation, as well as by a Simons Fellowship. 

\medskip

\emph{Statement on AI use.} At a later stage of preparing the manuscript, ChatGPT-6.1 Sol was used for proofreading and generating tikz code. The mathematical content is entirely due to the authors who take full responsibility for the contents of the paper.

\end{document}